\documentclass[11pt,reqno,tbtags]{amsart}
\usepackage{amssymb}
\usepackage{ifdraft}
\usepackage{url}
\usepackage[square,numbers]{natbib}

\numberwithin{equation}{section}

\allowdisplaybreaks

\newtheorem{theorem}{Theorem}[section]
\newtheorem{lemma}[theorem]{Lemma}

\newtheorem{corollary}[theorem]{Corollary}
\newtheorem{conjecture}[theorem]{Conjecture}
\newtheorem{problem}[theorem]{Problem}

\theoremstyle{definition}
\newtheorem{example}[theorem]{Example}
\newtheorem{definition}[theorem]{Definition}
\newtheorem{remark}[theorem]{Remark}

\newtheorem*{ack}{Acknowledgement}

\theoremstyle{remark}

{\end{example}}

\newenvironment{romenumerate}[1][0pt]{
\addtolength{\leftmargini}{#1}\begin{enumerate}
 }{\end{enumerate}}

\newcounter{oldenumi}
{\setcounter{oldenumi}{\value{enumi}}
\begin{romenumerate} \setcounter{enumi}{\value{oldenumi}}}
{\end{romenumerate}}

\newcounter{thmenumerate}
\newenvironment{thmenumerate}
{\setcounter{thmenumerate}{0}%
 \def\item{\par
 \refstepcounter{thmenumerate}\textup{(\roman{thmenumerate})\enspace}}
}
{}

\newcommand\pfitem[1]{\par(#1):}
\newcommand\pfitemx[1]{\par#1:}
\newcommand\pfitemref[1]{\pfitemx{\ref{#1}}}

\newcommand\pfcase[2]{\smallskip\noindent\emph{Case #1: #2} \noindent}

\newcommand{\refT}[1]{Theorem~\ref{#1}}
\newcommand{\refC}[1]{Corollary~\ref{#1}}
\newcommand{\refL}[1]{Lemma~\ref{#1}}
\newcommand{\refR}[1]{Remark~\ref{#1}}
\newcommand{\refS}[1]{Section~\ref{#1}}
\newcommand{\refSS}[1]{Section~\ref{#1}} 

\newcommand{\refE}[1]{Example~\ref{#1}}

\newcommand{\refand}[2]{\ref{#1} and~\ref{#2}}

\begingroup
  \divide\count255 by 60
  \multiply\count255 by -60
  \advance\count255 by \time
  \ifnum \count255 < 10 \xdef\klockan{\the\count1.0\the\count255}
  \else\xdef\klockan{\the\count1.\the\count255}\fi
\endgroup

\newcommand\nopf{\qed}   
\newcommand\noqed{\renewcommand{\qed}{}} 

\renewcommand\le{\leqslant}
\renewcommand\ge{\geqslant}

\newcommand{\sumj}{\sum_{j=0}^\infty}
\newcommand{\sumji}{\sum_{j=1}^\infty}
\newcommand{\sumk}{\sum_{k=0}^\infty}
\newcommand{\sumki}{\sum_{k=1}^\infty}
\newcommand{\sumkk}{\sum_{k=0}^K}
\newcommand{\suml}{\sum_{\ell=0}^\infty}

\newcommand{\sumni}{\sum_{n=1}^\infty}
\newcommand{\sumik}{\sum_{i=1}^k}
\newcommand{\sumil}{\sum_{i=1}^\ell}
\newcommand{\sumin}{\sum_{i=1}^n}
\newcommand{\sumjn}{\sum_{j=1}^n}
\newcommand{\sumkn}{\sum_{k=1}^n}
\newcommand{\prodil}{\prod_{i=1}^\ell}
\newcommand{\prodin}{\prod_{i=1}^n}

\newcommand\set[1]{\ensuremath{\{#1\}}}
\newcommand\bigset[1]{\ensuremath{\bigl\{#1\bigr\}}}
\newcommand\Bigset[1]{\ensuremath{\Bigl\{#1\Bigr\}}}

\newcommand\xpar[1]{(#1)}
\newcommand\bigpar[1]{\bigl(#1\bigr)}
\newcommand\Bigpar[1]{\Bigl(#1\Bigr)}
\newcommand\biggpar[1]{\biggl(#1\biggr)}
\newcommand\lrpar[1]{\left(#1\right)}

\newcommand\xcpar[1]{\{#1\}}
\newcommand\bigcpar[1]{\bigl\{#1\bigr\}}

\newcommand\abs[1]{|#1|}
\newcommand\bigabs[1]{\bigl|#1\bigr|}
\newcommand\Bigabs[1]{\Bigl|#1\Bigr|}
\newcommand\biggabs[1]{\biggl|#1\biggr|}

\def\rompar(#1){\textup(#1\textup)}    
\newcommand\xfrac[2]{#1/#2}

\newcommand\parfrac[2]{\lrpar{\frac{#1}{#2}}}

\newcommand\Bigparfrac[2]{\Bigpar{\frac{#1}{#2}}}

\def\xexp(#1){e^{#1}}
\newcommand\ceil[1]{\lceil#1\rceil}
\newcommand\floor[1]{\lfloor#1\rfloor}
\newcommand\lrfloor[1]{\left\lfloor#1\right\rfloor}

\newcommand\setn{\set{1,\dots,n}}
\newcommand\setm{\set{1,\dots,m}}
\newcommand\ntoo{\ensuremath{{n\to\infty}}}

\newcommand\mtoo{\ensuremath{{m\to\infty}}}
\newcommand\ktoo{\ensuremath{{k\to\infty}}}
\newcommand\Ktoo{\ensuremath{{K\to\infty}}}

\newcommand\ttoo{\ensuremath{{t\to\infty}}}
\newcommand\xtoo{\ensuremath{{x\to\infty}}}

\newcommand\downto{\searrow}
\newcommand\upto{\nearrow}

\newcommand\punkt[1]{\if.#1\else.\spacefactor1000\fi{#1}}
\newcommand\iid{i.i.d\punkt}    
\newcommand\ie{i.e\punkt}
\newcommand\eg{e.g\punkt}
\newcommand\viz{viz\punkt}
\newcommand\cf{cf\punkt}
\newcommand{\as}{a.s\punkt}

\newcommand\whp{w.h.p\punkt}

\newcommand\ii{\mathrm{i}}

\newcommand{\tend}{\longrightarrow}
\newcommand\dto{\overset{\mathrm{d}}{\tend}}
\newcommand\pto{\overset{\mathrm{p}}{\tend}}

\newcommand\eqd{\overset{\mathrm{d}}{=}}

\newcommand\op{o_{\mathrm p}}
\newcommand\Op{O_{\mathrm p}}

\newcommand\dapprox{\overset{\mathrm{d}}{\approx}}

\newcommand\bbR{\mathbb R}

\newcommand\bbN{\mathbb N}

\newcommand\bbZ{\mathbb Z}

\newcommand\bbZgeo{\mathbb Z_{\ge0}}

\newcounter{CC}
\newcommand{\CC}{\stepcounter{CC}\CCx} 
\newcommand{\CCx}{C_{\arabic{CC}}}     
\newcommand{\CCdef}[1]{\xdef#1{\CCx}}     
\newcommand{\CCreset}{\setcounter{CC}0} 
\newcounter{cc}
\newcommand{\cc}{\stepcounter{cc}\ccx} 
\newcommand{\ccx}{c_{\arabic{cc}}}     
\newcommand{\ccdef}[1]{\xdef#1{\ccx}}     

\renewcommand\Re{\operatorname{Re}}

\newcommand\E{\operatorname{\mathbb E{}}}
\renewcommand\P{\operatorname{\mathbb P{}}}
\newcommand\Var{\operatorname{Var}}
\newcommand\Cov{\operatorname{Cov}}

\newcommand\Po{\operatorname{Po}}
\newcommand\Bi{\operatorname{Bi}}

\newcommand\Be{\operatorname{Be}}
\newcommand\Ge{\operatorname{Ge}}
\newcommand\NBi{\operatorname{NBin}}
\newcommand\supp{\operatorname{supp}}

\newcommand\spann{\operatorname{span}}

\newcommand\ga{\alpha}
\newcommand\gb{\beta}
\newcommand\gd{\delta}

\newcommand\gf{\varphi}
\newcommand\gam{\gamma}
\newcommand\gG{\Gamma}

\newcommand\gl{\lambda}
\newcommand\gL{\Lambda}
\newcommand\go{\omega}
\newcommand\gO{\Omega}
\newcommand\gs{\sigma}
\newcommand\gS{\Sigma}
\newcommand\gss{\sigma^2}
\newcommand\eps{\varepsilon}

\renewcommand\phi{\varphi}

\newcommand\hgss{\hat\sigma^2}

\newcommand\cA{\mathcal A}
\newcommand\cB{\mathcal B}

\newcommand\cD{\mathcal D}
\newcommand\cE{\mathcal E}

\newcommand\cL{{\mathcal L}}

\newcommand\cP{\mathcal P}

\newcommand\cT{{\mathcal T}}

\newcommand\cZ{{\mathcal Z}}

\newcommand\tJ{{\tilde J}}

\newcommand\ett[1]{\boldsymbol1\xcpar{#1}} 
\newcommand\bigett[1]{\boldsymbol1\bigcpar{#1}}

\newcommand\qw{^{-1}}
\newcommand\qww{^{-2}}
\newcommand\qq{^{1/2}}
\newcommand\qqw{^{-1/2}}
\newcommand\qqc{^{3/2}}
\newcommand\qqcw{^{-3/2}}
\newcommand\qqq{^{1/3}}
\newcommand\qqqb{^{2/3}}

\newcommand\qqqbw{^{-2/3}}
\newcommand\qqqq{^{1/4}}

\renewcommand{\=}{:=}

\newcommand\intoo{\int_0^\infty}
\newcommand\intoooo{\int_{-\infty}^\infty}
\newcommand\oi{[0,1]}
\newcommand\ooo{[0,\infty)}
\newcommand\ooox{[0,\infty]}
\newcommand\oooy{(0,\infty)}
\newcommand\oooo{(-\infty,\infty)}
\newcommand\setoi{\set{0,1}}

\newcommand\dtv{d_{\mathrm{TV}}}
\newcommand\dk{d_{\mathrm{K}}}
\newcommand\dkk{\widetilde d_{\mathrm{K}}}

\newcommand\dd{\,\mathrm{d}}

\newcommand{\pgf}{probability generating function}
\newcommand{\mgf}{moment generating function}
\newcommand{\chf}{characteristic function}

\newcommand\rv{random variable}
\newcommand\lhs{left-hand side}
\newcommand\rhs{right-hand side}

\newcommand\etto{\bigpar{1+o(1)}}
\newcommand\st{\mathfrak T}
\newcommand\stf{\st_{\mathrm f}}
\newcommand\stl{\st_{\mathrm{lf}}}
\newcommand\stn{\st_n}
\newcommand\ctn{\cT_n}
\newcommand\ctN{{\cT_n}}
\newcommand\ctx{\cT^*}
\newcommand\ctxx[1]{\cT^{*(#1)}}
\newcommand\hcT{\widehat {\cT}}
\newcommand\hct{{\hcT}}
\newcommand\ttt{{\overline \cT}}
\newcommand\too{U_\infty}
\newcommand\voo{{V_\infty}}
\newcommand\No{\bbN_0}
\newcommand\Ni{\bbN_1}
\newcommand\bNo{\overline\bbN_0}
\newcommand\bNi{\overline\bbN_1}
\newcommand\hxi{\widehat\xi}
\newcommand\GW{Galton--Watson}
\newcommand\GWt{\GW{} tree}
\newcommand\cGWt{conditioned \GW{} tree}
\newcommand\GWp{\GW{} process}
\newcommand\GWf{\GW{} forest}
\newcommand\cGWf{conditioned \GW{} forest}
\newcommand\rot{o}
\newcommand\kooo{_{k=0}^\infty}

\newcommand\tpi{\widetilde \pi}
\newcommand\tPhi{\widetilde\Phi}
\newcommand\tPsi{\widetilde\Psi}
\newcommand\tnu{\widetilde\nu}
\newcommand\trho{\widetilde\rho}
\newcommand\ttau{\widetilde\tau}
\newcommand\tZ{\widetilde Z}

\newcommand\ppi{\ensuremath{(\pi_k)}}
\newcommand\ppix{\boldsymbol{\pi}}

\newcommand\ppiooo{\ensuremath{(\pi_k)\kooo}}
\newcommand\cbmn{\cB_{m,n}}
\newcommand\bmn{B_{m,n}}
\newcommand\cbnn{\cB_{n-1,n}}
\newcommand\bnn{B_{n-1,n}}
\newcommand\kk{^{(K)}}
\newcommand\taux{\tau_*}
\newcommand\nnx[1]{^{(#1)}}
\newcommand\nn{\nnx{n}}
\newcommand\delar{\mid}
\newcommand\tphi{\widetilde\phi}
\newcommand\intpipi{\int_{-\pi}^{\pi}}
\newcommand\txoo{t_\infty}
\newcommand\taun{\tau_n}

\newcommand{\pws}{probability \ws}
\newcommand\ws{weight sequence}
\newcommand\xin{\xi_1,\dots,\xi_n}
\newcommand\cce{c_\eps}
\newcommand\wk{w^{(K)}}
\newcommand\wwxk{\mathbf w^{(K)}}
\newcommand\Phik{\Phi_{K}}
\newcommand\Psik{\Psi_{K}}
\newcommand\pik{\pi^{(K)}}
\newcommand\ppik{\boldsymbol{\pi}^{(K)}}
\newcommand\rhok{\rho_{K}}
\newcommand\tauk{\tau_{K}}
\newcommand\bib{balls-in-boxes}
\newcommand\Iga{I$\ga$}
\newcommand\Igb{I$\gb$}
\newcommand\xmm{^{(m)}}

\newcommand\xjj{^{(j)}}
\newcommand\xxmm{^{[m]}}
\newcommand\tw{\widetilde w}
\newcommand\ww{(w_k)}
\newcommand\wwx{\mathbf{w}}
\newcommand\twwx{\widetilde\wwx}
\newcommand\tww{(\tw_k)}
\newcommand\wwkjx{\wwx^{(k_{j-1})}}
\newcommand\ty{\widetilde y}
\newcommand\YYn{(Y_1,\dots,Y_n)}

\newcommand\YYgsn{(Y_{\gs(1)},\dots,Y_{\gs(n)})}
\newcommand\yyn{(y_1,\dots,y_n)}
\newcommand\yyx{\mathbf{y}}
\newcommand\tyyn{(\ty_1,\dots,\ty_n)}
\newcommand\ttn{(t_1,\dots,t_n)}
\newcommand\TTn{(T_1,\dots,T_n)}

\newcommand\ddn{(d_1,\dots,d_n)}
\newcommand\ddnw{(d_2,\dots,d_n)}

\newcommand\Tx{{T'}}
\newcommand\Txw{{T'_u}}
\newcommand\ellr{r}

\newcommand\ageK{\cA_{\ge K}}
\newcommand\ak{\cA_{k}}
\newcommand\bk{\cB_{k}}
\newcommand\geko{_{k\ge0}}
\newcommand\ctnv{\cT_{n;v}}
\newcommand\bd{\bar d}
\newcommand\bddl{(\bd_1,\dots,\bd_\ell)}
\newcommand\dx{d^+}
\newcommand\elz{l}
\newcommand\rhoz{\rho_{\cZ}}
\newcommand\zk{Z_K}
\newcommand\ti{(\textsf{T1})}
\newcommand\tii{(\textsf{T2})}
\newcommand\win[1]{_{(#1)}}
\newcommand\ya{Y\win1}
\newcommand\yb{Y\win2}
\newcommand\yn{Y\win n}
\newcommand\yj{Y\win j}
\newcommand\xij{\xi\win j}
\newcommand\xijm{\xi^-\win j}
\newcommand\txi{\tilde\xi}
\newcommand\txij{\txi\win j}
\newcommand\xia{\xi\win 1}
\newcommand\txia{\txi\win 1}
\newcommand\etax[1]{\eta_{#1}}
\newcommand\etaj{\etax{j}}
\newcommand\taue{\tau_\eps}
\newcommand\ce{c_\eps}

\newcommand\spanw{\spann(\wwx)}
\newcommand\suppw{\supp(\wwx)}
\newcommand\Phim{\Phi_{\mathrm m}}
\newcommand\Phiim{\Phi_{\mathrm i}}
\newcommand\oct{\overline\cT}
\newcommand\vct{|\cT|}
\newcommand\xix{\xi^{*}}
\newcommand\ny[1]{N_{#1}}
\newcommand\nxi[1]{\overline N_{#1}}
\newcommand\nxin[1]{\overline N\nn_{#1}}
\newcommand\bn{B(n)}
\newcommand\jl{{J_l}}
\newcommand\al{{A_l}}
\newcommand\acl{{A^{\textsf c}_l}}
\newcommand\bl{{B_l}}
\newcommand\jll{{J_\ell}}
\newcommand\all{{A_\ell}}
\newcommand\Ax{A^*}
\newcommand\tN{\widetilde N}
\newcommand\nga{n^{1/\ga}}
\newcommand\ngaw{n^{-1/\ga}}
\newcommand\xga{^{1/\ga}}
\newcommand\xgaw{^{-1/\ga}}
\newcommand\xfn{X^f_n}
\newcommand\xabn{X^{a,b}_n}
\newcommand\xfoo{X^f_\infty}
\newcommand\xaboo{X^{a,b}_\infty}
\newcommand\xaoon{X^{a,\infty}_n}
\newcommand\xaoooo{X^{a,\infty}_\infty}
\newcommand\setz{\relax} 
\newcommand\srd{D_{\rho,\gd}}
\newcommand\srdy{D_{\rhoz,\gd'}}
\newcommand\frax[1]{\{#1\}}
\newcommand\sgt{simply generated tree}
\newcommand\sgrt{simply generated random tree}
\newcommand\sgf{simply generated forest}
\newcommand\sgrf{simply generated random forest}
\newcommand\pnx{p^*_n}
\newcommand\eei{\cE_{1i}}
\newcommand\eexi{\cE^*_{1i}}
\newcommand\eeexi{\cE^*_{2i}}
\newcommand\dij{\cD_{ij}}
\newcommand\dab{\cD_{12}}
\newcommand\eeq[1]{\cE_{1{#1}}}
\newcommand\eexq[1]{\cE^*_{1{#1}}}
\newcommand\eeexq[1]{\cE^*_{2{#1}}}
\newcommand\eqq{\widetilde\cE_{n}}

\newcommand\maxeta{\max_{i\le N}\eta_i}
\newcommand\logq{\log_{1/q}}
\newcommand\ff[2]{F_{#1,#2}^{\mathrm r}}
\newcommand\ffu[2]{F_{#1,#2}^{\mathrm u}}
\newcommand\dl{d^{\mathrm L}}
\newcommand\dr{d^{\mathrm R}}
\newcommand\qxi{\xi'}
\newcommand\bxi{\bar\xi}
\newcommand\bS{\bar S}
\newcommand\fen{\floor{\eps n}}
\newcommand\yxx{\widetilde D_n}
\newcommand\dxx{\widetilde d^+}
\newcommand\dxxhct{\dxx_\hct}
\newcommand\vx{v^*}
\newcommand\hcTx[1]{{\hcT_{#1n}}}
\newcommand\hcta{{\hcTx 1}}
\newcommand\hctb{{\hcTx 2}}
\newcommand\hctc{{\hcTx 3}}
\newcommand\hctj{{\hcTx j}}
\newcommand\xxxi{\check\xi_n}
\newcommand\gdd{\gd_1}
\newcommand\hyy{\widehat Y_n}

\newcommand\Cramer{Cram\'er}
\newcommand{\Holder}{H\"older}

\newcommand{\Takacs}{Tak\'acs}
\newcommand\cprime{'}  

\newcommand{\Polya}{P\'olya}

\newcommand\citetq[2]{\citeauthor{#2} \cite[{\frenchspacing #1}]{#2}}

\newcommand\urladdrx[1]{{\urladdr{\def~{{\tiny$\sim$}}#1}}}

\begin{document}
\title[Simply generated trees and random allocations]
{Simply generated trees, conditioned Galton--Watson trees,
 random allocations and condensation.}

\date{2 December 2011}%

\author{Svante Janson}
\address{Department of Mathematics, Uppsala University, PO Box 480,
SE-751~06 Uppsala, Sweden}
\email{svante.janson@math.uu.se}
\urladdrx{http://www2.math.uu.se/~svante/}


\begin{abstract} 
We give a unified treatment of the limit,  as the size tends to infinity,
of simply generated random trees, 
including both 
the well-known result in the standard case
of critical Galton--Watson trees 
and similar but less well-known results in the
other cases (i.e., when no equivalent critical Galton--Watson tree exists).
There is a well-defined limit in the form of an infinite
random tree in all cases;
for critical Galton--Watson trees this tree is locally finite but for the
other cases the random limit has exactly one node of infinite degree.

The proofs 
use a well-known connection to
a random allocation model that we call balls-in-boxes, and
we prove corresponding theorems for this model.

This survey paper contains 
many known results from many different sources, together
with some new results. 
\end{abstract}

\maketitle

\tableofcontents

\section{Introduction}\label{S:intro}

The main purpose of this survey paper is to study the asymptotic shape of simply
generated random trees in complete generality; this includes \cGWt{s}
as a special case, but we will also go beyond that case.
Definitions are given in \refS{Ssimply}; here we only recall that \sgt{s} are
defined by a \ws{} $\ww$, and that the case when the \ws{} is a probability
distribution yields \cGWt{s}.

It is well-known that in the case of a critical \cGWt{}, 
\ie, when the defining offspring distribution has expectation 1,
the random tree has a limit (as the size tends to infinity); 
this limit is an infinite random tree,
the size-biased \GWt{} defined by \citet{Kesten}, 
see also \citet{AldousII}, \citet{AldousPitman} and \citet{LPP}.
It is also well-known that this case is less special than it might seem;
there is a notion of equivalent \ws{s} defining the same \sgrt, see
\refS{Sequiv}, and a large class of \ws{s} have an equivalent \pws{}
defining a critical \cGWt.
Many probabilists, including myself, have often 
concentrated on this ``standard'' case of critical \cGWt{s}
and dismissed the remaining cases as uninteresting exceptional cases. 
However, some researchers, in particular mathematical physicists,
have studied such cases too.
\citet{BialasB} studied one case 
(\refE{Ezeta} below)
and found a phase
transition as we leave the standard case;
this can be interpreted as a condensation making the tree
bushy with one or a few nodes of very high degree. 
This interesting condensation was
studied further by
\citet{sdf}, who showed that (in the power-law case),
there is a limit tree of a different type, having one node of
infinite degree.

We give in the present paper
a unified treatment of the limit as the size tends to infinity
for all 
simply generated trees, including both 
the well-known result in the standard case
of critical \GWt{s}
and the 
``exceptional'' cases (\ie, when no equivalent probability weight
sequence exists, or when such a sequence exists but not with mean 1).
We will see that there is a well-defined limit in the form of an infinite
random tree for any weight sequence. 
In the non-standard cases, 
this infinite random limit has exactly one node of infinite degree,
so its form  
differs from the standard case  of a critical \GWt{} where all nodes in
the limit tree have finite degrees,
but nevertheless the trees are similar;
see Sections \refand{ShGW}{Smain} for details.

Some important notation, used throughout the paper, is introduced in
\refS{Snotation}, while Sections \refand{Sequiv}{SUlam} contain further
preliminaries.
The main limit theorem for \sgrt{s} is stated in \refS{Smain}, together with
some other, related, limit theorems concerning node degrees and fringe subtrees.
The differences between different types of \ws{s}
are discussed further in \refS{S3}.

The proofs of the limit theorems for random trees
use a well-known connection to
a random allocation model that we call
\bib; 
this model exhibits a similar behaviour, with condensation in the
non-classical cases, see \eg{} \citet{BialasetalNuPh97}.
The model is defined in \refS{SBB}, and
the relation between the models is described in \refS{Stree-balls}.
The \bib{} model is interesting in its own right, and it has been used for 
several
other applications; we give some examples from probability theory,
combinatorics and statistical physics in \refS{Sex+}.
We therefore also develop the general theory for \bib{} with arbitrary \ws{s}
(in the range where the mean occupancy is bounded). 
In particular, we give in \refS{SBB}
theorems corresponding to (and in some ways extending)
our main theorems for random trees.

The limit theorems for \bib{} are proved in Sections \ref{Sprel}--\ref{Sbbpf}, 
and then these results are used to prove the limit theorems for random trees
in Sections \ref{Stree-balls}--\ref{SpfXXX}.

The remaining sections contain additional results.
\refS{Spart} gives asymptotic results for the partition functions of the
models.
The very long \refS{Slarge} gives results on the largest degrees in random
trees, and the largest numbers of balls in a box in the \bib{} model; the
section is 
long because there are several different cases with different types of
behaviour. 
In particular, we study in \refSS{SSlarge+}
the case when there is condensation, and
investigate whether this appears as condensation to a single box (or node),
or whether 
the condensation is distributed over several boxes (nodes); it turns out
that both cases can occur.
We give also, in \refS{SSmaxforests}, applications to the size of the
largest tree in random forests.
In \refS{SlargeT}, the condensation in random trees is discussed in further
detail.
Finally, some additional comments, results and open problems are given in
Sections \refand{Sfurther}{Scondx}; \refS{Sfurther} mentions briefly various
other types of asymptotic results for \sgrt{s}, and \refS{Scondx} discusses
 alternative ways to condition \GWt{s}.

This paper contains many known results from many different sources, together
with some new results. (We believe, for example, 
that the theorems in \refS{Smain} are new in the present generality.) 
We have tried to give relevant references, but the absence of references
does not necessarily imply that a result is new.

\section{Simply generated trees}\label{Ssimply}

\subsection{Ordered rooted trees}
The trees that we consider are 
(with a few explicit exceptions)
\emph{rooted} and \emph{ordered} 
(such trees are also called \emph{plane} trees).
Recall that a tree is \emph{rooted} if  one node is distinguished as the 
\emph{root} $\rot$; this implies that we can arrange the nodes in a
sequence of 
generations (or levels), 
where generation $x$ consists of all nodes of distance $x$ to
the root. (Thus generation 0 is the root; generation 1 is the set of
neighbours of the root, and so on.) 
If $v$ is a node with $v\neq\rot$, then the
\emph{parent} of $v$ is the neighbour of $v$ on the path from $v$ to $\rot$;
thus, every node except the root has a unique parent, 
while the root has no parent. Conversely, for any node $v$, the neighbours of
$v$ that are further away from the root than $v$ are the \emph{children} of
$v$.
The number of children of $v$ 
is the \emph{outdegree} $\dx(v)\ge0$ of $v$.
Note that if $v$ is in generation $x$, then its parent is in
generation $x-1$ and its children are in generation $x+1$.

Recall further that a rooted tree is \emph{ordered} if the children of each node
are ordered in a sequence 
$v_1,\dots,v_d$, where $d=\dx(v)\ge0$ is the outdegree of $v$.
See \eg{} \citet{Drmota} for more information on these and other types of trees.
(The trees we consider are  called \emph{planted plane trees} in
\cite{Drmota}.)
We identify trees that are isomorphic in the obvious (order preserving) way.
(Formally, we can define our trees as equivalence classes. Alternatively, we
may select a specific representative in each equivalence class as in
\refS{SUlam}.) 

\begin{remark}
Some authors prefer to  
add an extra (phantom) node as a parent of the root; such trees are called
\emph{planted}. 
(An alternative version
is to add only a pendant edge at the root, with no second
endpoint.) 
There is an obvious one-to-one correspondence between trees
with and without the extra node, so the difference is just a matter of
formulations, but when comparing results one should be careful whether, for
example, the extra node is counted or not.
The extra node yields the technical advantage that also
the root has indegree 1 and thus total degree = $1+\dx(v)$; it further gives each
embedding in the plane a unique ordering of the children of every node
(in clockwise order from the parent, say).
Nevertheless, we find this device less natural and we will not use it in the
present paper. (We use outdegrees instead of degrees and assume that an
ordering of the children as above is given; then there are no problems.)
\end{remark}

We are primarily interested in (large) finite trees, but we will also
consider infinite trees, for example as limit objects in our main theorem
(\refT{Tmain}). The infinite trees may have nodes with infinite outdegree
$\dx(v)=\infty$; in this case we assume that the children are ordered
$v_1,v_2,\dots$ (i.e., the order type of the set of children is $\bbN$).

We let $\stn$ be the set of all ordered rooted trees with $n$ nodes
(including 
the root) and let
$\stf\=\bigcup_{n=1}^\infty\stn$ be the set of all finite ordered rooted trees;
see further \refS{SUlam}.

\begin{remark}\label{Rcatalan}
Note that $\stn$ is a finite set. 
In fact, it is well-known that its size $|\stn|$ is the $(n-1)$:th Catalan
number 
\begin{equation}
  \label{catalan}
C_{n-1}=\frac1n\binom{2n-2}{n-1}=\frac{(2n-2)!}{n!\,(n-1)!},
\end{equation}
see \eg{} \cite[Section 1.2.2 and Theorem 3.2]{Drmota},
\cite[Section I.2.3]{Flajolet} or \cite[Exercise 6.19(e)]{Stanley2}, 
but we do not need this. 
\end{remark}

For any tree $T$, we let $|T|$ denote the number of nodes; we  call
$|T|$ the \emph{size} of $T$.
As is well known, for any finite tree $T$,
\begin{equation}
  \label{sumd}
\sum_{v\in T}\dx(v)=|T|-1,
\end{equation}
since every node except the root is the child of exactly one node.

\subsection{\GWt{s}}
An important class of examples of random ordered rooted trees 
is given by the \emph{\GWt{s}}.
These are defined as the family trees of \GW{} processes:
Given a probability distribution $\ppiooo$ on $\bbZgeo$,
or, equivalently, a random variable $\xi$
with distribution $(\pi_k)_{k=0}^\infty$,
we build the tree $\cT$ recursively, starting with the root and giving each
node a 
number of children that is an independent copy of $\xi$. 
(We call $\ppi\kooo$ the \emph{offspring distribution} of $\cT$; we sometimes
also abuse the language and call $\xi$ the offspring distribution.)
In other words,
the outdegrees $\dx(v)$ are \iid{} with the distribution $(\pi_k)\kooo$.

Recall that the \GW{} process is called \emph{subcritical}, \emph{critical}
or \emph{supercritical} as the expected number of children 
$\E\xi=\sumk k\pi_k$ satisfies $\E\xi<1$, $\E\xi=1$ or $\E\xi>1$.
It is a standard basic fact of branching process theory that $\cT$ is finite
\as{} if $\E\xi\le1$ (\ie, in the subcritical and critical cases), but $\cT$
is infinite with positive probability if $\E\xi>1$ (the supercritical case),
see \eg{} \citet{AN}.

The \GWt{s} have random sizes. We are mainly interested in random trees
with a given size; we thus define $\ctn$ as $\cT$ conditioned on $|\cT|=n$.
These random trees $\ctn$ are called \emph{\cGWt s}. By definition, $\ctn$
has size $|\ctn|=n$.

It is well-known that several important classes of random trees can be seen
as \cGWt, see \eg{} \citet{AldousII}, \citet{Devroye}, \citet{Drmota}
and \refS{Sex}.

\subsection{Simply generated trees}

The random trees that we will study are a generalization of the \GWt{s}.
We suppose in this paper that we are given a fixed \emph{weight sequence}
$\wwx=(w_k)_{k\ge0}$ of non-negative real numbers.
We then define the \emph{weight} of a finite tree $T\in\stf$ by
\begin{equation}\label{wtree}
  w(T)\=\prod_{v\in T}w_{\dx(v)},
\end{equation}
taking the product over all nodes $v$ in $T$.
Trees with such weights are called \emph{simply generated trees} and were
introduced by \citet{MM78}.
To avoid trivialities, we assume that $w_0>0$ and that there exists some
$k\ge2$ with $w_k>0$.

We let $\ctn$ be the random tree obtained by picking an element of $\stn$ at
random with probability proportional to its weight, \ie,
\begin{equation}\label{ctn}
  \P(\ctn=T)=\frac{w(T)}{Z_n},\qquad T\in\stn,
\end{equation}
where the normalizing factor $Z_n$ is given by
\begin{equation}\label{zn}
  Z_n=Z_n(\wwx)\=\sum_{T\in\stn}w(T);
\end{equation}
$Z_n$ is known as  the \emph{partition function}.
This definition makes sense only when $Z_n>0$; we tacitly consider only such
$n$ when we discuss $\ctn$.
Our assumptions $w_0>0$ and $w_k>0$ for some $k\ge2$ imply that $Z_n>0$
for infinitely many $n$, see \refC{Cexists} for a more precise result.
(In most applications, $w_1>0$, and then $Z_n>0$ for every $n\ge1$, so there
is no problem at all. 
The archetypical example with a parity restriction is given
by the random (full) binary tree, see \refE{Ebinary},
for which $Z_n>0$ if and only if $n$ is odd.)

One particularly important case is when $\sumk w_k=1$, so the weight
sequence $(w_k)$ is a probability distribution on $\bbZgeo$.
(We then say that $(w_k)$ is a \emph{probability weight sequence}.)
In this case we let $\xi$ be a random variable with the corresponding
distribution: $\P(\xi=k)=w_k$; we further let $\cT$ be the random
Galton--Watson tree generated by $\xi$.
It follows directly from the definitions that for every finite tree
$T\in\stf$, $\P(\cT=T)=w(T)$. 
Hence 
\begin{equation}
  \label{znp}
Z_n=\P(|\cT|=n)
\end{equation}
and the simply generated random 
tree $\ctn$ is the same as the random
Galton--Watson tree $\cT$ conditioned on $|\cT|=n$, \ie,
it equals the \cGWt{} $\ctn$ defined above.

It is well-known, see \refS{Sequiv} for details, 
that in many cases it is possible to change the weight
sequence $\ww$ to a probability weight sequence
without changing  the distribution of the random trees $\ctn$;
in this case $\ctn$ can thus be seen as a \cGWt.
Moreover, in many cases this can be done such that the resulting probability
distribution has mean 1. In such cases it thus suffices to consider the
case of a probability weight sequence with mean $\E\xi=1$; then $\ctn$ is a
conditional critical Galton--Watson tree. It turns out that this is a nice
and natural 
setting, with many known results proved by many different authors.
(In many papers it is further assumed that $\xi$ has finite variance, or even
a finite exponential moment. This is not needed for the main results presented
here, but may be necessary for other results. See also Sections
\ref{S3}, \ref{Slarge} and \ref{Sfurther}.)

\section{Notation}\label{Snotation}

We consider a fixed \ws{} $\wwx=\ww\geko$.
The \emph{support} $\supp(\wwx)$
of the weight sequence $\wwx=(w_k)$ is \set{k:w_k>0}.
We define
\begin{equation}
\go= \go(\wwx)\=
\sup\supp(\wwx)
=
\sup\set{k:w_k>0}\le\infty,
\end{equation}
(When considering $\ctn$, we assume, as said above, $w_0>0$ and $w_k>0$ for
some $k\ge2$; this can be written $0\in\supp(\wwx)$ and $\go\ge2$.)

We further define (assuming that the support contains at least two
points) 
\begin{equation}\label{span}
  \spann(\wwx)\=\max\set{d\ge1: d\delar(i-j) \text{ whenever }w_i,w_j>0}.
\end{equation}
Since we assume $w_0>0$, \ie, $0\in\supp(\wwx)$,
we can simplify this to
\begin{equation}\label{span0}
  \spann(\wwx)=\max\set{d\ge1: d\delar i \text{ whenever }w_i>0},
\end{equation}
the greatest common divisor of $\supp(\wwx)$.

We let 
\begin{equation}
\Phi(z)\=\sumk w_k z^k  
\end{equation}
be the generating function of the given weight
sequence, and let 
$\rho\in\ooox$ be its radius of convergence. 
Thus
\begin{equation}\label{rho}
  \rho=1/\limsup_{\ktoo}w_k^{1/k}.
\end{equation}
$\Phi(\rho)$ is always defined, with $0<\Phi(\rho)\le\infty$.
Note that (assuming $\go>0$) $\Phi(\infty)=\infty$; in particular, if
$\rho=\infty$, then 
$\Phi(\rho)=\infty$. On the other hand, if $\rho<\infty$, then both
$\Phi(\rho)=\infty$ and $\Phi(\rho)<\infty$ are possible.
If $\rho>0$, then $\Phi(t)\upto\Phi(\rho)$ as $t\upto\rho$ by monotone
convergence. 

We further define, for  $t$ such that $\Phi(t)<\infty$,
\begin{equation}\label{psidef}
\Psi(t)\=\frac{t\Phi'(t)}{\Phi(t)}
=\frac{\sumk k w_kt^k}{\sumk  w_kt^k};
\end{equation}
$\Psi(t)$ is thus defined and finite at least for $0\le t<\rho$, and
if $\Phi(\rho)<\infty$, then $\Psi(\rho)$ is still defined by
\eqref{psidef}, 
with
$\Psi(\rho)\le\infty$ (note that the numerator in \eqref{psidef} may diverge
in this case, but not for $0\le t<\rho$). Moreover, if $\Phi(\rho)=\infty$,
we define $\Psi(\rho)\=\lim_{t\upto\rho}\Psi(t)\le\infty$.
(The limit exists by \refL{LPsi}\ref{L1a} below, but may be infinite.)

Alternatively, \eqref{psidef} may be written
\begin{equation}\label{psimgf}
\Psi(e^x)
=e^x\frac{\Phi'(e^x)}{\Phi(e^x)}
=\frac{\dd}{\dd x}\log\Phi(e^x).
\end{equation}

The function $\Psi$ will play a central role in the sequel.
This is mainly because of \refL{LEPsi} below, which gives a probabilistic
interpretation of $\Psi(t)$.
Its basic properties are given by the following lemma, which is proved in
\refS{Sprel}. 

\begin{lemma}
  \label{LPsi}
Let $\wwx=(w_k)\kooo$ be a given \ws{} with $w_0>0$ and $w_k>0$ for
some $k\ge1$ (\ie, $\go(\wwx)>0$).  
\begin{romenumerate}
  \item \label{L1a}
If\/ $0<\rho\le\infty$, then
the function 
\begin{equation}\label{l1a}
\Psi(t)\=\frac{t\Phi'(t)}{\Phi(t)}
=\frac{\sumk k w_kt^k}{\sumk  w_kt^k}
\end{equation}
is finite, continuous and (strictly) increasing on $[0,\rho)$, with $\Psi(0)=0$.

\item\label{L1x}
If\/ $0<\rho\le\infty$, then $\Psi(t)\to\Psi(\rho)\le\infty$ as
$t\upto\rho$.

\item\label{L1y}
For any $\rho$,
$\Psi$ is continuous  $[0,\rho]\to[0,\infty]$,
with $\Psi(\rho)\le\infty$.

\item \label{L1b}
If $\rho<\infty$ and $\Phi(\rho)=\infty$, then
$\Psi(\rho)\=\lim_{t\to\rho}\Psi(t)=\infty$.

\item \label{L1c}
If $\rho=\infty$, then $\Psi(\rho)\=\lim_{t\to\rho}\Psi(t)=\go\le\infty$.
\end{romenumerate}

Consequently, if $\rho>0$, then
\begin{equation}\label{l1}
\Psi(\rho)
=\lim _{t\upto\rho} \Psi(t) 
=\sup _{0\le t<\rho} \Psi(t)
\in(0,\infty].
\end{equation}
\end{lemma}

We define 
\begin{equation}\label{nu}
\nu\=\Psi(\rho).
\end{equation}
In particular, if $\Phi(\rho)<\infty$, then 
\begin{equation}\label{nu1}
\nu= 
\frac{\rho\Phi'(\rho)}{\Phi(\rho)}
\le\infty.
\end{equation}
It follows from \refL{LPsi} that $\nu=0\iff\rho=0$,
and that if $\rho>0$,  then
\begin{equation}\label{nusup}
\nu\=
\Psi(\rho)
=\lim _{t\upto\rho} \Psi(t) 
=\sup _{0\le t<\rho} \Psi(t)
\in(0,\infty].
\end{equation}
It follows from \eqref{l1a} 
that $\nu\le\go$.

Note that all these parameters depend on the weight sequence $\wwx=(w_k)$;
we may occasionally write \eg{} $\go(\wwx)$ and $\nu(\wwx)$, but
usually we for simplicity do not show $\wwx$ explicitly in the notation.

\begin{remark}\label{Rcz}
Let $\cZ(z)$ denote 
  the generating function 
$\cZ(z)\=\sumni Z_nz^n$. 
Then
  \begin{equation}
	\label{cz}
\cZ(z)=z\Phi(\cZ(z)),
  \end{equation} 
as shown already by \citet{Otter}.
This equation is the basis of much 
work on simply generated trees using algebraic and  analytic methods,
see \eg{} \citet{Drmota}, but the
  present paper uses different methods
and we will use \eqref{cz} only in a few minor remarks.
\end{remark}

\subsection{More notation}

We define $\No=\bbZgeo\=\set{0,1,2,\dots}$,
$\Ni=\bbZ_{>0}\=\set{1,2,\dots}$, 
$\bNo\=\No\cup\set\infty$
and
$\bNi\=\Ni\cup\set\infty$.

All unspecified limits are as \ntoo.
Thus, $a_n\sim b_n$ means $a_n/b_n\to1$ as \ntoo.
We use $\pto$ and $\dto$ for convergence in probability and distribution,
respectively, of random variables,
and $\eqd$ for equality in distribution. 
We  use $\op$ and $\Op$ in the standard senses:
$\op(a_n)$ is an unspecified random variable $X_n$ such that $X_n/a_n\pto0$
as \ntoo, and
$\Op(a_n)$ is a random variable $X_n$ such that $X_n/a_n$ is stochastically
bounded (usually called \emph{tight}).
We say that some event holds \emph{\whp} (with high probability) if
its probability tends to 1 as \ntoo.
(See further \eg{} \cite{SJN6}.)

A \emph{coupling} of two random variables $X$ and $Y$ 
is formally a pair of random variables $X'$ and $Y'$ defined on a common
probability space such that $X\eqd X'$ and $Y\eqd Y'$; with a slight abuse
of notation we may continue to write $X$ and $Y$, thus replacing the  original
variables with new ones having the same distributions.

We write $X_n\dapprox X'_n$ for two sequences of random variables or vectors
$X_n$ and $X'_n$ 
if there exists a coupling of $X_n$ and $X'_n$ with $X_n=X'_n$ \whp;
this is equivalent to 
$\dtv(X_n,X'_n)\to0$ as \ntoo, where $\dtv$ denotes the total variation
distance. 

We use $C_1,C_2,\dots$ to denote unimportant constants, possibly different
at different occurrences.

Recall that $\dx(v)=\dx_T(v)$ always denotes the \emph{outdegree} of a node
$v$ in a tree $T$. (We use the notation
$\dx(v)$ rather than $d(v)$ to emphasise this.)
We will not use the total degree $d(v)=1+\dx(v)$ (when $v\neq o$), but care
should be taken when comparing with other papers.

\section{Equivalent weights}\label{Sequiv}

If $a,b>0$ and we change $w_k$ to
\begin{equation}\label{tw}
  \tw_k\=ab^kw_k,
\end{equation}
then, for every tree $T\in\stn$, $w(T)$ is changed to,
using \eqref{sumd},
\begin{equation}
\tw(T)=a^nb^{\sum_v\dx(v)}w(T)=a^nb^{n-1}w(T).  
\end{equation}
Consequently, $Z_n$ is changed to 
\begin{equation}
  \label{tz}
\tZ_n\=
a^nb^{n-1}Z_n, 
\end{equation}
and the probabilities in
\eqref{ctn} are not changed. In other words, the new weight sequence
$(\tw_k)$ defines the same simply generated random trees $\ctn$ as $(w_k)$.
(This is essentially due to \citet{Kennedy}, who 
did not consider trees but
showed the corresponding result for
\GWp{es}. See also \citet{AldousII}.)
We say that \ws{} $(w_k)$ and $(\tw_k)$ related by \eqref{tw} (for some
$a,b>0$) are \emph{equivalent}. (This is clearly an equivalence relation on
the set of weight sequences.)

Let us see how replacing $(w_k)$ by the equivalent weight sequence $(\tw_k)$
affects the parameters defined above. The support, span and $\go$ are not
affected at all.

The generating function $\Phi(t)$ is replaced by
\begin{equation}\label{tPhi}
  \tPhi(t)\=\sumk\tw_kt^k=\sumk ab^kt^k=a\Phi(bt),
\end{equation}
with radius of convergence $\trho=\rho/b$.
Further, $\Psi(t)$ is replaced by
\begin{equation}\label{tPsi}
  \tPsi(t)\=\frac{t\tPhi'(t)}{\tPhi(t)}
=\frac{tab\Phi'(bt)}{a\Phi(bt)}
=\Psi(bt).
\end{equation}
Hence, if $\rho>0$,  $\nu$ is replaced by,
using \eqref{nusup},
\begin{equation*}
  \tnu\=\sup_{0\le t<\trho}\tPsi(t)
=\sup_{0\le t<\rho/b}\Psi(bt)
=\sup_{0\le s<\rho}\Psi(s)=\nu;
\end{equation*}
if $\rho=0$ then $\tnu=\trho=0=\nu$ is trivial. In other words, $\nu$ is
invariant and depends only on the equivalence class of the weight sequence.

\begin{lemma}\label{Lequivp}
 There exists a probability weight sequence equivalent to $(w_k)$ 
 if and only if  and only if $\rho>0$. In this case, the 
 probability weight sequences equivalent to $(w_k)$ 
 are given by 
 \begin{equation}\label{lequivp}
   p_k=\frac{t^kw_k}{\Phi(t)},
 \end{equation}
for any $t>0$ such that $\Phi(t)<\infty$.
\end{lemma}

\begin{proof}
The equivalent weight sequence $(\tw_k)$ given by \eqref{tw}
is a probability distribution if and only if
\begin{equation*}
  1=\sumk\tw_k=a\sumk w_kb^k=a\Phi(b),
\end{equation*}
\ie, if and only if $\Phi(b)<\infty$ and $a=\Phi(b)\qw$.
Thus, there exists a probability weight sequence equivalent to $(w_k)$ 
if and only if there exists $b>0$ with
$\Phi(b)<\infty$, \ie, if and only if $\rho>0$; in this case we can choose
any such $b$ and take $a\=\Phi(b)\qw$, which yields \eqref{lequivp} (with
$t=b$). 
\end{proof}
We easily find the \pgf{} and thus moments of the \pws{} in \eqref{lequivp};
we state this in a form including the trivial case $t=0$.

\begin{lemma}
  \label{LEPsi}
If $t\ge0$ and $\Phi(t)<\infty$, then 
\begin{equation}\label{lep1}
  p_k\=\frac{t^kw_k}{\Phi(t)},
\qquad k\ge0,
\end{equation}
defines a \pws{} $(p_k)$. This probability distribution has \pgf{}
\begin{equation}\label{lep2}
\Phi_t(z)\=  \sumk p_kz^k = \frac{\Phi(tz)}{\Phi(t)},
\end{equation}
and a random variable $\xi$ with this distribution has expectation
\begin{equation}\label{lep3}
\E\xi=\Phi_t'(1)=  \frac{t\Phi'(t)}{\Phi(t)}=\Psi(t)
\end{equation}
and variance
\begin{equation}\label{lep4}
\Var\xi=t\Psi'(t);
\end{equation}
furthermore, for any $s\ge0$ and $x\ge 0$,
\begin{equation}\label{lep5}
  \P(\xi\ge x)
\le e^{-sx}  \frac{\Phi(e^st)}{\Phi(t)}
\le
 e^{-sx}  \frac{\Phi(e^st)}{\Phi(0)}.
\end{equation}
\end{lemma}

If $t<\rho$, then $\E\xi$ and $\Var\xi$ are finite. If $t=\rho$, however,
$\E\xi$ and $\Var\xi$ may be infinite
(we define $\Var\xi=\infty$ when $\E\xi=\infty$, but $\Var\xi$ may be
infinite also when $\E\xi$ is finite); \eqref{lep3}--\eqref{lep4} still
hold, with $\Psi'(\rho)\le\infty$ defined 
as the limit $\lim_{s\upto \rho}\Psi'(s)$.
The tail estimate \eqref{lep5} is interesting only when $t<\rho$, when we
may choose any $s<\log(\rho/t)$ and obtain the estimate $O(e^{-sx})$.

\begin{proof}
Direct summations yield
\begin{equation}
  \sumk p_k = \frac{\sumk t^kw_k}{\Phi(t)}=1
\end{equation}
and, more generally,
\begin{equation}
  \sumk p_k z^k = \frac{\sumk w_k t^kz^k}{\Phi(t)}
= \frac{\Phi(tz)}{\Phi(t)},
\end{equation}
showing that $(p_k)$ is a probability distribution with the \pgf{} $\Phi_t$
given in \eqref{lep2}.

The expectation $\E\xi=\Phi_t'(1)$ is evaluated  by differentiating \eqref{lep2}
(for $z<1$ and then taking the limit as $z\to1$ to avoid convergence
problems if $t=\rho$),
or directly from \eqref{lep1} as 
\begin{equation*}
  \E\xi=\sumk kp_k
=\frac{\sumk k w_k t^k}{\Phi(t)}=\Psi(t).
\end{equation*}

Similarly, the variance is given by, using \eqref{lep2} and \eqref{lep3},
\begin{equation*}
\Var\xi=
  \Phi_t''(1)+\Phi_t'(1)-(\Phi_t'(1))^2
=
\frac{t^2\Phi''(t)}{\Phi(t)}+\frac{t\Phi'(t)}{\Phi(t)}-
\parfrac{t\Phi'(t)}{\Phi(t)}^2
=t\Psi'(t).
\end{equation*}
Alternatively,
\begin{equation*}
  \begin{split}
t\Psi'(t)
&=t\frac{\dd}{\dd t} \frac{\sumk kt^kw_k}{\Phi(t)}
=
\frac{\sumk k^2t^kw_k}{\Phi(t)}
-\lrpar{\frac{\sumk kt^kw_k}{\Phi(t)}}^2
\\&
= \sumk k^2p_k-\biggpar{\sumk kp_k}^2	
=\E\xi^2-(\E\xi)^2
=\Var\xi.
  \end{split}
\end{equation*}
(In the case $t=\rho$ and $\Var\xi=\infty$, we use this calculation for
$t'<t$ and let $t'\to t$.)

Finally, by \eqref{lep2},
\begin{equation*}
  \P(\xi\ge x)
\le e^{-sx}  \E e^{s\xi}
= e^{-sx}  {\Phi_{t}(e^s)}
= e^{-sx}  \frac{\Phi(e^st)}{\Phi(t)}.
\qedhere
\end{equation*}
\end{proof}

In particular, taking $t=1$, we recover that standard facts that
if $\ww$ is a probability
distribution, so $\Phi(1)=1$, 
then it has expectation $\Phi'(1)=\Psi(1)$ and
variance $\Psi'(1)$.

\begin{remark}\label{Rpt}
We see from \refL{Lequivp} that the \pws{s} equivalent to $(w_k)$ are given
by \eqref{lequivp}, where 
$t\in(0,\rho]$ when $\Psi(\rho)<\infty$ and
$t\in(0,\rho)$ when $\Psi(\rho)=\infty$.  
By \refL{LPsi},
$t\mapsto\E\xi=\Psi(t)$ is an increasing bijection
$(0,\rho]\to(0,\nu]$ and $(0,\rho)\to(0,\nu)$.
Hence, any
equivalent probability weight sequence is uniquely
determined by its expectation, and the possible expectations are
$(0,\nu]$ (when $\Psi(\rho)<\infty$) 
or $(0,\nu)$ (when $\Psi(\rho)=\infty$).  
\end{remark}

\begin{remark}
Note that we will frequently use \eqref{lequivp} to define a new \pws{} also
if we start with a \pws{} $(w_k)$. Probability distributions related in this
way are called \emph{conjugated} or \emph{tilted}.  
Conjugate distributions were introduced by \citet{Cramer} as an important
tool in large deviation theory, see \eg{} \cite{DemboZ}. The reason is
essentially the same as in the present paper: by conjugating the
distribution we can change its mean in a way that enables us to keep control
over sums $S_n$.
\end{remark}

\section{A modified \GWt}\label{ShGW}

Let $(\pi_k)_{k\ge0}$ be a probability distribution on $\No$ 
and let $\xi$ be a random variable on $\No$
with distribution $(\pi_k)\kooo$:
\begin{equation}\label{xi}
  \P(\xi=k)=\pi_k, \qquad k=0,1,2,\dots
\end{equation}
We assume that the
expectation $\mu\=\E\xi=\sum_k k\pi_k\le1$ (the subcritical or critical case).

In this case, 
we define 
(based on \citet{Kesten} and \citet{sdf})
a modified \GWt{} $\hcT$  as follows:
There are two types of nodes: \emph{normal} and \emph{special}, with the root
being special. Normal nodes
have offspring 
(outdegree) according to independent copies of $\xi$, while special nodes
have offspring according to independent copies of $\hxi$, where
\begin{equation}\label{hxi}
\P(\hxi=k)\=
  \begin{cases}
k\pi_k, & k=0,1,2,\dots,\\
1-\mu, & k=\infty.	
  \end{cases}
\end{equation}
(Note that this is a probability distribution on $\bNi$.)
Moreover, 
all children of a normal node are normal;
when a special node gets an infinite number of children, all are normal;
when a special node gets a finite number of children, one of its children
is selected uniformly at random and is special,  while all other children are
normal. 

Thus, for a special node, and any integers $j,k$ with $1\le j\le k<\infty$, the
probability that the node has exactly $k$ children and that the $j$:th of
them is special is $k\pi_k/k=\pi_k$.

Since each special node has at most one special child, the special nodes form a
path from the root; we call this path the \emph{spine} of $\hcT$.
We distinguish two different cases:
\begin{romenumerate}[-10pt]
\item[\ti]\label{ti}
If $\mu=1$ (the critical case), then $\hxi<\infty$ \as{} so each special
node has 
a special child and the spine is an infinite path.
Each outdegree $\dx(v)$ in $\hcT$ is finite, so
the tree is infinite but locally finite.

In this case, 
the distribution of $\hxi$ in \eqref{hxi} is  the \emph{size-biased}
distribution of $\xi$, and $\hcT$ is the size-biased \GWt{} 
defined by \citet{Kesten}, 
see also \citet{AldousII}, \citet{AldousPitman}, 
\citet{LPP} and \refR{Rsizebias} below.
The underlying size-biased \GWp{} 
is  the same as the \emph{Q-process} studied in
\citetq{Section I.14}{AN}, which is an instance of Doob's $h$-transform. 
(See \citet{LPP} for further related constructions
in other contexts and \citet{GeigerK} for a generalization.)

An alternative construction of the random tree $\hcT$ is to start with the
spine (an infinite path from the root) and then at each node in the spine
attach further branches;
the number of branches at each node in the spine is a copy of $\hxi-1$
and each branch is a copy of the \GWt{} $\cT$
with offspring distributed as $\xi$;
furthermore, at a node where $k$ new
branches are attached, the  number of them attached to the left of the spine
is uniformly distributed on \set{0,\dots,k}.
(All random choices are independent.)
Since the critical \GWt{} $\cT$ is \as{} finite, it follows that $\hcT$
\as{} has exactly one infinite path from the root, viz.\ the spine.

\item[\tii]\label{tii}
If $\mu<1$ (the subcritical case), then a special node has 
with probability $1-\mu$
no special child. 
Hence, the spine is \as{} finite and
the number $L$ of nodes in the spine has a (shifted) geometric
distribution
$\Ge(1-\mu)$,
\begin{equation}\label{spine}
  \P(L=\ell)=(1-\mu)\mu^{\ell-1},
\qquad \ell=1,2,\dots.
\end{equation}
The tree $\hcT$ has \as{} exactly one node with infinite outdegree, \viz{} the
top of the spine.
$\hcT$ has \as{} no infinite path.

In this case, an alternative construction of $\hcT$ is to
start with a spine of random length $L$, where $L$ has the geometric
distribution \eqref{spine}. We attach as in {\ti} further branches that are
independent copies of the \GWt{} $\cT$;
at the top of the spine we attach an infinite number of branches and
at all other nodes in the spine the number we attach is a copy of 
$\xi^*-1$ where $\xi^*\eqd(\hxi\mid\hxi<\infty)$ has the size-biased
distribution $\P(\xi^*=k)=k\pi_k/\mu$.
The spine thus ends with an explosion producing an infinite number of
branches, and this is the only node with an infinite degree.
This is the construction by \citet{sdf}.
\end{romenumerate}

\begin{example}\label{Emu=0}
  In the extreme case $\mu=0$, or equivalently  $\xi=0$ a.s.,
\ie, $\pi_0=1$ and $\pi_k=0$ for $k\ge1$, 
\eqref{hxi} shows that $\hxi=\infty$ a.s.
Hence, every normal node has no child and is thus a leaf,
while every special node has an
infinite number of children, all normal.
Consequently, the root is the only special node, 
the spine consists of the root only (\ie, its length $L=1$),
and the tree $\hcT$ consists of
the root with an infinite number of leaves attached to it, \ie, 
$\hcT$ is an infinite star. 
(This is also given directly by the alternative construction in {\tii} above.) 
In contrast, $\cT$ consists of the root only, so $|\cT|=1$.
In this case there is no randomness in $\cT$ or $\hcT$.
\end{example}

\begin{remark}\label{Rimmi}
  In case {\ti}, if we remove the spine, we obtain a random forest that can be
  regarded as coming from
a \GWp{} with immigration, where the immigration is described by an \iid{}
sequence of random variables with the distribution of $\hxi-1$, see \citet{LPP}.
(In the Poisson case, \citet{Grimmett80} gave a slightly different
description of $\hcT$ using a \GWp{} with immigration.)

In case {\tii}, we can do the same, but now the immigration is 
different: at a random (geometric) time, there is an infinite immigration,
and after that there is no more immigration at all.
\end{remark}

\begin{remark}\label{Rsin}
Some related modifications of \GWt{s} having a finite spine have been
considered previously.  
\citet{SagitovS} construct (as a limit for a certain two-type branching process)
a random tree similar to the one in {\tii} above (with a subcritical $\xi$),
with a finite spine having 
a length with the geometric distribution  \eqref{spine}; the difference is
that at the top of the spine, only a finite number of \GWt{s} $\cT$ are
attached. (This number may be a copy of $\xi^*-1$ as at the other points of the
spine, or it may have a different distribution, see \cite{SagitovS}.)
Thus there is no explosion, and the tree is finite.
Another modified \GWt{} is used by
\citet{SJ250}; the proofs use a truncated version of {\ti} above
(with a critical $\xi$),
where the spine has a fixed length $k$; at the top of the spine the special
node becomes normal and reproduces normally with $\xi$ children.
\citet{Geiger} studied $\cT$ conditioned on its height being at least $n$,
see \refS{Scondx}, and gave 
a construction of it using a spine of length $n$, but with more complicated
rules for the branches.
See also the modified trees $\hcta$, $\hctb$, $\hctc$ in \refS{SlargeT}.

The invariant random sin-tree constructed by \citet{AldousFringe}
in a more general situation, is
for a critical \GWp{} another related tree; it has an infinite spine as $\hcT$,
but differs from $\hcT$ in that the root has $\xi+1$ children (and thus
$\xi$ normal children) instead of $\hxi$.
In this case, it may be better to reverse the orientation of the spine and
consider the spine as an infinite path $\dotsm v_{-2}v_{-1}v_0$ starting at
$-\infty$ (there is thus no root); we attach further branches (copies of
$\cT$) as above, with all $v_i$, $i<0$, 
special (the number of children is a copy of $\hxi$),
but the top node $v_0$ normal (the number of children is a copy of $\xi$,
and all are normal).

\citet{KLPP} and \citet{ChDur} have constructed related trees with infinite
spines using multi-type \GWp{s}.
\end{remark}

\begin{remark}
  If $\xi$ has the \pgf{} $\gf(x)\=\E x^\xi=\sumk \pi_kx^k$, then $\hxi$ has
by \eqref{hxi} the \pgf{}
  \begin{equation}\label{pgfhxi}
\E x^{\hxi}=\sumk k\pi_kx^k
=x\gf'(x),
  \end{equation}
at least for $0\le x<1$. (Also for $\mu<1$ when $\hxi$ may take the value
$\infty$.) 
\end{remark}

\begin{remark}\label{Rgss}
  In case {\ti}, the random
  variable $\hxi$ is \as{} finite and has mean
  \begin{equation}\label{ehxi}
	\E\hxi=\sumk k\P(\hxi=k)=\sumk k^2\pi_k=\E\xi^2=\gss+1,
  \end{equation}
where $\gss\=\Var\xi\le\infty$.
In case {\tii}, we have $\P(\hxi=\infty)>0$ and thus $\E\hxi=\infty$.
This suggests that 
in results that are known in the critical case {\ti}, and where $\gss$ appears
as a parameter
(see \eg{} \refS{Sfurther}), the correct generalization of $\gss$
to the subcritical case {\tii} is not $\Var\xi$ but $\E\hxi-1=\infty$. 
(See \refR{Relz} below for a simple example.)
We thus define, for any distribution $\ppi\kooo$ with expectation $\mu\le1$,
\begin{equation}\label{hgss}
  \hgss\=
\E\hxi-1
=
  \begin{cases}
\gss, & \mu=1,\\
\infty, & \mu<1.	
  \end{cases}
\end{equation}
\end{remark}

\begin{remark}
  \label{Relz}
Let $\elz_k(T)$ denote the number of nodes with distance $k$ to the root in a
rooted tree $T$. (This is thus the size of the $k$:th generation.) Trivially,
$\elz_0(T)=1$, while $\elz_1(T)=\dx_T(o)$, the root degree.

It follows by the construction of $\hcT$ and induction that in case {\ti},
using \eqref{ehxi}, 
  \begin{equation}\label{elk0}
	\E \elz_k(\hcT)
=1+k(\E\hxi-1)=k\gss+1, 
\qquad  k\ge0.
  \end{equation}
In case {\tii}, we have if $\mu>0$ and $k\ge1$ a positive probability that $L=k$
and then $\elz_k(\hcT)=\infty$. Thus $\E\elz_k(\hcT)=\infty$.
Consequently, using \eqref{hgss}, if $0<\mu\le1$, then
  \begin{equation}\label{elk}
	\E \elz_k(\hcT)=k\hgss+1, 
\qquad  k\ge1.
  \end{equation}
However, this fails if $\mu=0$; in that case, $\elz_1(\hcT)=\infty$ but
$\elz_k(\hcT)=0$ for $k\ge2$, see \refE{Emu=0}.
\end{remark}

\begin{remark}
  \label{Rsizebias}
As said above, in the case $\mu=1$, the tree $\hcT$ is the size-biased \GWt,
see \cite{Kesten}, \cite{AldousPitman} and \cite{LPP}.
For comparison, we give the definition of the latter, for an arbitrary
distribution $ \ppi\geko$ with finite mean $\mu>0$:
Let, as above, $\xi$ have the distribution $\ppi$, see \eqref{xi}, and let
$\xix$ have the size-biased distribution defined by
\begin{equation}
  \label{c1}
\P(\xix=k)=\frac{k\pi_k}{\mu},
\qquad k=0,1,2,\dots
\end{equation}
(Note that this is a probability distribution on $\Ni$.)
Construct $\ctx$ as $\hcT$ above, with normal and special nodes, with the
only difference that the number of children of a special node has the
distribution of $\xix$ in \eqref{c1}.

In the critical case $\mu=1$, we have $\xix=\hxi$ and thus $\ctx=\hcT$, but
in the subcritical case $\mu<1$, $\ctx$ and $\hcT$ are clearly different.
(Note that $\ctx$ always is locally finite, but $\hcT$ is not when $\mu<1$.)
When $\mu>1$, $\hcT$ is not even defined, but $\ctx$ is. 
(As remarked by
\citet{AldousPitman}, in the supercritical case $\ctx$ has \as{} an
uncountable number of infinite paths from the root, in contrast to the case
$\mu\le1$ when the spine \as{} is the only one.)

$\ctx$ can also be constructed by the alternative construction in {\ti}
above starting with an infinite spine, again with the difference that
$\hxi-1$ is replaced by $\xix-1$.
$\ctx$ can also be seen as a \GWp{} with immigration in the same way as in
\refR{Rimmi}. 

By \eqref{c1}, the probability that a given special node in $\ctx$ has
$k\ge1$ children, 
with a given one of them special, is
\begin{equation}
  \label{c2}
\frac1k\P(\xix=k)=\frac{k\pi_k}{k\mu}
=\frac{\pi_k}{\mu}.
\end{equation}

Let $T$ be a fixed tree of height $\ell$, and let $u$ be a node in the
$\ell$:th (and last) generation in $T$.
Let $\ctxx\ell$ denote $\ctx$ truncated at height $\ell$.
It follows from \eqref{c2} and independence that the probability that
$\ctxx\ell=T$ and that $u$ is special (\ie, $u$ is the unique element of the
spine at distance $\ell$ from the root) equals
$\mu^{-\ell}\P(\cT\nnx\ell=T)$.
Hence, summing over the $\elz_\ell(T)$ possible $u$,
\begin{equation}
  \label{c3}
\P(\ctxx\ell=T)=
\mu^{-\ell}\elz_\ell(T)\P(\cT\nnx\ell=T),
\end{equation}
which explains the name size-biased \GWt.
(As an alternative, one can thus define $\ctx$ directly by \eqref{c3},
noting that this gives consistent distributions for $m=1,2,\dots$, see
\citet{Kesten}.) 
See further \refS{Scondh}.
\end{remark}

\section{The Ulam--Harris tree and convergence}\label{SUlam}
It is convenient, especially when discussing convergence, 
to regard our trees as subtrees of the infinite
Ulam--Harris tree defined as follows.
(See \eg{} \citet{Otter},
\citetq{\S\,VI.2}{Harris}, \citet{Neveu} and \citet{Kesten}.)  

\begin{definition}
The Ulam--Harris tree $\too$
is the  infinite rooted tree
with node set $\voo\=\bigcup_{k=0}^\infty \Ni^k$,  the set of all
finite strings $i_1\dotsm i_k$ of positive integers,
including the empty string $\emptyset$
which we take as the root $o$, and with 
an edge joining $i_1\dotsm i_k$ and $i_1\dotsm i_{k+1}$ for any $k\ge0$ and
$i_1,\dots,i_{k+1}\in\Ni$. 
\end{definition}

Thus every node $v=i_1\dotsm i_k$ has outdegree $\dx(v)=\infty$; the
children of $v$ are 
the strings $v1$, $v2$, $v3$, \dots, and we let them have this order so
$\too$ becomes
an infinite ordered rooted tree.
The parent of $i_1\dotsm i_k$ ($k>0$) is $i_1\dotsm i_{k-1}$. 

The family $\st$ of ordered rooted trees can be identified with the set of
all rooted subtrees $T$ of $\too$ that have the property
\begin{equation}
  \label{left}
i_1\dotsm i_k i \in V(T)
\implies
i_1\dotsm i_k j \in V(T)
\text{ for all } j\le i.
\end{equation}
Equivalently, 
by identifying $T$ and its node set $V(T)$,
we can regard $\st$ as the family of all subsets $V$ of $\voo$
that satisfy
\begin{align}
&\emptyset\in V, \label{voo1}
\\
&i_1\dotsm i_{k+1}\in V \implies i_1\dotsm i_{k}\in V,
\\&
i_1\dotsm i_{k}i\in V \implies i_1\dotsm i_{k}j\in V
\quad\text{for all}\quad j\le i.  \label{voo3}
\end{align}

We let $\stf\=\set{T\in\st:|T|<\infty}$ be the set of all finite
ordered rooted trees
and $\st_n\=\set{T\in\st:|T|=n}$ the set of all  ordered rooted trees
of size $n$.

If $T\in\st$, we let as above $\dx(v)=\dx_T(v)$ denote the outdegree of $v$ for
every 
$v\in V(T)$, For convenience, we also define $\dx(v)=0$ for $v\notin V(T)$;
thus $\dx(v)$ is defined for every $v\in\voo$, and the tree $T\in\st$ is
uniquely determined by the (out)degree sequence $(\dx_T(v))_{v\in\voo}$.
It is easily seen that this gives a bijection between $\st$ and the set of 
sequences $(d_v)\in\bNo^{\voo}$
with the property
\begin{equation}
  \label{dtree}
d_{i_1\dotsm i_k i}=0 \quad\text{when}\quad i>d_{i_1\dotsm d_k}.
\end{equation}

The family $\stl$ of locally finite trees corresponds to the subset of all
such sequences with all $d_v<\infty$, and the family $\stf$ of finite trees
correspond to the subset of all such sequences $(d_v)$ with all
$d_v<\infty$ and only finitely many $d_v\neq0$.

In this way we have $\stf\subset\stl\subset\st\subset\bNo^{\voo}$; 
note that 
$\stl=\st\cap\No^{\voo}$, so
$\stf\subset\stl\subset\No^{\voo}$.

We give $\bNo$ the usual compact topology as the one-point compactification
of the discrete space $\No$. Thus $\bNo$ is a compact metric space. (One metric,
among many equivalent ones, is given by the homomorphism $n\mapsto1/(n+1)$
onto $\set{1/n}_{n=1}^\infty\cup\set0\subset\bbR$.)
We give $\bNo^{\voo}$ the product topology
and its subspaces $\stf$, $\stl$ and $\st$ the induced topologies.
Thus $\bNo^{\voo}$  is a compact metric space, and its subspaces
$\stf$, $\stl$ and $\st$ are metric spaces. 
(The precise choice of metric on these spaces is irrelevant; we will not
use any explicit metric except briefly in \refS{SlargeT}.) 
Moreover, the condition
\eqref{dtree} defines $\st$ as a closed subset of $\bNo^{\voo}$; thus $\st$
  is a compact metric space. ($\stf$ and $\stl$ are not compact. In fact, it
  is easily seen that they are dense proper subsets of $\st$. 
$\stf$ is a countable discrete space.)

In other words, if $T_n$ and $T$ are trees in $\st$, 
then $T_n\to T$ if and only
if the outdegrees converge pointwise: 
\begin{equation}\label{dconv}
\dx_{T_n}(v)\to \dx_T(v) 
\qquad \text{for each $v\in\voo$}.  
\end{equation}
It is easily seen that it suffices to consider $v\in V(T)$, \ie,
\eqref{dconv} is equivalent to
\begin{equation}\label{dconvt}
\dx_{T_n}(v)\to \dx_T(v) 
\qquad \text{for each $v\in  V(T)$},  
\end{equation}
since \eqref{dconvt} implies that if $v\notin V(T)$, then $v\notin V(T_n)$
for sufficiently large $n$, and thus $\dx_{T_n}(v)=0$. (Consider the last
node $w$ 
in $V(T)$ on the path from the root to $v$ and use $\dx_{T_n}(w)\to \dx_T(w)$.)

Alternatively, 
we may as above consider the node set
$V(T)$ as a subset of $\voo$ 
and regard $\st$ as the family of all subsets of $\voo$ that
satisfy \eqref{voo1}--\eqref{voo3}.
We identify the family of all subsets of $\voo$ with $\setoi^\voo$, and give
this family the product topology, making it into a compact metric space.
(Thus, convergence means  convergence of the indicator
$\ett{v\in\cdot}$ for each $v\in\voo$.)
This induces a topology on $\st$, where $T_n\to T$ means that,
for each $v\in\voo$,
if $v\in V(T)$, then $v\in V(T_n)$ for all large $n$, and, conversely,
if $v\notin V(T)$, then $v\notin V(T_n)$ for all large $n$.

If $v=i_1\dots i_k$ with $k>0$, then $v\in V(T)$ if and only of 
$i_k\le \dx_T(i_1\dots i_{k-1})$. 
It follows immediately that $V(T_n)\to V(T)$ in the
sense just described, if and only if \eqref{dconv} holds. The two
definitions of $T_n\to T$ above are thus equivalent (for $\st$, and thus
also for its subsets $\stf$ and $\stl$).

Furthermore, we see, \eg{} from \eqref{dconv}, that the convergence of trees
can be described recursively:
Let $T_{(j)}$ denote the $j$:th subtree of $T$, \ie, the subtree rooted at the
$j$:th child of $T$, for $j=1,\dots,\dx_T(o)$. (We consider only finite $j$,
even when $\dx_T(o)=\infty$.)
Then,
$T_n\to T$ if and only if 
\begin{romenumerate}
\item 
the root degrees converge: $\dx_{T_n}(o)\to \dx_T(o)$, 
and further, 
\item 
for each $j=1,\dots,\dx_T(o)$, $T_{n,(j)}\to T_{(j)}$.
\end{romenumerate}
 (Note that $T_{n,(j)}$ is defined for large $n$, at least, by (i).)

It is important to realize that the notion of convergence used here is a
local (pointwise) one, so we consider only a single $v$ at a time, or,
equivalently, 
a finite set of $v$; there is no uniformity in $v$ required.

If $T$ is a locally finite tree, $T\in\stl$, then $\dx_T(v)<\infty$ for each
$v$, and thus \eqref{dconv} means that for each $v$, $\dx_{T_n}(v)=\dx(v)$ for
all sufficiently large $n$.

Let $T\xmm$ denote the tree $T$ truncated at height $m$, \ie, the subtree of
$T$ consisting of all nodes in generations $0,\dots,m$. If 
$T$ is locally finite,
then each $T\xmm$ is a finite tree, and it is easily seen 
from \eqref{dconvt} that convergence
to $T$ can be characterised as follows:
\begin{lemma}\label{LClf}
If $T$ is locally finite, then, for any trees $T_n\in\st$,
\begin{equation*}
  \begin{split}
  T_n\to T
&\iff T_n\xmm \to T\xmm
\quad\text{for each $m$}
\\&
\iff T_n\xmm = T\xmm
\quad \text{for each $m$ and all large $n$}. 
  \end{split}
\end{equation*}
(The last condition means for $n$ larger than some $n(m)$ depending on $m$.)
\nopf
\end{lemma}
This notion of convergence for locally finite trees is widely used; see
\eg{} \citet{Otter} and \citet{AldousPitman}.

In general, if $T$ is not locally finite, this characterization fails.
(For example, if $S_n$, $1\le n\le\infty$, is a star where the root has
outdegree $n$ and its 
children all have outdegree 0, then $S_n\to S_\infty$, but $S_n\xmm\neq
S_\infty\xmm$ for all $n$ and $m\ge1$.) Instead, we have to localise also
horizontally: 
Let $V\xxmm\=\bigcup_{k=0}^m\set{1,\dots,m}^k$, the subset of $\voo$
consisting of strings of length at most $m$ with all elements at most $m$.
For a tree $T\in\st$, let
$T\xxmm$ be the subtree with node set $V(T)\cap V\xxmm$, \ie, the tree $T$
truncated at height $m$ and pruned so that all outdegrees are at most $m$.
It is then easy to see from \eqref{dconv} that the following analogue and
generalization of \refL{LClf} holds:
\begin{lemma}\label{LC}
  For any trees $T,T_n\in\st$,
\begin{equation*}
  \begin{split}
  T_n\to T
&\iff T_n\xxmm \to T\xxmm
\quad\text{for each $m$}
\\&
\iff T_n\xxmm = T\xxmm
\quad \text{for each $m$ and all large $n$}. 
  \end{split}
\end{equation*}
(The last condition means for $n$ larger than some $n(m)$ depending on $m$.)
\nopf
\end{lemma}
Our notion of convergence for general trees $T\in\st$ was introduced in this
form by
\citet{sdf} (where the truncation $T\xxmm$ is called a \emph{left ball}).

\begin{remark}\label{RLC}
It is straightforward to obtain versions of Lemmas \ref{LClf}--\ref{LC} for
random trees $T$, $T_n$ and convergence in probability or distribution.
For example:
For any random trees $T,T_n\in\st$,
\begin{equation}\label{job}
  T_n\dto T
\iff T_n\xxmm \dto T\xxmm
\quad\text{for each $m$}.
\end{equation}
If $T\in\stl$, a.s., then we also have
\begin{equation}\label{jobl}
  T_n\dto T
\iff T_n\xmm \dto T\xmm
\quad\text{for each $m$},
\end{equation}
see \eg{} \citet{AldousPitman}.
The proofs are standard using the methods in \eg{} \citet{Billingsley}.
\end{remark}

\section{Main result for \sgrt{s}}\label{Smain} 

Our main result for trees is the following, proved in \refS{Spfmain}.
The case when $\nu\ge1$ was shown implicitly by \citet{Kennedy} (who
considered \GWp{es} and not trees), and explicitly by \citet{AldousPitman},
see also \citet{Grimmett80}, \citet{Kolchin},
\citet{Kesten} and \citet{AldousII}.
Special cases with $0<\nu<1$ and $\nu=0$ are given by \citet{sdf} 
and
\citet{SJ259},
respectively.

\begin{theorem}\label{Tmain}
  Let $\wwx=(w_k)_{k\ge0}$ be any weight sequence with $w_0>0$ and $w_k>0$
  for some $k\ge2$.
  \begin{romenumerate}[-10pt]
  \item \label{tmaincritical}
If $\nu\ge1$, let $\tau$ be the unique number in $[0,\rho]$ such that
$\Psi(\tau)=1$.
\item \label{tmainsub}
If $\nu<1$, let $\tau\=\rho$.	
  \end{romenumerate}
In both cases, $0\le\tau<\infty$ and $0<\Phi(\tau)<\infty$.
Let 
\begin{equation}\label{pk}
\pi_k\=\frac{\tau^kw_k}{\Phi(\tau) },
\qquad k\ge0;
\end{equation}
then
$\ppi\geko$ is a probability distribution, 
with expectation
\begin{equation}\label{mainmu}
\mu= \Psi(\tau)=\min(\nu,1)\le1
\end{equation}
and variance $\gss=\tau\Psi'(\tau)\le\infty$.
Let $\hcT$ be the infinite
modified \GWt{} constructed in \refS{ShGW} for the distribution $\ppi\geko$.
Then $\ctn\dto\hcT$ as \ntoo, in the topology defined in \refS{SUlam}.

Furthermore, in case \ref{tmaincritical}, $\mu=1$ (the critical case)
and $\hcT$ is locally finite with an infinite spine;
in case \ref{tmainsub}
$\mu=\nu<1$ (the subcritical case) and
$\hcT$ has a finite spine ending with an explosion.
\end{theorem}

\begin{remark}\label{Rmain}
Note that we can combine the two cases $\nu\ge1$ and $\nu<1$ and define,
using \refL{LPsi} and with $\Psi(\rho)=\nu$,
  \begin{equation}\label{tau}
  \tau\=\max\Bigset{t\le\rho:\Psi(t)\le1}.	
  \end{equation}
\end{remark}

\begin{remark}
  In case \ref{tmainsub}, there is no $\tau\ge0$ with $\Psi(\tau)=1$, see
  \refL{LPsi}. Hence the definition of $\tau$ can also be expressed as
  follows, recalling $\Psi(t)\=t\Phi'(t)/\Phi(t)$ from \eqref{psidef}:
$\tau$ is the unique number in $[0,\rho]$ such that
\begin{equation}\label{tau?}
  \tau\Phi'(\tau)=\Phi(\tau),
\end{equation}
if there exists any such $\tau$; otherwise $\tau\=\rho$.
(Equation \eqref{tau?} is used in many papers to define $\tau$, in the case
$\nu\ge1$.) 
\end{remark}

\begin{remark}  \label{R6.3.1}
If $0<t<\rho$, then
\begin{equation*}
  \frac{\dd}{\dd t}\parfrac{\Phi(t)}{t}
=\frac{t\Phi'(t)-\Phi(t)}{t^2}
=\frac{\Phi(t)}{t^2}\bigpar{\Psi(t)-1}.
\end{equation*}
Since $\Psi(t)$ is increasing by \refL{LPsi}, it follows that $\Phi(t)/t$
decreases on $[0,\tau]$ and increases on $[\tau,\rho]$, so $\tau$ can,
alternatively, be characterised as the (unique) minimum point in $[0,\rho]$
of the convex function $\Phi(t)/t$, \cf{} \eg{} \citet{Minami} and \citet{sdf}.
Consequently,
\begin{equation}
  \label{annaw}
\frac{\Phi(\tau)}{\tau}
=\inf_{0\le t\le\rho}\frac{\Phi(t)}{t}
=\inf_{0\le t<\infty}\frac{\Phi(t)}{t}.
\end{equation}
(This holds also when $\rho=0$, trivially, since then $\Phi(t)/t=\infty$ for
every $t\ge0$.)
\end{remark}

\begin{remark}\label{Rotter}
By \refR{R6.3.1},
$\tau$ is, equivalently, the (unique) maximum point in $[0,\rho]$
of $t/\Phi(t)$, which by \eqref{cz} is the inverse function of the
generating function $\cZ(z)$. It follows easily that
\begin{equation}\label{otter}
\tau=  \cZ(\rhoz),
\end{equation}
where $\rhoz=\tau/\Phi(\tau)$ is the radius of convergence of $\cZ$; see
also \refC{CZrho}. 
Note that $0\le\rhoz<\infty$ and that $\rhoz=0\iff\tau=0\iff\rho=0$.
\citet{Otter} uses \eqref{otter} as the definition of $\tau$ (by him denoted
$a$); see also \citet{Minami}. 
\end{remark}

\begin{remark}
  \label{Rnu=0}
When $\nu=0$ (which is equivalent to $\rho=0$), the limit $\hcT$ is
the non-random infinite star in \refE{Emu=0}, so \refT{Tmain} gives
$\ctn\pto\hcT$. 
\end{remark}

\begin{remark}
We consider briefly the cases excluded from \refT{Tmain}.
The case when $w_0=0$ is completely trivial, since then $w(T)=0$ for every
finite tree, so $\ctn$ is undefined. 
The same holds (for $n\ge2$)
when $w_0>0$ but $w_k=0$ for all $k\ge1$, \ie, when 
$\go=0$.

The case when $w_0>0$ and $w_1>0$
but $w_k=0$ for $k\ge2$, so $\go=1$, is also trivial. 
Then $w(T)=0$ unless $T$ is a rooted path $P_n$ for some $n$.
Thus $Z_n=w(P_n)=w_0w_1^{n-1}$, and (a.s.)\ $\ctn=P_n$, which converges as
\ntoo{} to the infinite path $P_\infty$.
We have $\nu=1=\go$, but,
in contrast to \refT{Tmain},
$\tau=\infty$, with $\tau$ defined \eg{} by \eqref{tau}.
Further,
interpreting \eqref{pk} as a limit, we have $\pi_k=\gd_{k1}$, 
so $\ppi$ is the
distribution concentrated at 1;
thus \eqref{hxi} yields $\hxi=1$ a.s., so $\hcT$
consists of an infinite spine only, \ie{} $\hcT=P_\infty$.
Consequently, $\ctn\dto\hcT$ holds in this case too.
\end{remark}

\begin{remark}
If we replace $(w_k)$ by the  equivalent \ws{} $(\tw_k)$ given by
\eqref{tw}, then \eqref{tau} and \eqref{tPsi} show that $\tau$ is replaced by
\begin{equation}
  \ttau\=\max\set{t\le\trho:\tPsi(t)\le1}
=\max\set{t\le\rho/b:\Psi(bt)\le1}=\tau/b.
\end{equation}
The corresponding \pws{} given by \eqref{pk} thus is,
using \eqref{tPhi},
\begin{equation}
  \tpi_k\=\frac{\ttau^k\tw_k}{\tPhi(\ttau)}
=\frac{(\tau/b)^k ab^kw_k}{a\Phi(\tau)}
=\frac{\tau^k w_k}{\Phi(\tau)}=\pi_k,
\end{equation}
so the distribution $\ppi$ is  invariant and depends only on the
equivalence class of $(w_k)$.  
\end{remark}

\begin{remark}\label{Rqk}
If $\rho>0$, then $\tau>0$ and the distribution \ppi{} is a \pws{}
equivalent to $(w_k)$.
There are other equivalent \pws{s}, 
see \refL{Lequivp},
but \refT{Tmain} and the theorems below
show that \ppi{} has a special role and therefore is a
canonical choice of a weight sequence in its equivalence class.
\refR{Rpt} shows that \ppi{}
is the unique probability distribution with mean 1 that is equivalent
to $(w_k)$, if any such distribution exists. If no such distribution exists
but $\rho>0$, then $(\pi_k)$ is the probability distribution equivalent to
$(w_k)$ 
that has the maximal mean.

A heuristic motivation for this choice of \pws{} is that when we construct
$\ctn$ as a \GWt{} $\cT$ conditioned on $|\cT|=n$, it is better to condition
on an event of not too small probability; in the critical case this
probability decreases as $n\qqcw$ provided $\gss<\infty$,
see \cite{Otter} ($\nu>1$) and \cite[Theorem 2.3.1]{Kolchin} ($\nu\ge1$,
$\gss<\infty$), and always subexponentially,
but in the subcritical and supercritical cases it typically decreases
exponentially fast, see Theorems \refand{Tpn}{TH1}.
\end{remark}

As a special case of \refT{Tmain}
we have the following result for the root degree
$\dx_{\ctn}(o)$, proved in \refS{Stree-balls}.

\begin{theorem}
  \label{Troot}
Let $\ww\geko$ and $\ppi\geko$ be as in \refT{Tmain}.
Then, as \ntoo,
\begin{equation}\label{troot1}
\P(\dx_{\ctn}(o)=d) \to d\pi_d, \quad d\ge0.
\end{equation}
Consequently, regarding $\dx_{\ctn}(o)$ as a random number in
$\bNo$,
\begin{equation}\label{troot2}
\dx_{\ctn}(o)\dto \hxi,   
\end{equation}
where $\hxi$ is a random variable in $\bNo$ with the distribution given in
\eqref{hxi}. 
\end{theorem}
Note that the sum $\sum_0^\infty d\pi_d=\mu$ of the limiting probabilities
in \eqref{troot1} may be less than 1; in that case we do not have convergence to
a proper finite random variable, which is why
we regard $\dx_{\ctn}(o)$ as a random number in
$\bNo$.

\refT{Troot}  describes the degree of the root. If we instead take a random
node, we obtain a different limit distribution, \viz{} \ppi.
We state two versions of this;
the two results are of the types called
\emph{annealed} and \emph{quenched} in statistical physics.
In the first (annealed) version, we take a random tree $\ctn$
and, simultaneously, a random node $v$ in it.
In the second (quenched) version 
we fix a random tree $\ctn$ and study the distribution of outdegrees in it.
(This yields a random probability distribution.
Equivalently, we study the outdegree of a random node conditioned on the
tree $\ctn$.) 

\begin{theorem}\label{Tdegree}
Let $\ww\geko$ and $\ppi\geko$ be as in \refT{Tmain}.
\begin{romenumerate}
\item \label{annealed}
Let $v$ be a uniformly random node in $\ctn$. 
Then, as \ntoo,
\begin{equation}\label{tdeg}
\P(\dx_{\ctn}(v)=d) \to \pi_d, \quad d\ge0.
\end{equation}

\item \label{quenched}
  Let $N_d$ be the number of nodes in $\ctn$ of outdegree $d$.
Then
\begin{equation}\label{tdegq}
  \frac{N_d}{n}\pto \pi_d, \quad d\ge0.
\end{equation}
\end{romenumerate}
\end{theorem}

The proof is given in \refS{SpfXXX}.
(When $\nu>1$, this was proved by \citet{Otter}, see also \citet{Minami}.)
See \refS{SSnormal} for further results.

Instead of considering just the outdegree of a random node, \ie, its number
of children, we may obtain a stronger result by considering the subtree
containing its children, 
grandchildren and so on. 
(This random subtree is called a \emph{fringe subtree} by
\citet{AldousFringe}.) 
We have an analogous result, also proved in
\refS{SpfXXX}.  
Cf.\ \cite{AldousFringe}, which in particular contains \ref{subannealed}
below in the case $\nu\ge1$ and $\gss<\infty$; this was extended by
\citet{BenniesK} to the general case $\nu\ge1$.
(Note that the limit distribution, \ie{} the distribution of $\cT$, is a
fringe distribution in the sense of \cite{AldousFringe} only if $\mu=1$,
\ie, if and only if $\nu\ge1$.)

\begin{theorem}
  \label{Tsubtree}
Let $\ww\geko$ and $\ppi\geko$ be as in \refT{Tmain}, and
let $\cT$ be the \GWt{} with offspring distribution \ppi.
Further, if $v$ is a node in $\ctn$,
let $\ctnv$ be the subtree rooted at $v$.
\begin{romenumerate}
\item \label{subannealed}
Let $v$ be a uniformly random node in $\ctn$.
Then, $\ctnv\dto\cT$, \ie,
for any fixed tree $T$,
\begin{equation}\label{ttree}
\P(\ctnv=T) \to \P(\cT=T).
\end{equation}

\item \label{subquenched}
Let $T$ be an ordered 
rooted tree and let $N_T\=|\set{v:\ctnv=T}|$ be the number of
nodes in $\ctn$ such that the subtree rooted there equals $T$.
Then
\begin{equation}\label{ttreeq}
  \frac{N_T}{n}\pto \P(\cT=T).
\end{equation}
\end{romenumerate}
\end{theorem}

\begin{remark}
  \citet{AldousFringe} considers also the tree obtained by a random
  re-rooting of $\ctn$, \ie, the tree obtained by declaring a uniformly
  random node $v$ to be the root. Note that this re-rooted tree contains
 $\ctnv$ as a subtree, and that, provided $v\neq\rot$, 
there is exactly one branch from the new root
 not in this subtree, \viz{} the branch starting with the original parent
  of $v$.
\citet{AldousFringe} shows, at least when $\nu\ge1$ and $\gss<\infty$,
convergence of this randomly re-rooted tree to the random sin-tree 
in \refR{Rsin}. The limit of the re-rooted tree
is thus very similar to the limit of $\ctn$
in \refT{Tmain},
but not identical to it.
\end{remark}

\section{Three different types of weights}\label{S3} 

Although \refT{Tmain} has only two cases, it makes sense to treat the case
$\rho=0$ separately. We thus have the following three (mutually exclusive)
cases for the \ws{} $(w_k)$:
\begin{enumerate}
\item[I.] 
$\nu\ge1$. Then $0<\tau<\infty$ and $\tau\le\rho\le\infty$.
The weight sequence $(w_k)$ is equivalent to $(\pi_k)$, which is a probability
distribution 
with mean $\mu=\Psi(\tau)=1$ and \pgf{} $\sumk \pi_kz^k$ with radius of convergence
$\rho/\tau\ge1$. 

\item[II.]
$0<\nu<1$. Then $0<\tau=\rho<\infty$. The weight sequence
$(w_k)$ is equivalent to $(\pi_k)$, which is a probability distribution
with mean $\mu=\Psi(\tau)<1$ and
\pgf{} $\sumk \pi_kz^k$ with radius of convergence $\rho/\tau=1$.

\item[III.]
$\nu=0$. Then $\tau=\rho=0$, and
$(w_k)$ is not equivalent to any probability distribution.
\end{enumerate}

If we consider the modified \GWt{} in \refT{Tmain},
then III is the case discussed in \refE{Emu=0}; 
excluding this case, I and II are
the same as {\ti} and {\tii} in \refS{ShGW}.

We can reformulate the partition into three cases in more probabilistic terms.
If $\xi$ is a non-negative integer valued random variable with distribution 
given by $p_k=\P(\xi=k)$, $k\ge0$, then the 
\emph{exponential moments} of $\xi$ are $\E R^\xi=\sumk p_kR^k$ for $R>1$.
(Equivalently, $\E e^{r\xi}$ for $r\=\log R>0$.)
We say that $X$, or the distribution $(p_k)$, has 
\emph{some finite exponential moment} if $\E R^X<\infty$ for some $R>1$; 
this is equivalent to the \pgf{} $\sumk p_kz^k$ having radius of convergence
strictly larger than $1$. 

Consider again a probability distribution $(\tw_k)$ equivalent to $(w_k)$,
with $\tw_k=t^kw_k/\Phi(t)$ for some $t\le\rho$. 
By \refS{Sequiv}, the radius of convergence of the \pgf{} $\tPhi(z)$ of this
distribution is $\rho/t$, \cf{} \eqref{tPhi}.
Hence, the distribution
$(\tw_k)$ has some finite exponential moment if and only if $0<t<\rho$.
The cases I--III can thus be described as follows:

\begin{enumerate}
\item[I.] 
$\nu\ge1$. 
Then $(w_k)$ is equivalent to a probability distribution with
mean $\mu=1$ (with or without some exponential moment). 
Moreover, $(\pi_k)$ in \eqref{pk} is the unique
such distribution.

\item[II.]
$0<\nu<1$. 
Then $(w_k)$ is equivalent to a probability distribution with
mean $\mu<1$ and no finite exponential moment. 
Moreover, $(\pi_k)$ in \eqref{pk} is the unique such distribution.

\item[III.]
$\nu=0$. Then 
$(w_k)$ is not equivalent to any probability distribution.
\end{enumerate}

Case I may be further subdivided. From an analytic point of view, it is
natural to split I into two subcases:

\begin{enumerate}
\item[Ia.] 
$\nu>1$; equivalently, $0<\tau<\rho\le\infty$.
The weight sequence $(w_k)$ is equivalent to $(\pi_k)$, which is a probability
distribution with mean $\mu=1$ and
\pgf{} $\sumk \pi_kz^k$ with radius of convergence $\rho/\tau>1$.
In other words, $(w_k)$ is equivalent to a probability distribution with
mean $\mu=1$ and some finite exponential moment. (Then $(\pi_k)$ is the unique
such distribution.)
By  \eqref{otter}, the condition 
can also be written
analytically as
$\cZ(\rhoz)<\rho$, a version used \eg{} in \cite{DurhuusJW}.
(This case is called \emph{generic} in \cite{DurhuusJW} and \cite{sdf}.)

\item[Ib.] 
$\nu=1$; then $0<\tau=\rho<\infty$.
The weight sequence $(w_k)$ is equivalent to $(\pi_k)$, which is a probability
distribution with mean 1 and
\pgf{} $\sumk \pi_kz^k$ with radius of convergence $\rho/\tau=1$.
In other words, $(w_k)$ is equivalent to a probability distribution with
mean $\mu=1$ and no finite exponential moment. 
(Then $(\pi_k)$ is the unique such distribution.)
\end{enumerate}

Case Ia is convenient when using analytic methods, since it says that the
point $\tau$ is strictly inside the domain of convergence of $\Phi$, which
is convenient for methods involving contour integrations in the complex
plane. 
(See \eg{} \citet{Drmota} for several such results of different types.) 
For that reason, many papers using such methods  consider
only case Ia. However, it has repeatedly turned out, for many different
problems, that results proved by
such methods often hold, by other proofs,
assuming only that we are in case I with finite
variance of $(\pi_k)$.
(In fact, as shown in \cite{SJ167}, it is at least sometimes possible to use
complex analytic methods also in the case when $\tau=\rho$ and $(\pi_k)$ has a
finite second moment.)
Consequently, it is often more important to partition case I into the
following two cases:
\begin{enumerate}
\item[\Iga.] 
$\nu\ge1$ and 
\ppi{} has variance $\gss<\infty$.
In other words, 
$(w_k)$ is equivalent to a probability distribution $(\pi_k)$ with
mean $\mu=1$ and finite second moment $\gss$. 

\item[\Igb.] 
$\nu=1$ and $\ppi$ has variance $\gss=\infty$.
In other words, $(w_k)$ is equivalent to a probability distribution with
mean $\mu=1$ and infinite variance.
\end{enumerate}

Note that Ia is a subcase of \Iga, since a finite exponential moment implies
that the second moment is finite.

When $\nu\ge1$, the quantity $\gss$ is another natural parameter of the
weight sequence $(w_k)$, which frequently occurs in asymptotic results,
see \eg{} \refS{Sfurther}.
(When $\nu<1$, the natural analogue is $\infty$,  see \refR{Rgss}.)
By \refT{Tmain} (or \eqref{lep4}),  $\gss=\tau\Psi'(\tau)$, so (assuming
$\nu\ge1$), 
we have case \Iga{} when $\Psi'(\tau)<\infty$ and
\Igb{} when $\Psi'(\tau)=\infty$.
Moreover, when $\nu\ge1$, then $(\pi_k)$ has mean $\mu=1$, 
and it follows from \eqref{lep2} that the variance $\gss$ of $\ppi$ 
also is given by the  formula \cite{AldousII}
\begin{equation}\label{gss}
  \gss=\Phi_\tau''(1)+\mu-\mu^2
=\Phi_\tau''(1)
=\frac{\tau^2\Phi''(\tau)}{\Phi(\tau)}.
\end{equation}
Hence \Iga{} is the case $\nu\ge1$ and $\Phi''(\tau)<\infty$; equivalently,
either $\nu>1$ or $\nu=1$ and $\Phi''(\rho)<\infty$.

\begin{remark}
We have seen that except in case III, we may without loss of generality
assume that the weight $(w_k)$ is a probability weight sequence. If this
distribution is critical, \ie{} has mean 1, we are in case I with
$\pi_k=w_k$, so we do not have to change the weights.

If the distribution $(w_k)$ is supercritical, then $\nu>1$ and we are in
case Ia; we can change to an equivalent critical probability weight.
Hence we never have to consider supercritical weights. 
(Recall that by \refR{Rpt}, $\nu$ is the supremum of the means of the equivalent
\pws{s}.) 

If the distribution $(w_k)$ is subcritical, we can only say that we are in
case I or II. We can often change to an equivalent critical probability
weight, but not always. 
\end{remark}

\section{Examples of \sgrt{s}}\label{Sex}

One of the reasons for the interest in simply generated trees 
is that many kinds of random trees occuring in various applications
can be seen as simply generated random trees and \cGWt.
We give some important examples here, see further 
\citet{AldousI,AldousII}, 
\citet{Devroye} and \citet{Drmota}.

We see from \refT{Tmain} and \refS{S3} that any simply generated random tree
defined by a \ws{} with $\rho>0$ can be defined by an equivalent \pws, and
then the tree is the corresponding \cGWt. Moreover,
the \pws{} $\ppi$ defined in \eqref{pk} is the canonical choice of offspring
distribution. 
Recall that $\ppi$ is characterised by having mean 1, whenever this is
possible (\ie, in case I), \ie, we prefer to have critical \GWt{s}.

\begin{example}[ordered trees]
  \label{Euniform}
The simplest example is to take $w_k=1$ for every $k\ge0$. Thus every tree
has weight 1, and $\ctn$ is a uniformly random ordered rooted tree with $n$
nodes. 
Further, $Z_n$ is the number of such trees; thus $Z_n$ is the Catalan number
$C_{n-1}$, see
\refR{Rcatalan} 
and \eqref{catalan}.
(For this reason, these random trees are sometimes called \emph{Catalan trees}.)
 
We have 
\begin{equation}\label{euphi}
  \Phi(t)=\sumk t^k=\frac{1}{1-t}
\end{equation}
and 
\begin{equation}\label{eupsi}
  \Psi(t)=\frac{t\Phi'(t)}{\Phi(t)}=\frac{t}{1-t}.
\end{equation}
Thus $\rho=1$ and $\nu=\infty$ (\cf{} \refL{LPsi}\ref{L1b}),
and $\Psi(\tau)=1$ yields $\tau=1/2$.
Hence \eqref{pk} yields the canonical \pws{}
\begin{equation}\label{eupi}
\pi_k=2^{-k-1}, 
\qquad k\ge0. 
\end{equation}
In other words, 
the uniformly random ordered rooted tree  
is the \cGWt{} with geometric offspring distribution  $\xi\sim\Ge(1/2)$.
(This is the geometric distribution with mean 1.
Any other geometric distribution yields an equivalent \ws, and thus the
same \cGWt.)

The size-biased random variable $\hxi$ in \eqref{hxi} has the distribution
\begin{equation}
\P(\hxi=k)=k\pi_k=k2^{-k-1}, 
\qquad k\ge1;
\end{equation}
thus $\hxi-1$ has a negative binomial distribution $\NBi(2,1/2)$.
It follows that in the infinite tree $\hcT$, if $v$ is a node on the spine
(for example the root) and $\dl(v),\dr(v)$ are the numbers of children of it
to the left and right of the spine, respectively, then
\begin{equation}
  \begin{split}
  \P\bigpar{\dl(v)=j\text{ and }\dr(v)=k}
&=\frac{1}{j+k+1}\P(\hxi=j+k+1)
=2^{-j-k-2}
\\&=2^{-j-1}\cdot2^{-k-1},
\qquad j,k\ge0;
  \end{split}
\raisetag\baselineskip
\end{equation}
thus $\dl(v)$ and $\dr(v)$ are independent and both have the same 
distribution $\Ge(1/2)$ as $\xi$.

We have $\gss\=\Var\xi=\tau\Psi'(\tau)=2$, see \refT{Tmain} and \eqref{gss}, and
$\E\hxi=\gss+1=3$, see \eqref{ehxi}.
\end{example}

\begin{example}[unordered trees]
  \label{Ecayley}
We have assumed that our trees are ordered, but it is possible to consider
unordered labelled rooted trees too by imposing a random order on the set of
children of each node. 
Note first that for ordered trees, the ordering of the children implicitly
yields a labelling of all nodes as in \refS{SUlam}. Hence, 
any ordered tree with $n$ nodes can be explicitly labeled
by $1,\dots,n$ in exactly $n!$ ways, and a uniformly random labelled ordered
rooted 
tree is the same as a uniformly random unlabelled ordered rooted tree with a
random 
labelling. (For unordered trees, a uniformly random labelled tree is
different from a uniformly random unlabelled   tree;
unlabelled unordered trees are not \sgt{s}.)

An unordered labelled rooted tree with outdegrees $d_i$ corresponds to
$\prod_i d_i!$ different ordered labelled rooted trees.
If we take $w_k=1/k!$, we give each of these ordered trees weight 
$\prod_i d_i!\qw$, so their total weight is 1.
Hence, the \ws{} $(1/k!)$ yields a uniformly random unordered labelled
rooted tree.

The number of unordered labelled \emph{unrooted} trees with $n$ nodes is
$n^{n-2}$, see \eg{} \cite[Section~5.3]{Stanley2},
a result given by \citet{Cayley} and known as Cayley's formula.
(Although attributed by Cayley to \citet{Borchardt} and
even earlier found by
\citet{Sylvester}, see \eg{} \cite[p.~66]{Stanley2}.)
Equivalently,
the number of unordered labelled rooted trees with $n$ nodes is $n^{n-1}$.
Hence random such trees are sometimes called
\emph{Cayley trees}. However, this name is also used for regular infinite
trees.

We have 
\begin{equation}\label{ecayleyphi}
  \Phi(t)=\sumk \frac{t^k}{k!}=e^t
\end{equation}
and 
\begin{equation}\label{ecayleypsi}
  \Psi(t)=\frac{t\Phi'(t)}{\Phi(t)}=t.
\end{equation}
Thus $\nu=\infty$ 
and $\Psi(\tau)=1$ yields $\tau=1$.
Hence \eqref{pk} yields the canonical \pws{}
\begin{equation}
\pi_k=\frac{e^{-1}}{k!}, 
\qquad k\ge0. 
\end{equation}
In other words, 
the uniformly random labelled unordered rooted tree  
is the \cGWt{} with Poisson offspring distribution $\xi\sim\Po(1)$.
(Any other Poisson distribution yields an equivalent \ws, and thus the same
\cGWt.) 

The size-biased random variable $\hxi$ in \eqref{hxi} has the distribution
\begin{equation}
\P(\hxi=k)=k\pi_k=\frac{e^{-1}}{(k-1)!}, 
\qquad k\ge1;
\end{equation}
thus $\hxi-1$ has also the Poisson distribution $\Po(1)$, \ie,
$\hxi-1\eqd\xi$.
(It is only for a Poisson distribution that $\hxi-1\eqd\xi$.)

We have $\gss\=\Var\xi=\tau\Psi'(\tau)=1$ and
$\E\hxi=\gss+1=2$, \cf{} \eqref{gss} and \eqref{ehxi}.

The partition function is given by
\begin{equation}\label{borel1}
  Z_n(\ppix)=\P(|\cT|=n)=\frac{n^{n-1}e^{-n}}{n!}.
\end{equation}
This is a special case of the \emph{Borel distribution} in \eqref{borel}
below; 
\citet{Borel} proved a result equivalent to \eqref{borel1} for 
a queueing problem,
see also 
\citet{Otter},
\citet{Tanner},
\citet{Dwass},
\citet{Takacs:ballots},
\citet{Pitman:enum},
\refE{Eforest}
and \refT{TZ} below.
Equivalently, using \eqref{tz},
\begin{equation}\label{borel2}
Z_n(\wwx)=e^n  Z_n(\ppix)=\frac{n^{n-1}}{n!}.
\end{equation}
Recall that $Z_n$ is defined by the sum \eqref{zn} over unlabelled ordered
rooted trees; 
if we sum over \emph{labelled} ordered rooted trees, we obtain $n!\,Z_n$,
which by the argument above corresponds to weight 1 on each labelled
unordered rooted tree; \ie, the number of labelled unordered rooted trees is
$n!\,Z_n(\wwx)=n^{n-1}$.
Thus \eqref{borel2} is equivalent to Cayley's formula
for the number of unordered trees given above.

By \eqref{borel2}, the generating function $\cZ(z)$ is $\sumni
n^{n-1}z^n/n!$, known as the the \emph{tree function}; see
\eqref{treefn}--\eqref{T'} in \refE{Eforest}.
\end{example}

\begin{example}[binary trees I]
  \label{Ebinary}
The namn \emph{binary tree} is used in (at least) two different, but
related, meanings. The first version 
(Drmota \cite[Section 1.2.1]{Drmota}), sometimes called 
\emph{full binary tree}
or \emph{strict binary tree},
is an ordered rooted tree where
every node has outdegree 0 or 2.
We obtain a uniformly random full binary tree by taking the \ws{} with
$w_0=w_2=1$, 
and $w_k=0$ for $k\neq0,2$. Note that this \ws{} has span 2; this is the
standard example of a \ws{} with span $>1$. As a consequence, a full binary
tree of 
size $n$ exists only if $n$ is odd. 
(This is easily seen directly; see \refC{Cexists} for a general result.)

We have 
\begin{equation}
  \Phi(t)=1+t^2
\end{equation}
and 
\begin{equation}
  \Psi(t)=\frac{t\Phi'(t)}{\Phi(t)}=\frac{2t^2}{1+t^2}.
\end{equation}
Thus $\rho=\infty$, $\nu=2$ (\cf{} \refL{LPsi}\ref{L1c}),
and $\Psi(\tau)=1$ yields $\tau=1$.
Hence \eqref{pk} yields the canonical \pws{}
\begin{equation}\label{fullbpi}
\pi_k=\tfrac12,
\qquad k=0,2. 
\end{equation}
In other words, the random full binary tree is the \cGWt{} with offspring
distribution $\xi=2X$ where $X\sim\Be(1/2)$. (In the \GWt{} $\cT$, thus each
node gets either twins 
or no children, each outcome with probability $1/2$.)

The size-biased random variable $\hxi$ has 
$\P(\hxi=2)=1$
by \eqref{fullbpi} and
\eqref{hxi},
so $\hxi=2$  and $\hxi-1=1$ a.s.

We have $\gss\=\Var\xi=1$ and
$\E\hxi=\gss+1=2$, \cf{} \eqref{gss} and \eqref{ehxi}.
\end{example}

\begin{example}[binary trees II]\label{Ebinary2}
The second version of a \emph{binary tree}
(Drmota \cite[Example 1.3]{Drmota})
is a rooted tree where
every node has at most one \emph{left child} and at most one \emph{right
  child}. Thus, each outdegree is 0, 1 or 2; if there are two children they
are ordered, and, moreover, if there is only one child, it is marked as
either left or right.
(There is a one-to-one correspondence between binary trees of this type with
$n$ nodes and the full binary trees  in \refE{Ebinary} with $2n+1$ nodes,
mapping a binary tree $T$ to a full binary tree $T'$, where 
$T'$ is obtained from $T$
by adding $2-d$ \emph{external nodes} at every node with outdegree $d$;
conversely, we obtain $T$ 
by deleting all leaves in $T'$ and keeping only the nodes that have
outdegree $2$ in $T'$ (the \emph{internal} nodes).)

Since there are two types of nodes with outdegree 1, we obtain the correct
count of these binary trees, and a uniformly distributed random binary tree,
by taking the \ws{} $w_0=1$, $w_1=2$, $w_2=1$, and $w_k=0$ for $k\ge3$, \ie,
$w_k=\binom2k$. 
Thus,
\begin{equation}
  \Phi(t)=1+2t+t^2=(1+t)^2
\end{equation}
and 
\begin{equation}
  \Psi(t)=\frac{t\Phi'(t)}{\Phi(t)}=\frac{2t}{1+t}.
\end{equation}
Thus $\rho=\infty$, $\nu=2$,
and $\Psi(\tau)=1$ yields $\tau=1$.
Hence \eqref{pk} yields the canonical \pws{}
\begin{equation}
\pi_k=\frac14\binom 2k,
\qquad k\ge0.
\end{equation}
In other words, 
a uniformly  random binary tree of this type is the \cGWt{} with binomial
offspring 
distribution $\xi\sim\Bi(2,1/2)$. 
(Any other distribution $\Bi(2,p)$, $0<p<1$, is equivalent and yields the
same  \cGWt{}.)

The size-biased random variable $\hxi$ has by \eqref{hxi} 
$\P(\hxi=1)=\P(\hxi=2)=\frac12$;
thus 
$\hxi-1\sim\Bi(1,1/2)$. 

We have $\gss\=\Var\xi=1/2$ and
$\E\hxi=\gss+1=3/2$, \cf{} \eqref{gss} and \eqref{ehxi}.
\end{example}

\begin{example}[Motzkin trees]
  A \emph{Motzkin tree} is a ordered rooted tree with each outdegree $\le2$. The
  difference from \refE{Ebinary2} is that there is only one type of a single
  child. Thus we count such trees 
and obtain uniformly random Motzkin trees
by taking $w_0=w_1=w_2=1$ and $w_k=0$, $k\ge3$.
We have 
\begin{equation}
  \Phi(t)=1+t+t^2
\end{equation}
and 
\begin{equation}
  \Psi(t)=\frac{1+2t}{1+t+t^2}.
\end{equation}
Thus $\rho=\infty$, $\nu=2$,
and $\Psi(\tau)=1$ yields $\tau=1$.
Hence \eqref{pk} yields the canonical \pws{}
\begin{equation}\label{motzpi}
\pi_k=\tfrac13,
\qquad k=0,1,2.
\end{equation}
In other words, 
a uniformly random Motzkin tree is the \cGWt{} with offspring
distribution $\xi$ uniform on \set{0,1,2}.

The size-biased random variable $\hxi$ has,
 by \eqref{hxi} and \eqref{motzpi}, the distribution
$\P(\hxi=1)=\frac13$, $\P(\hxi=2)=\frac23$;
thus 
$\hxi-1\sim\Bi(1,2/3)$. 

We have $\gss\=\Var\xi=2/3$ and
$\E\hxi=\gss+1=5/3$, \cf{} \eqref{gss} and \eqref{ehxi}.
\end{example}

\begin{example}[$d$-ary trees]\label{Ed-ary}
  In a \emph{$d$-ary} tree, each node has $d$ positions where a child may be
  attached, and there is at most one child per position. 
(Trees with children attached at different positions are regarded as
  different trees.) 
This generalises the binary trees in \refE{Ebinary2}, which is the special
case $d=2$. 

Since $k$ children may be attached in $\binom dk$ ways (with a given order),
we obtain a uniformly random $d$-ary trees by taking $w_k=\binom dk$.
We have 
\begin{equation}\label{d-aryphi}
  \Phi(t)=(1+t)^d
\end{equation}
and 
\begin{equation}\label{d-arypsi}
  \Psi(t)=\frac{t\Phi'(t)}{\Phi(t)}=\frac{dt}{1+t}.
\end{equation}
Thus $\rho=\infty$, $\nu=\go=d$,
and $\Psi(\tau)=1$ yields $\tau=1/(d-1)$.
Hence \eqref{pk} yields the canonical \pws{}
\begin{equation}
\pi_k=\binom dk (d-1)^{d-k}d^{-d}
=
\binom dk \Bigparfrac1{d}^k\Bigparfrac{d-1}{d}^{d-k},
\qquad k\ge0.
\end{equation}
In other words, 
a uniformly random $d$-ary tree is the \cGWt{} with binomial
offspring distribution $\xi\sim\Bi(d,1/d)$. 
(Any other distribution $\Bi(d,p)$, $0<p<1$, is equivalent and yields the
same  \cGWt{}.)

The size-biased random variable $\hxi$ has the distribution
\begin{equation}
\P(\hxi=k)=k\pi_k=
\binom {d-1}{k-1} \Bigparfrac1{d}^{k-1}\Bigparfrac{d-1}{d}^{d-k},
\qquad k\ge1;
\end{equation}
thus $\hxi-1$ has  the Binomial distribution 
$\Bi(d-1,1/d)$. 

We have $\gss\=\Var\xi=1-1/d$ and
$\E\hxi=\gss+1=2-1/d$, \cf{} \eqref{gss} and \eqref{ehxi}.
\end{example}

\begin{example}
  \label{Ezeta}
Let $\gb$ be a real constant and let $w_k=(k+1)^{-\gb}$. (The case $\gb=0$ is
\refE{Euniform}.) 
Then $\rho=1$.

If $-\infty<\gb\le1$, then $\Phi(\rho)=\infty$, so $\nu=\infty$ by
\eqref{nu} and \refL{LPsi}\ref{L1b}.

If $\gb>1$, then $\Phi(\rho)=\zeta(\gb)<\infty$ and
\begin{equation}\label{ezeta}
  \nu=\Psi(1)=\frac{\sum_k kw_k}{\Phi(1)}
=\frac{\zeta(\gb-1)-\zeta(\gb)}{\zeta(\gb)},
\qquad \gb>2,
\end{equation}
while $\nu=\Psi(1)=\infty$ if $\gb\le2$. 
Hence, 
see also \citet{BialasB}, 
\begin{equation}\label{ezetagb0}
  \nu=1 
\iff
\zeta(\gb-1)=2\zeta(\gb)
\iff
\gb=\gb_0=2.47875\dots   
\end{equation}
and $\nu>1\iff -\infty<\gb<\gb_0$.
(It can be shown that $\nu$ is a decreasing function of $\gb$ for $\gb>2$.)
In the case $\gb=\gb_0$, when thus $\nu=1$, we further have $\gss=\infty$ by 
\eqref{gss}, since $\Phi''(1)=\infty$ when $\gb\le3$. This is thus case
\Igb, in the notation of \refS{S3}.

In the case $\gb>\gb_0$ we thus have $0<\nu<1$, and $\ctn$ converges to a
random tree $\hcT$ with one node of infinite degree, see \refT{Tmain} and
\refS{ShGW}.
If $\gb\le\gb_0$, then $\nu\ge1$ and the limit tree $\hcT$ is locally
finite.
We thus see a phase transition at $\gb=\gb_0$ when we vary $\gb$ in this
example.

Note, however, that there is nothing special with the rate of decrease
$k^{-\gb_0}$; the value of $\gb_0$ depends on the exact form of our choice
of the weights $w_k$ in this example, and reflects the values for small $k$
rather than the asymptotic behaviour. For example, as remarked by
\citet{BialasB}, just changing $w_0$
would change $\gb_0$ to any desired value in $(2,\infty)$. 
With a different $w_0$, 
$\Phi(1)=\zeta(\gb)-1+w_0$, and a modification of \eqref{ezeta} shows 
that
the critical value $\gb_0$ yielding $\nu=1$ is given by,
see \cite{BialasB},
\begin{equation}
2\zeta(\gb_0)  -\zeta(\gb_0-1)=1-w_0.
\end{equation}
In particular, $\gb_0>3$ for $w_0<1+\zeta(2)-2\zeta(3)=0.24082\dots$; in
this case, for the critical $\gb=\gb_0$, we then have $\nu=1$ and $\gss<\infty$,
see \eqref{gss}.

See  \cite{BialasB} for some further analytic properties. For example, if
$\gb_0<3$ (for example when $w_0=1$), then,
as $\gb\upto\gb_0$, we have $1-\tau\sim c (\gb_0-\gb)^{1/(\gb_0-2)}$, 
where $c>0$ and the exponent can take any value $>1$.
\end{example}

\begin{example}\label{Ek!}
  Take $w_k=k!$. The generating function $\Phi(t)=\sumk k!\,t^k$ has radius of
  convergence $\rho=0$ so we are in case III, and there exists no equivalent
  \cGWt.

\refT{Tmain} shows that $\ctn$ converges to an infinite star, see
\refR{Rnu=0} and \refE{Emu=0}.
This means that the root degree converges in probability to $\infty$, and
that the outdegree of any fixed child converges to 0 in probability,
\ie, equals 0 \whp. Note, however, that we
cannot draw the conclusion that the outdegrees of \emph{all} children of the root
are 0 \whp{}; \refT{Tmain} and symmetry imply that the
proportion 
of children of the root with outdegree $>0$ tends to 0, but the number of
such children may still be large.
(\refT{Tdegree}\ref{quenched} yields the same conclusion.)

In fact, for this particular example $w_k=k!$, 
it is shown by \citet{SJ259}, using direct
calculations, that 
\whp{} all subtrees attached to the root have size 1 or 2, and that the
number of such subtrees of size 2 has an asymptotic Poisson distribution
$\Po(1)$. 
(This number thus \whp{} equals $N_1$, and $\elz_2(\ctn)$, and also the
number of children of the root with at least one child.)
\end{example}

\begin{example}\label{Ek!a}
  If we instead take $w_k=k!^\ga$ with $0<\ga<1$, then
as in \refE{Ek!},
$\rho=0$ and $\ctn$ converges to the infinite star in \refE{Emu=0}.
In this case, if  (for simplicity) $1/\ga\notin\Ni$, then 
$N_i(\ctn)/n^{1-i\ga}\pto i!^\ga$ for $1\le i\le\floor{1/\ga}$,
while $N_i=0$ \whp{} for each fixed $i>\floor{1/\ga}$; furthermore,
among the subtrees attached to the root, \whp{} there are subtrees of all sizes
$\le\floor{1/\ga}+1$, and all possible shapes of these trees, with the
number of each type tending to $\infty$ in probability, but no larger subtrees.
See \citet{SJ259} for details.

If we take $w_k=k!^\ga$ with $\ga>1$, then
\whp{} $\ctn$ is a star with $n-1$ leaves, so $N_d=0$ for $1\le d<n-1$.
\end{example}

See also the examples in \refS{Sex+}.

\section{Balls-in-boxes}\label{SBB}

The \emph{\bib} model is a model for random allocation of $m$
(unlabelled) balls in $n$ (labelled) boxes; here $m\ge0$ and $n\ge1$ are given
integers.
The set of possible allocations is thus 
\begin{equation}
\cbmn
\=\Bigset{\yyn\in\No^n:\sumin y_i=m} , 
\end{equation}
where $y_i$ counts the number of balls in box $i$.

We suppose again that $\wwx=(w_k)\kooo$ is a fixed weight sequence, and we
define the weight of an allocation $\yyx=\yyn$ as
\begin{equation}\label{wmm}
  w(\yyx)\=\prodin w_{y_i}.
\end{equation}

Given $m$ and $n$, we choose a random allocation $\bmn$ with probability
proportional to its weight, \ie, 
\begin{equation}\label{pbmn}
  \P(\bmn=\yyx)=\frac{w(\yyx)}{Z(m,n)},
\qquad \yyx\in\cbmn,
\end{equation}
where the normalizing factor $Z(m,n)$, again called the \emph{partition
  function}, is given by
\begin{equation}\label{zmn}
  Z(m,n)
=   Z(m,n;\wwx)
\=\sum_{\yyx\in\cbmn}w(\yyx).
\end{equation}
We consider only $m$ and $n$ such that $Z(m,n)>0$; otherwise $\bmn$ is
undefined. See further \refL{LBexists}.
We write $\bmn=\YYn$.

\begin{remark}
  The names balls-in-boxes and balls-in-bins are used in the literature
  for several different allocation models.
We use balls-in-boxes for the model defined here, following \eg{}
\citet{BialasetalNuPh97}.
\end{remark}

\begin{example}[probability weights]\label{EBprob}
In the special case when $(w_k)$ is a probability weight sequence, let 
$\xi_1,\xi_2,\dots$
be \iid{} random variables with the distribution $(w_k)$. Then
$w(\yyx)=\P\bigpar{(\xi_1,\dots,\xi_n)=\yyx}$ for any $\yyx=\yyn$.
Hence
\begin{equation}\label{ebprob}
 Z(m,n)=\P \bigpar{(\xi_1,\dots,\xi_n)\in\cbmn}
=\P(S_n=m),
\end{equation}
where we define
\begin{equation}
  S_n\=\sumin\xi_i.
\end{equation}
Moreover, $\bmn$ has the same distribution as $(\xi_1,\dots,\xi_n)$
conditioned on 
$S_n=m$:
\begin{equation}\label{ebprob2}
\YYn\eqd
\bigpar{(\xi_1,\dots,\xi_n)\mid S_n=m}.
\end{equation}
We will use this setting (and notation) several times below.
(This construction of a random allocation $\bmn$  is 
used by \citet{Kolchin} and there
called the \emph{general scheme of allocation}.)
\end{example}

We can replace the \ws{} by an equivalent \ws{} for the \bib\ model
just as we did for the random trees in \refS{Sequiv}.

\begin{lemma}\label{LB1}
Suppose that we replace the weights $(w_k)$ by  equivalent weights $(\tw_k)$
where $\tw_k\=ab^kw_k$ with $a,b>0$ as in \eqref{tw}.
Then the weight of an allocation 
$\yyx=\yyn\in\cbmn$ is changed to
\begin{equation}\label{lb1w}
  \tw(\yyx)=a^n b^{ m} w(\yyx),
\end{equation}
and the partition function $Z(m,n)=Z(m,n;\wwx)$ is changed to
\begin{equation}\label{lb1z}
  \tZ(m,n)\=Z(m,n;\twwx)=a^nb^mZ(m,n),
\end{equation}
while the distribution of $\bmn$ is invariant.
Thus $\bmn$ depends only on the
equivalence class of the weight sequence.
\end{lemma}

\begin{proof}
We have, by the definition \eqref{wmm},
  \begin{equation}
  \tw(\yyx)=\prodin\tw_{y_i}
=\prodin a b^{y_i}w_{y_i}
=a^n b^{\sumin y_i}\prodin w_{y_i}
=a^n b^{ m} w(\yyx),
\end{equation}
which shows \eqref{lb1w}, and \eqref{lb1z} follows by \eqref{zmn}.
Consequently, for every $\yyx\in\cbmn$, we have $\tw(\yyx)/\tZ(m,n)
=w(\yyx)/Z(m,n)$  
so the probability $\P(\bmn=\yyx)$ in \eqref{pbmn} is unchanged, which completes
the proof.
\end{proof}

Our aim is to describe the asymptotic distribution of the random allocation
$\bmn$ as $m,n\to\infty$; we consider the case when $m/n\to\gl$ for
some real $\gl$, and assume for simplicity that $0\le\gl<\go=\go(\wwx)$.
(Cases with  $m/n\to\infty$ are interesting too in some applications, 
for example in \refSS{SSmaxforests},
but will not be considered here. 
See \eg{} \citet{KolchinSCh}, \citet{Kolchin} and \citet{Pavlov}
for such results in special cases.)
The first step is to note that the distribution of 
$\bmn=(Y_1,\dots,Y_n)$ is \emph{exchangeable}, \ie, invariant under any
permutation of $Y_1,\dots,Y_n$. Hence, the distribution is completely
described by the (joint) distribution of the numbers of boxes with a certain
number of balls, so it suffices to study these numbers.

For any allocation of balls $\yyx=(y_1,\dots,y_n)\in\No^n$, and $k\ge0$, let
\begin{equation}
  N_k(\yyx)\=|\set{i:y_i=k}|,
\end{equation}
the number of boxes with exactly $k$ balls. Thus, if $\yyx\in\cbmn$, then
\begin{equation}
\sumk N_k(\yyx)=n
\qquad\text{and}\qquad
\sumk kN_k(\yyx)=m.  
\end{equation}

We thus want to find the asymptotic distribution of the random variables
$N_k(\bmn)$, $k=0,1,\dots$.
Our main result is the following, which will be proved in \refS{Sbbpf}
together with the other theorems in this section.

\begin{theorem}\label{TBmain}
  Let $\wwx=(w_k)_{k\ge0}$ be any weight sequence with $w_0>0$ and $w_k>0$
  for some $k\ge1$.
Suppose that $n\to\infty$ and $m=m(n)$ with $m/n\to\gl$ with $0\le\gl<\go$.
  \begin{romenumerate}[-10pt]
  \item \label{tbmaincritical}
If $\gl\le\nu$, let $\tau$ be the unique number in $[0,\rho]$ such
that 
$\Psi(\tau)=\gl$.
\item \label{tbmainsub}
If $\gl>\nu$, let $\tau\=\rho$.	
  \end{romenumerate}
In both cases, $0\le\tau<\infty$ and $0<\Phi(\tau)<\infty$.
Let 
\begin{equation}\label{pik}
\pi_k\=\frac{w_k\tau^k}{\Phi(\tau) },
\qquad k\ge0.
\end{equation}
Then
$\ppi\geko$ is a probability distribution, 
with expectation
\begin{equation}\label{tbm}
\mu=  \Psi(\tau)=\min(\gl,\nu)
\end{equation}
and variance $\gss=\tau\Psi'(\tau)\le\infty$.
Moreover, 
for every $k\ge0$,
\begin{equation}\label{tbmain}
  N_k(\bmn)/n\pto\pi_k.
\end{equation}
\end{theorem}

If we regard the \ws{} $\wwx$ as fixed and vary $\gl$ (\ie, vary $m(n)$), we
see that if $0<\nu<\infty$, there is a phase transition at $\gl=\nu$.

Note that $\tau$ and $\pi_k$ in \refT{Tmain} are the same as in
\refT{TBmain} with $\gl=1$.
Indeed, we will later see that the random trees correspond to $m=n-1$ and thus
$\gl=1$. 

\begin{remark}\label{RBanna}
The argument in \refR{R6.3.1} extends and shows that
$\tau$ is
the (unique) minimum point in $[0,\rho]$
of $\Phi(t)/t^\gl$; \ie, 
\begin{equation}
  \label{annagl}
\frac{\Phi(\tau)}{\tau^\gl}
=\inf_{0\le t\le\rho}\frac{\Phi(t)}{t^\gl}
=\inf_{0\le t<\infty}\frac{\Phi(t)}{t^\gl}.
\end{equation}
\end{remark}

By \eqref{tbmain}, there are roughly $n\pi_k$ boxes with $k$ balls. Summing
this approximation over all $k$ we would get $n$ boxes (as we should)
with a total of $n\sumk k\pi_k=n\mu$ balls. However, the total number of
balls is $m\approx n\gl$, so in the case $\gl>\nu$, \eqref{tbm} shows that
about $n(\gl-\mu)=n(\gl-\nu)$ balls are missing. Where are they?

The explanation is that the sums $\sumk kN_k(\bmn)/n=m$ are not uniformly
summable, and we cannot take the limit inside the summation sign. The
``missing balls'' appear in one or several boxes with very many balls, but
these ``giant'' boxes are not seen in the limit \eqref{tbmain} for
fixed $k$. In physical terminology, this can be regarded as condensation of
part of the mass (= balls).
We study this further in \refSS{SSlarge+}.
The simplest case is that there is a single giant box with $\approx
(\gl-\nu)n$ balls. We shall see that this happens in an
important case (\refT{TDzeta};
see also \citetq{Fig.\ 1}{BialasetalNuPh97} for some numerical examples), 
but that there are also other possibilities
(Examples \ref{Eoligo}--\ref{E00}).

Recall that for \sgrt{s}, which as said above correspond to
balls-in-boxes with $\gl=1$, \refT{Tmain} too shows that there is a
condensation when $\nu<\gl=1$ (since then $\mu<1$ by
\eqref{mainmu}); in this case the condensation appears as 
a node of infinite degree in the random limit
tree $\hcT$ of type {\tii}, 
see \refS{ShGW}.
We shall in \refS{SlargeT} study the relation between the forms of the
condensation shown 
in Theorems \refand{Tmain}{TBmain}.

We further have the following, essentially equivalent, version of
\refT{TBmain}, where we 
assume only 
that $m/n$ is bounded, but not necessarily convergent.

\begin{theorem}\label{TBmain2}
  Let $\wwx=(w_k)_{k\ge0}$ be any weight sequence with $w_0>0$ and $w_k>0$
  for some $k\ge1$.
Suppose that $n\to\infty$ and $m=m(n)$ with $m/n\le C$ for some $C<\go$.

Define the function $\tau:\ooo\to\ooox$ by 
$\tau(x)\=\sup\set{t\le\rho:\Psi(t)\le x}$.
Then $\tau(x)$ is the unique number in $[0,\rho]$ such
that $\Psi(\tau(x))=x$ when $x\le\nu$, and $\tau(x)=\rho$ when $x\ge\nu$;
furthermore, the function $x\mapsto\tau(x)$ is continuous.
We have
$0\le\tau(m/n)<\infty$ and $0<\Phi(\tau(m/n))<\infty$, and
for every $k\ge0$,
\begin{equation}\label{tbmain2}
  \frac{N_k(\bmn)}{n}-\frac{w_k(\tau(m/n))^k}{\Phi(\tau(m/n))}\pto0.
\end{equation}
Furthermore, for any $C<\go$,
this holds uniformly as \ntoo{} for all $m=m(n)$ with $m/n\le C$.
\end{theorem}

Returning to the random variables $Y_1,\dots,Y_n$, we have  the following
result, which is shown by a physicists' proof by \citet{BialasetalNuPh97}.

\begin{theorem}
  \label{TB3}
  Let $\wwx=(w_k)_{k\ge0}$ be any weight sequence with $w_0>0$ and $w_k>0$
  for some $k\ge1$.
Suppose that $n\to\infty$ and $m=m(n)$ with $m/n\to\gl$ where $0\le\gl<\go$,
and let $\ppi\geko$ be as in \refT{TBmain}.
Then, for every $\ell\ge1$ and $y_1,\dots,y_\ell\ge0$,
\begin{equation}\label{tb3}
  \P(Y_1=y_1,\dots,Y_\ell=y_\ell)\to \prod_{i=1}^\ell \pi_{y_i}.
\end{equation}
In other words, for every fixed $\ell$, the random variables $Y_1,\dots,Y_\ell$
converge jointly to independent random variables with the distribution
$(\pi_k)\geko$.
\end{theorem}

A more fancy way of describing the same result is that the sequence
$Y_1,\dots,Y_n$, arbitrarily extended to infinite length, converges in
distribution, as an element of $\No^\infty$, to a sequence of \iid{} random
variables with the distribution $\ppi\geko$. 
(See \eg{} \cite[Problem 3.7]{Billingsley}.)

\begin{remark}\label{Rmin}
  We have assumed $w_0>0$ in the results above for convenience, and because
  this condition is necessary when discussing simply generated trees, which
  is our main topic. The \bib{} model makes sense also when $w_0=0$, but
  this case is easily reduced to the case $w_0>0$: Let
  $\ga:=\min\set{k:w_k>0}$. If $\ga>0$, then this means that each box has to
  have at least $\ga$ balls. 
(In particular, we need $m\ge \ga n$.)
There is an obvious correspondence between such
  allocations in $\cbmn$ and allocations in $\cB_{m-\ga n,n}$ obtained by
  removing $\ga$ balls from each box. Formally, if $\yyx=\yyn\in\cbmn$ 
let $\widetilde\yyx=\tyyn$ with $\ty_i\=y_i-\ga$, and note that if we shift
the \ws{} to $\tw_k\=w_{k+\ga}$, then $\tw(\widetilde\yyx)=w(\yyx)$; 
thus $\bmn$ has the same distribution as $B_{m-\ga n,n}$ for $\twwx$, with
$\ga$ extra balls added in each box. 
It follows easily that the results above hold also in the case $w_0=0$.
(We interpret $w_k\tau^k/\Phi(\tau)$ for $\tau=0$ as the appropriate
limit value. 
Note also that it is essential to use \eqref{span} and not \eqref{span0} when
$w_0=0$.) 
\end{remark}

\begin{remark}
  \label{Rdelar}
Similarly, we can always reduce to the case $\spann(\wwx)=1$: 
If $\spann(\wwx)=d$, then the number of balls in each box has to be a
multiple of $d$, so we may instead consider an allocation of $m/d$
``superballs'', each consisting of $d$ balls. This means replacing each
$Y_i$ by $Y_i/d$ and using the weight sequence $(w_{dk})$.
We prefer, however, to allow a general span in our theorems, for ease of our
applications to simply generated trees where the corresponding reduction is
more complicated.
(For trees, we may replace each branch by a $d$-fold branch. In the \pws{}
case with \GWt{s}, this replaces the random variable $\xi$ by
$(\xi_1+\dots+\xi_d)/d$, with $\xi_i\eqd\xi$ \iid, but the roots gets a
different offspring distribution $\xi/d$; more generally, for a general
\ws{} $\wwx$, we replace $\Phi(t)$ by $\Phi(t^{1/d})^d$, except at the root
where we use different weights with the generating function $\Phi(t^{1/d})$.
We will not use this and leave the details to the reader.)
\end{remark}

\begin{remark}\label{Rsymmetry}
  We have assumed $m/n\to\gl<\go$ in Theorems \refand{TBmain}{TB3}, and
  similarly $m/n\le C<\go$ in \refT{TBmain2}; hence, for $n$ large at least, 
$m/n<\go$. In fact, $m/n\le\go$ is trivially necessary, see \refL{LBexists}.
When $\go<\infty$, the only remaining case (assuming $m/n$ converges)
is thus $m/n\to\go$ with $m/n\le \go$; in this case, it is easy to see that
\eqref{tbmain} and \eqref{tb3} hold with $\pi_\go=1$ and $\pi_k=0$, $k\neq\go$.
(This can be seen as a limiting case of \eqref{pik} with $\tau=\infty$.)

In fact, if $\go<\infty$, so the boxes have a finite maximum capacity $\go$,
then the complementation $y_i\mapsto\go- y_i$ yields a bijection of $\cbmn$
onto $\cB_{\go n-m,n}$, which preserves weights 
if $\ww$ simultaneously is reflected to $\twwx\=(w_{\go-k})$. Hence, 
$\bmn$ corresponds to $B_{\go n-m,n}$ (for $\twwx$), and
results
for $m/n\to\go<\infty$ follow from results for $m/n\to0$.

As said above, we do not consider the case $\go=\infty$ and
$m/n\to\infty$, when the average occupancy tends
to infinity. 
\end{remark}

\section{Examples of \bib}\label{Sex+}

Apart from the connection with simply generated trees, 
see  \refS{Stree-balls},
the balls-in-boxes model is interesting in its own
right.

We begin with three classic examples of balls-in-boxes, see \eg{}
\citetq{II.5}{FellerI} and \citet{Kolchin}, followed by
further examples from probability theory, combinatorics
and statistical physics, including several examples of random
forests. 
(We return to these examples of random forests in
\refS{SSmaxforests}, where we study the size of the largest tree in them.)

\begin{example}[Maxwell--Boltzmann statistics; multinomial distribution]
  Consider \label{EMB}
a uniform random allocation of $m$ \emph{labelled} balls in $n$
  boxes. This is the same as throwing $m$ balls into $n$ boxes at random,
  independently and with each ball uniformly distributed.
(In statistical mechanics, this is known as the 
\emph{Maxwell--Boltzmann statistics}.)
It is elementary that the resulting random allocation $\YYn$ has a
multinomial distribution
\begin{equation}\label{xlab}
  \P\bigpar{\YYn=\yyn}
=n^{-m}\binom{m}{y_1,\dots,y_n}
=m!\, n^{-m}\prodin \frac1{y_i!}.
\end{equation}
If we take $w_k=1/k!$, we see that the probabilities in \eqref{xlab} and
\eqref{pbmn} are proportional, and thus must be identical, so the \ws{}
$(1/k!)$ yields the uniform random allocation of labelled balls.
We see also that then 
\begin{equation}
Z(m,n)=n^m/m!.  
\end{equation}

Alternatively, we may take a Poisson distribution $\Po(a)$:
$w_k=a^ke^{-a}/k!$; this is an equivalent \ws{} for any $a>0$. We see
directly that then $S_n\sim \Po(na)$ so \eqref{ebprob} yields 
\begin{equation}
  Z(m,n)=(na)^me^{-na}/m!;
\end{equation}
hence we see again that \eqref{pbmn} and \eqref{xlab} agree.

Comparing with \refE{Ecayley}, and using \refL{Lx} below, we see that the
multiset 
of degrees in a random unordered labelled tree of size $n$ has exactly the 
distribution obtained when throwing $n-1$ balls into $n$ boxes at random.

With $w_k=1/k!$ we have, as in \refE{Ecayley},  
\eqref{ecayleyphi}--\eqref{ecayleypsi} and
$\rho=\go=\nu=\infty$.
Hence, if $m/n\to\gl$, we have $\tau=\gl$ and thus $\pi_k=\gl^ke^{-\gl}/k!$,
so $\ppi$ is the $\Po(\gl)$ distribution, which thus is the canonical choice
of weights. (In the asymptotic case; for given $m$ and $n$ one might choose
$\Po(m/n)$, \cf{} \eqref{tbmain2}.) 

\refT{TB3} (or \eqref{tbmain}) shows that if $m/n\to\gl<\infty$, then the
asymptotic distribution of the 
numbers of balls in a given urn is $\Po(\gl)$.

The idea to study the multinomial distribution as a vector of \iid{} Poisson
variables conditioned on the sum is an old one that has been used
repeatedly, see \eg{} \citet{KolchinSCh}, 
\citet{Holst:conditional,Holst:urn},  
\citet{Kolchin},
\citet{SJ132}.
\end{example}

\begin{example}[Bose--Einstein statistics] 
The \label{EBE}
\ws{} $w_k=1$ yields a uniform distribution over all allocations of $m$
identical and indistinguishable balls in $n$ boxes; thus each allocation
$\YYn\in\cbmn$ has the same probability $1/|\cbmn|=1/\binom{n+m-1}{m}$.

This is known as 
\emph{Bose--Einstein statistics} in statistical quantum mechanics;
it is the distribution followed by \emph{bosons}. (In the simple case with
no forces acting on them.)

Comparing with \refE{Euniform}, and using \refL{Lx} below, we see that the
multiset 
of degrees in a random ordered tree of size $n$ has exactly the 
distribution obtained by a uniform random allocation of $n-1$ balls  into
$n$ boxes. 

As in \refE{Euniform} we have \eqref{euphi}--\eqref{eupsi} and $\rho=1$,
$\nu=\infty$. 
If $m/n\to\gl<\infty$, then the equation $\Psi(\tau)=\gl$ is,
by \eqref{eupsi}, 
$\xfrac{\tau}{(1-\tau)}=\gl$, 
  and thus 
\begin{equation}\label{ebetau}
  \tau=\frac{\gl}{1+\gl}.
\end{equation}
 
Any geometric distribution 
$\Ge(p)$ with $0<p<1$ 
is a \ws{} equivalent to $\ww$, and
\eqref{ebetau} shows that the canonical choice \eqref{pk} is, using
\eqref{euphi},  
\begin{equation}
  \pi_k=(1-\tau)\tau^k
=\frac{\gl^k}{(\gl+1)^{k+1}},
\end{equation}
which is the distribution
$\Ge(1-\tau)=\Ge(1/(\gl+1))$.
By \refT{TB3}, this is also
the asymptotic distribution of balls in a given urn.

See also \citet{Holst:conditional, Holst:urn}
and \citet{Kolchin}.
\end{example}

\begin{example}[Fermi--Dirac statistics]\label{EFD}
The other type of particles in statistical quantum mechanics is
\emph{fermions}; 
they exclude each other (the \emph{Pauli exclusion principle})
so all allocations of them have to satisfy $Y_i\le 1$, \ie, $Y_i\in\setoi$.
A random allocation uniform among all such possibilities is  known as 
\emph{Fermi--Dirac statistics}; this is thus equivalent to a uniform random
choice of one of the $\binom nm$ subsets of $m$ boxes.

We obtain this distribution by the choice $w_0=w_1=1$ and $w_k=0$ for
$k\ge2$; thus
\begin{equation}\label{fermiphi}
  \Phi(t)=1+t
\end{equation}
and 
\begin{equation}
  \Psi(t)=\frac{t}{1+t}.
\end{equation}
We have $\rho=\infty$ and $\nu=\go=1$.
(Formally, \eqref{fermiphi} is the case $d=1$ of
\eqref{d-aryphi}, but note that  we assume $d\ge2$ in \refE{Ed-ary}.)

If $m/n\to\gl<1$, we thus have a rather trivial example of the general
theory with $\tau/(1+\tau)=\gl$ and thus
\begin{equation}\label{fermitau}
  \tau=\frac{\gl}{1-\gl},
\end{equation}
and $\ppi=(1-\gl,\gl,0,0,\dots)$, \ie,
the Bernoulli distribution $\Be(\gl)$.
(Any Bernoulli distribution $\Be(p)$ with $0<p<1$ is equivalent.)

Since $\go=1$, the corresponding \cGWt{} is trivially the deterministic  path
$P_n$,
a case which we have excluded above.
\end{example}

\begin{example}[\Polya\ urn \cite{Holst:urn}]\label{Epolya}
Consider a multicolour \Polya{} urn 
containing balls of $n$ different colours,
see \citet{EggPol}.
Initially, the urn contains $a>0$ balls of each colour. Balls are drawn at
random, one at a time. After each drawing, the drawn ball is replaced
together with $b>0$ additional balls of the same colour.
(It is natural to take $a$ and $b$ to be integers, but the model is easily
interpreted also for arbitrary real $a,b>0$, see \eg{} \cite{SJ154}.)

Make $m$ draws, and let $Y_i$ be the number of times that a ball of colour
$i$ is drawn; then $\YYn$ is a random allocation in $\cbmn$.

A straightforward calculation, see \cite{EggPol}, \cite{JohnsonKotz},
\cite{Holst:urn},
shows that
	\begin{multline}
\P\bigpar{\YYn=\yyn}
\\
\begin{aligned}
&=
\binom{m}{y_1,\dots,y_n}
\frac{\prodin a(a+b)\dotsm(a+(y_i-1)b)}{na(na+b)\dotsm(na+(m-1)b)}	
\\&
=
\frac{\prodin \binom{a/b+y_i-1}{y_i}}{\binom{na/b+m-1}{m}}.	    
\end{aligned}
	\end{multline}
Hence, as noted by \citet{Holst:urn},
this equals the random allocation given by the weights
\begin{equation}
  w_k=\binom{a/b+k-1}{k}=(-1)^k\binom{-a/b}{k},
\qquad k=0,1,\dots.
\end{equation}
Note that the case $a=b$ yields $w_k=1$ and the uniform random allocation in
\refE{EBE} (Bose--Einstein statistics).
We have 
\begin{equation}
  \Phi(t)=\sumk \binom{a/b+k-1}{k}t^k=(1-t)^{-a/b},
\end{equation}
with radius of convergence $\rho=1$, and thus
\begin{equation}
  \Psi(t)=\frac ab\cdot\frac{t}{1-t}.
\end{equation}
Hence, $\nu=\Psi(1)=\infty$, and for any $\gl\in\ooo$, 
\begin{equation}
  \tau=\frac{b\gl}{a+b\gl}.
\end{equation}

The equivalent \pws{s} are, by \refL{Lequivp}, given by 
\begin{equation}
\frac{t^kw_k}{\Phi(t)} = 
 \binom{a/b+k-1}{k}t^k(1-t)^{a/b}, 
\qquad 0<t<1,
\end{equation}
which is the negative binomial distribution $\NBi(a/b,1-t)$
(where the parameter $a/b$ is not necessarily an integer).
The canonical choice, which by Theorems \refand{TBmain}{TB3} is the
asymptotic distribution of 
the number of balls of a given colour, 
is $\NBi(a/b,1-\tau)=\NBi(a/b,a/(a+b\gl))$.
See also \citet{Holst:urn} and \citet{Kolchin}.

Note that the case $b=0$ (excluded above)
means drawing with replacement; this is
\refE{EMB}, which thus can be seen as a limit case.
(This corresponds to the Poisson limit
$\NBi(a/b,a/(a+b\gl))\dto\Po(\gl)$ as $b\to0$.)
\end{example}

\begin{example}[drawing without replacement]
Consider again an urn with balls of $n$ colours, with initially $a$ balls of
each colour. 
(This time, $a\ge1$ is an integer.)
Draw $m$ balls without replacement, and let as above $Y_i$ be the number of
drawn balls of colour $i$.
(The case $a=1$ yields the Fermi--Dirac statistics in \refE{EFD}.)

Formally, this is the case $b=-1$ of \refE{Epolya}, and a similar
calculation shows that
\begin{equation}
\P\bigpar{\YYn=\yyn}
=
\frac{\prodin \binom{a}{y_i}}{\binom{na}{m}};
\end{equation}
hence this is the random allocation given by the weights
\begin{equation}
  w_k=\binom ak,
\qquad k=0,1,\dots
\end{equation}
We have thus $\Phi(t)=(1+t)^a$, exactly as in \refE{Ed-ary}, with $d=a$.

The equivalent \pws{s} are the binomial distributions $\Bi(a,p)$, $0<p<1$,
and the canonical choice is, for $0<\gl<a$, 
$  \ppi=\Bi(a,\gl/a)$, \ie{}
\begin{equation}
\pi_k=\binom ak\parfrac \gl{a}^k \parfrac{a-\gl}{a}^{a-k}
=\binom ak\frac{ \gl^k (a-\gl)^{a-k}}{a^a}.
\end{equation}
See also \citet{Holst:urn} and \citet{Kolchin}.

Note that taking the limit as $a\to\infty$, 
we obtain drawing with replacement, which is \refE{EMB};
this corresponds to the Poisson limit $\Bi(a,\gl/a)\dto\Po(\gl)$ as
$a\to\infty$. 
\end{example}

\begin{example}[random rooted forests \cite{Kolchin}] \label{Eforest} 
Consider \emph{labelled rooted forests} consisting of $n$ unordered
rooted trees with
  together $m$ labelled nodes. (Thus $m\ge n$.)
We may assume that the $n$ roots are labelled $1,\dots,n$; let $T_i$ be the
tree with root $i$ and let $t_i\=|T_i|$.
Then the node sets $V(T_i)$ form a partition of \setm, so $\sumin t_i=m$ and
$\ttn$ is an allocation in $\cbmn$, with each $t_i\ge1$.
Furthermore, given $\ttn\in\cbmn$ with all $t_i\ge1$, the node sets $V(T_i)$
can be chosen in $\binom{m-n}{t_1-1,\dots,t_n-1}$ ways, and given $V(T_i)$,
the tree $T_i$ can by Cayley's formula be chosen in $t_i^{t_i-2}$ ways.
(The trees are rooted but the roots are given.)
Hence, the number of forests with the allocation $\ttn$ is
\begin{equation}\label{forest}
  \binom{m-n}{t_1-1,\dots,t_n-1}\prodin t_i^{t_i-2}
=
(m-n)!\prodin \frac{t_i^{t_i-2}}{(t_i-1)!}
=
(m-n)!\prodin \frac{t_i^{t_i-1}}{t_i!}.
\end{equation}
Hence, a uniformly random labelled rooted forest corresponds to a random
allocation $\bmn$ with the \ws{} $w_k=k^{k-1}/k!$, $k\ge1$, and $w_0=0$.
Note that here $w_0=0$ unlike almost everywhere else
in the present paper; in the notation of \refR{Rmin}, we have $\ga=1$.
(As discussed in \refR{Rmin}, we can reduce to the case $w_0>0$ by
considering $(t_1-1,\dots,t_n-1)$, which is an allocation in $\cB_{m-n,n}$;
this means that we count only non-root nodes. We prefer, however, to keep
the setting above with $w_0=0$, noting that the results above still hold by
\refR{Rmin}.)

If $\ff mn$ denotes
the number of labelled rooted forests
with $m$ labelled nodes of which $n$ are given as roots,
then \eqref{forest} implies
\begin{equation}
  \label{fforest}
\ff mn = (m-n)!\, Z(m,n).
\end{equation}
It is well-known that 
$\ff mn = n m^{m-n-1}$, a formula also given by \citet{Cayley},
see \eg{} \cite[Proposition~5.3.2]{Stanley2} or \cite{Pitman:enum}; thus
\begin{equation}
  \label{zforest}
Z(m,n)=\frac{ n m^{m-n-1}}{ (m-n)!}.
\end{equation}

We have
\begin{equation}\label{treefn}
  \Phi(t)\=\sumki \frac{k^{k-1}}{k!}t^k=T(t),
\end{equation}
the well-known
\emph{tree function}
(known by this name since it is the exponential generating function for
rooted unordered labelled trees, \cf{} \refE{Ecayley}).
Note that $T(z)$ satisfies the functional equation
\begin{equation}
  \label{treeeq}
T(z)=ze^{T(z)};
\end{equation}
see \eg{} \cite[Section II.5]{Flajolet}.
Equivalently, 
\begin{equation}\label{treeeq2}
  z=T(z)e^{-T(z)}, 
\end{equation}
which by differentiation leads to
\begin{equation}
  \label{T'}
T'(z)=\frac{T(z)}{z(1-T(z))}\;. 
\end{equation}
Hence,
\begin{equation}\label{forestpsi}
  \Psi(t)\=\frac{t\Phi'(t)}{\Phi(t)}=\frac1{1-T(t)}\;.
\end{equation}
By \eqref{treefn} and Stirling's formula, $\Phi(t)$ has radius of convergence
$\rho=e\qw$. 
Furthermore, \eqref{treeeq2} implies that
$\Phi(\rho)=T(e\qw)=1$. 
Hence, \eqref{forestpsi} yields $\nu=\Psi(\rho)=\infty$, and if
$1\le\gl<\infty$, then $\gl=\Psi(\tau)$ is solved by
\begin{equation}\label{forestTtau}
T(\tau)=1-\frac1\gl=\frac{\gl-1}{\gl}
\end{equation}
and thus, using \eqref{treeeq2},
\begin{equation}\label{foresttau}
  \tau=\frac{\gl-1}{\gl}e^{-(\gl-1)/\gl}.
\end{equation}

The \pws{s} equivalent to $\ww$ are by \refL{Lequivp} given by,
substituting $x=T(t)$, and thus $t=xe^{-x}$ by \eqref{treeeq2},
\begin{equation}\label{borel}
p_k=
  \frac{t^k}{T(t)}w_k=\frac{k^{k-1} t^k}{T(t)k!}
=\frac{(kx)^{k-1} e^{-kx}}{k!},
\qquad k\ge1,
\end{equation}
where $0\le t\le e\qw$ and thus $0\le x\le 1$.
This is known as a \emph{Borel distribution};  it appears for example as
the distribution of the size $|\cT|$
of the Galton--Watson tree with offspring distribution $\Po(x)$.
(This was first proved by \citet{Borel}. It follows by \refT{TZ} below, with
the \pws{} $\Po(x)$;  
see also 
\citet{Otter},
\citet{Tanner},
\citet{Dwass},
\citet{Takacs:ballots},
\citet{Pitman:enum}.)
It follows that the random rooted forest considered here
has the same distribution as the forest
defined by a \GWp{} with 
starting with $n$ individuals (the roots)
and $\Po(x)$ offspring distribution, conditioned
to have total size $m$; \cf{} \refE{EGWforest} below.
See further \citet{Kolchin} and \citet{Pavlov}.

In particular, the canonical distribution for a given $\gl\ge1$ is,
using \eqref{foresttau},
\begin{equation}
  \pi_k=
\frac{k^{k-1} \tau^k}{T(\tau)k!}
=\frac{k^{k-1}}{k!} \parfrac{\gl-1}{\gl}^{k-1}e^{-k(\gl-1)/\gl}.
\end{equation}
By Theorems \refand{TBmain}{TB3}, and \refR{Rmin},
this is the asymptotic distribution of the
size of a given (or random) tree in the forest, say $T_1$.
The asymptotic distribution of $|T_1|$
is thus the distribution of the size $|\cT|$ of
a Galton--Watson tree with offspring distribution $\Po(1-1/\gl)$.
Moreover, $T_1$ is, given its size $|T_1|$, uniformly distributed over all
trees on $|T_1|$ nodes, and the same is true for the Poisson Galton--Watson
tree $\cT$ by \refE{Ecayley}. Consequently,
$T_1\dto \cT$ as \ntoo{} with $m/n\to\gl$.
(We may regard $T_1$ as an ordered tree, ordering the children of a node \eg{}
by their labels.)

The same random allocation $\bmn$ also describes the block lengths in
hashing with linear probing; see \citet{SJ133}.
Indeed, there is a
one-to-one correspondence between hash tables and
rooted forests, see \eg{} \citet[Exercise 6.4-31]{KnuthIII} and \citet{CL}.
\end{example}

\begin{example}[random unrooted forests] \label{Euforest} 
  Consider \emph{labelled unrooted for\-ests} consisting of $n$ trees with
  together $m$ labelled nodes. (Thus $m\ge n$.)
We may assume that the $n$ trees are labelled $T_1,\dots,T_n$;
 let $t_i\=|T_i|$.
As in \refE{Eforest}, the node sets $V(T_i)$ form a partition of \setm, so
$\sumin t_i=m$ and 
$\ttn$ is an allocation in $\cbmn$, with each $t_i\ge1$.
In the unrooted case, given $\ttn\in\cbmn$ with all $t_i\ge1$, the node sets
$V(T_i)$ 
can be chosen in $\binom{m}{t_1,\dots,t_n}$ ways, and given $V(T_i)$,
the tree $T_i$ can by Cayley's formula be chosen in $t_i^{t_i-2}$ ways.
Hence, the number of unrooted forests with the allocation $\ttn$ is
\begin{equation}\label{uforest}
  \binom{m}{t_1,\dots,t_n}\prodin t_i^{t_i-2}
=
m!\prodin \frac{t_i^{t_i-2}}{t_i!}.
\end{equation}
Hence, a uniformly random labelled unrooted forest corresponds to a random
allocation $\bmn$ with the \ws{} $w_k=k^{k-2}/k!$, $k\ge1$, and $w_0=0$.
As in \refE{Eforest}, we have $w_0=0$, but this is no problem by \refR{Rmin}.

If $\ffu mn$ denotes
the number of labelled unrooted forests
with $m$ labelled nodes and $n$ labelled trees,
then \eqref{uforest} implies
\begin{equation}
  \label{fuforest}
\ffu mn = m!\, Z(m,n).
\end{equation}
There is no simple general formula for $\ffu mn$, as there is for the rooted
forests in \refE{Eforest}, and hence no simple formula for $Z(m,n)$.
Asymptotics are given by \citet{Britikov}.
(See \refE{Ebritikov2} for one case. The asymptotic formula when
$m/n\to\gl>2$ follows similary
from \refT{TDzeta}(ii), and when $m/n\to\gl<2$ with $m=\gl n+o(\sqrt n)$ from
\refT{THB1}.) 

We have
\begin{equation}\label{uforestphi}
  \Phi(t)\=\sumki \frac{k^{k-2}}{k!}t^k=T(t)-\tfrac12T(t)^2,
\end{equation}
where $T(t)$ is the tree function in \eqref{treefn}.
(The latter equality is well-known, see \eg{} 
\cite[II.5.3]{Flajolet}; it can be shown \eg{} by showing that both sides
have the same derivative $T(t)/t$; there are also combinatorial proofs.)
Hence, using \eqref{T'},
\begin{equation}\label{uforestpsi}
  \Psi(t)\=\frac{t\Phi'(t)}{\Phi(t)}
=\frac{T(t)}{\Phi(t)}
=\frac1{1-T(t)/2};
\end{equation}
\cf{} the similar \eqref{forestpsi} in the rooted case.

As for 
\eqref{treefn}, $\Phi$ has the radius of convergence $\rho=e\qw$, but now,
by \eqref{uforestpsi},
$\nu=\Psi(\rho)=2$ is finite, so there is a phase transition at $\gl=2$.
The parameter $\tau$ is by the definition in \refT{Tmain} and
\eqref{uforestpsi} given by
$T(\tau)=2-2/\gl=2(\gl-1)/\gl$ 
for $\gl\le 2$; thus, using \eqref{treeeq2},
\begin{equation}\label{uforesttau}
  \tau=
  \begin{cases}
2\frac{\gl-1}{\gl}e^{-2(\gl-1)/\gl}, & \gl\le2.	
\\
e\qw, & \gl\ge2.
  \end{cases}
\end{equation}

The \pws{s} equivalent to $\ww$ are by \refL{Lequivp} given by,
again substituting $t=xe^{-x}$ or $x=T(t)$,
\begin{equation}\label{pku}
p_k=
 \frac{k^{k-2} t^k}{T(t)(1-T(t)/2)k!}
=\frac{x(kx)^{k-2} e^{-kx}}{(1-x/2)k!},
\qquad k\ge1,
\end{equation}
where $0\le t\le e\qw$ and thus $0\le x\le 1$.
In particular, the canonical distribution for a given $\gl\ge1$ is,
by \eqref{uforesttau} and \eqref{pku}, for $k\ge1$,
\begin{equation}\label{pkk}
  \pi_k=
\frac{k^{k-2} \tau^k}{T(\tau)(1-T(\tau)/2)k!}
=
\begin{cases}
\frac{k^{k-2}}{k!} \gl\Bigpar{2\frac{\gl-1}{\gl}}^{k-1}e^{-2k(\gl-1)/\gl},
& \gl\le2,
\\
\frac{2k^{k-2} e^{-k}}{k!}, & \gl\ge2.
\end{cases}
\end{equation}
By Theorems \refand{TBmain}{TB3}, and \refR{Rmin},
this is the asymptotic distribution of the
size of a given (or random) tree in the forest, say $T_1$.

We shall see in 
\refT{Tmaxu} that 
the phase transition at $\gl=2$ is seen clearly in the size of the largest
tree in the forest:
if $m/n\to\gl<2$, then the largest tree is of size $\Op(\log n)$,
while if $m/n\to\gl>2$, then there is a unique giant tree 
of size $(\gl-2)n+\op(n)$; 
for details
see Theorems \refand{TDzeta}{Tmaxu}, and, 
more generally,  \citet{LP}.
This is thus an example of the condensation discussed after \refT{TBmain}
(and similar to the condensation in \refT{Tmain} when $\nu<1$).
\end{example}

\begin{example}[simply generated forests and \GWf{s}]  
A \label{EGWforest} 
 \emph{\sgf} is a sequence $(T_1,\dots,T_n)$ of rooted trees, with weight
\begin{equation}\label{sgf}
  w\TTn\=\prodin w(T_i),
\end{equation}
where $w(T_i)$ is given by \eqref{wtree}, for some fixed \ws{} $\wwx$.
A \emph{\sgrf} with $n$ trees and $m$ nodes, where $n$ and $m$ are given
with $m\ge n$, 
is such a forest chosen at random, with probability proportional to its
weight. Note that in the special case $n=1$, this is the same as a \sgrt{}
defined in \refS{Ssimply}. More generally, for any $n$, a \sgrf{} $\TTn$ is,
conditioned on the sizes $|T_1|,\dots,|T_n|$, a sequence of independent
\sgrt{s} with the given sizes (all defined by the same \ws{} $\wwx$).
Moreover, the sizes $(|T_1|,\dots,|T_n|)$ form an allocation in $\cbmn$, and
it is easily seen that
this is a random allocation $\bmn$ defined by the \ws{}
$(Z_k)_{k=0}^\infty$, where $Z_k$ is the partition function \eqref{zn} for
\sgt{s} with \ws{} $\wwx$ (and $Z_0=0$).

A \sgrf{} can thus be obtained by a two-stage process, combining the
constructions in Sections \refand{Ssimply}{SBB}.
Note that equivalent \ws{s} $\wwx$ yield equivalent \ws{s} $(Z_k)$ by
\eqref{tz}, and thus the same \sgrf.

In the special case when $\wwx$ is a \pws, we also define
a \emph{\GWf} with $n$ trees, for a given $n$, 
as a sequence $(\cT_1,\dots,\cT_n)$ of $n$
\iid{} \GWt{s}; it describes the evolution of a \GWp{} started with $n$
particles. (It can also be seen as a single \GWt{} $\cT$ with the root
chopped off, conditioned on the root degree being $n$, provided that this root
degree is possible.) Note that the probability  distribution of the forest
is given by the weights in \eqref{sgf}.
Hence, in the \pws{} case, the \sgrf{} 
equals the \emph{\cGWf} with $n$ trees and $m$ nodes, 
defined as a \GWf{} with $n$ trees conditioned on the total size
being $m$; in other words,
it describes a \GWp{} started with $n$ particles conditioned on the
total size being $m$.

Random forests of this type are studied by \citet{Pavlov}, see also
\citet[Example III.21]{Flajolet}.

For example, taking $w_k=1/k!$, we have by \eqref{borel2} $Z_k=k^{k-1}/k!$,
$k\ge1$; this is the \ws{} used in \refE{Eforest}, so we obtain the same
random allocation of tree sizes as there; moreover, given the tree sizes,
the trees are uniformly random labelled unordered rooted trees by
\refE{Ecayley}. Consequently, for this \ws, the \sgrf{} is the random
labelled forest with unordered rooted trees in \refE{Eforest}.
The same random forest is obtained by the equivalent \pws{}
$w_k=x^ke^{-x}/k!$,
with $0<x\le1$,
so it equals also the \cGWf{} with offspring distribution $\Po(x)$, 
\cf{} \refE{Eforest}. 

Another example is obtained by taking $w_k=1$ for all $k\ge0$. Then every
forest has weight $1$, so the this \sgrf{} is a uniformly random forest
of  ordered rooted trees. (An \emph{ordered rooted forest}.)
By \refE{Euniform}, the \ws{} $(Z_k)$ is then given by the Catalan numbers
in \eqref{catalan}: 
$Z_k=C_{k-1}=(2k-2)!/(k!\,(k-1)!)$, $k\ge1$.

Further examples are given by starting with the other examples of random trees
in \refS{Sex}.

We shall see in \refT{TH1} that if the \ws{} $\wwx$ is as in \refT{Tmain},
and further $\spann(\wwx)=1$, $\nu\ge1$ and $\gss<\infty$, then
\begin{equation}\label{zorro}
  Z_k\sim \frac{\tau}{\sqrt{2\pi\gss}}\parfrac{\Phi(\tau)}{\tau}^k k^{-3/2} .
\end{equation}
Recalling $\cZ(\tau/\Phi(\tau))=\tau$ by  \eqref{otter},
we may replace $Z_k$ by the equivalent \pws{}
\begin{equation}\label{sgfz'}
  \tZ_k\=\frac{Z_k}{\cZ(\tau/\Phi(\tau))}\parfrac{\tau}{\Phi(\tau)}^k
=\frac{Z_k}{\tau}\parfrac{\tau}{\Phi(\tau)}^k
\sim \frac{1}{\sqrt{2\pi\gss}} k^{-3/2},
\end{equation}
so we have the asymptotic behaviour $\tZ_k\sim c k\qqcw$ for every such \ws{}
$\wwx$, where only the constant $c=1/\sqrt{2\pi\gss}$ depends on $\wwx$.
This explains why random forests of this type have similar asymptotic
behaviour, in contrast to the unrooted forests in \refE{Euforest}.
\end{example}

\begin{example}
  \label{EBzeta}
Let, as in \refE{Ezeta}, 
$w_k=(k+1)^{-\gb}$ for some real constant $\gb$. 
Then $\rho=1$. 
As shown in \refE{Ezeta}, $\nu=\infty$ if $\gb\le2$, and $\nu<\infty$ if
$\gb>2$; in the latter case, $\nu$ is given by \eqref{ezeta}.
This example is studied further in \eg{} 
\citet{BialasetalNuPh97}. 
\end{example}

\begin{example}[power-law]
\label{EBpower}
More generally, suppose that 
$w_k\sim c k^{-\gb}$ as \ktoo, for some real constant $\gb$ and $c>0$,
\ie, that $w_k$ asymptotically satisfies a power-law.
Qualitatively, we have the same behaviour as in Examples
\refand{Ezeta}{EBzeta}, but numerical values such as the critical $\gb$ in
\eqref{ezetagb0} will in general be different.

We repeat some easy facts:
first, $\rho=1$, $\go=\infty$ and $\spann(\wwx)=1$.

If $-\infty<\gb\le1$, then $\Phi(\rho)=\Phi(1)=\infty$; hence $\nu=\infty$ by
\refL{LPsi}\ref{L1b}.

If $1<\gb\le2$, then $\Phi(\rho)<\infty$ but $\Phi'(\rho)=\sumk kw_k=\infty$;
hence again $\nu=\Psi(\rho)=\infty$ by \eqref{nu1}.

On the other hand, if $\gb>2$, then $\Phi(1)<\infty$ and $\Phi'(1)<\infty$,
and thus $\nu<\infty$ by \eqref{nu1}.
Summarising: 
\begin{equation}\label{beta2}
 \nu<\infty\iff\gb>2. 
\end{equation}
In the case $\gb>2$, there is thus a phase transition when we vary $\gl$. 

Suppose $\gb>2$, so $\nu<\infty$. If $\gl\ge\nu$, then $\tau=\rho=1$, and the
canonical distribution $\ppi$ is by \eqref{pik} given simply by
$\pi_k=w_k/\Phi(1)$. This distribution then has mean $\mu=\nu<\infty$ by
\eqref{tbm}; since $\pi_k\asymp k^{-\gb}$ as \ktoo, the variance
$\gss=\infty$ if $2<\gb\le3$, while $\gss<\infty$ when $\gb>3$.

Note that Examples \ref{Eforest} and \ref{Euforest} with random forests
are of this type,
provided we replace $w_k$ by the equivalent $\tw_k\=e^{-k}w_k$; 
Stirling's formula shows that $\tw_k\sim ck^{-\gb}$ where 
$\gb=3/2$ for rooted forests and $\gb=5/2$ for unrooted forests
(and $c=1/\sqrt{2\pi}$). 
The different values of $\gb$ explains the different asymptotical behaviours of
these two types of random forests: by the results above, the tail behaviour
of $w_k$ implies that
$\nu=\infty$ for rooted forests but $\nu<\infty$ for unrooted forests,
as we have shown by explicit calculations in Examples \ref{Eforest} and
\ref{Euforest}. Recall that this means that
there is a phase transition and condensation for high $m/n$ in the
unrooted case but not in the rooted case.

More generally, \eqref{sgfz'} shows that \sgrf{s} under weak assumptions
have the same power-law behaviour of the \ws{} with $\gb=3/2$ as the special
case of (unordered) rooted forests in \refE{Eforest}. Thus
$\nu=\infty$ and there is no phase transition. (At least not in the range
$m=O(n)$ that we consider. \citet{Pavlov} show a phase transition at
$m=\Theta(n^2)$.) 
\end{example}

\begin{example}[unlabelled forests]\label{Eunlabelledforests}
Consider, as \citet{Pavlov:unlabelled}, rooted forests consisting of $n$
rooted \emph{unlabelled} trees, assuming that the trees, or equivalently the
roots, 
are labelled $1,\dots,n$, but otherwise  the nodes are unlabelled.
A uniformly random forest of this type with $m$ nodes can
be seen as \bib{} with the \ws{} $(t_k)$, where $t_k$ is the number of
unlabelled rooted trees with $k$ nodes.
In this case there is no simple formula for the generating function
$\Phi(z)$, but there is a functional equation, from which it can be shown
that $t_k\sim c_1 k^{-3/2}\rho^{-k}$, where 
$\rho\approx0.3382$ as usual
is the radius of convergence of $\Phi(z)$
and $c_1\approx 0.4399$, see \citet{Otter:trees} 
or, \eg, \citet[Section 3.1.5]{Drmota}.
Furthermore, $\Phi(\rho)=1$; thus $(t_k\rho^k)$ gives an equivalent \pws{}
with $t_k\rho^k\sim c_1 k\qqcw$ as \ktoo.
The asymptotic behaviour of the \ws{} is thus the same as for labelled
rooted forests in \refE{Eforest}, and more generally for 
\GWf{s} (under weak conditions) in \refE{EGWforest}, and we expect
the same type of asymptotic behaviour in spite of the fact that the
unlabelled forest is not simply generated; this is seen in detail in
\citet{Pavlov} for the size of the largest tree.
In particular,  we have $\nu=\infty$ by \refE{EBpower} and \eqref{beta2},
and thus there is no 
phase transition at finite $\gl$.

Similarly, \citet{Pavlov:unun}  considered \emph{unrooted} forests consisting of
$n$  trees labelled $1,\dots,n$ with a total of $m$ unlabelled nodes.
These are described by the \ws{} $(\check t_k)$ where $\check t_k$ is the
number of unrooted unlabelled trees with $k$ nodes. Again, there is no
no simple formula for the generating function
$\check\Phi(z)\=\sum_k\check t_kz^k$, but there is the relation
$\check\Phi(z)=\Phi(z)-\frac12\Phi(z)^2+\frac12\Phi(z^2)$ found by
\citet{Otter:trees}, which leads to the asymptotic formula
$\check t_k\sim c_2 k^{-5/2}\rho^{-k}$, where
$\rho$ is as above and
$c_2\approx 0.5347$,
see also
 \citet[Section 3.1.5]{Drmota}.
In this case,
$(\check t_k\rho^k/\check\Phi(\rho))$ gives an equivalent \pws{}
which is  $\sim (c_2/\check\Phi(\rho)) k^{-5/2}$ as \ktoo, which is the same
type of asymptotic behaviour as for the \ws{} for labelled unrooted forests in
\refE{Euforest}; we thus expect the same type of asymptotic behaviour as for
those forests. In particular, $\nu<\infty$ by \refE{EBpower}; a numerical
calculation 
gives $\nu\=\rho\check\Phi'(\rho)/\check\Phi(\rho)\approx 2.0513$, 
see \citet{Pavlov:unun}.

Note that both types of ``unlabelled'' forests considered here have the
trees labelled (\ie, ordered). Completely unlabelled forests
cannot be described by \bib{} (as far as we know), since the number of
(non-isomorphic) ways to number the trees depends on the forest.
\end{example}

\begin{example}[the backgammon model]
  The model with $w_k=1/k!$ for $k\ge1$ as in \refE{EMB}, but $w_0>0$
  arbitrary, was considered by \citet{BackI} and \citet{BackII,BackIII},
who called it the \emph{backgammon model}.
We have 
\begin{equation}
  \Phi(t)=w_0+\sumki \frac{t^k}{k!}=e^t+w_0-1
\end{equation}
and 
\begin{equation}\label{backpsi}
  \Psi(t)
=\frac{t e^t}{\Phi(t)}
=\frac{t }{1+(w_0-1)e^{-t}}.
\end{equation}
Thus $\rho=\nu=\infty$. The equation 
$\Psi(\tau)=\gl$ can be written
\begin{equation}
  (\tau-\gl)e^\tau=(w_0-1)\gl, 
\end{equation}
and the solution can be written
\begin{equation}
  \tau=\gl+W\bigpar{(w_0-1)\gl e^{-\gl}}
=\gl-T\bigpar{(1-w_0)\gl e^{-\gl}},
\end{equation}
where $W(z)$ is the Lambert $W$ function
\cite{Corless} defined by $W(z) e^{W(z)}=z$
and $T(z)$ is the tree function 
in \eqref{treefn}
(analytically
extended to all real $z<e^{-1}$); note that $W(z)=-T(-z)$ 
by \eqref{treeeq2}, see \cite{Corless}. 

The canonical \pws{} \eqref{pik} is,
using \eqref{backpsi} and $\Psi(\tau)=\gl$, 
\begin{equation}
\pi_k=
\frac{\tau^k}{\Phi(\tau)k!}
=
\frac{\gl\tau^{k-1}e^{-\tau}}{k!}
=
\frac{\gl}{\tau}\cdot\frac{\tau^{k}e^{-\tau}}{k!},
\qquad k\ge1,
\end{equation}
and $\pi_0=\gl\tau\qw e^{-\tau}w_0$.
\end{example}

\begin{example}[random permutations and recursive forests]\label{Eperm}
Consider \emph{permutations} of \setm{} with exactly $n$ cycles.
Let us label the cycles $1,\dots, n$, in arbitrary order, and let $y_i$ be
the length of the $i$:th cycle. Then $\yyn$ is an allocation in $\cbmn$ with
each $y_i\ge1$,
and for each such $\yyn\in\cbmn$, the number of permutations with $y_i$ elements
in cycle $i$ is 
\begin{equation}\label{permutations}
  \binom{m}{y_1,\dots,y_n}\prodin (y_i-1)!=m!\prodin\frac1{y_i},
\end{equation}
since there are $(y-1)!$ cycles with $y$ given elements.
Consequently, a uniformly random permutation of \setm{} with exactly $n$ cycles
corresponds to a random allocation $\bmn$ defined by the weights $w_0=0$ and
$w_k=1/k$ for $k\ge1$. Note that here, as in \refE{Eforest}, $w_0=0$, and
\refR{Rmin} applies 
with $\ga=1$. 

The number of permutations with $n$ (unlabelled) cycles is by 
\eqref{permutations}
\begin{equation}
m!\,Z(m,n)/n!,  
\end{equation}
where we divide by $n!$ in order to ignore the labelling above.

The same \bib{} model with $w_k=1/k$, $k\ge1$,
also describes \emph{random recursive forests}, see
\citet{Pavlov:recursive}. 

We have 
\begin{equation}\label{phiperm}
  \Phi(t)=\sum_{k=1}^\infty \frac{t^k}k=-\log(1-t)
\end{equation}
with radius of convergence $\rho=1$ and 
\begin{equation}\label{psiperm}
\Psi(t)\=\frac{t\Phi'(t)}{\Phi(t)}
=\frac{t}{-(1-t)\log(1-t)},
\end{equation}
so $\nu=\Psi(1)=\infty$, \cf{} \refE{EBpower} ($\gb=1$).

The equivalent \pws{s} are by \refL{Lequivp} given by
\begin{equation}
  p_k=\frac{x^k}{k|\ln(1-x)|},
\qquad 0<x<1,
\end{equation}
with probability generating function
$\Phi(xz)/\Phi(x)=\log(1-xz)/\log(1-x)$.
This distribution is called the \emph{logarithmic distribution}.
See further \citet{KolchinSCh} and \citet{Kolchin}.

By \refR{Rmin}, we obtain results on random permutations with $m$ cycles as
$m/n\to\gl\in[1,\infty)$, see for example \citet{Kazimirov:permutation}.
However, it is of greater interest to consider random permutations without
constraining the number of cycles. This can be done using methods similar to
the ones used here, but is outside the scope of the present paper; see \eg{}
\citet{KolchinSCh}, 
\citet{Kolchin} and
\citet{ABT}. Note that even if we condition on the number of cycles, a
typical random permutation of $\setm$ has about $\log m$ cycles, so we are
interested in the case $n\approx \log m$ and thus $m/n\to\infty$, which we
do not considered here.

Other random objects that can be decomposed into components can be studied
similarly, for example random mappings \cite{Kolchin}; our results apply
only to random objects with a given number of 
components (in some cases), but similar methods are useful for the general case;
see \citet{Kolchin} 
and \citet{ABT}.
\end{example}

\section{Preliminaries}\label{Sprel}

\begin{proof}[Proof of \refL{LPsi}]
\pfitemref{L1a}
  Since $\Phi'(t)=\sumk kw_kt^{k-1}$ has the same radius of convergence
  $\rho$ as $\Psi$, and $\Phi(t)\ge w_0>0$ for $t\ge0$,
it is immediate that $\Psi$ is well-defined, finite and continuous for
$t\in[0,\rho)$. Furthermore, 
if $0<t<\rho$, then $t\Psi'(t)$ is by \eqref{lep4} the variance of a
non-degenerate random variable, and thus $t\Psi'(t)>0$.
Hence $\Psi(t)$ is increasing,
completing the proof of \ref{L1a}.

\pfitemref{L1x} 
If $\Phi(\rho)=\infty$, the claim 
is just the definition of $\Psi(\rho)$ in \refS{Ssimply}.
(Note that the existence of the limit follows from \ref{L1a}.)
We may thus assume $\Phi(\rho)<\infty$; then
$t\upto\rho$ implies $\Phi(t)\to\Phi(\rho)<\infty$ 
and $\Phi'(t)\to\Phi'(\rho)\le\infty$ 
by monotone convergence, and thus
\begin{equation*}
\Psi(t)\=\frac{t\Phi'(t)}{\Phi(t)}
\to
\frac{\rho\Phi'(\rho)}{\Phi(\rho)}
=\Psi(\rho).
\end{equation*}

\pfitemref{L1y} 
The case $\rho=0$ is trivial, and the case $\rho>0$
follows from \ref{L1a} and \ref{L1x}.

\pfitemref{L1b}
For any $\ell\ge0$,
\begin{equation}
  \label{psil}
\Psi(t)-\ell
=
\frac{\sumk (k-\ell)w_kt^k}{\sumk w_kt^k}
\ge
\frac{\sum_{k=0}^{\ell-1} (k-\ell)w_kt^k}{\sumk w_kt^k}.
\end{equation}
If $\rho<\infty$ and $\Phi(\rho)=\infty$, we thus have, 
\begin{equation*}
  \Psi(t)-\ell\ge\frac{O(1)}{\Phi(t)}\to0
\qquad \text{as $t\upto\rho$},
\end{equation*}
so $\Psi(\rho)-\ell\ge0$. Since $\ell$ is arbitrary, this shows
$\Psi(\rho)=\infty$, 
proving \ref{L1b}.

\pfitemref{L1c} 
If $\rho=\infty$, choose $\ell$ with $w_\ell>0$. Then \eqref{psil} implies
\begin{equation*}
\Psi(t)-\ell
\ge
\frac{-\ell\sum_{k=0}^{\ell-1} w_kt^k}{ w_\ell t^\ell}\to0
\qquad\text{as \ttoo,} 
\end{equation*}
so $\Psi(\infty)-\ell\ge0$. Hence, $\Psi(\infty)\ge\sup\set{\ell:w_\ell>0}=\go$.

Conversely, 
\begin{equation*}
\Psi(t)
=
\frac{\sum_{k=0}^{\go} kw_kt^k}{\sum_{k=0}^{\go} w_kt^k} \le \go
\qquad\text{for all $t\in[0,\rho)$,} 
\end{equation*}
so $\Psi(\rho)\le \go$, completing the proof of \ref{L1c}.

Finally, \eqref{l1} follows from \ref{L1a} and \ref{L1x}.
\end{proof}

\begin{remark}
Alternatively, the fact that $\Psi(t)$ is increasing 
can also be seen as follows:
Let $0<a<b<\rho$ and let $Y$
be a random variable with distribution $\P(Y=k)=w_ka^k/\Phi(a)$
(\cf{} \refL{LEPsi}).
Then $\Psi(a)=\E Y$ and $\Psi(b)=\E \bigpar{Y(b/a)^Y}/\E(b/a)^Y$,
so $\Psi(a)\le\Psi(b)$ is equivalent to the correlation inequality
$\E \bigpar{Y(b/a)^Y}\ge\E Y\E(b/a)^Y$, which says that the two random
variables $f(Y)\=Y$ and $g(Y)\=(b/a)^Y$ are positively correlated; it is
well-known that this holds (as long as the expectations are finite)
for any two increasing functions $f$ and $g$ and any $Y$, 
see \cite[Theorem 236]{HLP} where the result is attributed to Chebyshev,
and it is easy to
see that, in fact, strict inequality holds in the present case.
(The latter inequality is an analogue of Harris' correlation inequality 
\cite{Harris} for variables $Y$ with values in a discrete cube
$\setoi^N$; in fact, the inequalities have a common
extension to variables with values in
$\bbR^N$. 
Cf.\ also the related FKG inequality, which extends Harris' inequality;
see for example \cite{Grimmett:RCM} where also its history is described.)

For a third proof that $\Psi(t)$ is increasing, note that 
\eqref{psimgf} shows that
$\Psi$ is (strictly) increasing if and only if $\log\Phi(e^x)$ is (strictly)
convex, which is an easy consequence of \Holder's inequality, 
(See \eg{} \cite[Lemma 2.2.5(a)]{DemboZ} and note that 
$\Phi(e^x)=\sumk e^{kx}w_k$ is the \mgf{} of $\ww$ in the case that $\ww$
is a \pws.)
\end{remark}

\refL{LPsi} shows that $\Psi$ is a bijection
$[0,\rho]\to[0,\Psi(\rho)]=[0,\nu]$, so it has a well-defined inverse
$\Psi\qw:[0,\nu]\to[0,\rho]$. We extend this inverse to $\ooo$ as follows.

\begin{lemma}
  \label{Ltau}
For $x\ge0$ define $\tau=\tau(x)\in[0,\infty]$ by 
\begin{equation}\label{taube}
\tau(x)\=\sup\set{t\le\rho:\Psi(t)\le x}.  
\end{equation}
Then $\tau(x)$ is the unique number in $[0,\rho]$ such
that $\Psi(\tau(x))=x$ when $x\le\nu$, and $\tau(x)=\rho$ when $x\ge\nu$.
Furthermore, the function $x\mapsto\tau(x)$ is continuous, and,
for any $x\ge0$,
\begin{equation}\label{kia0}
  \Psi(\tau(x))=\min(x,\nu).
\end{equation}
If\/ $x<\go$, then
$0\le\tau(x)<\infty$ and $0<\Phi(\tau(x))<\infty$.
On the other hand, if $x\ge\go$, then $\tau(x)=\Phi(\tau(x))=\infty$.
\end{lemma}

\begin{proof}
By \refL{LPsi} and the definition \eqref{nu},
$\Psi$ is an increasing continuous
bijection $[0,\rho]\to[0,\Psi(\rho)]=[0,\nu]$;
thus if $0\le x\le\nu$, there exists a unique $\Psi\qw(x)\in[0,\rho]$ with
$\Psi(\Psi\qw(x))=x$, 
and \eqref{taube} yields $\tau(x)=\Psi\qw(x)$. Since $\Psi$ is a continuous
bijection of one compact space onto another, its inverse
$\Psi\qw:[0,\nu]\to[0,\rho]$ is continuous too; thus
$x\mapsto\tau(x)=\Psi\qw(x)$ is 
continuous on $[0,\nu]$.
Furthermore, \eqref{kia0} holds for $x\le\nu$.

If $x\ge\nu=\Psi(\rho)$, then \eqref{taube} yields $\tau(x)=\rho$, and thus
$\Psi(\tau(x))=\Psi(\rho)=\nu$, so \eqref{kia0} holds in this case too.

Combining the two cases we see that $x\mapsto\tau(x)$ is continuous on
$\ooo$, and that 
\eqref{kia0} holds.

Now suppose that $x<\go$ and $\tau(x)=\infty$. Since $\tau(x)\le\rho$ we
then have $\rho=\infty$, and \refL{LPsi}\ref{L1c} yields
$\Psi(\tau(x))=\Psi(\rho)=\go>x$, contradicting \eqref{kia0}.
Thus $\tau(x)<\infty$ when $x<\go$. 
Furthermore, if $\Phi(\tau(x))=\infty$, 
then $\tau(x)=\rho$, since $\Phi(t)<\infty$ for $t<\rho$, and thus
$\Phi(\rho)=\infty$. 
If further $x<\go$, and thus $\rho=\tau(x)<\infty$ as just shown, 
then \refL{LPsi}\ref{L1b} 
would give $\Psi(\tau(x))=\Psi(\rho)=\infty$, again contradicting
\eqref{kia0} since $x<\infty$. Thus $\Phi(\tau(x))<\infty$ when $x<\go$.

Conversely, if $x\ge\go$, then $\go<\infty$, so $\Phi(t)$ is a polynomial
and $\rho=\infty$. \refL{LPsi}\ref{L1c} shows that
$\Psi(\rho)=\go\le x$, so \eqref{taube} yields $\tau(x)=\rho=\infty$, whence
also $\Phi(\tau(x))=\Phi(\infty)=\infty$.
\end{proof}

Next, we investigate when $Z(m,n)>0$.
We say than an allocation $\yyn$ of $m$ balls in $n$ boxes 
is \emph{good} if it has positive weight,
\ie, if $y_i\in\supp(\wwx)$ for every $i$.
Thus, $Z(m,n)>0$ if and only if there is a good allocation in $\cbmn$;
in this case, the  random allocation $\bmn$ is defined and is always good.

Provided $m$ is not too small or too large,
the $m$ and $n$ for which good allocations exist are easily characterised;
the following lemma shows that a simple necessary condition also is sufficient.
(The exact behaviour for very small $m$ is complicated. The largest $m$ such
that $Z(m,n)=0$ for all $n$ is called the \emph{Frobenius number} of the set
$\supp(\wwx)$; it is a well-known, and in general difficult, problem to
compute this, see \eg{} \cite{frobenius}. 
The case when $m$ is close to $\go n$ (with finite $\go$) is 
essentially the same by the symmetry in \refR{Rsymmetry}.) 

\begin{lemma}
  \label{LBexists}
Suppose that $w_0>0$.
  \begin{romenumerate}
  \item 
If $Z(m,n)>0$, then $\spann(\wwx)\delar m$ and $0\le m\le \go n$.
\item 
If $\go<\infty$, then there exists a constant $C$ (depending on $\wwx$) such
that if 
$\spann(\wwx)\delar m$ and $C\le m\le \go n-C$, then $Z(m,n)>0$.
\item 
If $\go=\infty$, then for each $C'<\infty$, there exists a constant $C$
(depending on $\wwx$ and $C'$) such that if 
$\spann(\wwx)\delar m$ and $C\le m\le C'n$, then $Z(m,n)>0$.
\end{romenumerate}
\end{lemma}

\begin{proof}\CCreset
\pfitem{i}
$Z(m,n)>0$ if and only if $m=\sumin y_i$ for some $y_i$ with $w_{y_i}>0$,
  \ie, $y_i\in\supp(\wwx)$. This implies $0\le y_i\le \go$ and
  $\spann(\wwx)\delar y_i$ for each $i$, and the necessary conditions in (i)
  follow immediately.

\pfitem{ii}
We may for convenience assume that $\spann(\wwx)=1$, see
  \refR{Rdelar}; then, by \eqref{span0}, 
$\supp(\wwx)\setminus\set0$ is a finite set of integers 
with greatest common divisor 1.
Thus, by a well-known theorem by Schur, see \eg{} 
\cite[3.15.2]{Wilf} or \cite[Proposition IV.2]{Flajolet}, there is a
constant $\CC\CCdef\CCsofie$ such that
every integer $m\ge\CCsofie$ can be written as a finite sum
$m=\sum_i{y_i}$ with $y_i\in\supp(\wwx)$ (repetitions are allowed); \ie{} we
have a good allocation of $m$ balls in some number $\ell(m)$ boxes.
Choose one such allocation for each $m\in[\CCsofie,\CCsofie+\go)$,
and let $\CC\CCdef\CCemma$ be the maximum number of boxes in any of them. 

If $\CCsofie\le m\le \go n-\CCemma\go$, let $a\=\floor{(m-\CCsofie)/\go}$.
Then $m-a\go\in[\CCsofie,\CCsofie+\go)$, and has thus a good allocation in
at most $\CCemma$ boxes. We add $a$ boxes with
$\go$ balls each, and have obtained a good allocation of $m$ balls
using at most \begin{equation*}
  \CCemma+a
=
\CCemma+\floor{(m-\CCsofie)/\go}
\le
\CCemma+\floor{(\go n-\CCemma\go-\CCsofie)/\go}
\le n
\end{equation*}
boxes. Hence we may add empty boxes and obtain a good allocation in $\cbmn$.
(Recall that $0\in\supp(\wwx)$.) Thus $Z(m,n)>0$ when 
$\CCsofie\le m\le \go n-\CCemma\go$.

\pfitem{iii} 
We may again assume $\spann(\wwx)=1$.
Let $K$ be a large integer and consider the truncated \ws{}
$\wwxk=(\wk_k)$ defined by
\begin{equation}\label{trunc}
  \wk_k\=
  \begin{cases}
	w_k, & k\le K,
\\
0, & k>K;
  \end{cases}
\end{equation}
we assume that $K\in\supp(\wwx)$ and that $K$ is so large that $K\ge C'+1$ and 
$\spann(\wwxk)=\spann(\wwx)=1$.
Then $\go(\wwxk)=K$, and (ii) shows that for some $\CC\CCdef\CCmagnus$, 
if $\CCmagnus\le m\le Kn-\CCmagnus$, then $Z(m,n;\wwx)\ge Z(m,n;\wwxk)>0$.
Hence, if $m\ge\CCmagnus$ and $Z(m,n)=0$, then $Kn-\CCmagnus < m \le C'n$,
and thus $n< \CCmagnus$, whence $m< C'\CCmagnus$.
Consequently, if $C'\CCmagnus \le m\le C' n$, then $Z(m,n)>0$.
\end{proof}

\begin{remark}
  In the case $\go=\infty$, it is not always true that there is a constant
  $C$ such that $Z(m,n)>0$ whenever $m\ge C$.
For example, suppose that $w_k=1$ when $k=0$ or $k=j!$ for some $j\ge0$, and
$w_k=0$ otherwise. Then $Z(m,n)=0$ when $m=(n+1)!-1$ and $n\ge2$.
\end{remark}

\begin{remark}
\refL{LBexists} is easily modified for the case $w_0=0$; if
  $\ga\=\min\set{k:w_k>0}$ as in \refR{Rmin}, 
then the necessary condition (i) is 
$\ga n\le m\le \go n$ and $\spann(\wwx)\delar (m-\ga n)$, and again this is
sufficient if $m$ stays away from the boundaries.
\end{remark}

\section{Proofs of  Theorems \ref{TBmain}--\ref{TB3}}\label{Sbbpf}

We now prove the theorems in \refS{SBB}; we begin with some lemmas.

First 
we state and prove a version of the local central limit theorem (for
integer-valued variables) that is convenient for our application below. 
We will need it for a triangular array, where the variables we sum depend on
$n$. 

We define the span of an integer-valued random variable to be the span of
its distribution, defined as in \eqref{span}.

\begin{lemma}\label{LCLT}
  Let\/ $\xi$ and $\xi\nnx1,\xi\nnx2,\dots$ 
be integer-valued random variables with  $\xi\nn\dto\xi$ as \ntoo,
and let $S\nn_n\=\sumin\xi\nn_i$, where $\xi\nn_i$ are independent copies of
$\xi\nn$. 
Suppose further that $\xi$ is non-degenerate, with span $d$ 
and finite variance $\gss>0$, and
that
$\sup_n\E|\xi\nn|^3<\infty$.
If\/ $d>1$, we assume for simplicity that $d\delar\xi$ and  $d\delar\xi\nn$
for each $n$.

Let $m=m(n)$ be a sequence of integers that are multiples of $d$, and assume
that $\E\xi\nn=m(n)/n$. 
Then, as \ntoo,
\begin{equation}\label{lclt}
  \P(S\nn_n=m)=\frac{d+o(1)}{\sqrt{2\pi \gss n}}.
\end{equation}
\end{lemma}

\begin{proof} \CCreset
The proof uses standard arguments, see \eg{} \citetq{Theorem 1.4.2}{Kolchin};
we only have to check uniformity in $\xi\nn$ of our estimates.

If the span $d>1$, we may divide $\xi$, $\xi\nn$ and $m$ by $d$, and
  reduce to the case $d=1$. Hence we assume in the proof that
  $\spann(\xi)=1$.

Let $\phi(t)\=\E e^{\ii t\xi}$ and
$\phi_n(t)\=\E e^{\ii t \xi\nn}$ be the \chf{s} of $\xi$ and $\xi\nn$.
Further, let $\tphi_n(t)\=e^{-\ii t m/n}\phi_n(t)$ be the \chf{} of the
centred random variable $\xi\nn-\E\xi\nn=\xi\nn-m/n$.
 
Then $S\nn_n$ has \chf{} $\phi_n(t)^n$, and thus, by the inversion formula and
a change of variables,
\begin{equation}\label{sofie}
  \begin{split}
\P(S\nn_n=m)
&=
\frac1{2\pi}\intpipi e^{-\ii mt}\phi_n(t)^n\dd t
=
\frac1{2\pi}\intpipi \tphi_n(t)^n\dd t
\\&=
\frac1{2\pi\sqrt n}\int_{-\pi\sqrt n}^{\pi\sqrt n} \tphi_n(x/\sqrt n)^n\dd x
\\&=
\frac1{2\pi\sqrt n}\intoooo  \tphi_n(x/\sqrt n)^n \bigett{|x|<\pi\sqrt n}\dd x.
  \end{split}
\end{equation}

Let $\gss_n$ be the variance of $\xi\nn$. 
Since $\E|\xi\nn|^3$ are uniformly bounded, 
$\gss_n<\infty$; moreover,
the random variables $\xi\nn$
are uniformly square integrable
and it follows from $\xi\nn\dto\xi$ that
$\gss_n\to\gss$. 
(See e.g.{}
\citetq{Theorems 5.4.2 and 5.4.9}{Gut} for this standard argument.)
In particular, $\gss_n\ge\gss/2$ for all sufficiently large $n$;
we consider in the remainder of the proof only such $n$.

Since $\tphi_n(t)$ is the \chf{} of $\xi\nn-\E\xi\nn$ which has mean 0 and,
by assumption, an absolute third moment that is uniformly bounded,
we have by a standard expansion
(see \eg{} \cite[Theorems 4.4.1]{Gut})
\begin{equation}\label{erika}
  \tphi_n(t)
= 1-\tfrac12 {\gss_n}t^2+O(\E|\xi\nn|^3|t|^3)
= 1-\tfrac12 {\gss_n}t^2+O(|t|^3),
\end{equation}
uniformly in all $n$ and $t$.
In particular, for any fixed real $x$,
\begin{equation}
  \tphi_n(x/\sqrt n)
=1-\frac{\gss_n x^2}{2n}+O(n\qqcw)
=1-\frac{\gss x^2+o(1)}{2n},
\end{equation}
and thus
\begin{equation}\label{anna}
  \tphi_n(x/\sqrt n)^n\to e^{-\gss x^2/2}.
\end{equation}
We are aiming at estimating the integral in \eqref{sofie} by dominated
convergence, so we also need a suitable bound that is uniform in $n$.

We write \eqref{erika} as 
$\bigabs{\tphi_n(t)- (1-\frac12 {\gss_n}t^2)}\le \CC|t|^3$.
Let $\gd\=\gss/8\CCx>0$. Then, if $|t|\le\gd$, recalling our assumption
$\gss_n\ge\frac12\gss$,
\begin{equation}\label{magnus}
 |\tphi_n(t)|\le
 1-\tfrac12 {\gss_n}t^2+\CCx|t|^3
\le 
 1-\tfrac14 {\gss}t^2+\CCx\gd t^2
= 1-\tfrac18 {\gss}t^2.
\end{equation}

For $\gd\le |t|\le\pi$ we claim that there exists $n_0$ and $\eta>0$
such that if $n\ge n_0$ and $\gd\le|t|\le\pi$, then 
\begin{equation}
  \label{olof}
|\tphi_n(t)|\le1-\eta.
\end{equation}
In fact, if this were not true, then there would exist sequences $n_k\ge
k$ and $t_k\in[\gd,\pi]$ (by symmetry, it suffices to consider $t>0$) such
that 
$|\phi_{n_k}(t_k)|=|\tphi_{n_k}(t_k)|>1-1/k$.
By considering a subsequence, we may assume that $t_k\to\txoo$ as \ktoo{} for
some $\txoo\in[\gd,\pi]$. Since $\xi_n\dto\xi$, $\phi_{n_k}(t)\to\phi(t)$
uniformly for $|t|\le\pi$, and thus 
$\phi_{n_k}(t_k)\to\phi(\txoo)$. It follows that $|\phi(\txoo)|=1$ for some
$\txoo\in[\gd,\pi]$, but this is impossible when $\spann(\xi)=1$, as is 
well-known (and easily seen from
$\E e^{\ii \txoo(\xi-\xi')}=|\phi(\txoo)|^2=1$, where $\xi'$ is an independent
copy of $\xi$).
This contradiction shows that \eqref{olof} holds.

We can combine \eqref{magnus} and \eqref{olof}; we let
$\cc\=\min\set{\gss/8,\eta/\pi^2}\ccdef\ccnu$ and obtain, for $n\ge n_0$,
\begin{equation*}
  |\tphi_n(t)|\le 1-\ccnu t^2
\le \exp(-\ccnu t^2),
\qquad |t|\le\pi,
\end{equation*}
and thus
\begin{equation*}
  |\tphi_n(x/\sqrt n)|^n
\le \exp(-\ccnu x^2),
\qquad |x|\le\pi\sqrt n.
\end{equation*}
This justifies the use of dominated convergence in \eqref{sofie}, 
and we obtain by \eqref{anna}
\begin{equation*}
  \begin{split}
2\pi\sqrt{n}\P(S\nn_n=m)
&=
\intoooo  \tphi_n(x/\sqrt n)^n \bigett{|x|<\pi\sqrt n}\dd x
\\&
\to
\intoooo  e^{-\gss x^2/2}\dd x
=\sqrt{2\pi/\gss},
  \end{split}
\end{equation*}
which yields \eqref{lclt}.  (Recall that we have assumed $d=1$.)
\end{proof}

\begin{remark}
  \label{RCLT}
A simple modification of the proof shows that the result still holds if the
condition $\E\xi\nn=m(n)/n$ is relaxed to
$m(n)=n\E\xi\nn+o(\sqrt{n})$. 
Furthermore, for {any} $m=m(n)$,
$\P(S\nn_n=m)\le
\frac1{2\pi}\intpipi |\tphi_n(t)|^n\dd t$, and it follows by the proof
above that 
\begin{equation}\label{lcltle}
  \P(S\nn_n=m)\le\frac{d+o(1)}{\sqrt{2\pi \gss n}},
\end{equation}
uniformly in all $m\in\bbZ$.

Moreover, both \refL{LCLT} and the remarks above hold, with only minor
modifications in the proof,  also if the condition 
$\sup_n\E|\xi\nn|^3<\infty$ is relaxed to uniform square integrability of
$\xi\nn$. 
In particular, if $\xi\nn=\xi$, this assumption is not needed at all; then the
assumption $\gss<\infty$ is the only moment condition that we need.
(This is the classical local central limit theorem for discrete
distributions,
see \eg{} \citetq{\S\ 49}{GneKol}
or \citetq{Theorem 1.4.2}{Kolchin}.)
\end{remark}

We use \refL{LCLT} to obtain lower bounds of the (rather weak) type $\exp(o(n))$
for $\P(S_n=m)$ in the case of a \pws, for suitable $m$. We treat the cases
$\rho>1$ and $\rho=1$ separately.

\begin{lemma}
  \label{LB2}
Let $\wwx$ be a probability weight sequence with $0<w_0<1$ and $\rho>1$.
Let $\xi_1,\xi_2,\dots$ be \iid{} random variables with distribution $\wwx$ 
and let $S_n\=\sumin \xi_i$.

Assume that $m=m(n)$ are integers  that are multiples of $d\=\spann(\wwx)$,
and that $m(n)/n\to\E\xi_1$.
Then
$$\P(S_n=m)=Z(m,n)=e^{o(n)}.$$
\end{lemma}

\begin{proof}
Let $\xi\=\xi_1$ and $\gl\=\E\xi=\Phi'(1)=\Psi(1)$.
Since $\rho>1$, we have $\nu>\Psi(1)=\gl$.
Thus, by assumption, $m/n\to\gl<\nu$, so $m/n<\nu$ for all large $n$; we
consider in the sequel only such $n$. By \refL{LPsi} we may then define
$\taun\in[0,\rho)$ by $\Psi(\taun)=m/n$. 
Since $\Psi\qw$ is continuous on $[0,\nu)$, and $\Psi(1)=\E\xi=\gl$, we have
  \begin{equation}
	\label{taunus}
\taun=\Psi\qw(m/n)\to\Psi\qw(\gl)=1
\qquad
\text{as \ntoo}.
 \end{equation}

Let $\xi\nn$ have the conjugate distribution
\begin{equation}\label{xinn}
  \P(\xi\nn=k)=\frac{\taun^k}{\Phi(\taun)}w_k,
\qquad k\ge0;
\end{equation}
by \refL{LEPsi} this is a probability distribution with expectation
\begin{equation}
  \E\xi\nn=\Psi(\taun)=m/n.
\end{equation}
The conditions of \refL{LCLT} are easily verified: 
Since $\taun\to1$ by \eqref{taunus}, we have
$\P(\xi\nn=k)\to w_k=\P(\xi=k)$ and thus $\xi\nn\dto\xi$.
Furthermore, taking any $\tau_*\in(1,\rho)$ and considering only $n$ that
are so large that $\taun<\tau_*$,
\begin{equation*}
  \E|\xi\nn|^3
=\sumk k^3 \frac{\taun^k}{\Phi(\taun)}w_k
\le \frac1{\Phi(0)}\sumk k^3 \tau_*^kw_k 
<\infty.
\end{equation*}
Furthermore, if $d=\spann(\xi)$, then $w_k>0\implies d\delar k$ 
by \eqref{span0};
thus $d\delar\xi$ and $d\delar\xi\nn$ (a.s.).
\refL{LCLT} thus applies, and
if $\wwx\nn$ denotes the distribution of $\xi\nn$ in \eqref{xinn},
then by \eqref{ebprob} and
\eqref{lclt},
\begin{equation}\label{xia}
  Z(m,n;\wwx\nn)
=\P\biggpar{\sumin\xi\nn_i=m}
\sim\frac{d}{\sqrt{2\pi\gss n}},
\end{equation}
where 
$\gss\=\Var\xi$.
By \eqref{lb1z}, we have $Z(m,n;\wwx\nn)=\Phi(\taun)^{-n}\taun^m Z(m,n)$, and
thus, recalling that $\taun\to1<\rho$ and hence 
$\Phi(\taun)\to\Phi(1)=1$,
\begin{equation*}
  \begin{split}
	\P(S_n=m)
&= Z(m,n)
=
\taun^{-m}\Phi(\taun)^n Z(m,n;\wwx\nn)
\\&
=\exp\bigpar{-m\log\taun+n\log\Phi(\taun)+\log Z(m,n;\wwx\nn)}
\\&
=\exp\bigpar{o(n)}.
\qedhere
  \end{split}
\end{equation*}
\end{proof}

\begin{lemma}
  \label{LB2sub}
Let $\wwx$ be a probability weight sequence with $0<w_0<1$ and $\rho=1$.
Let $\xi_1,\xi_2,\dots$ be \iid{} random variables with distribution $\wwx$
and let $S_n\=\sumin \xi_i$.

Assume that $m=m(n)$ are integers  that are multiples of $d\=\spann(\wwx)$,
and that $m(n)/n\to\gl<\infty$ with $\gl\ge\E\xi_1$.
Then
$$\P(S_n=m)=e^{o(n)}.$$
\end{lemma}

\begin{proof}
Let $K$ be a large integer and consider the truncated \ws{}
$\wwxk=(\wk_k)$ defined by, as in \eqref{trunc},
\begin{equation}\label{trunc2}
  \wk_k\=
  \begin{cases}
	w_k, & k\le K,
\\
0, & k>K,
  \end{cases}
\end{equation}
having generating function 
$\Phik(t)=\sumkk w_kt^k$, and the corresponding
$\Psik(t)\=t\Phik'(t)/\Phik(t)$.
We assume that $K$ is so large that $\spann(\wwxk)=\spann(\wwx)$,
and that $K>k$ for some $k>\gl$ with $w_k>0$. (Such $k>\gl$ exists since
$\rho<\infty$.) Thus the \ws{}  $\wwxk$ has, by \refL{LPsi}\ref{L1c},
$\nu(\wwxk)=\Psik(\infty)=\go(\wwxk)>\gl$.
Hence, by \refL{LPsi} again, there exists $\tauk\in\ooo$ 
such that $\Psik(\tauk)=\gl$.
Thus the probability distribution $\ppik=(\pik_k)$ defined by
\begin{equation}\label{pikk}
  \pik_k\=\frac{\tauk^k}{\Phik(\tauk)}\wk_k
\end{equation}
has expectation $\gl$.
Since this distribution has finite support it has radius of convergence
$\rhok=\infty$; furthermore, $m/n\to\gl$ by assumption.
Hence \refL{LB2} applies to $\ppik$ and yields
\begin{equation}
  \label{u1}
Z(m,n;\ppik)=e^{o(n)}.
\end{equation}
By \eqref{lb1z} and \eqref{pikk},
\begin{equation}
  \label{u2}
Z(m,n;\ppik)=\Phik(\tauk)^{-n}\tauk^mZ(m,n;\wwxk).
\end{equation}
Moreover, $Z(m,n;\wwx)\ge Z(m,n;\wwxk)$ since $w_k\ge\wk_k$ for each
$k$. Hence, by \eqref{u1} and \eqref{u2},
\begin{equation}
  \label{jesper}
  \begin{split}
Z(m,n;\wwx)
&\ge Z(m,n;\wwxk)
=\tauk^{-m}\Phik(\tauk)^{n}Z(m,n;\ppik)
\\&
=\tauk^{-m}\Phik(\tauk)^{n}e^{o(n)}.	
  \end{split}
\end{equation}

This holds for every large fixed $K$.
Now let $K\to\infty$.

If $0<t<\rho=1$, then $\Phik(t)\to\Phi(t)$ and $\Phik'(t)\to\Phi'(t)$ as \Ktoo,
so $\Psik(t)\to\Psi(t)<\Psi(1)=\E\xi_1\le\gl$.
Hence, for large $K$, $\Psik(t)<\gl=\Psik(\tauk)$, so $\tauk>t$.
Consequently, $\liminf_{\Ktoo}\tauk\ge1$.

On the other hand, 
if $t>\rho=1$, let $\ell\=\ceil{\gl}+1>\gl$, and assume $K\ge\ell$.
Then
\begin{equation}\label{qqq}
  \Psik(t)=\frac{\sumkk kw_kt^k}{\sumkk w_kt^k}
\ge
\frac{\sum_{k=\ell}^K \ell w_kt^k}{\sumkk w_kt^k}
=
\ell - \frac{\sum_{k=0}^{\ell-1} \ell w_kt^k}{\Phik(t)}
\to\ell>\gl,
\end{equation}
as \Ktoo, since $\Phik(t)\to\Phi(t)=\infty$.
Hence, for large $K$, $\Psik(t)>\gl=\Psik(\tauk)$, and thus $\tauk<t$.
Consequently, $\limsup_{\Ktoo}\tauk\le1$.

Combining these upper and lower bounds, we have
\begin{equation*}
  \tauk\to1,\qquad\text{as }\Ktoo.
\end{equation*}
If we take $t<1$, we thus have for large $K$, 
$\tauk>t$ and hence $\Phik(\tauk)>\Phik(t)$.
Thus, $\liminf_{\Ktoo}\Phik(\tauk)\ge\lim_{\Ktoo}\Phik(t)
=\Phi(t)$ 
for every $t<1$, so
\begin{equation*}
  \liminf_{\Ktoo}\Phik(\tauk)\ge\Phi(1)=1.
\end{equation*}

Given any $\eps>0$, we may thus take $K$ so large that $\tauk<e^\eps$ and
$\Phik(\tauk)>e^{-\eps}$. Then \eqref{jesper} yields
\begin{equation*}
Z(m,n;\wwx)\ge 
e^{-\eps m-\eps n+o(n)}
\ge
e^{-\eps m-2\eps n}
\end{equation*}
for large $n$. Since $\eps$ is arbitrary and $m=O(n)$, this shows 
$Z(m,n;\wwx)\ge e^{-o(n)}$, and the result follows since $Z(m,n)\le1$ for
  any \pws{} by \eqref{ebprob}.  
\end{proof}

\begin{proof}[Proof of \refT{TBmain}] 
First, \refL{Ltau} shows that
$\tau$ defined by \ref{tbmaincritical}
and \ref{tbmainsub} is well-defined and
equals $\tau(\gl)$ defined in \refL{Ltau};
since $\gl<\go$ we have $\tau<\infty$ and $\Phi(\tau)<\infty$.
Further, \eqref{kia0} yields
\begin{equation}\label{kia}
  \Psi(\tau)=\min(\gl,\nu).
\end{equation}

Since $\tau<\infty$ and $\Phi(\tau)<\infty$, $\pi_k$ is well-defined by
\eqref{pik}; furthermore, by
\refL{LEPsi} and \eqref{kia}, $\ppi$ is a probability distribution with mean
and variance as asserted. 

We now turn to proving \eqref{tbmain}, the main assertion.
We study three cases separately.

\pfcase{(a)}{$\tau>0$.}
Then $\ppix=(\pi_k)$ is a \pws{} equivalent to $\wwx=(w_k)$, so we may
replace $(w_k)$ 
by $(\pi_k)$ without changing $\bmn$. Note that this changes
$\rho$ and $\tau$ to $\rho(\ppix)=\rho(\wwx)/\tau$
and $\tau(\ppix)=\tau(\wwx)/\tau=1$ by \eqref{tPhi} and \eqref{tPsi}.
We may thus assume that $\ww$ equals the
\pws{} $\ppi$, and that $\rho\ge\tau=1$.
By \eqref{kia}, then $\Psi(1)=\min(\gl,\nu)$. 

We employ the notation of \refE{EBprob}.
Note that by \eqref{tbm},
\begin{equation}\label{esub}
  \E\xi_1=\Psi(1)=
\min(\gl,\nu)\le \gl.
\end{equation}
Moreover, if $\rho>1$, then $\nu=\Psi(\rho)>\Psi(1)$ by \refL{LPsi}, so 
\eqref{kia} shows that in this case,
\begin{equation}\label{egeneric}
 \E\xi_1=\Psi(1)=\gl. 
\end{equation}

The allocation $(\xi_1,\dots,\xi_n)$ (with a random sum $S_n$) consists of
$n$ \iid{} components, so 
\begin{equation}
N_k(\xin)
=\sumin\ett{\xi_i=k}\sim\Bi(n,\pi_k)  
\end{equation}
has a binomial distribution. For every $k$ and $\eps>0$,
we have by Chernoff's inequality, see \eg{} 
\cite[Theorem 2.1 or Remark 2.5]{JLR},
\begin{equation}
  \P\bigpar{\abs{N_k(\xin)-n\pi_k}>\eps n}
\le \exp(-\cce n),
\end{equation}
for some constant $\cce>0$ depending on $\eps$.

We condition on $S_n=m$, recalling that
\begin{equation}\label{bdd}
  \bmn\eqd\bigpar{(\xin)\mid S_n=m}.
\end{equation}
When $\rho>1$ we apply \refL{LB2}, using $m/n\to\gl$ and \eqref{egeneric},
and when $\rho=1$ we apply \refL{LB2sub}, using \eqref{esub}. In both cases we
obtain $\P(S_n=m)=\exp(o(n))$ and thus by \eqref{bdd},
\begin{multline*}
  \P\bigpar{\abs{N_k(\bmn)-n\pi_k}>\eps n}
=  \P\bigpar{\abs{N_k(\xin)-n\pi_k}>\eps n \mid S_n=m}  
\\
\le \frac{  \P\bigpar{\abs{N_k(\xin)-n\pi_k}>\eps n}}{\P(S_n=m)}
\le\exp\Bigpar{-\cce n+o(n)}
\to0.	 
\end{multline*}
Since $\eps$ is arbitrary, this shows that
\begin{equation*}
  \frac{N_k(\bmn)}{n}-\pi_k\pto0
\end{equation*}
as asserted, which completes the proof when $\tau>0$.

\pfcase{(b)}{$\tau=0$ and $\rho>0$.}
We write $N_k$ for $N_k(\bmn)$.
By \eqref{pik} we have $\pi_0=1$ and $\pi_k=0$ for $k>0$; hence, 
\eqref{tbmain} says that $N_0/n\pto1$ and $N_k/n\pto0$ for $k>0$.

Since $\tau<\rho$, we are in case \ref{tbmaincritical}, so
$\gl=\Psi(\tau)=\Psi(0)=0$. In other words, $m/n\to0$.
The result is trivial (and deterministic) in this case. We have
\begin{equation}\label{victoria}
  \frac1n\sum_{k=1}^\infty N_k
\le\frac1n\sum_{k=1}^\infty kN_k
=\frac mn\to\gl=0.
\end{equation}
Hence $N_k/n\to 0=\pi_k$ for every $k\ge1$. Moreover, \eqref{victoria} also
implies
\begin{equation}
\frac{N_0}n=  \frac{n-\sum_{k=1}^\infty N_k}{n}
\to1=\pi_0,
\end{equation}
which completes the proof when $\tau=0<\rho$. 
\noqed
\end{proof}

Before we treat the remaining case in \refT{TBmain}, we show that
\refT{TBmain2} too holds in the cases treated so far.

\begin{proof}[Proof of \refT{TBmain2} from \refT{TBmain}]
We prove that \refT{TBmain} for some \ws{} $\ww$ implies \refT{TBmain2} for
the same weights.
  The assertions about $\tau$ follow from \refL{Ltau}, so we turn to
  \eqref{tbmain2}.

Consider a subsequence of $(m(n),n)$. It suffices to show that every such
subsequence has a subsubsequence such that \eqref{tbmain2} holds.
(See \eg{} \cite[Section 5.7]{Gut}, \cite[p.~12]{JLR} or \cite[Theorem
  2.3]{Billingsley} for this standard argument.)

Since $m/n\le C$ by assumption, we can select a subsubsequence such that
$m/n\to\gl$ for some $\gl\le C<\go$. Then \refT{TBmain} applies
and thus (along the subsubsequence), 
\begin{equation}\label{anna1}
  \frac{N_k(\bmn)}{n}-\frac{w_k(\tau(\gl))^k}{\Phi(\tau(\gl))}
=
  \frac{N_k(\bmn)}{n}-\pi_k
\pto0.
\end{equation}
Furthermore, since $m/n\to\gl$ and $x\mapsto\tau(x)$ is continuous,
$\tau(m/n)\to\tau(\gl)$ (along the subsubsequence); hence
\begin{equation}\label{anna2}
\frac{w_k(\tau(m/n))^k}{\Phi(\tau(m/n))}
-\frac{w_k(\tau(\gl))^k}{\Phi(\tau(\gl))}
\to0.
\end{equation}
Combining \eqref{anna1} and \eqref{anna2}, we see that \eqref{tbmain2} holds
along the subsubsequence, which as said above completes the proof of
\eqref{tbmain2}. 

That \eqref{tbmain2} holds uniformly is, in fact, automatic since we have
shown it for an arbitrary $m(n)$ (although we stated it for emphasis):
Let $X_{m,n}$ denote the \lhs{} of \eqref{tbmain2}, and let $\eps>0$.
Choose $m(n)$ as the integer $m\in[0,Cn]$  that maximises
$\P(|X_{m,n}|>\eps)$. Since 
\eqref{tbmain2} says that $\P(|X_{m(n),n}|>\eps)\to0$, we 
have
$\sup_{m\le Cn}\P(|X_{m,n}|>\eps)\to0$.
\end{proof}

\begin{proof}[Completion of the proof of \refT{TBmain}] \CCreset
\pfcase{(c)}{$\rho=0$.}
We write again $N_k$ for $N_k(\bmn)$, recalling that this is a random variable.
In this case $\nu=0$ and $\tau=\rho=0$ for every $\gl\ge0$.
By \eqref{pik} we thus have $\pi_0=1$ and $\pi_k=0$ for $k>0$; hence, as in
case (b), we have
to show that $N_0/n\pto1$ and $N_k/n\pto0$ for $k>0$.
By assumption, $m/n$ converges, so the sequence $m/n$ is bounded; let $C$ be
a large constant such that $m/n\le C$. Further, let $K$ be a large integer;
we assume $K>2C$ and (for simplicity) $w_K>0$. (Note that such $K$ exist
since $\go=\infty$ when $\rho=0$.)

We say that a  box is \emph{small} if it contains at most $K$ balls, and
\emph{large} otherwise. 
Let $N'\=\sum_{0}^KN_k$ be the number of small boxes and 
$M'\=\sum_{0}^KkN_k$ the number of balls in them. 
Note first that by our assumptions, $m/n\le C<K/2$. Hence,
\begin{equation}\label{small}
  m
\ge
m-M'
=
\sum_{K+1}^\infty kN_k
\ge K\sum_{K+1}^\infty N_k
=K(n-N')
\ge \frac{2m}n({n-N'}).
\end{equation}
Thus, $n-N'\le n/2$ and $N'\ge n/2$; in particular $N'\to\infty$.
Moreover,
\begin{equation}\label{emma}
  0\le\frac{M'}{N'}
\le\frac{m}{n/2}
\le 2C<K.
\end{equation}

The weight $w(\yyx)$ in \eqref{wmm} factorizes as the product over the small
boxes times the product over the large boxes. Thus,
if we condition on $M'$ and $N'$, and moreover on the set of the $N'$ boxes
that are small, then the allocations of the small boxes and the large boxes
are independent; moreover, the allocations to the small boxes form a random
allocation of the type $B_{M',N'}$ for the truncated weight sequence
$\wwx\kk$ given by
\eqref{trunc} above.
By assumption, $w_K>0$, and thus the truncated sequence has
$\go\kk\=\go(\wwx\kk)=K$. 

The truncated weight sequence $\wwx\kk$ has a polynomial generating function
$\Phi\kk(t)=\sum_{0}^Kw_kt^k$
with an infinite radius of convergence $\rho\kk=\infty$.
We have already proved \refT{TBmain} in this case, and thus \refT{TBmain2}
also holds in this case, by the proof above.
Applying \refT{TBmain2} to the truncated weight sequence and the allocations
of small boxes we see that there exists a continuous function
$\tau_K:[0,K)\to\ooo$ such that, conditioned on $(M',N')$, 
\begin{equation}\label{kok}
 \frac{N_k}{N'}-\frac{w_k(\tau_K(M'/N'))^k}{\Phi\kk\bigpar{\tau_K(M'/N')}}\pto0,
\qquad k\le K.
\end{equation}
Moreover, \eqref{kok} holds uniformly in all $(M',N')$ by \refT{TBmain2} and
\eqref{emma}.
Hence, denoting the
\lhs{} of \eqref{kok} by $X$, we have for every $\eps>0$
$\P(|X|>\eps\mid M',N')\le\gd(n)$, for some function $\gd(n)\to0$.
Taking the expectation, it follows that also 
$\P(|X|>\eps)\le\gd(n)\to0$, and thus \eqref{kok} holds also
unconditionally. Thus,
\begin{equation}\label{kok1}
 \frac{N_k}{N'}
=\frac{w_k(\tau_K(M'/N'))^k}{\Phi\kk\bigpar{\tau_K(M'/N')}}+\op(1),
\qquad k\le K.
\end{equation}

By \eqref{emma}, $M'/N'\le 2C$, and thus,
using \refL{Ltau} and $2C<K=\go(\wwx\kk)$,
$\tau_K(M'/N')\le\tau_K(2C)<\infty$.
Hence, with $\CC\=\tau_K(2C)$,
\begin{equation*}
w_0\le\Phi\kk(\tau_k(M'/N'))
\le\Phi\kk(\CCx)=\CC,
\CCdef\Cphi
\end{equation*}
say. 
Taking $k=0$ in \eqref{kok1} we now find
\begin{equation}\label{koks}
  \frac{N_0}{N'}
=
\frac{w_0}{\Phi\kk\bigpar{\tau_K(M'/N')}}+\op(1)
\ge
\frac{w_0}{\Cphi}+\op(1).
\end{equation}
Since $N'\ge n/2$ this shows that there exists $\cc>0$ 
(for example $\ccx\=w_0/(3\Cphi)$)
such that 
\whp{}
\begin{equation}\label{nobig}
  \frac{N_0}n\ge\ccx.
\end{equation}
It follows further from \eqref{koks} that we can invert \eqref{kok1} for
$k=0$ (since $x\mapsto x\qw$ is continuous for $x>0$); thus
\begin{equation}\label{kok-1}
 \frac{N'}{N_0}=\frac{\Phi\kk\bigpar{\tau_K(M'/N')}}{w_0}+\op(1).
\end{equation}
Multiplying \eqref{kok1} and \eqref{kok-1} we find the simpler relation
\begin{equation}\label{koko}
 \frac{N_k}{N_0}
=\frac{w_k}{w_0}(\tau_K(M'/N'))^k+\op(1),
\qquad k\le K.
\end{equation}

Let $\ell\=\min\set{k>0:w_k>0}$ be the smallest non-zero index with positive
weight, and define a random variable by
\begin{equation}\label{taux}
  \taux\= \parfrac{w_0N_\ell}{w_\ell N_0}^{1/\ell}.
\end{equation}
It follows from \eqref{koko}, with $k=\ell$, that
$\taux=\tau_K(M'/N')+\op(1)$. Consequently, \eqref{koko} yields
\begin{equation}\label{kokox}
 \frac{N_k}{N_0}
=\frac{w_k}{w_0}\taux^k+\op(1),
\qquad k\le K.
\end{equation}

We have so far worked with a fixed, large $K$. However, the definition
\eqref{taux} does not depend on the choice of $K$, and since $K$ may be
chosen arbitrarily large, we see that, in fact, \eqref{kokox} holds for
every $k\ge0$, with the same (random) $\taux$.

Fix again $K>0$, and sum \eqref{kokox} for $k\le K$. This yields
\begin{equation}\label{kokoy}
\frac{n}{N_0}
\ge
 \sum_{0}^K\frac{N_k}{N_0}
=\sum_{0}^K\frac{w_k}{w_0}\taux^k+\op(1)
=\frac{\Phi\kk(\taux)}{w_0}+\op(1).
\end{equation}
Recall that $N_0/n\ge\ccx$ \whp{} by \eqref{nobig}. 
We thus have from \eqref{kokoy}
\begin{equation}\label{kokoz}
\Phi\kk(\taux)
\le
w_0\frac{n}{N_0}+\op(1)
\le
w_0/\ccx+1
\end{equation}
\whp.
By assumption, $\rho=0$, so $\Phi(t)=\infty$ for every $t>0$.
Hence, for every $\eps>0$ we have $\Phi\kk(\eps)\to\Phi(\eps)=\infty$ as
$K\to\infty$, so we may choose $K$ with $\Phi\kk(\eps)>w_0/\ccx+1$.
Then \eqref{kokoz} shows that $\taux<\eps$ whp; since $\eps>0$ is arbitrary,
this says that 
\begin{equation*}
\taux\pto0.  
\end{equation*}
We substitute this in \eqref{kokox}, and obtain
$N_k/N_0\pto 0$ for every $k\ge1$; hence also
\begin{equation}\label{julie}
  N_k/n\pto 0,\qquad k\ge1.
\end{equation}

Finally, we return to \eqref{small}, and see that
\begin{equation}\label{sjw}
  K(n-N')\le m\le Cn.
\end{equation}
Let $\eps>0$ and choose $K>C/\eps$; then \eqref{sjw} yields 
$n-N'<\eps n$ and thus $N'>(1-\eps)n$.
Further, by \eqref{julie},
\begin{equation*}
N_0=N'-\sum_{1}^K N_k=N'-\op(n)>(1-\eps)n-\op(n),
\end{equation*}
so \whp{} $N_0>(1-2\eps)n$. This shows that $N_0/n\pto1$, which together
with \eqref{julie} completes the proof in the case $\rho=0$.
\end{proof}

This completes the proof of \refT{TBmain}, and thus also of \refT{TBmain2}.

\begin{proof}[Proof of \refT{TB3}]
  Conditioned on the numbers $N_k=N_k(\bmn)$, $k=0,1,\dots$, the 
numbers $Y_1,\dots,Y_n$ are obtain by placing $N_0$ 0's, $N_1$ 1's, \dots,
in (uniformly) random order; thus the conditional
probability is
\begin{equation}\label{pyn}
  \P(Y_1=y_1,\dots,Y_\ell=y_\ell\mid N_0,N_1,\dots)
= \prod_{i=1}^\ell \frac{N_{y_i}-c_i}{n-i+1}
= \prod_{i=1}^\ell \frac{N_{y_i}+O(1)}{n+O(1)},
\end{equation}
where $c_i\=|\set{j<i:y_j=y_i}|$.
By \refT{TBmain}, this product converges in probability to 
$\prod_{i=1}^\ell \pi_{y_i}$ as \ntoo, and the result follows by taking the
expectation (using dominated convergence).
\end{proof}

\section{Trees and \bib}\label{Stree-balls}

The proofs of the results for random trees are based on a connection with
the \bib{} model. This connection is well-known, 
see \eg{} \citet{Otter}, \citet{Dwass}, \citet{Kolchin}, \citet{Pitman:enum},
but for completeness we
give full proofs.

We consider a fixed \ws{} $\wwx=\ww$ and the corresponding random trees
$\ctn$ and random allocations $\bmn$; we write $\bmn=\YYn$.

We begin with some deterministic considerations.
The idea is to regard the outdegrees of the nodes of a tree $T$ as an
allocation; we regard the nodes as both balls and boxes, and if $v$ is a
node, we put the children of $v$ as balls in box $v$. There are two
complications, which will be dealt with in detail below: we have to specify
an ordering of the nodes and we will not obtain all allocations.

Let $T$ be a finite tree, with $|T|=n$. Take the nodes in some prescribed
order $v_1,\dots,v_n$, for definiteness we use the depth-first order (this
is the lexicographic order on $\voo$), and list the outdegrees 
as $d_1=\dx(v_1),\dots,d_n=\dx(v_n)$. 
We call this the \emph{degree sequence} of $T$ and denote it by
$\gL(T)\=\ddn$. Note that the tree $T$ can be reconstructed from $\ddn$, so
$T$ is determined by $\gL(T)=\ddn$.

By \eqref{sumd}, $d_1+\dots+d_n=n-1$, so $\ddn$ can be seen as an allocation
of $n-1$ balls in $n$ boxes: $\gL(T)=\ddn\in\cbnn$.
Consequently, $\gL$ is an injective map $\stn\to\cbnn$.
Note also that $\gL$ preserves the weight: 
\begin{equation}\label{wp}
w(T)=w(\gL(T))  
\end{equation}
by the definitions \eqref{wtree} and \eqref{wmm}.
However, not every allocation corresponds to a tree, so $\gL$ is not onto.
We begin by characterizing the image $\gL(\stn)$.
We use a simple and well-known extension of \eqref{sumd}.

\begin{lemma} 
  \label{Ldsub}
Let $T$ be a tree and $T'$ a subtree with the same root. Let 
$\partial T'\=\set{v\in V(T)\setminus V(T'):v\sim w \text{ for some $w\in T'$}}$
be the set of nodes outside $T'$ with a parent inside it.
Then,
\begin{equation}
  \sum_{v\in T'}\dx_T(v)=|T'|+|\partial T'|-1.
\end{equation}
\end{lemma}
\begin{proof}
  The set of children of the nodes in $T'$ consists of
  $\bigpar{V(T')\setminus\set o}\cup\partial T'$.
\end{proof}

\begin{lemma}\label{Ld}
  A sequence $\ddn\in\No^n$ is the degree sequence of a tree $T\in\stn$ if
  and only if
  \begin{align}\label{ld}
\sum_{i=1}^k d_i &\ge k,
\qquad 1\le k<n,
\\	\label{ld2}
\sum_{i=1}^n d_i &=n-1.
  \end{align}
\end{lemma}
Of course, \eqref{ld2} is just the requirement that $\ddn\in\cbnn$.

\begin{proof}
  For any $k\le n$, the nodes $v_1,\dots,v_k$ form a subtree $T_k$ of $T$,
  and \refL{Ldsub} yields
\begin{equation}
  \sum_{i=1}^k\dx_T(v_i)=|\partial T_k|+k-1,
\end{equation}
which yields \eqref{ld} since $|\partial T_k|\ge1$ when $k<n$.

Conversely, if $\ddn$ satisfies \eqref{ld}--\eqref{ld2}, a tree with degree
sequence $\ddn$ is easily constructed. (The point is that \eqref{ld} assures
that the construction will not stop before we have $n$ nodes.)
  \end{proof}

The amazing fact is that for any allocation in $\cbnn$, exactly one of its
cyclic shifts satisfies \eqref{ld}.
(In particular, exactly $1/n$ of all allocations satisfy \eqref{ld}.)
To see this, it is simplest to consider the sequence $(d_i-1)_{i=1}^n$; we
state a more general result that we will use later, see \eg{}
\citet{Takacs}, \citet{Wendel}, \citet{Pitman:enum}.

\begin{lemma}\label{Lcyclic}
Let $x_1,\dots,x_n\in\set{-1,0,1,\dots}$ with $x_1+\dots+x_n=-\ellr\le0$.
For $j\in\bbZ$, let $x_1\xjj,\dots,x_n\xjj$ be the cyclic shift defined by
$x_i\xjj\=x_{i+j}$ with the index taken modulo $n$, and consider the
corresponding partial sums $S_k\xjj\=\sumik x_i\xjj$,
$k=0,\dots,n$. Then there are exactly $\ellr$ values of $j\in\setn$ such that
\begin{equation}
  \label{lc}
S_k\xjj>-\ellr, \qquad 0\le k<n.  
\end{equation}
\end{lemma}
Note that $S_0\xjj=0$ and $S_n\xjj=-\ellr$ for every $i$. The condition 
\eqref{lc} thus says that the walk $S_0\xjj,\dots,S_n\xjj$ first reaches
$-\ellr$ at time $n$.
The case $r=0$ is trivial: since $S_0\xjj=0$, \eqref{lc} then is never satisfied
for $k=0$. 

\begin{proof}
We extend the definition of $x_j$ for all $j\in\bbZ$ by taking the index
modulo $n$; thus $x_{j+n}=x_j$. We further define $S_k$ for all $k\in\bbZ$
by $S_0=0$ and $S_k-S_{k-1}=x_k$, $k\in\bbZ$; thus
$S_k=\sum_{i=1}^k x_i$ when $k\ge0$  and
$S_k=-\sum_{i=k+1}^0 x_i$ when $k<0$. 
Then $S_{k+n}=S_k-\ellr$ for all $k\in\bbZ$, and $S_k\xjj=S_{k+j}-S_j$.

Let further 
\begin{equation*}
  M_k\=\min_{-\infty < i \le k} S_i
= \min_{k-n< i \le k} S_i;
\end{equation*}
note that $M_k$ is finite and $M_{k+n}=M_k-\ellr$.
Moreover, $M_{k+1}\le M_k$ and $M_{k+1}-M_k$ is 0 or $-1$, since
$S_{k+1}=S_k+x_{k+1}\ge S_k-1$.
We have
\begin{equation*}
  \begin{split}
S_k\xjj >-\ellr,  \text{ for $0\le k<n$}
&\iff
S_{k+j}-S_j	>-\ellr,  \text{ for $0\le k<n$}
\\
&\iff
S_{k+j}+\ellr>S_j, \text{ for $0\le k<n$}
\\&\iff
S_{k+j-n}>S_j, \text{ for $0\le k<n$}
\\&\iff
S_{i}>S_j, \text{ for $j-n\le i<j$}
\\&\iff
M_{j-1}>S_j
\\&\iff
M_{j-1}>M_j.
  \end{split}
\end{equation*}
In each interval of $n$ integers, $M$ decreases by $r$ in steps of 1, so
there are exactly $\ellr$ steps down, which completes the proof.
\end{proof}

\begin{corollary}\label{Ccyclic}
  If $\ddn\in\cbnn$, then exactly one of the $n$ cyclic shifts of $\ddn$ is
  the degree sequence $\gL(T)$ of a tree $T\in\stn$.
\end{corollary}
\begin{proof}
  Let $x_i\=d_i-1$. Then 
$\sumik x_i=\sumik d_i-k$, so 
\eqref{ld} is equivalent to $\sumik x_i\ge 0$ for $k<n$, 
which for the shifted sequence is \eqref{lc}
with $\ellr=1$; further,
$\sumin x_i=n-1-n=-1$.
Hence the result follows by \refL{Lcyclic} with $\ellr=1$.
\end{proof}

We now use our fixed weight sequence $\ww$.
We begin with the partition function for simply generated trees.
This was proved (in the \pws{} case, which is no real loss of generality) by
\citet{Otter}, see also \citet{Dwass};
an algebraic proof uses the Lagrange inversion formula \cite{Lagrange},
see \eg{} \citet{Boyd} and \citetq{Theorem 2.11}{Drmota};
\citet{Kolchin} gives a different proof by induction.
See also \citet{Pitman:enum} where the relation between different approaches
is discussed.

\begin{theorem}\label{TZ}
  \begin{equation*}
	Z_n=\frac1n Z(n-1,n).
  \end{equation*}
\end{theorem}

\begin{proof}
  By \refC{Ccyclic}, the mapping $(T,j)\mapsto \gL(T)\xjj$, where ${}\xjj$
  denotes a cyclic shift as in \refL{Lcyclic}, is a bijection of
  $\stn\times\setn\to\cbnn$. 
Consequently, by \eqref{zmn}, \eqref{wp} and \eqref{zn}, 
since the weight $w(\yyx)$ is not changed by cyclic shifts,
\begin{equation*}
  Z(n-1,n)=\sum_{T\in\stn}\sum_{j=1}^n w\bigpar{\gL(T)\xjj} 
  =\sum_{T\in\stn}n w\xpar{\gL(T)} 
=\sum_{T\in\stn}n w(T)
=n Z_n. 
\qedhere
\end{equation*}
\end{proof}

\begin{corollary}
  \label{Cexists}
Suppose that $w_0>0$ and $\go(\wwx)\ge2$, with $d\=\spann(\wwx)\allowbreak\ge1$.
If $Z_n>0$, then $n\equiv 1\pmod{d}$.
Conversely, for some $n_0$ (depending on $\wwx$),
if $n\equiv 1\pmod{d}$ and $n\ge n_0$, then $Z_n>0$.
\end{corollary}
\begin{proof}
  By \refT{TZ}, $Z_n>0\iff Z(n-1,n)>0$. The result follows from \refL{LBexists}.
\end{proof}

In the same way we can compute various
probabilities for the random tree $\ctn$. We
begin with the root degree $\dx(o)$; note that for any tree $T$, $v_1$ is the
root $o$, so $\dx(o)=\dx(v_1)=d_1$.
(\refL{Lroot} is a special case of \refL{Ldeg} below, but we prefer to study
this 
simpler case first because it shows the main ideas in the proof without the
complications (notational and others) in the more general version.)

\begin{lemma}
  \label{Lroot}
For any $d\ge0$ and $n\ge2$,
\begin{equation}  \label{lroot}
  \P(\dx_{\ctn}(o)=d)
=
\frac{n}{n-1} d\P(Y_1=d).
\end{equation}
Thus, the distribution of the root degree $\dx_{\ctn}(\rot)$ of $\ctn$ is
the size-biased 
distribution of $Y_1$.
\end{lemma}

\begin{proof}
  Let $T\in\stn$ have degree sequence $\ddn$.
If $d_1=d$, then $d_2,\dots,d_n$ is an allocation in $\cB_{n-1-d,n-1}$, and
by \refL{Ld}, such an allocation $(d_2,\dots,d_n)$ comes from a tree $T$
with $d_1=d$ if and only if
\begin{equation}\label{d2c}
  d+\sum_{i=2}^k d_i \ge k, \qquad 1\le k<n,
\end{equation}
or, equivalently,
\begin{equation*}
  \sum_{i=1}^{k} d_{i+1} \ge k+1-d, \qquad 0\le k<n-1.
\end{equation*}
We use \refL{Lcyclic} again, now with $x_i=d_{i+1}-1$ and $\ellr=d$ and see
that exactly $\ellr=d$ of the $n-1$ cyclic shifts of $d_2,\dots,d_n$ satisfy
\eqref{d2c}. 
Thus, by considering all trees $T$ with $d_1=d$ and the $n-1$ cyclic shifts
of $d_2,\dots,d_n$, we obtain each allocation $\ddn\in\cbnn$ with $d_1=d$
exactly $\ellr=d$ times. 
(It is possible that some shifts of $\ddnw$ coincide, but this does not matter.)
Consequently,
\begin{equation*}
  \begin{split}
  (n-1) Z_n \P(\dx_{\ctn}(o)=d)
&=(n-1)\sum_{T\in\stn:\;d_1(T)=d} w(T)
\\&
=d \sum_{\ddn\in\cbnn:\;d_1=d} w\bigpar{\ddn}
\\&
=d\, Z(n-1,n)\P(Y_1=d).	
  \end{split}
\end{equation*}
This yields the result by \refT{TZ}. 
\end{proof}

\begin{remark}\label{Rproot}
  More explicitly we have
\begin{equation*}
  \begin{split}
Z(n-1,n)\P(Y_1=d)
&=\sum_{\ddn\in\cbnn:\;d_1=d} w\bigpar{\ddn}
\\
&=	\sum_{\ddnw\in\cB_{n-1-d,n-1}} w_d w\bigpar{\ddnw}
\\&
=w_d Z(n-1-d,n-1),
  \end{split}
\end{equation*}
and thus
\begin{equation}\label{rproot}
    \P(\dx_{\ctn}(o)=d)
=
d w_d \frac{n}{n-1}\cdot\frac{Z(n-1-d,n-1)}{Z(n-1,n)}\,.
\end{equation}
\end{remark}

\begin{proof}[Proof of \refT{Troot}]
  By \refT{TB3} (with $m=n-1$ and $\gl=1$), 
$\P(Y_1=d)\to\pi_d$, and \eqref{troot1} follows from \refL{Lroot}.

The space $\bNo$ is 
compact, so every sequence of random variables in it is tight, and therefore
has a 
subsequence converging in distribution, see
\cite[Section 6]{Billingsley}.
It follows from \eqref{troot1} that if
$\dx_{\ctn}(o)\dto X$ along a subsequence, then $\P(X=k)=k\pi_k$ for every
$k\in\No$, and thus $\P(X=\infty)=1-\sumk k\pi_k=1-\mu$. Consequently, 
$X\eqd\hxi$ 
so $\dx_{\ctn}(o)\dto \hxi$ 
for every convergent subsequence, 
which means
that the entire sequence converges to $\hxi$, see
  \cite[Theorem 2.3]{Billingsley}.
\end{proof}

This proves the part of \refT{Tmain} that describes the
root degree. It remains to consider all other nodes. This will be done by
similar arguments.
We begin with a generalization of \refL{Lroot}.

\begin{lemma}\label{Ldeg}
  Let $T'\in\stf$ 
be a fixed finite subtree of the Ulam--Harris tree $\too$,
  let $\ell\=|T'|$ be its size and let $v_1,\dots,v_\ell$ be its nodes in
  depth-first order, and let $d'_1,\dots,d'_\ell$ be its degree sequence.
  (I.e., $d'_i=\dx_{T'}(v_i)$.) Suppose that $d_1,\dots,d_\ell\in\No$ and that
  $d_i\ge d_i'$ for every $i$. Then, for every $n>\ell$,
\begin{multline} \label{ldeg}
\P\bigpar{\dx_{\ctn}(v_i)=d_i \text{ for	$i=1,\dots,\ell$}}
\\
=
\biggpar{\sumil d_i-\ell+1} \frac{n}{n-\ell}
 \P(Y_i=d_i \text{ for }i=1,\dots \ell).
  \end{multline}
\end{lemma}

Note that $\dx_T(v_i)\ge d'_i$ for $i=1,\dots,\ell$ implies that 
$T\supset T'$.

\begin{proof}

We have earlier used the depth-first order of the nodes
to define the degree sequence,
but many other orders could be used. In this proof, we consider only trees
$T$ that contain the given $T'$ as a subtree, and then we choose the order
which first takes the nodes of $T'$ in depth-first order (this is
$v_1,\dots,v_\ell$), and then the 
remaining nodes of $T$ in depth-first order; let $\gL'(T)$ be the degree
sequence in this order.

Let $\cA_n$ be the set of trees $T\in\stn$ with 
$\dx_{T}(v_i)=d_i$ for all $i$ (which implies $T\supset T'$). 
If $T\in\cA_n$, then the degree sequence $\gL'(T)$ thus begins with
the given $d_1,\dots,d_\ell$; furthermore, it satifies \eqref{ld}.
Conversely, every sequence beginning with the 
given $d_1,\dots,d_\ell$ that satifies \eqref{ld} is the degree sequence
$\gL'(T)$ of a unique tree in $\cA_n$.
Note also that \eqref{ld} is automatically satisfied for $k<\ell$, since
then $d_i\ge d_i'$ for $i\le k$ and $\sumik d_i'\ge k$ by \refL{Ld} applied
to $T'$.

Let $D:=d_1+\dots+d_\ell$. Consider a sequence $\ddn\in\cbnn$ beginning with
the given $d_1,\dots,d_\ell$, and let $x_i\=d_{\ell+i}-1$, for
$i=1,\dots,n-\ell$. Then $\ddn$ satisfies \eqref{ld} if and only if
\begin{equation*}
D+\sumik (x_i+1) \ge \ell+k  
\end{equation*}
for $k=0,\dots,n-\ell-1$, which is equivalent to 
\begin{equation*}
\sumik x_i \ge -(D-\ell),
\qquad 0\le k<n-\ell.
\end{equation*}
Furthermore, 
$$
\sum_{i=1}^{n-\ell}x_i=\sum_{\ell+1}^n d_i-(n-\ell)
=(n-1-D)-(n-\ell)=
-(D-\ell+1).
$$
\refL{Lcyclic} with $\ellr=D-\ell+1$ thus shows that of the $n-\ell$
cyclic permutations of $d_{\ell+1},\dots,d_n$, exactly $D-\ell+1$ yield a
degree sequence $\gL'(T)$ of a tree $T\in\cA_n$. In other words, if we take
the degree sequences $\gL'(T)$ for all trees
$T\in\cA_n$ and make these $n-\ell$
permutations of each of them, then we obtain every 
allocation $\yyx=\yyn\in\cbnn$ with $y_i=d_i$, $i=1,\dots,\ell$, exactly
$D-\ell+1$ times each.
Consequently,
\begin{equation*}
  \begin{split}
&(n-\ell)Z_n\P(\ctn\in\cA_n)
=(n-\ell)\sum_{T\in\cA_n}w(T)
=\sum_{T\in\cA_n}(n-\ell)w(\gL'(T))
\\&
=\sum_{\yyx\in\cbnn:\;y_i=d_i \text{ for }i\le \ell} (D-\ell+1) w(\yyx)
\\&
=(D-\ell+1) Z(n-1,n)\P(Y_i=d_i \text{ for }i\le \ell). 
  \end{split}
\end{equation*}
The result follows by \refT{TZ}.
\end{proof}

\begin{remark}\label{Rptree}
Arguing as in \refR{Rproot}, we obtain from \refL{Ldeg}
the explicit formula, generalizing \eqref{rproot},
with 
$D\=\sumil d_i$ and other notations as above,
\begin{multline}\label{rptree}
\P\bigpar{\dx_{\ctn}(v_i)=d_i \text{ for	$i=1,\dots,\ell$}}
\\
=
\frac{n}{n-\ell}
(D-\ell+1) 
 \frac{w_{d_1}\dotsm w_{d_\ell} Z(n-D-1,n-\ell)}{Z(n-1,n)}.
  \end{multline}
\end{remark}

\begin{remark}
  Note that \refL{Ldeg} (or \eqref{rptree})
shows that the probability remains exactly the same
  if we permute $d_1,\dots,d_\ell$, provided that the permuted sequence
  $(d_{\gs(i)})$ still is allowed, \ie, $d_{\gs(i)}\ge d'_i$ for all $i\le\ell$.
However, if the latter condition fails for some $i$, then the probability
typically becomes 0. (This is an interesting case of a symmetry that is not
complete.) 

For example, considering only the root $o$ and its first child 1,
we have 
\begin{equation*}
 \P(\dx_{\ctn}(o)=d \text{ and } \dx_{\ctn}(1)=d')
=
 \P(\dx_{\ctn}(o)=d' \text{ and } \dx_{\ctn}(1)=d)
\end{equation*}
whenever $d,d'\ge1$; however, if, say, $d\ge1$ and $d'=0$, then the \rhs{}
is 0 while the \lhs{} in general is not.
\end{remark}

\begin{remark}
  \refL{Ldeg} extends with minor modifications (mainly notational) to
  arbitrary finite rooted subtrees $T'$ of $\too$ (not
  necessarily satisfying \eqref{left}). 
We omit the details.
\end{remark}

\section{Proof of \refT{Tmain}}\label{Spfmain} 
First, as in the proof of \refT{TBmain},
\refL{Ltau} shows that
$\tau$ defined by \ref{tmaincritical}
and \ref{tmainsub} is well-defined and
equals $\tau(1)$ defined in \refL{Ltau};
since $1<2\le\go$ we have $\tau<\infty$ and $\Phi(\tau)<\infty$.
Further, \eqref{kia0} yields
$ \Psi(\tau)=\min(1,\nu)$.
Hence, by
\refL{LEPsi}, $\ppi$ is a probability distribution with mean
and variance as asserted. (This is a special case of the corresponding claims
in \refT{TBmain}, with $\gl=1$. We have $\gl=1$ here since
we relate the random trees to allocations with $m=n-1$,
and thus $m/n\to1$.)

The final claims follow by \eqref{mainmu} and the construction in \refS{ShGW}.

We turn to the main assertion, $\ctn\dto\hcT$.
Since $\st$ is a compact metric space, any sequence of random trees in $\st$
is tight, and has thus a convergent subsequence. 
(See e.g.\ \cite[Section 6]{Billingsley}.)
In particular, this holds for $\ctn$.

Consider a limiting random tree $\ttt$ in $\st$ such that $\ctn\dto\ttt$
along some subsequence.
We will show that then $\ttt\eqd\hcT$, regardless of the subsequence; this
implies $\ctn\dto\hcT$ for the full sequence, which then completes the proof.

We have defined $\st$ in \refS{SUlam} such that $\st\subset\bNo^\voo$ using
the embedding $T\mapsto (\dx_T(v))_{v\in\voo}$.
In order to show $\ttt\eqd\hcT$, it thus suffices to show that the
distributions agree on cylinder sets, \ie, that
$\bigpar{\dx(v_1),\dots,\dx(v_m)}\in \bNo^m$ has the same distribution for
$\ttt$ and $\hcT$, for any finite set $V=\set{v_1,\dots,v_m}\subset\voo$.
Since $\bNo^m$ is a countable set, this is equivalent to
\begin{equation}
  \label{eqd}
\P\bigpar{\dx_{\ttt}(v_1)=d_1,\dots,\dx_{\ttt}(v_m)=d_m}
=
\P\bigpar{\dx_{\hcT}(v_1)=d_1,\dots,\dx_{\hcT}(v_m)=d_m},
\end{equation}
for any finite set $V=\set{v_1,\dots,v_m}\subset\voo$ and any
$d_1,\dots,d_m\in\bNo$. 

It thus suffices to show \eqref{eqd}. 
Furthermore,
given any finite set $V\subset\voo$, we may enlarge it to a finite set $V$
satisfying \eqref{voo1}--\eqref{voo3}, \ie, a set that is the node set of
some finite tree in $\stf$.
It thus suffices to show \eqref{eqd} for $V=V(\Tx)$ with $\Tx\in\stf$.

We make one more reduction. 
Suppose that $V=V(\Tx)$ with $\Tx\in\stf$
and that \eqref{eqd} contains a condition $\dx(v_i)=d_i$
with $d_i<\dx_\Tx(v_i)$. Let $v\=v_i$ and let $u$ be the last child of $v$ in
$\Tx$; thus (recalling the notation in \refS{SUlam}) $u=vj$ for some integer
$j=\dx_\Tx(v)>d_i$. By \eqref{dtree}, any tree $T\in\st$ with $\dx_T(v)=d_i$ has
$\dx_T(u)=0$, and further (\eg{} by \eqref{dtree} and induction)
$\dx_T(s)=0$ for every descendant $s$ of $u$. 
Thus, letting $\Txw$ denote the subtree of $\Tx$ rooted at $u$,
for any $s\in\Txw$, the event 
$\set{\dx_{\hcT}(v)=d_i\text{ and }\dx_{\hcT}(s)>0}$
is impossible and has probability 0; furthermore, the same holds for each
$\ctn$, and thus, since $\ctn\to\ttt$ along a subsequence, 
$\P\bigpar{\dx_{\ttt}(v)=d_i\text{ and }\dx_{\ttt}(s)>0}=0$.
Consequently, if \eqref{eqd} contains a condition $\dx(v_j)=d_j$ with
$v_j\in\Txw$ and $d_j>0$, then both sides are trivially
$0$. On the other hand, if $d_j=0$ for all $v_j\in\Txw$, then the
conditions 
$\dx(v_j)=d_j$ 
are redundant in \eqref{eqd} and may be deleted, so we may
replace $\Tx$ by the smaller tree with $\Txw$ removed. Repeating this
pruning, if necessary, we see that it suffices to show \eqref{eqd}
for $V=V(\Tx)$ when $\Tx\in \stf$ is a finite tree and $d_i\ge \dx_\Tx(v_i)$
for every $i$.

Recall that $d_i$ in \eqref{eqd} may be infinite. We study three different
cases separately.

\pfcase{(a)}{Every $d_i<\infty$.}
This is the case treated in \refL{Ldeg}; we take the limit as \ntoo{} in
\eqref{ldeg} and obtain by \refT{TB3}
(with $m=n-1$ and $\gl=1<\go(\wwx)$), letting again $D\=\sumil d_i$,
\begin{equation*}
  \P\bigpar{\dx_{\ctn}(v_i)=d_i \text{ for	$i=1,\dots,\ell$}}
\to 
\lrpar{D-\ell+1} \prodil \pi_{d_i}.
\end{equation*}
Since we have assumed $\ctn\dto\ttt$ along a subsequence, this yields
\begin{equation}\label{ematt}
  \P\bigpar{\dx_{\ttt}(v_i)=d_i \text{ for	$i=1,\dots,\ell$}}
=
\lrpar{D-\ell+1} \prodil \pi_{d_i}.
\end{equation}

Now consider the modified \GWt{} $\hcT$. 
(Recall its construction in \refS{ShGW}.)
If the tree $\hcT$ has $\dx_{\hcT}(v_i)=d_i<\infty$ for all $v_i\in T'$,
then the spine has to 
extend outside $T'$. The first point on the spine outside $T'$ is a node in
$\partial T'$ (regarding $T'$ as a subtree of $\hcT$). The condition
$\dx_{\hcT}(v_i)=d_i$ for 
$v_i\in T'$ determines the boundary $\partial T'$ of $T'$
in $\hcT$, which thus not depend on
$\hcT$, and \refL{Ldsub} shows that $|\partial T'|=D-\ell+1$.

Fix a node $u\in\partial T'$, and consider the event $\cE_u$ that the spine
of $\hcT$ passes through $u$ and that $\dx_{\hcT}(v_i)=d_i$ for
$i=1,\dots,\ell$. 
The event $\cE_u$ thus specifies the nodes in $T'$ that are special in the
construction of $\hcT$ (\viz{} the nodes on the path from $o$ to $u$), and
for each special node it specifies which of its children that will be
special; furthermore it specifies the number of children for each node in
$T'$, special or not. Recall that the probability that a special node has
$d<\infty$ children, with a given one of them being special, is $\pi_d$, just
as the 
probability that a normal node has $d$ children. Thus, by independence,
for every $u\in\partial T'$,
$\P(\cE_u)=\prodil \pi_{d_i}$.
This probability thus does not depend on $u$, so summing over the $D-\ell+1$
nodes $u\in\partial T'$ we obtain
\begin{equation*}
  \P\bigpar{\dx_{\hcT}(v_i)=d_i \text{ for $i=1,\dots,\ell$}}
=
\sum_{u\in\partial T'}\P(\cE_u)
=
\lrpar{D-\ell+1} \prodil \pi_{d_i},
\end{equation*}
which together with \eqref{ematt} shows \eqref{eqd} in this case.
(Cf.\ \refR{Rsizebias} for a similar argument.)

\pfcase{(b)}{Exactly one $d_i=\infty$.}
Suppose that $d_j=\infty$ and $d_i<\infty$ for $i\neq j$.
Define, for $0\le k\le\infty$, 
\begin{equation*}
  \cA_k\=\set{T\in\st: \dx_T(v_i)=d_i \text{ for } i\neq j
              \text { and } \dx_T(v_j)=k}.
\end{equation*}
We thus want to show $\P(\ttt\in\cA_\infty)=\P(\hcT\in\cA_\infty)$.
We define further 
$$
\ageK\=\bigcup_{K\le k\le\infty} \cA_k,
$$
and note that since $\ctn\dto\ttt$ (along a subsequence), we have
(along the subsequence), for any finite $K$,
\begin{equation}\label{rosen}
  \P(\ctn\in\ageK)\to\ 
 \P(\ttt\in\ageK).
\end{equation}

We define also (for finite $k$) the analogous
\begin{equation*}
  \bk\=\set{\yyn\in\cbnn: 
y_j=k \text { and } 
  y_i=d_i \text{ for } i\le\ell \text{ with }i\neq j
}.
\end{equation*}
Then  \refL{Ldeg} can be written, with $D'\=\sum_{i\neq j}d_i$, for $k<\infty$,
\begin{equation}\label{teak}
  \P(\ctn\in\ak)=(k+D'-\ell+1)\frac{n}{n-\ell}\P(\bnn\in\bk).
\end{equation}

Consider, for simplicity, $k>\max_{i\neq j}d_i$. Then \eqref{pyn} shows
that,
with $N_i=N_i(\bnn)$,
\begin{equation*}
  \P(\bnn\in\bk)=
 \E \P(\bnn\in\bk\mid N_0,N_1,\dots)=
\E\biggpar{ \frac{N_k}{n} \prod_{i\neq j} \frac{N_{d_i}+O(1)}{n+O(1)}}.
\end{equation*}
(The implicit constants in the
$O$ in this proof may depend on $\ell$ and $d_1,\dots,d_\ell$, but not
on $n$ or $k$.) 
Consequently, by \eqref{teak},
\begin{equation*}
  \begin{split}
  \P(\ctn\in\ak)
&=
\bigpar{k+O(1)}\bigpar{(1+O(n\qw)}
\E\biggpar{ \frac{N_k}{n} \prod_{i\neq j} \frac{N_{d_i}}{n}}
\\
&=
\bigpar{1+O(k\qw)+O(n\qw)}
\E\biggpar{\frac{k N_k}{n} \prod_{i\neq j} \frac{N_{d_i}}{n}}.	
  \end{split}
\end{equation*}
Summing over $k\ge K$, we obtain for any $K$, using $\sumk kN_k=n-1$
for any allocation $\bnn$,  
\begin{equation}\label{dal}
  \begin{split}
  \P(\ctn\in\ageK)
&=
\sum_{k= K}^\infty \P(\ctn\in\ak)
\\&=
\bigpar{1+O(K\qw)+O(n\qw)}
\E\biggpar{\frac{\sum_{k\ge K}k N_k}{n} \prod_{i\neq j} \frac{N_{d_i}}{n}}.	
\\&=
\bigpar{1+O(K\qw)+O(n\qw)}
\E\biggpar{\frac{n-1-\sum_{k< K}k N_k}{n} \prod_{i\neq j} \frac{N_{d_i}}{n}}.	
  \end{split}
\end{equation}
By \refT{TBmain}, for any fixed $K$, as \ntoo,
\begin{equation*}
  \frac{n-1-\sum_{k< K}k N_k}{n} \prod_{i\neq j} \frac{N_{d_i}}{n}
\pto \bigpar{1-\sum_{k< K}k\pi_k} \prod_{i\neq j} \pi_{d_i}.
\end{equation*}
By dominated convergence, the expectation converges to the same limit, and thus
\eqref{rosen} and \eqref{dal} yield
\begin{equation}\label{dalq}
  \P(\ttt\in\ageK)
=
\bigpar{1+O(K\qw)}
\Bigpar{1-\sum_{k< K}k\pi_k} \prod_{i\neq j} \pi_{d_i}.
\end{equation}

Finally, let $K\to\infty$ to obtain
\begin{equation}\label{emm}
  \P(\ttt\in\cA_\infty)
=
\Bigpar{1-\sum_{k< \infty}k\pi_k} \prod_{i\neq j} \pi_{d_i}
=
\xpar{1-\mu} \prod_{i\neq j} \pi_{d_i}.
\end{equation}

Now consider $\hcT$. If $\dx_{\hcT}(v_j)=d_j=\infty$, then the spine ends with
an explosion at $v_j$. This fixes the spine, and the event that
$\dx_{\hcT}(v_i)=d_i$ for $i\neq j$ then means, just as in case (a) when we
considered a specific $\cE_u$, that we have specified the number of children 
to be $d_i$ for these nodes, and for the special nodes (except $v_j$)
we have also specified which child is special. The probability of this is
$\pi_{d_i}$ for each $i\neq j$, and the probability that the special node
$v_j$ has an infinite number of children is, by \eqref{hxi}, $1-\mu$.
Hence, by independence,
\begin{equation}
  \P(\hcT\in\cA_\infty)
=
\xpar{1-\mu} \prod_{i\neq j} \pi_{d_i},
\end{equation}
which together with \eqref{emm} shows 
$  \P(\ttt\in\cA_\infty)=  \P(\hcT\in\cA_\infty)$, which is
\eqref{eqd} in this case.

\pfcase{(c)}{More than one $d_i=\infty$.}
By the definition of the modified \GWt{} $\hcT$, there is at most one node
with infinite degree, so in this case,
\begin{equation*}
  \P\bigpar{\dx_{\hcT}(v_i)=d_i \text{ for $i=1,\dots,\ell$}}
= 0.
\end{equation*}
This means that the sum of these probabilities for all sequences $\ddn$ 
with at most one infinite value is 1. But we have shown that for such
sequences, the probability is the same for $\ttt$ as for $\hcT$, so the
probabilities for $\ttt$ for these sequences also sum up to 1. Consequently,
if more than one $d_i=\infty$, then
\begin{equation*}
  \P\bigpar{\dx_{\ttt}(v_i)=d_i \text{ for $i=1,\dots,\ell$}}
= 0
\end{equation*}
too, which shows \eqref{eqd} in this case.

This shows that \eqref{eqd} holds for any 
$v_1,\dots,v_m$ such that
$\set{v_1,\dots,v_m}=V(T')$ where $T'\in\stf$ is
a finite tree and $\ddn$ is any sequence in $\bNo^m$ with $d_i\ge
\dx_{T'}(v_i)$ for every $i$. As discussed above, this implies \eqref{eqd} in
full generality and thus $\ttt\eqd\hcT$, which shows that $\ctn\dto\hcT$.
\qed

\section{Proofs of Theorems \refand{Tdegree}{Tsubtree}}\label{SpfXXX}

We begin by stating another version of the correspondence between simply
generated trees and the \bib{} model.

\begin{lemma}\label{Lx}
We may couple $\ctn$ and $\bnn$ such that the degree sequence
$\gL(\ctn)$ is a cyclic shift of $\bnn$, and, 
conversely, $\bnn$ is a uniformly random cyclic shift of $\gL(\ctn)$.
\end{lemma}

\begin{proof}
  Let $\bnn=\YYn$
and let $\YYgsn$ be the unique cyclic shift of $\YYn$
that is the degree sequence of a tree in $\stn$, see
\refC{Ccyclic}.
Then $\YYgsn\eqd\gL(\ctn)$,
as a consequence of \refC{Ccyclic} and the invariance of the weight
  $w\YYn$ under cyclic shifts.
Consequently, we may couple $\bnn$ and $\ctn$ such that
$\YYgsn=\gL(\ctn)$, and the result follows.
\end{proof}

\begin{proof}[Proof of \refT{Tdegree}]
We use the coupling in \refL{Lx}. Then $N_d$ in \refT{Tdegree} equals
$N_d(\bnn)$ in \refT{TBmain}, and thus \eqref{tdegq} follows by
\eqref{tbmain}.

We obtain \eqref{tdeg} as a simple consequence of \eqref{tdegq}, using
$\P(\dx_{\ctn}(v)=d\mid N_d)=N_d/n$
and thus
$\P(\dx_{\ctn}(v)=d)=\E N_d/n$, 
\cf{} the proof of
\refT{TB3}. Alternatively, we can arrange so that
$\dx_{\ctn}(v)=Y_1$, and the result then follows by \refT{TB3}.
\end{proof}

\begin{proof}[Proof of \refT{Tsubtree}]
We use again the coupling in \refL{Lx}.
Let $T$ be a fixed tree of size $\ell$ and let its degree sequence be
$\bddl$.  Recall that we have defined the degree sequence using depth-first
search. It follows that if a tree has degree sequence $\ddn$ and a
node $v$  is visited as node $v_{j}$ in the 
depth-first search, then the subtree rooted at $v$ has degree sequence
$(d_{j}, \dots,d_{k})$, where we stop when this is a degree sequence of a
tree, \ie, when it satisfies the condition in \refL{Ld}.
In particular, the subtree rooted at $v$ equals $T$ if and only if
$(d_j,\dots,d_{j+\ell-1})=\bddl$.
(Clearly, this is impossible if $j>n-\ell+1$, since then a tree would be
completed with less size than $\ell$.)

Consequently, $N_T$ equals the number of substrings $\bddl$ in $\YYn$,
regarded as a cyclic sequence. In other words,
if we let $I_j$ be the indicator of the event
$(Y_j,\dots,Y_{j+\ell-1})=\bddl$, where we define $Y_i\=Y_{i-n}$ for $i>n$,
then 
\begin{equation}\label{nt}
N_T=\sumjn I_j.  
\end{equation}
In particular, taking the expectation and using the rotational symmetry,
\begin{equation*}
 \P(\ctnv=T)=\frac1n \E N_T = \E I_1=\P\bigpar{(Y_1,\dots,Y_\ell)=\bddl},
 \end{equation*}
and thus \refT{TB3} yields
 \begin{equation*}
 \P(\ctnv=T)
\to\prodil \pi_{\bd_i}
=\P(\cT=T),
\end{equation*}
which proves \eqref{ttree}.

In order to show the stronger result \eqref{ttreeq}, we condition as in the
proof of \refT{TB3} on $N_0,N_1\dots $ and obtain, see \eqref{pyn},
\begin{equation}\label{emms}
  \begin{split}
  \E(I_j\mid N_0,N_1,\dots)
&=
\P\bigpar{(Y_1,\dots,Y_\ell)=\bddl\mid N_0,N_1,\dots}
\\&
= \prod_{i=1}^\ell \frac{N_{\bd_i}-c_i}{n-i+1}
= \prod_{i=1}^\ell \frac{N_{\bd_i}}{n}
+O\parfrac1n,	
  \end{split}
\end{equation}
where $c_i\=|\set{j<i:\bd_j=\bd_i}|$.
If $|j-k|\ge \ell$ and $|j-k\pm n|\ge \ell$ (\ie, $j$ and $k$ have distance
at least $\ell$, regarded as point on a circle of length $n$), then 
similarly, with $c'_i\=|\set{j\le\ell:\bd_j=\bd_i}|$,
\begin{equation*}
  \begin{split}
  \E(I_jI_k\mid N_0,N_1,\dots)
= \prod_{i=1}^\ell \frac{N_{\bd_i}-c_i}{n-i+1}
\prod_{i=1}^\ell \frac{N_{\bd_i}-c_i-c'_i}{n-\ell-i+1},
  \end{split}
\end{equation*}
and it follows that
\begin{equation}\label{covo}
  \begin{split}
  \Cov(I_j,I_k\mid N_0,N_1,\dots)
=O(1/n). 
  \end{split}
\end{equation}
For $j$ and $k$ of distance less than $\ell$, we use the trivial 
\begin{equation}
  \label{cov1}
| \Cov(I_j,I_k\mid N_0,N_1,\dots)|\le1.
\end{equation}
There are less than $n^2$ pairs $(j,k)$ of the first type and $O(n)$ pairs
of the second type, and thus by \eqref{nt} and \eqref{covo}--\eqref{cov1},
\begin{equation*}
  \begin{split}
  \Var(N_T\mid N_0,N_1,\dots)
=\sumjn\sumkn  \Cov(I_j,I_k\mid N_0,N_1,\dots)
=O(n).
  \end{split}
\end{equation*}
Consequently, $N_T/n-  \E(N_T/n\mid N_0,N_1,\dots)\pto0$, and thus by
\eqref{nt}, \eqref{emms} and \refT{TBmain},
\begin{equation*}
  \begin{split}
\frac{N_T}n&=
  \E\Bigpar{\frac{N_T}{n}\Bigm| N_0,N_1,\dots}+\op(1)
= \prod_{i=1}^\ell \frac{N_{\bd_i}}{n}
+\op(1)
\\&
\pto\prodil \pi_{\bd_i}
=\P(\cT=T).
\qedhere
  \end{split}
\end{equation*}
\end{proof}

\section{Asymptotics of the partition functions}\label{Spart}

We have a simple asymptotic result for the partition
function $Z(m,n)$ (to the first order in the exponent, at least if $\rho>0$):

\begin{theorem}\label{TBZ}
  Let $\wwx=(w_k)_{k\ge0}$ be a weight sequence with $w_0>0$ and $w_k>0$
  for some $k\ge1$.
Suppose that $n\to\infty$ and $m=m(n)$ with 
$\spann(\wwx)\delar m$, $m\to\infty$ and
$m/n\to\gl$ where $0\le\gl<\go$,
and let $\tau$ be as in \refT{TBmain}.
\begin{romenumerate}
\item 
If $\rho>0$, then
  \begin{equation}\label{tbz}
	\frac1n \log Z(m,n)\to 
\log\Phi(\tau)-\gl\log\tau
\in\oooo
.
  \end{equation}
\item 
If $\rho=0$ and $\gl>0$, then
  \begin{equation}\label{tbzoo}
	\frac1n \log Z(m,n)\to \infty.
  \end{equation}
\end{romenumerate}
In both cases, the result can be written
  \begin{equation}\label{tbz2}
	\frac1n \log Z(m,n)\to 
\log\inf_{0\le t\le\rho} \frac{\Phi(t)}{t^\gl}
=\log\inf_{0\le t<\infty} \frac{\Phi(t)}{t^\gl}
\le\infty.
  \end{equation}
\end{theorem}

If $0\le\gl\le\nu$ and $\rho>0$, the limit 
can also be written
$\log\Phi(\tau)-\Psi(\tau)\log\tau$.

The formula \eqref{tbz} is shown by a physicists' proof by
\citet{BialasetalNuPh97}. 

\begin{remark}\label{RZ0}
If $\gl=0$, then $\tau=0$, and we interpret the \rhs{} of \eqref{tbz} as
$\log\Phi(0)=\log w_0$; this is in accordance with \eqref{tbz2}.

It is easily seen that the result holds, with this limit,
also in the rather trivial case when $m$ is bounded, provided $Z(m,n)>0$.
\end{remark}

\begin{remark}
  If $\go<\infty$, then the result holds also when $\gl=\go$, provided
  $Z(m,n)>0$,
if we let $\tau=\infty$ as in \refR{Rsymmetry} and interpret the
  \rhs{} of \eqref{tbz} as the limit value $\log w_\go$,
which again is in accordance with \eqref{tbz2}. This follows from
  \refR{RZ0} by the
  symmetry argument in \refR{Rsymmetry}.
\end{remark}

\begin{remark}
  Using the function $\tau(x) $ defined in \refT{TBmain2}, the result
  \eqref{tbz} can also be written, using the continuity of $\tau(x)$ and an
  extra argument (which we omit) when $\gl=0$,
  \begin{equation}
\log Z(m,n)= n\log\Phi(\tau(x))-m\log\tau(x)+o(n)
  \end{equation}
or, equivalently,
  \begin{equation}\label{tbzmn}
	Z(m,n)=\Phi(\tau(x))^n\tau(x)^{-m}e^{o(n)}.
  \end{equation}
As in \refT{TBmain2}, it suffices here that $m/n\le C<\go$ (and $m\to\infty$).
\end{remark}

\begin{proof}[Proof of \refT{TBZ}]
Note that the assumptions imply that
$Z(m,n)>0$ (at least for $n$, and thus $m$, large) by \refL{LBexists}.
The equivalence between \eqref{tbz}--\eqref{tbzoo}
and \eqref{tbz2} follows from \eqref{annagl}.

\pfitem{i}
Assume first $\gl>0$.
  Since $\rho>0$ and $\gl>0$, we then have $\tau>0$.
Thus $\wwx=\ww$ is equivalent to $\ppix=\ppi$, and \refL{LB1} yields
\begin{equation*}
Z(m,n)=  Z(m,n;\wwx)
=
\Phi(\tau)^{n}\tau^{-m} Z(m,n;\ppix).
\end{equation*}
We saw in the proof of \refT{TBmain}, case (a), that Lemmas
\refand{LB2}{LB2sub} yield $Z(m,n;\ppix)=\exp(o(n))$, and thus
\begin{equation*}
Z(m,n)
=\exp\bigpar{n\log\Phi(\tau)-m\log\tau+o(n)},
  \end{equation*}
which yields \eqref{tbz}.

It remains to consider the case $\gl=0$. Then $m/n\to0$, and we may assume
$m<n/2$. 
In any allocation of $m$ balls, there are at most $m$ non-empty boxes. 
Let us
mark $2m$ boxes, including all non-empty boxes. For each choice of the
marked boxes, we have in them an allocation in $\cB_{m,2m}$,
and only empty boxes outside; since there are $\binom {n}{2m}$ choices of
marked boxes,
\begin{equation}
  \label{b1}
Z(m,n)
\le \binom{n}{2m} w_0^{n-2m} Z(m,2m).
\end{equation}
On the other hand, any allocation of $m$ balls in $2m$ boxes can be extended
to an allocation in $\cbmn$ with the last $n-2m$ boxes empty; thus
\begin{equation}
  \label{b2}
Z(m,n)
\ge w_0^{n-2m} Z(m,2m).
\end{equation}
We have, by Stirling's formula, using $m/n\to\gl=0$,
\begin{equation}\label{b3}
\frac1n\log\binom{n}{2m}
\le \frac1n\log \parfrac{en}{2m}^{2m} 
=\frac{2m}{n}\log\frac{e}2 - \frac{2m}{n}\log\frac mn
\to0.
\end{equation}
Moreover, by the case $\gl>0$ just proved, we have from \eqref{tbz} 
$\log Z(m,2m)=O(m)=o(n)$. Consequently, \eqref{b1}--\eqref{b3} yield
\begin{equation*}
  \log Z(m,n)=(n-2m)\log w_0 +o(n)= n\log w_0 + o(n),
\end{equation*}
showing \eqref{tbz} in the case $\gl=0$.

\pfitem{ii}
As in the proof of \refL{LB2sub}, we use the truncated \ws{} 
$\wwxk$ defined in \eqref{trunc2}, where
$K$ is so large that $\spann(\wwxk)=\spann(\wwx)$ and $\go(\wwxk)>\gl$,
and we let again $\Phik$ and $\Psik$ be the corresponding functions for
$\wwxk$ and define $\tauk$ by $\Psik(\tauk)=\gl$.

For any $t>0$, $\Phik(t)\to\Phi(t)=\infty$ as $
K\to\infty$, and thus \eqref{qqq} holds, showing that for large $K$,
$\Psik(t)>\gl$ and thus $\tauk<t$. 
Since $t$ is arbitrary, this shows that $\tauk\to0$ as \Ktoo.
Applying (i) to $\wwxk$ and its partition function $\zk$ we obtain, for
every large $K$,
\begin{equation*}
  \begin{split}
  \liminf_\ntoo\frac1n Z(m,n)
&\ge
  \lim_\ntoo\frac1n \zk(m,n)
=\log\Phik(\tauk)-\gl\log\tauk
\\&
\ge\log w_0-\gl\log\tauk.	
  \end{split}
\end{equation*}
As \Ktoo, $\tauk\to0$ so the \rhs{} tends to $\infty$, which completes the
proof. 
\end{proof}

\begin{remark}
  \label{RBZ}
The case $\rho=0$ and $\gl=0$ is excluded from \refT{TBZ}; in this case,
almost anything can happen.
To see this, note first that
by \eqref{b1}--\eqref{b3}, if $m/n\to\gl=0$,
\begin{equation}\label{ulla}
  \frac1n\log Z(m,n)=\log w_0+\frac1n \log Z(m,2m)+o(1),
\end{equation}
and by \refT{TBZ}(ii), $\frac1m\log Z(m,2m)\to\infty$ as \mtoo,
and hence $m/Z(m,2m)\to0$.
We can choose $m=m(n)\to\infty$ with $m/n\to0$ so rapidly that 
$m/n\ll m/\log Z(m,2m)$, and then
$\frac1n\log Z(m,2m)\to0$ and \eqref{ulla} yields $\frac1n\log Z(m,n)\to\log
w_0=\log\Phi(0)$.

We can also choose $m$ with $m/n\to0$ so slowly that
$m/n\gg m/{\log Z(m,2m)}$, and then
$\frac1n\log Z(m,2m)\to\infty$ and \eqref{ulla} yields 
$\frac1n\log Z(m,n)\to\infty$. 

Furthermore, we can choose $m(n)$ oscillating between these two cases,
and then  
$\liminf \frac1n\log Z(m,n)=\log\Phi(0)$ and
$\limsup \frac1n\log Z(m,n)=\infty$, and we can arrange so that every number
in $[\log\Phi(0),\infty)$ is a limit point of some subsequence.

For many \ws{s} with $\rho=0$, one can choose $m(n)$ such that 
$ \frac1n\log Z(m,n)\to a$ for any given $a\in[\log\Phi(0),\infty]$. For
example for $w_k=k!$ as in \refE{Ek!}, we have by \cite{SJ259} and \refT{TZ}
$Z(n-1,n)\sim en!$ and it follows, arguing similarly to \eqref{b1} and
\eqref{b2},
 that $\frac1m\log Z(m,2m)=\log m+O(1)$, so taking $m\sim an/\log n$, we
 obtain $ \frac1n\log Z(m,n)\to a$ by \eqref{ulla}.

However, if $w_k$ increases very rapidly, it may be impossible to obtain
convergence of the full sequence to a limit different from $\log\Phi(0)$ or
$\infty$, so we can only achieve convergence of subsequences.
For example, if $w_0=1$ and $w_{k+1}\ge Z(k,2k)^2$, then
$Z(k+1,2(k+1))\ge w_{k+1}\ge Z(k,2k)^2$, and it follows easily from \eqref{ulla}
that 
$\limsup \frac1n\log Z(m,n) \ge 2 \liminf \frac1n\log Z(m,n)$.
\end{remark}

We apply \refT{TBZ} to simply generated trees.

\begin{theorem}\label{TZlim}
  Let $\wwx=(w_k)_{k\ge0}$ be any weight sequence with $w_0>0$ and $w_k>0$
  for some $k\ge2$.
Suppose that $n\to\infty$  with 
$n\equiv1\pmod{\spann(\wwx)}$,
and let $\tau$ be as in \refT{Tmain}.
Then 
  \begin{equation*}
	\frac1n \log Z_n\to 
\log\Phi(\tau)-\log\tau
=\log\inf_{0\le t<\infty}\frac{\Phi(t)}{t}\in(-\infty,\infty].
  \end{equation*}
The limit is finite if $\rho>0$, and $+\infty$ if $\rho=0$.
\end{theorem}

\begin{proof}
  An immediate consequence of Theorems \refand{TZ}{TBZ}.
\end{proof}

For \pws{s}, \refT{TZlim} can be expressed as follows, \cf{} \refR{Rqk}.

\begin{theorem}\label{Tpn}
  Let $\cT$ be a \GWt{} with offspring distribution $\xi$, and assume that
  $\P(\xi=0)>0$ and $\P(\xi>1)>0$. 
Suppose that $n\to\infty$  with 
$n\equiv1\pmod{\spann(\xi)}$,
and let $\tau$ be as in \refT{Tmain}.
Then 
  \begin{equation*}
	\frac1n \log \P(|\cT|=n)\to 
\log\Phi(\tau)-\log\tau
=\log\inf_{0\le t<\infty}\frac{\Phi(t)}{t}\in(-\infty,0].
  \end{equation*}
If\/ $\E\xi=1$, or if\/ $\E\xi<1$ and $\rho=1$, then the
limit is $0$; otherwise it is strictly negative.
In other words, $\P(|\cT|=n)$ decays exponentially fast in the supercritial
case (then $\tau<1$)
and in the subcritical case with $\rho>1$ (then $\tau>1$), 
but only subexponentially
in the critical case and in the subcritical case with $\rho=1$
(then $\tau=1$).
\end{theorem}
\begin{proof}
  We have $\P(|\cT|=n)=Z_n$, see  \refS{Ssimply}, and we apply \refT{TZlim}.
Since now $\ww$ is  a \pws, we have $\rho\ge1$ and
$\inf_{0\le t<\infty}\xfrac{\Phi(t)}{t}\le\Phi(1)/1=1$, with equality if and
only if $\tau=1$, see \refR{R6.3.1}. The final claims follow using the
definition of $\tau$ in \refT{Tmain}.
\end{proof}

When $\rho>0$ and $\gl>0$ (which are equivalent to $\tau>0$), we can also
prove stronger ``local'' versions of 
Theorems \ref{TBZ} and \ref{TZlim}, showing that the partition function
behaves smoothly for small changes in $m$ or $n$.
\begin{theorem}
  \label{TH2}
  Let $\wwx=(w_k)_{k\ge0}$ be a weight sequence with $w_0>0$ and $w_k>0$
  for some $k\ge1$.
Suppose that $n\to\infty$ and $m=m(n)$ with $m/n\to\gl$ where $0<\gl<\go$,
and let $\tau$ be as in \refT{TBmain}.
If $\rho>0$, then, for every fixed $k\in\bbZ$ such that $\spanw\delar k$,
\begin{equation}
  \label{temma}
\frac{Z(m+k,n)}{Z(m,n)}\to\tau^{-k}.
\end{equation}
\end{theorem}
\begin{proof}
  For any $k\ge0$, by \eqref{wmm}--\eqref{zmn},
\begin{equation}
  \label{sigrid}
\P(Y_1=k)
=\frac{w_kZ(m-k,n-1)}{Z(m,n)},
\end{equation}
and thus
\begin{equation*}
\frac{\P(Y_1=k)}{\P(Y_1=0)}
=\frac{w_kZ(m-k,n-1)}{w_0Z(m,n-1)}.
\end{equation*}
Since \refT{TB3} yields
\begin{equation*}
\frac{\P(Y_1=k)}{\P(Y_1=0)}
\to\frac{\pi_k}{\pi_0}
=\tau^k \frac{w_k}{w_0},
\end{equation*}
we see (replacing $n$ by $n+1$) that \eqref{temma} holds when $-k\in\suppw$.
Furthermore, the set of $k\in\bbZ$ such that \eqref{temma} holds for any
allowed sequence $m(n)$ is easily seen to be a subgroup of $\bbZ$
(since we may replace $m$ by $m\pm k'$ for any fixed $k'$).
Consequently, by \eqref{span0}, this set contains every multiple of $\spanw$.
\end{proof}

\begin{theorem}
  \label{TH3}
Let $\wwx=(w_k)_{k\ge0}$ be a weight sequence with $w_0>0$ and $w_k>0$
  for some $k\ge1$.
Suppose that $n\to\infty$ and $m=m(n)$ with $m/n\to\gl$ where $0\le\gl<\go$,
and let $\tau$ be as in \refT{TBmain}.
Then,
\begin{equation}
  \label{temma3}
\frac{Z(m,n+1)}{Z(m,n)}\to\Phi(\tau).
\end{equation}
\end{theorem}

\begin{proof}
  By \eqref{sigrid} with $k=0$ and \refT{TB3},
  \begin{equation*}
\frac{w_0Z(m,n-1)}{Z(m,n)}
=\P(Y_1=0)\to\pi_0=\frac{w_0}{\Phi(\tau)},	
  \end{equation*}
and the result follows since $w_0\neq0$.
\end{proof}

For trees we have a corresponding result:

\begin{theorem}
  \label{TH4}
Let $\wwx=(w_k)_{k\ge0}$ be a weight sequence with $w_0>0$ and $w_k>0$
for some $k\ge2$.
If $\rho>0$ and $\spanw=1$, then
\begin{equation*}
  \frac{Z_{n+1}}{Z_n}\to\frac{\Phi(\tau)}{\tau}.
\end{equation*}
\end{theorem}

\begin{proof}
By Theorems \ref{TZ} and \ref{TH2}--\ref{TH3},
\begin{equation*}
  \frac{Z_{n+1}}{Z_n}
=\frac{n Z(n,n+1)}{(n-1)Z(n-1,n)}
=\frac{n}{n-1}\cdot\frac{Z(n,n+1)}{Z(n,n)}\cdot\frac{ Z(n,n)}{Z(n-1,n)}
\to\Phi(\tau){\tau}\qw.
\end{equation*}  
\end{proof}

We assumed here span 1 for convenience only; if $\spanw=d$, we instead
obtain, by a similar argument,
$Z_{n+d}/Z_n\to(\Phi(\tau)/\tau)^d$. 

\medskip

In the case  $\nu\ge1$ and $\gss=\tau\Psi'(\tau)<\infty$ (which is automatic
if $\nu>1$), 
\ie{} our case \Iga, 
\refT{TZlim} can be sharpened substantially as follows, see 
\citet{Otter},
\citet{MM78},
\citet{Kolchin},
\citet{Drmota}.

\begin{theorem}
  \label{TH1}
Let $\wwx=\ww$, $\tau$ and $\gss$ be as in \refT{Tmain}, and let
$d\=\spann(\wwx)$. 
If $\nu\ge1$ and $\gss<\infty$, then, for $n\equiv 1 \pmod d$, 
\begin{equation}
  \label{asymp}
Z_n\sim 
\frac{d}{\sqrt{2\pi\gss}}\cdot\frac{\Phi(\tau)^n\tau^{1-n}}{n\qqc}
=d\sqrt{\frac{\Phi(\tau)}{2\pi\Phi''(\tau)}}
\parfrac{\Phi(\tau)}{\tau}^n n\qqcw.
\end{equation}
\end{theorem}
\begin{proof}
  Replacing $\ww$ by $\ppi$ and using \eqref{tz}, we see that it suffices to
  consider the case of a \pws{} with $\tau=\Phi(\tau)=1$.
By \refT{TZ}, \eqref{ebprob} and \eqref{gss}, in this case the result is
equivalent to 
\begin{equation*}
  \P(S_n=n-1) \sim \frac{d}{\sqrt{2\pi\gss n}},
\end{equation*}
which is the local central limit theorem in this case, see \eg{}
\citetq{Theorem 1.4.2}{Kolchin} or use \refL{LCLT} and \refR{RCLT}.
\end{proof}

There is a corresponding improvement of \refT{TBZ}.

\begin{theorem}
  \label{THB1}
Let $\wwx=\ww$, $m=m(n)$, $\tau$ amd $\gss$ be as in \refT{TBmain}, and let
$d\=\spann(\wwx)$. 
If\/ $0<\gl<\nu$, or $\gl=\nu$  and $\gss<\infty$, then, 
for $m=\gl n+o(\sqrt n)$ with $m\equiv 0 \pmod d$,
\begin{equation}
  \label{asympb}
Z(m,n)\sim 
\frac{d}{\sqrt{2\pi\gss n}}{\Phi(\tau)^n\tau^{-m}}.
\end{equation}
\end{theorem}
\begin{proof}
Again it suffices to
  consider the case of a \pws{} with $\tau=\Phi(\tau)=1$; this time using
  \eqref{lb1z}. 
In this case the result is by  \eqref{ebprob}
equivalent to 
\begin{equation*}
  \P(S_n=m) \sim \frac{d}{\sqrt{2\pi\gss n}},
\end{equation*}
which again is the local central limit theorem and follows  \eg{} by
\refL{LCLT} and \refR{RCLT}.
\end{proof}

\begin{remark}
The asymptotic formula \eqref{asympb} 
holds for arbitrary $m=m(n)$ with $0<c \le m/n \le C<\go$ and
$m\equiv 0 \pmod d$, 
and either $C<\nu$ or $C=\nu$ and $\Phi''(\rho)<\infty$
(which means that $\Psi'(\rho)<\infty$ and thus
the distribution $\eqref{pik}$ has finite variance for $\tau=\nu$),
provided $\tau$ is replaced by $\tau(m/n)$ given by
$\Psi(\tau(m/n))=m/n$. (Cf.\ \refT{TBmain2}.)
The proof is essentially the same (as in the proof of \refT{TBmain2}, it
suffices to consider subsequences where $m(n)/n$ converges); we omit the
details. 
\end{remark}

In the case $\nu=\gl$ ($\nu=1$ in the tree case)
and $\gss=\infty$, 
we have no general results but 
we can obtain similar more precise versions of Theorems
\refand{TZlim}{TBZ}
in the important case of a power-law \ws, \refE{EBpower}.
(We need $1<\ga\le2$ here; if $\ga\le1$, then $\nu=\infty>\gl$, and if
$\ga>2$, then $\gss<\infty$ so Theorems \refand{TH1}{THB1} apply, see
\refE{EBpower} with $\gb=\ga+1$. Note also that $\spann(\wwx)=1$.) 
The case $\gl>\nu$ is treated in \refT{TDzeta}  and \refR{RDzeta}.

\begin{theorem}\label{THstab}
Suppose for some $c>0$ and $\ga$ with $1<\ga\le2$,
\begin{equation}\label{tail}
 w_k\sim ck^{-\ga-1}\qquad\text{as \ktoo}.
\end{equation}
  \begin{romenumerate}
  \item \label{thstab1}
If $\nu=1$, then, 
\begin{align}
Z_n&\sim 
\frac{\Phi(1)\xga}{c^{1/\ga}\Gamma(-\ga)\xga|\Gamma(-1/\ga)|}
 {\Phi(1)^n} n^{-1-1/\ga},
\qquad \text{when }
1<\ga<2,\\
  \intertext{and}  \label{mga2}
Z_n&\sim 
\Bigparfrac{\Phi(1)}{\pi c}\qq
 {\Phi(1)^n} n^{-3/2} (\log n)\qqw,
 \qquad \text{when }\ga=2. 
\end{align}

  \item \label{thstabm}
If $m=\nu n+o(n\xga)$, then
\begin{align} \label{mgb1}
Z(m,n)&\sim 
\frac{\Phi(1)\xga}{c^{1/\ga}\Gamma(-\ga)\xga|\Gamma(-1/\ga)|}
 {\Phi(1)^n} n^{-1/\ga},
\qquad \text{when }
1<\ga<2,\\
  \intertext{and}  \label{mgb2}
Z(m,n)&\sim 
\Bigparfrac{\Phi(1)}{\pi c}\qq
\frac{\Phi(1)^n}{\sqrt{ n \log n}},
 \qquad \text{when }\ga=2. 
\end{align}
  \end{romenumerate}
\end{theorem}

\begin{proof}
This time, we did not assume $w_0>0$, but we may do so without loss of
generality in the proof.
In fact, if $w_0=0$, then $\nu>1$, so in \ref{thstab1} we always have
$w_0>0$, and in \ref{thstabm} we can reduce to the case $w_0>0$ 
by the method in \refR{Rmin}.

\ref{thstab1} follows from \refT{TZ} and \ref{thstabm}, taking $m=n-1$; 
hence it suffices to prove \eqref{mgb1}--\eqref{mgb2}.

We have $\rho=1$, and in the usual notation $\gl=\nu$ and thus
$\tau=\rho=1$. 
We reduce to the \pws{}  case by dividing each $w_k$ by
$\Phi(1)$ (which changes $c$ to $c/\Phi(1)$).
Let $\xi$ be a random variable with the distribution $\ppi=\ww$.
Then $\E\xi=\nu$.
Furthermore, \eqref{tail} yields
\begin{equation}\label{tail2}
\P(\xi\ge k)=\sum_{l=k}^\infty w_l\sim c\ga\qw k^{-\ga}.   
\end{equation}
Hence $\xi$ is in the domain of attraction of an
$\ga$-stable distribution, see \citetq{Section XVII.5}{FellerII}.
More precisely, if we first consider the case $1<\ga<2$,
then there exists an $\ga$-stable random variable $X_\ga$
such that
\begin{equation}
  \frac{S_n-n\nu}{n^{1/\ga}}\dto X_\ga.
\end{equation}
(The distribution of $X_\ga$ is given by 
\eqref{stabchf} and
\eqref{tdzetastabL} below.)
Moreover, a local limit law holds, see \eg{}
\citetq{\S\ 50}{GneKol},
\citetq{Theorem 4.2.1}{IbragimovLinnik}
or \citetq{Corollary 8.4.3}{BinghamGoldieTeugels},
which says
\begin{equation}
  \label{l9}
\P(S_n=\ell)=n^{-1/\ga}\Bigpar{g\Bigparfrac{\ell-n\nu}{n^{1/\ga}}+o(1)},
\end{equation}
uniformly for all integers $\ell$,
where $g$ is the density function of $X_\ga$.
In particular,
\begin{equation}
  Z(m,n)=\P(S_n=m)\sim\ngaw g(0).
\end{equation}
The results in
\cite[Sections XVII.5--6]{FellerII} show,
if we keep track of the constants
(see \eg{} \cite{SJN12} for calculations), 
that 
\begin{equation}
  \label{g0}
g(0)=(c\,\Gamma(-\ga))\xgaw|\Gamma(-1/\ga)|\qw,
\end{equation}
and \eqref{mgb1} follows.

In the case $\ga=2$, \cite[Section XVII.5]{FellerII} similarly yields
\begin{equation}\label{gg}
\frac{S_n}{\sqrt{n\log n}}\dto N(0,c/2);
\end{equation}
again a local limit theorem holds by
\cite[Theorem 4.2.1]{IbragimovLinnik}
or \cite[Corollary 8.4.3]{BinghamGoldieTeugels},
and thus
\begin{equation}
  \label{l92}
\P(S_n=\ell)
=\frac1{\sqrt{n\log n}}\Bigpar{g\Bigparfrac{\ell-n\nu}{\sqrt{n\log n}}+o(1)},
\end{equation}
uniformly in $\ell\in\bbZ$,
where now $g(x)$ is the density function 
$(\pi c)\qqw e^{-x^2/c}$
of $N(0,c/2)$.
In particular,
\begin{equation}
  Z(m,n)=\P(S_n=m)\sim \frac1{\sqrt{n\log n}}g(0)=
\frac1{\sqrt{n\log n}}\cdot\frac1{\sqrt{\pi c}},
\end{equation}
which proves \eqref{mgb2}.
\end{proof}

\begin{remark}
  The proof shows that \eqref{tail} can be relaxed to \eqref{tail2} together
  with $\spann(\wwx)=1$.
\end{remark}

\begin{example}\label{Ebritikov2}
Let $\ffu mn$ be the number of labelled unrooted forests
with $m$ labelled nodes and $n$ labelled trees,
see \refE{Euforest}.  
Using the weights $w_k=k^{k-2}/k!$ and  
$\tw_k=e^{-k}w_k\sim(2\pi)\qqw k^{-5/2}$, 
we have by \eqref{fuforest} and \eqref{lb1z}
\begin{equation}
\ffu mn= m!\,Z(m,n;\wwx)=m!\,e^{m}Z(m,n;\twwx).
\end{equation}
At the phase  transition $m=2n$, \refT{THstab} applies to $\twwx$
with $\ga=3/2$.
We have $c=(2\pi)\qqw$ and, by \eqref{uforestphi}, 
$\Phi(1)=\Phi(\rho)=1/2$. Hence \eqref{mgb1} yields, after simplifications,
\begin{equation}\label{fu2n}
  \begin{split}
\frac{\ffu{2n}{n}}{2n!}=
Z(2n,n;\wwx)=e^{2n}Z(2n,n;\twwx)
\sim \frac{2^{-2/3}3^{-1/3}}{\Gamma(1/3)}
e^{2n} 2^{-n}n^{-2/3}.	
  \end{split}
\end{equation}
(The constant can also be written $2^{-5/3}3^{1/6}\pi\qw\gG(2/3)$.)
A more general result is proved by the same method by \citet{Britikov}.
\citet[Proposition VIII.11]{Flajolet},  
show \eqref{fu2n}  by a different method
(although there is a computational error in the constant given in the result
there). 
\end{example}

\smallskip

We end this section by considering the behaviour of the generating function
$\cZ(z)\=\sumni Z_nz^n$.
The following immediate corollary of \refT{TZlim} was shown by
\citet{Otter}, see 
\citet{Minami} and, for $\nu>1$, \citetq{Proposition IV.5}{Flajolet}.
See also also \refR{Rotter}.

\begin{corollary}
  \label{CZrho}
Let $\ww\geko$ and $\tau$ be as in  \refT{Tmain},
and let $\rhoz$ be the radius of convergence of 
the generating function 
$\cZ(z)\=\sumni Z_nz^n$. Then
$\rhoz\=\tau/\Phi(\tau)$.
\nopf
\end{corollary}

Moreover, 
by \eqref{otter}, 
$ \cZ(\rhoz)=\tau<\infty$. Since the generating function $\cZ(z)$ has
non-negative coefficients, it follows that $\cZ(z)$ is continuous on the
closed disc $\setz{|z|\le\rhoz}$, and $|\cZ(z)|\le\tau$ there.
If we, for simplicity, assume that $\spann(\wwx)=1$,
then 
$|\cZ(z)|<|\cZ(\rhoz)|=\tau$ for $|z|\le\rhoz$, $z\neq\rhoz$. Since
$|Z|<\tau$ implies 
\begin{equation*}
  \begin{split}
  |\Phi(Z)-Z\Phi'(Z)|
&=\biggabs{w_0-\sumki(k-1)w_kZ^k}
\ge
w_0-\sumki(k-1)w_k|Z|^k
\\&
>
w_0-\sumki(k-1)w_k\tau^k
=\Phi(\tau)-\tau\Phi'(\tau)=0,	
  \end{split}
\end{equation*}
it follows that
$\Phi(Z)-Z\Phi'(Z)\neq0$ if $Z=\cZ(z)$ with 
$|z|=\rhoz$, $z\neq\rhoz$; hence the implicit function theorem and
\eqref{cz} show that $\cZ(z)$ has an analytic continuation to some
neighbourhood of $z$.
Consequently, $\cZ$ then can be extended across $|z|=\rhoz$ everywhere except at
$z=\rhoz$.
(If $\spann(\wwx)=d$, the same holds except at $z=\rhoz e^{2\pi\ii j/d}$,
$j\in\bbZ$.)

In our case Ia ($\nu>1$, or equivalently $\tau<\rho$), much more is known:
$\cZ$ has a square root
  singularity at $\rhoz$ with a local expansion of $\cZ(z)$ as an analytic
  function of $\sqrt{1-z/\rhoz}$:
  \begin{equation}
	\label{sing}
\cZ(z)=\tau-b\sqrt{1-z/\rhoz}+\dots,
  \end{equation}
where, with $\gss\=\Var\xi$ given by \eqref{gss}, 
\begin{equation}\label{singb}
  b\=
\sqrt{\frac{2\Phi(\tau)}{\Phi''(\tau)}}
=\sqrt2\,\frac{\tau}{\gs},
\end{equation}
see \citet{MM78}, 
\citetq{Theorem VI.6}{Flajolet} and
\citetq{Section 3.1.4 and Theorem 2.19}{Drmota};
in particular, $\cZ$ then extends analytically to a neighbourhood of $\rho$ cut
  at the ray $[\rho,\infty)$.
In fact, this extends (in a weaker form) to the case $\nu\ge1$ and
$\gss<\infty$
(case \Iga): \eqref{sing} holds in a suitable region, with an error term 
$o(\sqrt{1-z/\rhoz})$, see \citet{SJ167}. 

\begin{remark}
  In the case $\nu>1$, 
\eqref{sing} and \eqref{singb} 
yield another proof of \eqref{asymp} 
by standard singularity analysis, 
see \eg{} \citetq{Theorem  3.6}{Drmota} and 
\citetq{Theorem VI.6 and  VII.2}{Flajolet}; this argument can be extended to
the case $\nu\ge1$ and $\gss<\infty$, see 
\citetq{Remark 3.7}{Drmota} and \citetq{Appendix}{SJ167}.
When $\nu>1$, an expansion with further terms can also be obtained, see
\citet{Minami} and 
\citetq{Theorem VI.6}{Flajolet}.  
\end{remark}

In the other cases ($\gss=\infty$ or $\nu<1$),
the asymptotic behaviour of $\cZ$ at the singularity $\rhoz$ depends on the
behaviour of $\Phi(z)$ at its singularity $\rho$. It seems difficult to say
anything detailed in general, so we study only a few examples.
We assume $\nu\le1$ and $\go>1$; thus \refL{LPsi} implies that
$\rho<\infty$, $\Phi(\rho)<\infty$ 
and $\Phi'(\rho)<\infty$. We assume also $\rho>0$ and $\spann(\wwx)=1$.

\begin{example}
Suppose that $0<\rho<\infty$ and that $\Phi(z)$ has an analytic extension to a
sector  
$\srd\=\set{z:|\arg(\rho-z)|<\pi/2+\gd \text{ and } |z-\rho|<\gd}$  for
some $\gd>0$, and that in this sector $\srd$, for some $a\neq0$ and
non-integer $\ga>1$, and some $f(z)$ analytic at $\rho$
(which can be taken as a polynomial of degree $<\ga$),
\begin{equation}\label{eaa}
  \Phi(z)=f(z)+a(\rho-z)^\ga+ o\bigpar{|\rho-z|^\ga},
\qquad \text{as }z\to\rho.
\end{equation}
(We have to have $\ga>1$ since $\Phi'(\rho)<\infty$. For $\ga\ge2$ integer,
see instead \refE{Ebb}.)
If we assume that
$\Phi$ has no further singularities on ${|z|=\rho}$,
this implies by singularity analysis, see \citetq{Section VI.3}{Flajolet},
\begin{equation}
  \label{eac}
w_k\sim \frac{a}{\Gamma(-\ga)} k^{-\ga-1}\rho^{\ga-k},
\qquad\text{as \ktoo}. 
\end{equation}
The converse does not
hold in 
general, but can be expected if the \ws{} is very regular.
For example, \eqref{eaa} holds (in the plane cut at $[\rho,\infty)$) 
if $w_k=(k+1)^{-\gb}$, $k\ge1$, as in \refE{Ezeta}, with $\gb=\ga+1>2$,
$\rho=1$ 
and $a=\Gamma(-\ga)$, see
\eg{} \cite[Section VI.8]{Flajolet}.

Let $F(Z)\=Z/\Phi(Z)$, so \eqref{cz} can be written 
\begin{equation}\label{tycho}
  F(\cZ(z))=z.
\end{equation}
Since $\nu\le1$, we have $\tau=\rho$, and thus by \refC{CZrho} and \eqref{otter}
$\rhoz=F(\tau)=F(\rho)$ and $\cZ(\rhoz)=\rho$.
Note that 
\begin{equation}\label{brahe}
  F'(\rho)=\frac{\Phi(\rho)-\rho\Phi'(\rho)}{\Phi(\rho)^2}
=\frac{1-\Psi(\rho)}{\Phi(\rho)}
=\frac{1-\nu}{\Phi(\rho)}.
\end{equation}
If $\nu<1$, then \eqref{brahe} yields
$F'(\rho)>0$ and \eqref{tycho} shows that 
$\rho-\cZ(z)\sim F'(\rho)\qw(\rhoz-z)$ as $z\to\rhoz$.
Moreover, $F$ is defined in a sector $\srd$, and its image contains some
similar sector $\srdy$ (with $0<\gd'<\gd$) such that $\cZ(z)$ extends
analytically to $\srdy$ by \eqref{tycho}, 
and it follows easily by \eqref{tycho} and
\eqref{eaa}
that in $\srdy$, with
 some  $f_1(z)$ analytic at $\rhoz$,
\begin{equation}\label{eab}
  \cZ(z)=f_1(z)+a_1(\rhoz-z)^\ga+ o\bigpar{|\rhoz-z|^\ga},
  \qquad \text{as }z\to\rhoz,
\end{equation}
where 
\begin{equation}\label{ea1}
  a_1=a\frac{\rho\Phi(\rho)^{\ga-1}}{(1-\nu)^{\ga+1}}.
\end{equation}

As noted above, $\cZ(z)$ has no other singularities on $\setz{|z|=\rho}$, and
singularity analysis \cite{Flajolet} applies and shows,
using \eqref{eac},
\begin{equation}\label{zeac}
Z_n\sim \frac{a_1}{\Gamma(-\ga)}n^{-\ga-1}\rhoz^{\ga-n}
\sim \frac{\rho}{(1-\nu)^{\ga+1}} \Phi(\rho)^{n-1} w_n.
\end{equation}
However, we will show in greater generality in \refT{TDzeta} and
\refR{RDzeta} (by a straightforward reduction to the case $\rho=1$ using
\eqref{tz})
that \eqref{eac} always implies \eqref{zeac} when $\nu<1$,
without any assumption 
like \eqref{eaa} on $\Phi(z)$.

If $\nu=1$, we assume $1<\ga<2$, since \eqref{eaa} with $\ga>2$ implies
$\Phi''(\rho)<\infty$ and thus $\gss<\infty$, so \eqref{sing} and \refT{TH1}
would apply. We now have $F'(\rho)=0$, and \eqref{eaa}--\eqref{tycho} yield,
in some domain $\srdy$,
  \begin{equation}
\cZ(z)=\rho-\parfrac{\Phi(\rho)}{a}^{1/\ga}\Bigpar{1-\frac{z}{\rhoz}}^{1/\ga}
+\dots.
  \end{equation}
Singularity analysis yields
  \begin{equation}\label{bon}
Z_n\sim\frac{1}{|\Gamma(-1/\ga)|}
\parfrac{\Phi(\rho)}{a}^{1/\ga}n^{-1-1/\ga}{\rhoz}^{-n}.
  \end{equation}
However, we have already proved in \refT{THstab}\ref{thstab1} 
(assuming $\rho=1$, without loss of generality)
that \eqref{eac} implies \eqref{bon} in this case, without any assumption
like \eqref{eaa} on $\Phi(z)$.
\end{example}

\begin{example}\label{Ebb}
If $\ga\ge2$ is an integer, \eqref{eaa} does not exhibit a
singularity. Instead we consider $\wwx$ with, for some $f$ analytic at $\rho$, 
\begin{equation}\label{ebb}
  \Phi(z)=f(z)+a(\rho-z)^\ga\log(\rho-z)+ O\bigpar{|\rho-z|^\ga},
\end{equation}
as $z\to\rho$
in some sector $\srd$. This includes the case $w_k=(k+1)^{-\ga-1}$, see 
\citetq{Section VI.8}{Flajolet}.

In the case $\nu<1$, we obtain as above
\begin{equation}
  \cZ(z)=f_1(z)+a_1(\rhoz-z)^\ga\log(\rhoz-z)+ O\bigpar{|\rhoz-z|^\ga},
\end{equation}
as $z\to\rhoz$ in some sector, 
with $f_1(z)$ analytic at $\rhoz$ and $a_1$ given by \eqref{ea1}.
We again obtain by singularity analysis
\begin{equation}
Z_n
\sim \frac{\rho}{(1-\nu)^{\ga+1}} \Phi(\rho)^{n-1} w_n,
\end{equation}
which is another instance of \eqref{rdzeta}.

In the case $\nu=1$, we consider only $\ga=2$, since $\gss<\infty$ if
$\ga>2$.
Then \eqref{ebb} yields (we have $a<0$ in this case)
  \begin{equation}
\cZ(z)=\rho-\parfrac{2\Phi(\rho)}{-a}^{1/2}\bigpar{1-\xfrac{z}{\rhoz}}^{1/2}
\bigpar{-\log\lrpar{1-\xfrac{z}{\rhoz}}}^{-1/2}
+\dots.
  \end{equation}
Singularity analysis 
\cite[Theorems VI.2--3]{Flajolet}
gives another proof of \eqref{mga2} in the special case
\eqref{ebb} (again assuming $\rho=1$, as we may).  
\end{example}

\begin{example}
 Define $\wwx$ by 
$\Phi(z)=w_0+\sumj 2^{-2j} z^{2^j}$, for some $w_0>0$;
thus $\supp(\wwx)$ is the lacunary sequence $\set0\cup\set{2^j}$. 
Then $\rho=1$, $\Phi(\rho)=w_0+4/3$ and $\Phi'(\rho)=2$; hence
$\nu=\Psi(\rho)=2/(w_0+4/3)$.
The function $\Phi(z)$ is analytic in the unit disc and has the unit circle
as a natural boundary; it cannot be extended analytically at any point.
(See \eg{} \citetq{Remark 16.4 and Theorem 16.6}{Rudin}.)

Taking $w_0>2/3$, we have $\nu<1$; hence, $F'(\rho)>0$ by \eqref{brahe}.
Thus $F$ maps the unit circle onto a closed curve $\gG$ that goes vertically
through
$F(1)=\rho_z$, and since $F$ cannot be continued analytically across the
unit circle, $\cZ(z)$ cannot be continued analytically across the curve
$\gG$. In particular, $\cZ(z)$ is not analytic in any sector $\srdy$.
\end{example}

\section{Largest degrees and boxes}\label{Slarge}

Consider a random allocation $\bmn=\YYn$ and arrange 
$Y_1,\dots,Y_n$ in decreasing order as $\ya\ge\yb\ge\dots$. Thus, $\ya$ is
the largest number of balls in any box, $\yb$ is the second largest, and so
on.

By \refL{Lx}, we may also consider the random tree
$\ctn$ by taking $m=n-1$; then
$\ya$ is the largest outdegree in $\ctn$, 
$\yb$ is the second largest outdegree, and so on.

As usual, we consider asymptotics as \ntoo{} and $m/n\to\gl$.
(Thus $\gl=1$ in the tree case.)
We usually ignore the cases $m/n\to0$ and $m/n\to\infty$; 
these are left to the reader as open problems. (See \eg{} \citet{KolchinSCh},
\citet{Kolchin}, \citet{Pavlov} and \citet{Kazimirov}
for examples of such results.)

The results in Sections \refand{Smain}{SBB} suggest that 
$\ya$ is small when $\gl<\nu$, but large (perhaps of order $n$) when
$\gl>\nu$, which is one aspect of the phase transition at $\gl=\nu$.
We will see that this roughly is correct, but that the full story is
somewhat more complicated.

We study the cases $\gl\le\nu$ and $\gl>\nu$ separately;
we also consider separately several subcases of the first case
where we can give more precise results.  

We
first note that the case $\go<\infty$, when the box capacities (node degrees in
the tree case)
are bounded is trivial: \whp{} the maximum is attained in many boxes.

\begin{theorem}\label{Tgo}
  Let $\wwx=(w_k)_{k\ge0}$ be a weight sequence with $w_0>0$ and 
$\go<\infty$.
Suppose that $n\to\infty$ and $m=m(n)$ with $m/n\to\gl>0$.
  Then $\yj=\go$ \whp{} for every fixed $j$.
\end{theorem}
\begin{proof}
  Clearly, each $Y_i\le\go$, so $\yj\le\ya\le\go$.

We assume tacitly, as always,  that $\bmn$ exists, \ie{} $Z(m,n)>0$, 
and thus $m\le\go n$, so $\gl\le \go$.
By \refT{TBmain} if $\gl<\go$, and \refR{Rsymmetry} if $\gl=\go$,
$N_\go(\bmn)/n\pto\pi_\go>0$. In particular, $N_\go(\bmn)\pto\infty$, and thus
$\P(\yj=\go)\to1$. 
\end{proof}

\subsection{The case $\gl\le\nu$}

In the case $\gl\le\nu$, we show that, indeed, all $Y_i$ are small.
Theorems \ref{TD1a}--\ref{TD1} yield (\whp) a bound $o(n)$ when $\gl=\nu$,
and a much 
stronger logarithmic bound $O(\log n)$ when $\gl<\nu$.
(In the tree case, we have $\gl=1$, so these are the cases $\nu=1$ and $\nu>1$.)

\refE{EL8} shows that in general, the bound $o(n)$ when $\gl=\nu$ is
essentially best possible; at least, we can have $\ya>n^{1-\eps}$ \whp{} for
any given $\eps>0$.

\begin{theorem}
  \label{TD1a}
  Let $\wwx=(w_k)_{k\ge0}$ be a weight sequence with $w_0>0$ and $w_k>0$
  for some $k\ge1$.
Suppose that $n\to\infty$ and $m=m(n)$ with $m/n\to\gl$ where $0\le\gl<\infty$.
If\/ $\gl\le\nu$,
then $\ya=\op(n)$.
\end{theorem}

Equivalently, $\ya/n\pto0$.

\begin{proof}
   The case $\gl=0$ is trivial, since $\ya/n\le  m/n\to\gl$. 
The case $\gl=\go$ is also trivial, since then $\go<\infty$ and $\ya\le\go$.
As above, $\gl>\go$ is impossible.
Hence we may assume 
$0<\gl<\go$ and $\nu\ge\gl>0$, which implies $\tau>0$, where
$\Psi(\tau)=\gl$, \cf{} \refT{TBmain}. 
We may then for convenience replace $\ww$ by the equivalent \ws{} $\ppi$ in
\eqref{pik};  we may thus assume
that $\wwx$ is a \pws{} with $\tau=1$, and thus $\rho\ge\tau=1$, and then the
corresponding random variable $\xi$ has $\E\xi=\gl$.

By \eqref{sigrid} and symmetry, for any $k\ge0$,
\begin{equation}
  \label{l5}
\P(\ya=k)\le 
n\P(Y_1=k)=n\frac{w_kZ(m-k,n-1)}{Z(m,n)}.
\end{equation}
Furthermore, $w_k=\pi_k=\P(\xi=k)\le1$ and, using \refE{EBprob}, 
$Z(m,n)=\P(S_n=m)=e^{o(n)}$ by
Lemma \ref{LB2} ($\rho>1$)
or \ref{LB2sub} ($\rho=1$). We turn to estimating $Z(m-k,n-1)$.

Let $0<\eps<\gl$, and define $\taue$ by $\Psi(\taue)=\gl-\eps$.
Since $\Psi(\tau)=\gl$, we have $0<\taue<\tau=1$.

For each $n$, choose $k=k(n)\in[\eps n,m]$ such that $Z(m-k,n-1)$ is
maximal.
We have $\eps\le k/n\le m/n\to\gl$;
choose a subsequence such that $k/n$ converges, say $k/n\to\gamma$ with
$\eps\le\gamma\le\gl$.
Then, along the subsequence, $(m-k)/(n-1)\to\gl-\gam$.

By \refT{TBZ} (and \refR{RZ0}, ignoring the trivial case $Z(m-k,n-1)=0$),
using $\taue<1$, $\gam\ge\eps$ and \eqref{annagl},
\begin{equation*}
  \begin{split}
\frac1n\log Z(m-k,n-1)
&\to
\log\inf_{t\ge0}\frac{\Phi(t)}{t^{\gl-\gam}}	
\le
\log\inf_{0\le t\le\taue}\frac{\Phi(t)}{t^{\gl-\gam}}	
\\&
\le
\log\inf_{0\le t\le\taue}\frac{\Phi(t)}{t^{\gl-\eps}}	
=
\log\frac{\Phi(\taue)}{\taue^{\gl-\eps}}	
=: \ce,
  \end{split}
\end{equation*}
say, where \refR{RBanna} shows that, since $\taue\neq1$,
\begin{equation}
  \label{l6}
\ce<\log\bigpar{\Phi(1)/1^{\gl-\eps}}=0.
\end{equation}

We have shown that
\begin{equation}
  \label{l7}
\limsup_\ntoo\frac1n\log Z(m-k,n-1)
\le\ce
\end{equation}
for $k=k(n)$ and any subsequence such that $k/n$ converges; it follows  that
\eqref{l7} holds for the full sequence.
In other words,
\begin{equation}
  \label{l7aj}
\log Z(m-k,n-1)
\le\ce n + o(n)
\end{equation}
for our choice $k=k(n)$ that maximises the \lhs, and thus uniformly for all
$k\in[\eps n,m]$.
Using \eqref{l7aj} and, as said above, \refL{LB2sub} in \eqref{l5} we obtain,
recalling \eqref{l6},
\begin{equation*}
  \P(\ya\ge\eps n)
=\sum_{k=\eps n}^m \P(\ya=k)\le m n e^{\ce n+o(n)}e^{o(n)}
= e^{\ce n+o(n)}\to0.
\end{equation*}
In other words, for any $\eps>0$, $\ya<\eps n$ \whp, which is equivalent to
$\ya=\op(n)$. 
\end{proof}

The following logarithmic bound when $\gl<\nu$ is essentially due to
\citet{MM90} (who studied the tree case).

\begin{theorem}
  \label{TD1}
  Let $\wwx=(w_k)_{k\ge0}$ be a weight sequence with $w_0>0$ and $w_k>0$
  for some $k\ge1$.
Suppose that $n\to\infty$ and $m=m(n)$ with $m/n\to\gl$.
Assume $0<\gl<\nu$, and define  $\tau\in(0,\rho)$ by $\Psi(\tau)=\gl$.
\begin{romenumerate}
\item \label{td1b}
Then $\tau<\rho$ and
\begin{equation}\label{td1}
\ya \le  \frac1{\log(\rho/\tau)}\log n + \op(\log n).
\end{equation}
\item \label{td1c}
In particular, if $\rho=\infty$, then $\ya=\op(\log n)$.
\item \label{td1d}
If further $w_k^{1/k}\to1/\rho$ as $\ktoo$, 
then, for every fixed $j\ge1$,
\begin{equation}\label{td1z}
\frac{Y\win j}{\log n} \pto  \frac1{\log(\rho/\tau)}.
\end{equation}
\end{romenumerate}
\end{theorem}

Recall that $1/\rho=\limsup_{\ktoo}w_k^{1/k}$, see \eqref{rho},
so the extra assumption $w_k^{1/k}\to1/\rho$ as \ktoo{} in \ref{td1d} holds
unless the \ws{} is rather irregular.
(The proof shows that the assumption can be weakened to $\P(\xi\ge
k)^{1/k}\to\tau/\rho$.) 

It is not difficult to show \refT{TD1} directly, but we prefer to postpone the
proof and use parts of the more refined \refT{TDx} below, in order to avoid some
repetitions of arguments.

We conjecture that \refT{TD1} holds also for $\gl=0$.
Since then $\tau=0$,
this means the following.
(This seems almost obvious given the result for positive $\gl$ in
\refT{TD1},
where the constant $1/\log(\rho/\tau)\to0$ as $\gl\to0$ and thus $\tau\to0$,
but there is no general monotonicity and we leave this as an open problem.)

\begin{conjecture}
If\/ $\rho>0$ and $m/n\to0$, then $\ya=\op(\log n)$.   
\end{conjecture}

\subsection{The subcase $\gss<\infty$}

In the case $\gss\=\Var\xi<\infty$ (which includes the case $\gl<\nu$),  
there is a much more precise result, which says that, simply,
the largest numbers $\ya,\yb\dots$ asymptotically have the same distribution
as the largest elements in the \iid{} sequence $\xin$. (Provided we choose
the distribution of $\xi$ correctly, and possibly depending on $n$, see
below for details.) 
In other words, the conditioning in \refE{EBprob} then
has asymptotically no effect on the largest elements of the sequence.
(When $\gss=\infty$ this is no longer necessarily true, however, as we
shall see in \refE{EL8}.)

In order to state this precisely, we  now assume that $\go=\infty$
(see \refT{Tgo} otherwise)
and $0<\gl\le\nu$, and 
define as usual $\tau$ 
by $\Psi(\tau)=\gl$, and
let $\xi$  be a random variable with the distribution in
\eqref{pik}.

If $m/n\le\nu$, we further define $\taun$ by
$\Psi(\tau_n)=m/n$, and let $\xi\nn$ be the random variable with the
distribution in \eqref{xinn}.
We will only use $\taun$ and $\xi\nn$ in the case $\gl<\nu$, so
$m/n\to\gl<\nu$ and $\taun$ really is defined (at least for large $n$);
furthermore $\taun\to\tau<\rho$ and $\xi\nn\dto\xi$. 

We further let $\xin$ and (when $\gl<\nu$)
$\xi\nn_1,\dots,\xi\nn_n$ be \iid{} sequences of
copies of $\xi$ and $\xi\nn$, respectively, and we arrange them in
decreasing order as $\xi\win1\ge\dots\ge\xi\win n$ and
$\xi\nn\win1\ge\dots\ge\xi\nn\win n$.
Finally, we introduce the counting variables, for any subset $A\subseteq
\No$,
\begin{align}
  \ny{A}&\=|\set{i\le n:Y_i\in A}|,\label{yna}
\\
  \nxi{A}&\=|\set{i\le n:\xi_i\in A}|,
\\
  \nxin{A}&\=|\set{i\le n:\xi\nn_i\in A}|.
\end{align}
($\ny A$ and $\nxi A$ also depend on $n$, but as usual, we for simplicity do
not show this in the notation.)
Note that $\nxi A$ and $\nxin A$ simply have binomial distributions
$\nxi A \sim \Bi(n,\P(\xi\in A))$ and
$\nxin A \sim \Bi(n,\P(\xi\nn\in A))$.

We have
\begin{equation}\label{cqk}
  Y\win j \le k \iff \ny{[k+1,\infty)}<j,
\end{equation}
and similarly for $\xi\win j$ and $\xi\nn\win j$. Thus it is elementary to
obtain asymptotic results for the maximum $\xi\win1$ of \iid{} variables,
and more generally  for
$\xi\win j$ and $\xi\nn\win j$, see \eg{} \citet{LLR}.

We introduce three different probability metrics to state the results.
For discrete random variables $X$ and $Y$ with values in $\No$ (the 
case we are interested in here), we define the \emph{Kolmogorov distance}
\begin{equation}\label{dk}
  \dk(X,Y)\=\sup_{x\in\No}|\P(X\le x)-\P(Y\le x)|
\end{equation}
and the \emph{total variation distance}
\begin{equation}\label{dtv}
  \dtv(X,Y)\=\sup_{A\subseteq\No}|\P(X\in A)-\P(Y\in A)|.
\end{equation}
In order to treat also the case with variables tending to $\infty$, we
further define the
\emph{modified Kolmogorov distance}
\begin{equation}\label{dkk}
  \dkk(X,Y)\=\sup_{x\in\No}\frac{|\P(X\le x)-\P(Y\le x)|}{1+x}.
\end{equation}
For $\dkk$, we also allow random variables in $\bNo$, \ie, we allow the
value $\infty$. 
(Furthermore, the definitions of $\dk$ 
and $\dtv$ and the results for  them in the lemma below extend to random
variables with values in $\bbZ$. The definitions extend further to 
\rv{s} with values in
$\bbR$ for $\dk$, and in any space for $\dtv$, but not all properties below
hold in this generality.) 

Note that these distances depend only on the distributions $\cL(X)$ and
$\cL(Y)$, so $d(\cL(X),\cL(Y))$ might be a better notation, but we find it
convenient to allow both notations, as well as the mixed 
$d(X,\cL(Y))$.

It is obvious that the three distances above are metrics on the space of
probability measures on $\No$ (or on $\bNo$).

We collect a few simple, and mostly well-known, facts for these three
metrics in a lemma; the proofs are left to the reader.

\begin{lemma}
\label{LM}
\begin{thmenumerate}
\item \label{lma}
For any random variables $X$ and $Y$ with values in $\No$,
\begin{equation*}
  \dkk(X,Y) \le \dk(X,Y) \le \dtv (X,Y).
\end{equation*}
\item \label{lmb}
For any $X$ and $X_1,X_2,\dots$ with values in $\No$,
\begin{equation*}
  \begin{split}
X_n\dto X &\iff \dtv(X_n,X)\to0\iff \dk(X_n,X)\to0
\\&
\iff \dkk(X_n,X)\to0.	
  \end{split}
\end{equation*}

\item \label{lmc}
For any $X$ and $X_1,X_2,\dots$ with values in $\bNo$,
\begin{equation*}
X_n\dto X \iff \dkk(X_n,X)\to0.
\end{equation*}
In particular,
\begin{equation*}
X_n\pto \infty \iff \dkk(X_n,\infty)\to0.
\end{equation*}

\item \label{lmkk}
For any $X_n$ and $X'_n$ with values in $\bNo$,
$ \dkk(X_n,X'_n)\to0 \iff$ 
$\bigabs{\P(X_n\le x)-\P(X'_n\le x)}\to0$
for every fixed $x\ge0$.

\item \label{lmd}
For any $X_n$ and $X'_n$, $\dtv(X_n,X'_n)\to0\iff$
there exists a coupling $(X_n,X'_n)$ with $X_n=X'_n$ \whp.
(We denote this also by $X_n\dapprox X'_n$.)

\item \label{lme}
The supremum in \eqref{dtv} is attained, and the absolute value sign is
redundant. In fact, if $A\=\set{i:\P(X=i)>\P(Y=i)}$, then 
$\dtv(X,Y)=\P(X\in A)-\P(Y\in A)$.

\item \label{lmf}
For any $X$ and $Y$ with values in $\No$,
\begin{equation*}
  \dtv(X,Y)
=\sum_{x\in\No}\bigpar{\P(X= x)-\P(Y= x)}_+
=\tfrac12\sum_{x\in\No}|\P(X= x)-\P(Y= x)|.
\end{equation*}
\end{thmenumerate}
\nopf
\end{lemma}

\begin{remark}
  The three metrics are, by \refL{LM}\ref{lmb}, equivalent in the usual
  sense that they define the same topology, but they are not uniformly
  equivalent.
For example, if $X_n\sim\Po(n)$,
$X'_n\=2\floor{X_n/2}$
(\ie, $X_n$ rounded down to an even integer) and $X_n''\=X_n'+1$, then
$\dk(X_n',X_n'')\to0$ 
as \ntoo, but $\dtv(X_n',X_n'')=1$.
\end{remark}

We define $\Po(\infty)$ as the distribution of a random variable that equals
$\infty$ identically.

After all these preliminaries, we state the result (together with some
supplementary results). There are really two
versions; it turns out that for general sequences $m(n)$, we have to use the
random variables $\xi\nn$, with $\E\xi\nn=m(n)/n$ exactly tuned to $m(n)$,
but under a weak assumption we can replace $\xi\nn$ by $\xi$ and obtain a
somewhat simpler statement, which we choose as our main formulation.
(This goes back to \citet{MM91}, who proved
\ref{tdxkn} in the tree case, assuming $\gl<\nu$; see also
\citet[Theorem 1.6.1]{KolchinSCh} and
\citet[Theorem 1.5.2]{Kolchin} for $\ya$ in the special case
in \refE{EMB}.)

\begin{theorem}\label{TDx}
Let $\wwx=(w_k)_{k\ge0}$ be a weight sequence with $w_0>0$ and $\go=\infty$.
Suppose that $n\to\infty$ and $m=m(n)$ with 
$m=\gl n+o(\sqrt n)$  
where $0<\gl\le\nu$, and 
use the notation above.
Suppose further that $\gss\=\Var\xi<\infty$.
(This is redundant when $\gl<\nu$.)
  \begin{romenumerate}
  \item \label{tdxkn}
If (possibly for $n$ in a subsequence) $h(n)$ are integers such that 
\\$nP(\xi\ge h(n))\to \ga$, for some $\ga\in[0,\infty]$, then 
\begin{equation*}
\ny{[h(n),\infty)}\=|\set{i:Y_i\ge h(n)}|\dto\Po(\ga).  
\end{equation*}

  \item \label{tdx0}
If 
$h(n)$ are integers such that 
$nP(\xi\ge h(n))\to 0$, then \whp{}
$\ya <h(n)$.

  \item \label{tdxoo}
If 
$h(n)$ are integers such that 
$nP(\xi\ge h(n))\to \infty$, then, for every fixed $j$, \whp{}
$\yj \ge h(n)$.

\item \label{tdxkk}
For any sequence $h(n)$,
$\dkk\bigpar{\ny{[h(n),\infty)},\nxi{[h(n),\infty)}}\to0$.

\item \label{tdxk}
For every fixed $j$, 
$\dk\bigpar{\yj,\xi\win j}\to0$.

\item \label{tdxtv}
$\dtv\bigpar{\ya,\xi\win 1}\to0$.
  \end{romenumerate}

If $\gl<\nu$, the condition $m=\gl n+o(\sqrt n)$ can be weakened to 
$m/n=\gl+o(1/\log n)$.

Moreover, if $\gl<\nu$, then the results hold for any $m=m(n)$ with
$m/n\to\gl$, provided $\xi$ is replaced by $\xi\nn$, $\nxi{}$ by $\nxin{}$
and $\xi\win j$ by $\xi\nn\win j$. 
\end{theorem}

\begin{remark}
  In the version with $\xi\nn$, we do not need $\gl$ at all. By considering
  subsequences, it follows that it suffices that $0<c\le m/n\le C<\nu$.
(Cf.\ \refT{TBmain2}.) 
Furthermore, this version extends to the case $\gl=\nu$ and
$m/n\le\nu$, but we have ignored this case for simplicity.
\end{remark}

\begin{problem}
  Is \refT{TDx} (in the $\xi\nn$ version) true also for $\gl=0<\nu$?
\end{problem}

The total variation approximation in \ref{tdxtv} is stronger than the
Kolmogorov distance approximation in \ref{tdxk}, 
and our proof is considerably longer,
but for many purposes \ref{tdxk} is enough. 
We conjecture that total variation approximation holds for
every $Y\win j$, and not just for $\ya$; presumably this can be shown by a
modification of the proof for $\ya$ below, but we have not checked the
details and leave this as an open problem. 
Furthermore, we believe that the
result extends to the joint distribution of finitely many $Y\win j$. 
(The
corresponding result in \ref{tdxk}, using a multivariate version of the
Kolmogorov distance, is easily verified by the methods below.)

\begin{problem}
Does $\dtv\bigpar{Y\win j,\xi\win j}\to0$ hold for every fixed $j$,
under the assumptions of \refT{TDx}?
\end{problem}

\begin{proof}[Proof of \refT{TDx}]

As in the proof of \refT{TD1a}, 
we may replace $\ww$ by the equivalent \ws{} $\ppi$ in
\eqref{pik}. 
We may thus assume 
that $\wwx$ is a \pws{} with $\tau=1$, and thus $\rho>\tau=1$, and the
corresponding random variable $\xi$ has $\E\xi=\gl$.
We consider first the version with $\xi$, assuming
$m=\gl n+o(\sqrt n)$, and discuss afterwards the modifications for $\xi\nn$.

We begin by looking again at \eqref{sigrid}:
\begin{equation}
  \label{l41a}
\P(Y_1=k)
=\frac{w_kZ(m-k,n-1)}{Z(m,n)}.
\end{equation}
When $m=\gl n+o(\sqrt n)$, 
we may apply \refL{LCLT} and \refR{RCLT} and thus,
with $d\=\spann(\wwx)$,
\begin{equation}\label{zm=}
Z(m,n)=\P(S_{n}=m)= \frac{d+o(1)}{\sqrt{2\pi\gss n}}.  
\end{equation}
Moreover, by \eqref{lcltle}, for any $k$,
\begin{equation}\label{zm-kle}
Z(m-k,n-1)=\P(S_{n-1}=m-k)\le \frac{d+o(1)}{\sqrt{2\pi\gss n}}.  
\end{equation}
Consequently, \eqref{l41a} yields, uniformly for all $k$,
\begin{equation}\label{yabound}
  \P(Y_1 = k) \le (1+o(1)) w_k
= (1+o(1))\P(\xi=k).
\end{equation}
In particular, we may sum over $k\ge K$ and obtain, for any $K=K(n)$,
\begin{equation}\label{y1bound}
  \P(Y_1 \ge K) \le (1+o(1))\P(\xi \ge K).
\end{equation}

Since, by assumption,  $\E\xi^2<\infty$, we have $\P(\xi\ge K)=o(K^{-2})$ as
$\Ktoo$. Hence, for every fixed $\gd>0$, 
$\P(\xi\ge\gd\sqrt n) = o(n\qw)$. It follows that there exists a sequence
$\gd_n\to0$ such that
$\P(\xi\ge\gd_n\sqrt n) = o(n\qw)$. Consequently, defining $\bn\=\gd_n\sqrt
n$, we have $\bn=o(\sqrt n)$ and 
\begin{equation}
  \label{xibn}
\P(\xi\ge \bn) = o(n\qw),
\end{equation}
and thus, by \eqref{y1bound} and symmetry,
  \begin{equation}  \label{ybn}
\P(\ya\ge \bn) \le n \P(Y_1\ge \bn)
=n \bigpar{1+o(1)}\P(\xi\ge \bn) = o(1).
  \end{equation}
Hence, $\ya<\bn$ \whp.

Similarly, $\P(\xi\win1\ge \bn) \le n \P(\xi_1\ge \bn)=o(1)$, so 
$\xi\win1<\bn$ \whp.

\pfitemref{tdxkn}
Write, for convenience, $N\=N_{[h(n),\bn]}$, and note that \whp{} 
$\ya\le \bn$ and then $N=N_{[h(n),\infty)}$.
(We assume for simplicity $h(n)\le B(n)$; otherwise we let $N\=0$, leaving
  the trivial modifications in this case to the reader.)

Moreover, for $k\le \bn=o(\sqrt n)$, 
we have $(m-k)-(n-1)\gl=o(\sqrt n)$, and thus
\refR{RCLT} shows that, for any $k=k(n)\le \bn$,
\begin{equation}\label{zm-k=}
Z(m-k,n-1)=\P(S_{n-1}=m-k)= \frac{d+o(1)}{\sqrt{2\pi\gss n}}.  
\end{equation}
Since we here may take $k=k(n)$ that maximises or minimises this for
$k\le \bn$, it follows that \eqref{zm-k=} holds uniformly for all $k\le \bn$.
Consequently, by 
\eqref{l41a}, \eqref{zm=} and \eqref{zm-k=},
\begin{equation}\label{1722}
  \P(Y_1 = k) = (1+o(1)) w_k
= (1+o(1))\P(\xi=k),
\end{equation}
uniformly for all $k\le \bn$. By
the assumption and \eqref{xibn}, this yields
\begin{equation*}
  \begin{split}
\E N 
&= n\sum_{k=h(n)}^{\bn} \P(Y_1=k)
= n\sum_{k=h(n)}^{\bn} \etto\P(\xi=k)
\\&
= \etto n\P\bigpar{h(n)\le\xi\le\bn}
\\&
= \etto n\bigpar{\P(\xi\ge h(n))-\P(\xi> \bn)}
\to \ga.
  \end{split}
\end{equation*}

Similarly, again using the symmetry as well as \refL{LCLT} and \refR{RCLT},
\begin{equation*}
  \begin{split}
\E N(N-1)
&=n(n-1)\P\bigpar{Y_1,Y_2\in[h(n),\bn]} 
\\&
= n(n-1)\sum_{k_1,k_2=h(n)}^{\bn} \P(Y_1=k_1 \text{ and } Y_2=k_2)
\\&
= n(n-1)\sum_{k_1,k_2=h(n)}^{\bn}
\frac{w_{k_1}w_{k_2}Z(m-k_1-k_2,n-2)}{Z(m,n)}
\\&
= n(n-1) \sum_{k_1,k_2=h(n)}^{\bn} \P(\xi=k_1) \P(\xi=k_2) \etto
\\&
= \etto n^2\bigpar{\P(\xi\ge h(n))-\P(\xi> \bn)}^2
\\&
\to \ga^2.
  \end{split}
\end{equation*}
Moreover,
the same argument works for any factorial moment $\E(N)_\ell$ and yields
$\E(N)_\ell\to\ga^\ell$ for every
$\ell\ge1$.
If $\ga<\infty$, we thus obtain $N\dto\Po(\ga)$ by the method of moments,
and the result follows, since $N=N_{[h(n),\infty)}$ \whp. 

If $\ga=\infty$, 
this argument yields
\begin{equation}\label{3monkeys}
\E(N)_\ell\sim \bigpar{n\P(\xi\ge h(n))}^\ell
\to\infty
\end{equation}
for every $\ell\ge1$, and we make a thinning: 
Let $A$ be a
constant and let $q\=A/\bigpar{n\P(\xi\ge h(n))}$; then $q\to A/\ga=0$.
We consider only $n$ that are so large that $q<1$. We then randomly, and
independently, mark each box with probability $q$.
Let $N'$ be the random number of marked boxes $i$ such that
$Y_i\in[h(n),B(n)]$ . Then, for every $\ell\ge1$, using \eqref{3monkeys},
\begin{equation}
\E (N')_\ell
=(n)_\ell q^\ell\P\bigpar{Y_1,\dots,Y_\ell\in[h(n),\bn]}   
= q^\ell \E (N)_\ell
\to A^\ell.
\end{equation}
Consequently, by the method of moments,
$N'\dto\Po(A)$. 
In particular, this shows, for every fixed $x$,
\begin{equation*}
  \P(N< x) \le   \P(N'< x) 
\to \P(\Po(A)< x),
\end{equation*}
which can be made arbitrarily small by taking $A$ large.
Hence, $\P(N< x) \to0$ for every fixed $x$, \ie, $N\pto\infty$ 
and thus $N_{[h(n),\infty)}\pto\infty$, 
as we claim
in this case.

\pfitemref{tdx0}
Part \ref{tdxkn} applies with $\ga=0$, and yields 
$N_{[h(n),\infty)}\pto0$, which means
$N_{[h(n),\infty)}=0$ \whp.
Thus $\yj<h(n)$ \whp{} by \eqref{cqk}.

\pfitemref{tdxoo}
Part \ref{tdxkn} applies with $\ga=\infty$, and yields 
$N_{[h(n),\infty)}\pto\infty$.
Thus, for every fixed $j$, 
by \eqref{cqk},
$\P(\yj <h(n))=\P(N_{[h(n),\infty)}<j)\to0$.

\pfitemref{tdxkk}
Suppose not.
Then there exists a sequence $h(n)$ and an $\eps>0$ such that, for some
subsequence,
\begin{equation}
  \label{mm4}
\dkk\bigpar{\ny{[h(n),\infty)},\nxi{[h(n),\infty)}}>\eps.
\end{equation}
We may select a subsubsequence such that $n\P(\xi\ge h(n))\to\ga$ for some
$\ga\in[0,\infty]$; then 
$\dkk\bigpar{\ny{[h(n),\infty)},\Po(\ga)}\to0$ by \ref{tdxkn} and
  \refL{LM}\ref{lmc}.  
Moreover, along the same subsubsequence,
$\nxi{[h(n),\infty)}\sim\Bi\bigpar{n,\P(\xi\ge h(n)}\dto\Po(\ga)$, by the
  standard Poisson approximation for binomial distributions (and rather
  trivially if $\ga=\infty$); hence
$\dkk\bigpar{\nxi{[h(n),\infty)},\Po(\ga)}\to0$.
The triangle inequality yields
$\dkk\bigpar{\ny{[h(n),\infty)},\nxi{[h(n),\infty)}}\to0$ along the
	subsubsequence, which contradicts \eqref{mm4}. This contradiction proves 
\ref{tdxkk}. 

\pfitemref{tdxk}
Suppose not. Then, by \eqref{dk}, there is an $\eps>0$ and a subsequence
such that for some $h(n)$,
\begin{equation}\label{mm5}
  \bigabs{\P(\yj\le h(n))-\P(\xi\win j\le h(n))}\ge\eps.
\end{equation}
However, by \eqref{cqk}, \eqref{dkk} and \ref{tdxkk},
\begin{equation*}
  \begin{split}
\bigabs{\P(\yj\le h(n))-&\P(\xi\win j\le h(n))}
\\&
=
\bigabs{\P(N_{[h(n)+1,\infty)}\le j-1)-\P(\nxi{[h(n)+1,\infty)}\le j-1)}
\\&
\le j \dkk\bigpar{N_{[h(n)+1,\infty)},\nxi{[h(n)+1,\infty)}}
\to0,
  \end{split}
\end{equation*}
which contradicts \eqref{mm5}. This contradiction proves 
\ref{tdxk}.

\pfitemref{tdxtv}
Let $A=A(n)\=\set{i:\P(\yj=i)>\P(\xi\win j=i)}$; 
thus, see \refL{LM}\ref{lme},
\begin{equation}\label{j51}
\dtv(\yj,\xi\win j)=\P(\yj\in A)-\P(\xi\win j\in A).
\end{equation}

Let $\gd>0$. For each $n$, we partition $\No$ into a finite family
$\cP=\set{J_l}_{l=1}^L$ of intervals as follows. First, each $i\in\No$ with
$\P(\xia=i)\ge \gd/2$ is a singleton \set{i}; note that there are at most
$2/\gd$ such $i$.
The complement of the set of these $i$ consists of at most $2/\gd+1$
intervals $\tJ_k$ (of which one is infinite).
We partition each such interval $\tJ_k$ further into intervals $J_l$ with
$\P(\xia\in J_l)\le\gd$ by repeatedly chopping off the largest such
subinterval starting at the left endpoint. Since only points with
$\P(\xia=i)<\gd/2$ remain, each such interval $J_l$ except the last in each
$\tJ_k$ satisfies $\P(\xia\in J_l)>\gd/2$.
Hence, our final partition $\set{J_l}$ contains at most $2/\gd+1$ intervals
$J_l$ with $\P(\xia\in J_l)<\gd/2$, while the number of intervals $J_l$ with
$\P(\xia\in J_l)\ge\gd/2$ is clearly at most $2/\gd$. Consequently, $L$, the
total number of intervals, is at most $4/\gd+1$.

We write $J_l=[a_l,b_l]$.
We say that an interval $J_l\in\cP$ is \emph{fat} if $\P(\xia\in
J_l)>\gd$, and 
\emph{thin} otherwise. Note that by our construction, a fat interval is a
singleton \set{a_l}.

Next, fix a large number  $D$.
We say that an interval $J_l=[a_l,b_l]\in\cP$ is \emph{good} if 
$n\P(\xi\ge a_l) \le D$, and \emph{bad} otherwise.

For any interval $J_l$,
\begin{equation}
\label{jl}
\bigabs{\P(\ya\in J_l)-\P(\xia\in J_l)}\le 2 \dk(\ya,\xia)=o(1)
\end{equation}
by \ref{tdxk}.

Let $A_l\=A\cap J_l$. Thus $A$ is the disjoint union $\bigcup_lA_l$.
($A$, $J_l$ and $A_l$ depend on $n$.)

We note that if 
$J_l$ is fat, then $J_l$ is a singleton, and either $\al=\jl$ or
$\al=\emptyset$; in both cases we have, using \eqref{jl},
\begin{equation}
\label{fat}
{\P(\ya\in \al)-\P(\xia\in \al)}\le 2 \dk(\ya,\xia)=o(1).
\end{equation}

We next turn to the good intervals. We claim that, uniformly for all good
intervals $J_l$, as \ntoo,
\begin{equation}
\label{good}
\P(\ya\in A_l)\le e^{\gd e^D}\P(\xia\in A_l)+o(1).
\end{equation}
As usual, we suppose that this is not true and derive a contradiction.
Thus, assume that there is an $\eps>0$ and, for each $n$ in some
subsequence, a good interval $\jl=[a_l,b_l]$ (depending on $n$) such that
\begin{equation}
  \label{mk4}
\P(\ya\in A_l)> e^{\gd e^D}\P(\xia\in A_l)+\eps.
\end{equation}
If $\jl$ is fat, then \eqref{mk4} contradicts \eqref{fat} for large $n$, so
we may assume that  $\jl$ is thin, \ie, $\P(\xia\in\jl)\le\gd$.

Let $\acl\=\jl\setminus\al$ and $\bl\=[b_l+1,\infty)$.
Let $\ga_n\=n\P(\xi\in \al)$,
$\gb_n\=n\P(\xi\in \bl)$ and
$\gam_n\=n\P(\xi\in \acl)$.
The assumption that $\jl$ is good implies that 
$\ga_n+\gb_n+\gam_n=n\P(\xi\ge a_l)\le D$. By
selecting a subsubsequence we may assume that $\ga_n\to\ga$, $\gb_n\to\gb$
and $\gam_n\to\gam$ for some real $\ga,\gb,\gam$ with
$\ga+\gb+\gam\le D$. 
Then \ref{tdxkn} shows that $N_\bl\dto\Po(\gb)$; moreover, the proof extends
easily (using joint factorial moments) to show that
$N_\al\dto\Po(\ga)$, $N_\bl\dto\Po(\gb)$ and $N_\acl\dto\Po(\gam)$, jointly
and with independent limits.

Similarly, by the method of moments or otherwise (this is a standard Poisson
approximation of a multinomial distribution),
$\nxi\al\dto\Po(\ga)$, $\nxi\bl\dto\Po(\gb)$ and $\nxi\acl\dto\Po(\gam)$,
jointly and with independent limits.

Note that 
\begin{equation*}
\ya\in\al\implies N_\al\ge1 \text{ and } N_\bl=0.  
\end{equation*}
Conversely,
\begin{equation*}
N_\al\ge1 \text{ and } N_\bl=N_\acl=0 \implies \ya\in\al.
\end{equation*}
The corresponding results hold for $\xia$.
Thus,
\begin{equation}
  \label{mk6}
\P(\ya\in \al)
\le\P(N_\al\ge1,\;N_\bl=0)
\to\P\bigpar{\Po(\ga)\ge1}\P\bigpar{\Po(\gb)=0}
\end{equation}
and
  \begin{multline}
  \label{mk6''}
\P(\xia\in \al)
\ge\P(\nxi\al\ge1,\;\nxi\bl=\nxi\acl=0)
\\
\to\P\bigpar{\Po(\ga)\ge1}\P\bigpar{\Po(\gb)=0}\P\bigpar{\Po(\gam)=0}.	
  \end{multline}
Since $\P\bigpar{\Po(\gam)=0}=e^{-\gam}$, \eqref{mk6}--\eqref{mk6''} yield
\begin{equation}
  \label{mk6b}
\P(\ya\in \al)-e^\gam\P(\xia\in\al)\le o(1).
\end{equation}
Moreover, $\nxi\jl=\nxi\al+\nxi\acl\dto\Po(\ga+\gam)$, and thus
\begin{multline}
  \label{mk7}
\P(\xia\in\jl)
=\P(\nxi\jl\ge1,\;\nxi\bl=0)
\ge\P(\nxi\jl=1,\;\nxi\bl=0)
\\
\to(\ga+\gam)e^{-\ga-\gam}e^{-\gb}.
\end{multline}
We are assuming that $\jl$ is thin, \ie, $\P(\xia\in\jl)\le\gd$, and thus
\eqref{mk7} yields
$(\ga+\gam)e^{-\ga-\gam}e^{-\gb}\le\gd$ and consequently
\begin{equation*}
  \gam\le \ga+\gam \le \gd e^{\ga+\gb+\gam}\le \gd e^D.
\end{equation*}
Hence, \eqref{mk6b} implies
\begin{equation*}
\P(\ya\in \al)\le e^{\gd e^D}\P(\xia\in\al)+ o(1),
\end{equation*}
which contradicts \eqref{mk4}. 
This contradiction shows that \eqref{good} holds uniformly
for all good intervals.

It remains to consider the bad intervals.

Let $\jll=[a_\ell,b_\ell]$ be the rightmost bad interval.
If $\jll$ is fat we use \eqref{fat} and if $\jll$ is thin we use \eqref{jl}
which gives
\begin{equation*}
  \P(\ya \in\all) \le   \P(\ya \in\jll)
\le
  \P(\xia \in\jll)+o(1)
\le\gd+o(1).
\end{equation*}
In both cases,
\begin{equation}\label{mk8a}
  \P(\ya \in\all) 
\le
  \P(\xia \in\all)+\gd+o(1).
\end{equation}

Finally, let $\Ax$ be the union of the remaining bad intervals. Then
$\Ax=[0,a_\ell-1]$ and by \ref{tdxk},
\begin{equation}
  \label{mk8b}
\P(\ya\in\Ax)
=\P(\ya<a_\ell)
\le \P(\xia<a_\ell)+o(1).
\end{equation}
Furthermore, recalling $n\P(\xi\ge a_\ell)>D$ since $\jll$ is bad, 
\begin{equation}
  \label{mk8c}
\P(\xia<a_\ell)
=\P(\nxi{[a_\ell,\infty)}=0)
=\bigpar{1-\P(\xi\ge a_\ell)}^n
\le e^{-n\P(\xi\ge a_\ell)}
\le e^{-D}.
\end{equation}

We obtain by summing \eqref{good} for all good intervals together with
\eqref{mk8a} and \eqref{mk8b},
recalling that the number of 
intervals is bounded (for a fixed $\gd$) and using \eqref{mk8c}, 
\begin{equation*}
  \begin{split}
	\P(\ya\in A)
&=
\sum_l\P(\ya\in\al)
\\&
\le e^{\gd e^D} \sum_l\P(\xia\in\al)+o(1)+\gd+\P(\xia<a_\ell)
\\&
\le
e^{\gd e^D} \P(\xia\in A)+o(1)+\gd+e^{-D}.
  \end{split}
\end{equation*}
Consequently,
\begin{equation*}
  \begin{split}
\dtv(\ya,\xia)
&=	\P(\ya\in A)-\P(\xia\in A)
\\&
\le
\bigpar{e^{\gd e^D}-1} \P(\xia\in A)+\gd+e^{-D}+o(1)
\\&
\le
\bigpar{e^{\gd e^D}-1} +\gd+e^{-D}+o(1),
  \end{split}
\end{equation*}
and thus
\begin{equation}
\limsup_\ntoo\dtv(\ya,\xia)
\le
\bigpar{e^{\gd e^D}-1}+\gd+e^{-D}.
\end{equation}

Letting first $\gd\to0$ and then $D\to\infty$, we obtain
$\dtv(\ya,\xia)\to0$, which proves \ref{tdxtv}.

This completes the proof of the version with $\xi$ and the assumption $m=\gl
n+o(n\qq)$. 
Now remove this assumption, but assume $\gl<\nu$ and thus $\tau<\rho$.
We consider only $n$ with $0<m/n<\nu$ and thus $0<\taun<\rho$.
Denote the distribution \eqref{xinn} of $\xi\nn$ by $\wwx\nn$ 
(this is a \pws{} equivalent to $\wwx$)
and let $S_n\nn\=  \xi\nn_1+\dots+\xi\nn_n$. 
Then, by \refE{EBprob} applied to $\wwx\nn$, 
in analogy with \eqref{l41a} (and equivalent to it by \eqref{lb1z}),
\begin{equation}
  \label{l41x}
\P(Y_1=k)
=\frac{w_k\nn Z(m-k,n-1;\wwx\nn)}{Z(m,n;\wwx\nn)}
=\frac{Z(m-k,n-1;\wwx\nn)}{Z(m,n;\wwx\nn)}\P(\xi\nn=k).
\end{equation}

Furthermore, for any $y\ge0$,
using \eqref{xia},
\begin{equation}\label{l7a}
  \begin{split}
  \P(\ya\ge y)
\le
n\P(Y_1\ge y)
&=
n\P\bigpar{\xi\nn_1\ge y\bigm|S_n\nn=m}
\\&
\le
n\P\bigpar{\xi\nn_1\ge y}\P\bigpar{S_n\nn=m}\qw
\\& \le n\P(\xi\nn\ge y)\cdot O(n\qq).	
  \end{split}
\end{equation}
Choose $\taux\in(\tau,\rho)$. Then, for $s>0$ and
$n$ so large than $\taun<\taux$,
by \eqref{lep5},
\begin{equation}\label{l77}
  \P(\xi\nn\ge y)
\le e^{-sy}  \frac{\Phi(e^s\taun)}{\Phi(\taun)}
\le
 e^{-sy}  \frac{\Phi(e^s\taux)}{\Phi(0)}.
\end{equation}
Choosing $s>0$ with $e^s<\rho/\taux$, we thus find 
$  \P(\xi\nn\ge y)=O(e^{-sy})$ and, 
by \eqref{l7a},
\begin{equation*}
  \P(\ya\ge y) =O\bigpar{n\qqc e^{-sy}}.
\end{equation*}
We now define $\bn\=2 s\qw\log n$, and obtain
\begin{equation}
  \P(\ya\ge \bn) =O\bigpar{n\qqc e^{-s\bn}}=O\bigpar{n\qqw}
\to0.
\end{equation}
Hence, $\ya<\bn$ \whp.
Similarly, using \eqref{l77} again,
\begin{equation}
\P(\xi\nn\ge\bn)=o(n\qw)
\end{equation}
and thus 
$
 \P(\xia\nn\ge \bn) \le n\P(\xi\nn\ge\bn)\to0$,
so $\xia<\bn$ \whp. 

We have shown that \eqref{xibn} (with $\xi\nn$) and \eqref{ybn} hold.
Moreover, \refL{LCLT} yields, see \eqref{xia} again, 
$Z(m,n;\wwx\nn)\sim d/(2\pi\gss n)\qq$, and 
for $k\le\bn=O(\log n)$,
the same
argument yields also, using \refR{RCLT},
$Z(m-k,n-1;\wwx\nn)\sim d/(2\pi\gss n)\qq$, because
$m-k-(n-1)\E\xi\nn=m-k-(n-1)m/n=-k+m/n=o(n\qq)$. 
Consequently, \eqref{l41x} yields
\begin{equation}
  \P(Y_1=k)=\etto\P(\xi\nn=k),
\end{equation}
uniformly for $k\le\bn$.

We can now argue exactly as above, using $\xi\nn$, $\xi\win j\nn$ and $\nxin
A$, which proves this version of the theorem.

Finally, if $\gl<\nu$ and $m/n=\gl+o(1/\log n)$, then
$\taun\=\Psi\qw(m/n)=\tau+o(1/\log n)$, because $\Psi\qw$ is differentiable
on $(0,\nu)$.
Since we are assuming $\tau=1$ and $\Phi(\tau)=1$ in the proof, we thus have,
uniformly for all $k\le\bn=O(\log n)$,
\begin{equation}
  \P(\xi\nn=k)=\frac{\taun^k}{\Phi(\taun)}w_k
=\etto w_k
=\etto\P(\xi=k).
\end{equation}
Since also $\P(\xi\nn\ge\bn)=o(n\qw)$ and $\P(\xi\ge\bn)=o(n\qw)$, it
follows that
$n\P(\xi\nn\ge h(n))\to\ga \iff n\P(\xi\ge h(n))\to\ga$,
and thus we may in \ref{tdxkn}--\ref{tdxoo} replace $\xi\nn$ by $\xi$
again. Finally, \ref{tdxkk}--\ref{tdxtv} follow as above in this case too.
\end{proof}

\begin{proof}[Proof of \refT{TD1}]
Recall that $\gl<\nu\iff \tau<\rho$ by \refL{LPsi}.
We have $\nu>\gl>0$, so $\rho>0$, and $\tau>0$.
Thus $1<\rho/\tau\le\infty$.

\pfitemref{td1b}
Fix $a>1/\log(\rho/\tau)$. Choose $b$ with $e^{1/a}<b<\rho/\tau$. Then
$0\le1/\rho<(b\tau)\qw$. Choose $c$ with $1/\rho<c<(b\tau)\qw$.

Since $\limsup_\ktoo w_k^{1/k}=1/\rho<c$, we have $w_k^{1/k}<c$ for large
$k$, and then,
defining
$\taun$ and $\xi\nn$ by \eqref{taunus}--\eqref{xinn},
\begin{equation}
  \label{v1}
\P(\xi\nn=k)=\frac{\taun^k}{\Phi(\taun)}w_k
\le \frac{(c\taun)^k}{\Phi(0)}.
\end{equation}
As \ntoo, $c\taun\to c\tau<b\qw$.
Let $h\=\floor{a\log n}$. For large $n$, \eqref{v1} applies for $k\ge h$,
and $c\taun<b\qw<1$, and then
\begin{equation*}
  \P(\xi\nn\ge h) 
\le \sum_{k=h}^\infty\frac{(c\taun)^k}{\Phi(0)}
\le \sum_{k=h}^\infty\frac{b^{-k}}{w_0}
=O\bigpar{b^{-h}}
=O\bigpar{n^{-a\log b}}
\end{equation*}
Since $a\log b>1$, thus $n\P(\xi\nn\ge h)\to0$, and \refT{TDx}\ref{tdx0}
yields $\ya\le h\le a\log n$ \whp.

\pfitemref{td1c}
If $\rho=\infty$, then 
\ref{td1b} applies with 
$\rho/\tau=\infty$ and thus
$1/\log(\rho/\tau)=0$.

\pfitemref{td1d}
If $\rho=\infty$, the result follows by \ref{td1c}, so we may assume
$1=\tau<\rho<\infty$. 
Let $a\=1/\log(\rho/\tau)$ and $0<\eps<1$.
The upper bound $\yj\le\ya\le(a+\eps)\log n$ \whp{} follows from \ref{td1b}, 
and it remains to find a matching lower bound.

Let 
$k\=\ceil{(1-\eps)a\log n}$.
Then, since $\taun\to\tau$,
\begin{equation*}
  \begin{split}
\log\P(\xi\nn=k)&=
  \log w_k +k\log\taun-\log\Phi(\taun) 
\\&
= -k(\log\rho+o(1))+k(\log\tau+o(1))+O(1)
\\&
=-k\log(\rho/\tau)+o(k)
=-(1-\eps+o(1))\log n	
  \end{split}
\end{equation*}
and thus
\begin{equation*}
n\P(\xi\nn\ge k) 
\ge n\P(\xi\nn= k)
=n^{\eps+o(1)}\to\infty.
\end{equation*}
By \refT{TDx}\ref{tdxoo} (and the last sentence in \refT{TDx}), 
this implies \whp{}
\begin{equation*}
\yj\ge k \ge (1-\eps)a\log n
\end{equation*}
This completes the proof, since we can take $\eps$ arbitrarily small.
\end{proof}

Specialising \refT{TDx} to the tree case ($m=n-1$), we obtain the following.
(Recall that $\gss<\infty$ is automatic when $\nu>1$.)
\begin{corollary}\label{CDx}
Let $\wwx=(w_k)_{k\ge0}$ be a weight sequence with $w_0>0$ and $w_k>0$
  for some $k\ge2$, and 
let $\xi$ have the distribution given by $\ppi$ in \eqref{pk}.
Suppose that $\nu\ge1$ and $\gss\=\Var\xi<\infty$.
Then, as \ntoo, for the largest degrees $\ya\ge\yb\ge\dots$ in $\ctn$,
$\dtv(\ya,\xi\win1)\to0$ and, for every fixed $j$,
$\dk(Y\win j,\xi\win j)\to0$.
\end{corollary}
\begin{proof}
The case $\go=\infty$ is a special case of \refT{TDx}, with $\gl=1$.

The case $\go<\infty$ is trivial:
for every fixed $j$,
$\yj=\go$ \whp{} by \refT{Tgo}, and, trivially, $\xi\win j=\go$ \whp.
\end{proof}

The comparison with $\xij$ in \refT{TDx} and \refC{CDx} is appealing since
$\xij$ is the $j$:th largest of $n$ \iid{} random variables. For
applications it is often convenient to modify this a little by taking a
Poisson number of variables instead.

Consider an infinite \iid{} sequence $\xi_1,\xi_2,\dots$, let as above
$\xij$ be the $j$:th largest among the first $n$ elements of the sequence
and define $\txij$ as the $j$:th largest among the first $N(n)$ elements
$\xi_1,\dots,\xi_{N(n)}$, where $N(n)\sim\Po(n)$ is a random Poisson
variable independent of $\xi_1,\xi_2,\dots$.

\begin{lemma}
  \label{Ltxi}
W.h.p.\ $\txij=\xij$ and thus $\dtv(\txij,\xij)\to0$ as \ntoo{}
for every fixed $j\ge1$.
\end{lemma}

\begin{proof}
  Let $n_\pm\=\floor{n\pm n\qqqb}$, and let 
$\xijm$ be the $j$:th largest of $\xi_1,\dots,\xi_{n_-}$.
By symmetry, the positions of the $j$ largest among $\xin$ are uniformly
random (we resolve any ties in the ordering at random); thus the probability
that one of 
them  has index $>n_-$ is at most 
$j(n-n_-)/n=o(1)$.
Hence, \whp{} all $j$ are among  $\xi_1,\dots,\xi_{n_-}$, and then
$\xij=\xijm$.

Furthermore, \whp{} $n_-\le N(n)\le n_+$, and a similar argument (using
conditioning on $N(n)$) shows that \whp{} $\txij=\xijm$.
Hence, \whp{} $\xij=\xijm=\txij$.
Now use \refL{LM}\ref{lmd}.
\end{proof}

We can thus replace $\xij$ by $\txij$ in \refT{TDx} and \refC{CDx}.
(We can similarly replace $\xij\nn$ by $\txij\nn$ defined in the same way.)
The advantage is that, by standard properties of the Poisson distribution,
the corresponding counting variables 
\begin{equation*}
  \tN_k\=|\set{i\le N(n):\xi=k}|
\end{equation*}
are \emph{independent} Poisson variables with $\tN_k\sim\Po(n\P(\xi=k))$.
We similarly define
$\tN_{[k,\infty)}\=\sum_{l=k}^\infty\tN_l\sim\Po\bigpar{n\P(\xi\ge k)}$. 

\begin{remark}
  An equivalent way to express this is that the multiset
  $\Xi_n\=\set{\xi_i:i\le N(n)}$ is a Poisson process on $\No$ with
  intensity measure $\gL_n$ given by $\gL_n\set k=n\P(\xi=k)$.
\end{remark}

We thus have (exactly), for any $j$ and $k$,
\begin{equation}\label{txij}
  \P(\txij \le k)
=\P\bigpar{\tN_{[k+1,\infty)}<j}
=\P\bigpar{\Po(n\P(\xi>k))<j};
\end{equation}
in particular
\begin{equation}\label{txi1}
  \P(\txia\le k)=e^{-n\P(\xi>k)}.
\end{equation}
This gives the following special case of \refT{TDx}.  
(There is a similar version with $\xi\nn$.)

\begin{corollary}
  \label{CDx2}
Suppose that $w_0>0$ and $\go=\infty$.
Suppose further that $n\to\infty$ and $m=m(n)$ with 
$m=\gl n+o(\sqrt n)$  
where $0<\gl\le\nu$, and 
that either $\gl<\nu$ or $\gss\=\Var\xi<\infty$.
Then, uniformly in all $k\ge0$,
\begin{equation}\label{cdx1}
  \P(\yj \le k)
=
\P\bigpar{\Po(n\P(\xi>k))<j}+o(1)
\end{equation}
for each fixed $j\ge1$;
in particular
\begin{equation}\label{cdx2}
  \P(\ya\le k)=e^{-n\P(\xi>k)}+o(1).
\end{equation}
\end{corollary}
\begin{proof}
  Immediate by \refT{TDx}\ref{tdxk},
\refL{Ltxi} and \eqref{txij}--\eqref{txi1}.
\end{proof}

\begin{remark}
Since $\tN_{[h(n),\infty)}\ge j\iff \txij\ge h(n)$, it follows easily from
Lemmas  \ref{Ltxi} and \ref{LM}\ref{lmkk} that for any sequence $h(n)$,
\begin{equation}
\dkk\bigpar{\nxi{[h(n),\infty)},\tN_{[h(n),\infty)}}\to0.  
\end{equation}
Hence, \refT{TDx}\ref{tdxkk} is equivalent to
$\dkk\bigpar{\ny{[h(n),\infty)},\tN_{[h(n),\infty)}}\to0$, 
and thus
\begin{equation}
\dkk\Bigpar{\ny{[h(n),\infty)},\Po\bigpar{n\P(\xi\ge h(n))}
}
\to0.
\end{equation}
This is another, essentially equivalent, way to express the results above.
\end{remark}

\subsection{The subcase $\gl<\nu$}

When $\gl<\nu$, we have $\tau<\rho$ and the random variable $\xi$ has some
finite exponential moment, \cf{} \refS{S3}; hence the probabilities $\pi_k$
decrease rapidly. 
\refT{TDx} and \refC{CDx2} show that $\ya$ (and each $\yj$) has its
distribution concentrated on $k$ such that $\P(\xi\ge k)$ is of the order $1/n$.
If the decrease of $\pi_k$ is not too irregular, this implies 
strong concentration of $\ya$, with, rougly speaking, $\ya\approx k$ when 
$\P(\xi\ge k)\approx 1/n$.
To make this precise, we define three versions of a suitable such estimate
$k=k(n)$. Let, as above, $\pi_k=\P(\xi=k)=\tau^k w_k/\Phi(\tau)$ and let
\begin{equation}\label{Pik}
\Pi_k\=\P(\xi\ge k)=\sum_{l=k}^\infty \pi_l.
\end{equation}
Define
\begin{align}
k_1(n)&\=\max\set{k:\pi_k\ge 1/n}, \label{k1}\\ 
k_2(n)&\=\max\set{k:\Pi_k\ge 1/n}, \label{k2}\\ 
k_3(n)&\=\max\set{k:\sqrt{\Pi_k\Pi_{k+1}}\ge 1/n}. \label{k3}
\end{align}
Note that $k_1(n)\le k_2(n)$ and $k_2(n)-1\le k_3(n)\le k_2(n)$.

We consider 
the typical case when $w_{k+1}/w_k$ converges as \ktoo.
We assume implicitly that $w_{k+1}/w_k$ is defined for all large $k$; thus
$w_k>0$ and $\go=\infty$. 
If $w_{k+1}/w_k\to a$ as \ktoo,
then \eqref{rho} yields $\rho=1/a$; hence $\rho=\infty$ if $a=0$ and
$0<\rho<\infty$ if $a>0$.

\begin{theorem}\label{TDa}
Suppose that $w_0>0$ and that $w_{k+1}/w_k\to a<\infty$ as \ktoo.
Suppose further that $n\to\infty$ and $m=m(n)$ with 
$m=\gl n+o(\sqrt n)$  
where $0<\gl<\nu$.
\begin{romenumerate}
\item 
Then, for each $j\ge1$, 
\begin{equation*}
\yj=k_1(n)+\Op(1)=k_2(n)+\Op(1)=k_3(n)+\Op(1).
\end{equation*}
\item 
If $a=0$, then, moreover, \whp,
\begin{align*}
|\yj-k_1(n)|\le1,&&
|\yj-k_2(n)|\le1,&&
\yj\in\set{k_3(n),k_3(n)+1}.
\end{align*}
\end{romenumerate}
\end{theorem}

\begin{proof}
\pfitem{i}
  We have, as  said above, $\rho=1/a>0$.
Furthermore, since $\gl<\nu$, we have $\tau<\rho$ and thus, as \ktoo,
\begin{equation}\label{juh}
\frac{\pi_{k+1}}{\pi_k}  
=\tau \frac{w_{k+1}}{w_k}  \to\tau a = \frac{\tau}{\rho} <1.
\end{equation}
It follows from \eqref{juh} and \eqref{Pik}, using dominated convergence, that,
as \ktoo, 
\begin{equation}\label{ul}
\frac{\Pi_{k}}{\pi_k}  
=\sum_{i=0}^\infty \frac{\pi_{k+i}}{\pi_k}
\to \sum_{i=0}^\infty (\tau a)^i
=\frac{1}{1-\tau a}.
\end{equation}
If $\ell$ is chosen such that $(\tau a)^\ell<1-\tau a$, then
\eqref{ul} and \eqref{juh} imply $\Pi_{k+\ell}/\pi_k\to(\tau a)^\ell/(1-\tau
a)<1$ as \ktoo, and thus, for large $k$,
$\Pi_{k+\ell}<\pi_k<\Pi_k$; hence, for large $n$,
$k_1(n)\le k_2(n)\le k_1(n)+\ell$. Thus, recalling that $|k_2(n)-k_3(n)|\le1$,
\begin{equation}\label{k123}
  k_1(n)=k_2(n)+O(1)=k_3(n)+O(1).
\end{equation}
Furthermore, \eqref{juh} and \eqref{ul} yield also
\begin{equation}\label{palm}
\frac{\Pi_{k+1}}{\Pi_k}  
\to\tau a <1.
\end{equation}
By \eqref{k2}, $n\Pi_{k_2(n)}\ge 1 >n\Pi_{k_2(n)+1}$. This and 
\eqref{palm} imply that if $\gO(n)$ is any sequence with $\gO(n)\to\infty$,
then $n\Pi_{k_2(n)-\gO(n)}\to\infty$ and $n\Pi_{k_2(n)+\gO(n)}\to0$. 
Consequently, recalling the definition \eqref{Pik},
by \refT{TDx}\ref{tdx0}--\ref{tdxoo} (or by \refC{CDx2})
\whp{} $\yj\ge k_2(n)-\gO(n)$ and  $\yj < k_2(n)+\gO(n)$. Since
$\gO(n)\to\infty$ is arbitrary, this yields $\yj=k_2(n)+\Op(1)$.
(See \eg{} \cite{SJN6}.)
The result follows by \eqref{k123}.

\pfitem{ii}
When $a=0$, \eqref{ul} yields $\Pi_k\sim\pi_k$, 
 \eqref{juh} yields $\pi_{k+1}/\pi_k\to0$ 
and \eqref{palm} yields $\Pi_{k+1}/\Pi_k\to0$ 
as \ktoo.
It follows easily from \eqref{k1}--\eqref{k3} that
$n\Pi_{k_1(n)-1}\to\infty$, $n\Pi_{k_1(n)+2}\to0$,
$n\Pi_{k_2(n)-1}\to\infty$, $n\Pi_{k_2(n)+2}\to0$,
$n\Pi_{k_3(n)}\to\infty$, $n\Pi_{k_3(n)+2}\to0$,
and the results follow by
\refT{TDx}\ref{tdx0}--\ref{tdxoo}.
\end{proof}

If $a=0$, \ie{} $w_{k+1}/w_k\to0$ as \ktoo, thus $\ya$ is asymptotically
concentrated at 
one or two values.
(This was shown, in the tree case, by 
\citet{MM92}, after showing concentration to at most three values in
\cite{MM91}; see also \citet{KolchinSCh}, \citet{Kolchin} 
and \citet{CarrGS} for special cases.) 
If $a>0$, we still have a strong concentration, but not to any finite number
of values as is seen by \refT{Ta} below.

We consider two important examples, where we apply this to random trees,
so $m=n-1$ and $\gl=1$. (Recall
that $\ya$ then is the largest outdegree in $\ctn$. The largest degree is
\whp{} $\ya+1$, since \whp{} it is not attained at the root,
\eg{} because the root degree is $\Op(1)$ by \refT{Troot}; this should be
kept in mind when comparing with results in other papers.)

\begin{example}\label{EUmax}
  For uniform random labelled ordered rooted trees, we have by \refE{Euniform}
$\xi\sim\Ge(1/2)$ with $\pi_k=2^{-k-1}$ and thus
$\P(\xi\ge k)=2^{-k}$.
Hence $\ya$ has asymptotically the same distribution as the maximum of
$n$ \iid{} geometrically distributed random variables, which is a simple and
well-studied example, see \eg{} \citet{LLR}. Explicitly, 
\refC{CDx2} applies and \eqref{cdx2} yields, uniformly in $k\ge0$,
\begin{equation}\label{ku1}
  \P(\ya\le k)=e^{-n2^{-k-1}}+o(1).
\end{equation}
(This was, essentially, shown by \citet{MM91}.)

One way to express this is to introduce a random variable $W$ with the
Gumbel distribution
\begin{equation}\label{gumbel}
  \P(W\le x)=e^{-e^{-x}},
\qquad -\infty<x<\infty.
\end{equation}
Then \eqref{ku1} yields, uniformly for $k\in\bbZ$,
\begin{equation}\label{ku2}
  \begin{split}
\P(\ya\le k)	
&=\P\bigpar{W<(k+1)\log 2-\log n}+o(1)
\\&
=\P\Bigpar{\frac{W+\log n}{\log 2}<k+1}+o(1)
\\&
=\P\Bigpar{\lrfloor{\frac{W+\log n}{\log 2}}\le k}+o(1)
.
  \end{split}
\end{equation}
In other words,
extending $\dk$ to $\bbZ$-valued random variables,
\begin{equation}
  \dk\bigpar{\ya,\floor{(W+\log n)/\log 2}}\to0.
\end{equation}
Thus, the maximum degree $\ya$ can be approximated (in distribution) by 
$\floor{(W+\log n)/\log 2}=\floor{W/\log 2+\log_2n}$.
Hence $\ya-\log_2n$ is tight but no asymptotic distribution exists;
$\ya-\log_2n$ can be approximated by 
$\floor{W/\log 2+\log_2n}-\log_2n=
\floor{W/\log 2+\frax{\log_2n}}-\frax{\log_2n}$
(where we let $\frax x\=x-\floor x$ denote the fractional part of $x$),
which shows convergence in
distribution for any subsequence such that $\frax{\log_2n}$ converges to
some $\ga\in\oi$, but
the limit depends on $\ga$. 
See further \citetq{in particular Lemma 4.1 and Example 4.3}{SJ175}.

In the same way we see that $\yj$ can be approximated in distribution by 
$\floor{W_j/\log2+\log_2n}$ where $W_j$ has the distribution
\begin{equation}
  \P(W_j\le x)
=\P(\Po(e^{-x})<j)
=\sum_{i=0}^{j-1}\frac{e^{-ix}}{i!}e^{-e^{-x}},
\qquad -\infty<x<\infty,
\end{equation}
with density function $e^{-jx}e^{-e^{-x}}/(j-1)!$;
further $W_j\eqd-\log V_j$, where $V_j$ has the Gamma distribution 
$\mathrm{Gamma}(j,1)$.
(Cf.\ \citetq{Section 2.2}{LLR} for the relation between the distributions of
$\xi\win j$ and $\xia$ in the \iid{} case.)
\end{example}

\begin{example}\label{EPomax}
  For uniform random labelled unordered rooted trees, we have by \refE{Ecayley}
$\xi\sim\Po(1)$ with $\pi_k=e\qw/k!$.
We have $w_{k+1}/w_k\to0$, so
\refT{TDa}(ii) applies and shows that $\ya$ is concentrated on at most two
values, as proved by \citetq{Theorem 2.5.2}{Kolchin};
see also \citet{MM92} and \citet{CarrGS}. 

Explicitly, \eqref{cdx2} yields
(treating the rather trivial case $n>k\qq\cdot k!$ separately)
\begin{equation}
  \P(\ya<k)=e^{-ne\qw/k!(1+O(1/k))}+o(1)
=e^{-ne\qw/k!}+o(1)
\end{equation}
which by Stirling's formula yields
\begin{equation}\label{pox}
  \P(\ya<k)
=\exp\bigpar{-e^{\log n-(k+\frac12)\log k +k -\log(e\sqrt{2\pi})}}+o(1)
\end{equation}
uniformly in $k\ge1$, \cf{} \citet{CarrGS}.
It follows easily from Stirling's formula, or from \eqref{pox}, that
$k_1(n),k_2(n),k_3(n)\sim\log n/\log\log n$,
and more precise asymptotics can be found too;
\cf{}
\cite{Moon68},
\cite{MM91},
\cite{CarrGS}.
\end{example}

In fact, the simple \refE{EUmax} is typical for the case $w_{k+1}/w_k\to a>0$
as \ktoo; then $\ya$ always has asymptotically the same distribution as the
maximum of 
\iid{} geometric random variables, provided we adjust the number of these
variables according to $\wwx$. We state some versions of this in the next
theorem. For simplicity we consider only the maximum $\ya$, and leave the
extensions to $\yj$ for general fixed $j$ to the reader.

\begin{theorem}
  \label{Ta}  
Suppose that $w_0>0$ and that $w_{k+1}/w_k\to a$ as \ktoo, with $0<a<\infty$.
Suppose further that $n\to\infty$ and $m=m(n)$ with 
$m=\gl n+o(\sqrt n)$  
where $0<\gl<\nu$.

Let $q\=\tau a=\tau/\rho<1$.
Let $k(n)$ be any sequence such that $\pi_{k(n)}=\Theta(1/n)$; equivalently,
$k(n)=k_1(n)+O(1)$, and let $N=N(n)$ be integers such that
\begin{equation}\label{nq}
  N\sim 
\frac{n\pi_{k(n)}q^{-k(n)}}{1-q}
=
\frac{nw_{k(n)}a^{-k(n)}}{\Phi(\tau)(1-q)}.
\end{equation}
\begin{romenumerate}[-10pt]
\item 
Let $\eta_1,\dots,\eta_N$ be \iid{}  random variables with a geometric
distribution $\Ge(1-q)$, \ie,
$\P(\eta_i=k)=(1-q)q^{-k}$, $k\ge0$. Then
\begin{equation}\label{ta1}
\ya\dapprox\max_{i\le N}\eta_i.
\end{equation}
\item 
Let $W$ have the Gumbel distribution \eqref{gumbel}. Then
\begin{equation}\label{ta2}
\ya\dapprox\floor{W/\log(1/q)+\log_{1/q}N}.
\end{equation}
\item 
Let $b_n\=n\pi_{k(n)}$; thus $b_n=\Theta(1)$.
Then
\begin{equation}\label{ta3}
\ya-k(n)\dapprox\lrfloor{\bigpar{W+\log(b_n/(1-q))}/\log(1/q)}.
\end{equation}
Thus $\ya-k(n)$ is tight, and converges for every
subsequence such that $b_n$ converges.
\end{romenumerate}
\end{theorem}

Hence $\ya-k(n)$  converges for every
subsequence such that $b_n$ converges, but the limit depends on the
subsequence so $\ya-k(n)$ does not have a limit distribution.
(For the distributions that appear as subsequence limits, see 
\citetq{Examples  4.3 and   2.7}{SJ175}.)
Note that necessarily $k(n)\to\infty$ and thus $N\to\infty$ as \ntoo.

We show first a simple lemma, similar to \refL{LM}.

\begin{lemma}\label{LM2}
  Let $X_n$ and $X'_n$ be integer-valued \rv{s} and suppose that
  there exists a sequence of integers $k(n)$ such that $X_n-k(n)$ is tight.
(Equivalently: $X_n=k(n)+\Op(1)$.) Then the following are equivalent:
  \begin{romenumerate}
  \item $\P(X_n\le k(n)+\ell)-\P(X'_n\le k(n)+\ell)\to0$
{for each fixed $\ell\in\bbZ$};
  \item $\dk(X_n,X'_n)\to0$;
  \item $X_n\dapprox X'_n$, \ie, $\dtv(X_n,X'_n)\to0$.
  \end{romenumerate}
  \begin{equation*}
  \end{equation*}
\end{lemma}

\begin{proof}
  By considering $X_n-k(n)$ and $X'_n-k(n)$ we may assume that $k(n)=0$.
Let $\eps>0$. Since $X_n$ is tight, there exists $L$ such that
$\P(|X_n|>L)<\eps$ for every $n$. Suppose that (i) holds. Then
\begin{equation*}
  \begin{split}
  \dtv(X_n,X_n')
&=\sum_{\ell=-\infty}^\infty\bigpar{\P(X_n=\ell)-\P(X'_n=\ell)}_+
\\&
\le\sum_{\ell=-L}^L\bigpar{\P(X_n=\ell)-\P(X'_n=\ell)}_++\P(|X_n|>L)	
\le o(1)+\eps.
  \end{split}
\end{equation*}
This shows (iii). The implications (iii)$\implies$(ii) and (ii)$\implies$(i)
are trivial.
\end{proof}

\begin{proof}[Proof of \refT{Ta}] 
By \eqref{juh}, $\pi_{k+1}/\pi_k\to q$ as \ktoo, and it follows 
from \eqref{k1} that
$n\pi_{k_1(n)}\in[1,q\qw+o(1)]$.
It follows further that $\pi_{k(n)}=\Theta(1/n)\iff k(n)=k_1(n)+O(1)$, as
asserted, and then
$\pi_{k(n)}q^{-k(n)}\sim \pi_{k_1(n)}q^{-k_1(n)}$; thus we may replace
$k(n)$ by $k_1(n)$ in \eqref{nq}.

\pfitem{i}
For each fixed $\ell\in\bbZ$, by \eqref{ul}, \eqref{juh} and \eqref{nq},
\begin{equation}
  \begin{split}
n\P(\xi\ge k(n)+\ell)
&=n\Pi_{k(n)+\ell}	
\sim n\pi_{k(n)+\ell}/(1-q)	
\sim n\pi_{k(n)}q^{\ell}/(1-q)	
\hskip-3em
\\&
\sim Nq^{k(n)+\ell}
=N\P(\eta_1\ge k(n)+\ell);
  \end{split}
\raisetag\baselineskip
\end{equation}
furthermore, this is $\Theta(1)$.
Hence, \eqref{cdx2} yields
\begin{equation*}
  \begin{split}
  \P\bigpar{\ya<k(n)+\ell}
&=e^{-n\P(\xi\ge k(n)+\ell)}+o(1)
=e^{-N\P(\eta_1\ge k(n)+\ell)}+o(1)\hskip-3em
\\&
=\bigpar{1-\P(\eta_1\ge k(n)+\ell)}^N+o(1)		
\\&
=\P\bigpar{\max_{i\le N}\eta_i<k(n)+\ell}+o(1),
  \end{split}
\raisetag\baselineskip
\end{equation*}
and \eqref{ta1} follows by \refL{LM2}, since $\ya-k_1(n)$ is tight by
\refT{TDa}. 

\pfitem{ii}
As in \eqref{ku2}, uniformly in $k\in\bbZ$,
\begin{equation}\label{glunt}
  \begin{split}
\P\Bigpar{\max_{i\le N}\eta_i <k}
&=\bigpar{1-q^k}^N
=e^{-Nq^k}+o(1)\\
&=\P\bigpar{W<k\log(1/q)-\log N}+o(1)
\\&
=\P\lrpar{\lrfloor{\frac{W+\log N}{\log (1/q)}}< k}+o(1)
.
  \end{split}
\end{equation}
Hence,
 $ \dtv\bigpar{\maxeta,\floor{W/\log(1/q)+\log_{1/q}N}}\to0$,
and \eqref{ta2} follows from \eqref{ta1} and \refL{LM2}.
 
\pfitem{iii}
By \eqref{nq}, $\logq N=k(n)+\logq(b_n/(1-q))+o(1)$, and \eqref{ta3} follows
easily from \eqref{ta2}, using \refL{LM2} and the fact that $W$ is
absolutely continuous. 
\end{proof}

\begin{remark}\label{RTa}
  For later use we note that \refT{Ta}, as other results, extends to the
  case $w_0=0$
by the argument in \refR{Rmin}; we now have to assume
$\gl>\ga\=\min\set{k:w_k>0}$.
The extension of \refT{Ta}(i) is perhaps more subtle that other applications
of this argument since $N$ will change by a factor $\sim q^\ga$,
but (ii) and (iii) are straightforward, and then (i) follows by
\eqref{glunt} and \refL{LM2}.
\end{remark}

If the \ws{} is very irregular, $\ya$ can fail to be concentrated even in
the case $\gl<\nu$.

\begin{example}
  Let $\ell_j\=2^{2^j}$ and $\gS\=\set{\ell_j}_{j\ge1}$.
Let $w_k=1/k^2$ if $k\in\gS$, $w_k=0$ if $k\ge3$ and $k\notin\gS$, and
choose $w_0>0$, $w_1>0$ and $w_2$ such that $\ww$ is a \pws{} with
$\mu\=\sumk kw_k=1$.
Then $\rho=1$ and, by \eqref{nu1}, $\nu=\infty$. Choose $m=n-1$ (the tree
case); thus $\gl=1<\nu$. 

Note that $\ell_{j+1}=\ell_{j}^2$.
If $n=\ell_j$, then 
$\P(\xi\ge\ell_{j})\sim1/\ell_{j}^2=n\qww$, 
$\P(\xi\ge\ell_{j-1})\sim1/\ell_{j-1}^2=n\qw$, and
$\P(\xi\ge\ell_{j-2})\sim1/\ell_{j-2}^2=n\qqw$, and it follows from
\eqref{cdx2} that for $n$ in the subsequence $\gS$,
$\P(\ya<\ell_{j})\to1$, 
$\P(\ya<\ell_{j-1})\to e\qw$ and
$\P(\ya<\ell_{j-2})\to0$.
Hence, along this subsequence,
$\P(\ya=n\qq)\to1-e\qw$ and
$\P(\ya=n\qqqq)\to e\qw$.
\end{example}

\subsection{The subcase $w_{k+1}/w_k\to0$ as \ktoo}

We have seen in \refT{TDa} that when $w_{k+1}/w_k\to0$ as \ktoo, the maximum
$\ya$ is asymptotically concentrated at one or two values.
We shall see that for ``most'' (in a sense specified below) values of $n$,
$\ya$ is 
concentrated at one value, but there are also rather large transition
regions where $\ya$ takes two values with rather large probabilities.

We have, as said before \refT{TDa},
$\go=\infty$ and $\rho=\infty$. Furthermore, by
\refL{LPsi}\ref{L1c}, $\nu=\infty$. 

We define
\begin{equation}
  \label{harad}
n_k\=\floor{1/\pi_k},
\end{equation}
noting that $n_{k+1}/n_k\sim \pi_k/\pi_{k+1}\to\infty$ as \ktoo; in
particular, $n_{k+1}>n_k$ (for large $k$, at least).
The results above then can be stated as follows.

\begin{theorem}
  \label{TDB}
Suppose that $w_0>0$ and that $w_{k+1}/w_k\to 0$ as \ktoo.
Suppose further that $n\to\infty$ and $m=m(n)$ with 
$m=\gl n+o(\sqrt n)$  
where $0<\gl<\infty$.
\begin{romenumerate}
\item 
Consider $n$ in a subsequence such that for some $k(n)$ and some
$x\in\oooy$, $n/n_{k(n)}\to x$. Then
\begin{align*}
  &\P\bigpar{\ya=k(n)-1}\to e^{-x}, \\
  &\P\bigpar{\ya=k(n)}\to 1-e^{-x}.
\end{align*}
\item 
Let $\gO_k\to\infty$ as \ktoo. If \ntoo{} with
$n\notin\bigcup_{k=1}^\infty[\gO_k\qw n_k,\gO_k n_k]$, then, for $k(n)$ such
that $n_{k(n)}<n<n_{k(n)+1}$,
  \begin{align*}
  \P\bigpar{\ya=k(n)}&\to 1.
  \end{align*}
\end{romenumerate}
\end{theorem}

\begin{proof}
  \pfitem{i}
Along the subsequence, using \eqref{Pik}, \eqref{ul} and \eqref{harad},
\begin{equation}
  \label{kee}
n\P(\xi\ge k(n))
=n\Pi_{k(n)}
\sim n\pi_{k(n)}
\sim\frac{n}{n_{k(n)}}
\to x.
\end{equation}
Hence, \eqref{cdx2} yields $\P(\ya\le k(n)-1)\to e^{-x}$.
Furthermore, by \eqref{kee} and \eqref{palm}, 
$n\P(\xi> k(n))\to0$ and 
$n\P(\xi\ge k(n)-1)\to\infty$; hence \eqref{cdx2} yields
$\P(\ya\le k(n))\to1$ and $\P(\ya\le k(n)-2)\to 0$.

\pfitem{ii}
We may assume $\gO_k>1$. Then the assumptions imply 
$\gO_{k(n)}n_{k(n)}<n<\gO_{k(n)+1}\qw n_{k(n)+1}$, where $k(n)\to\infty$ and
thus $\gO_{k(n)}\to\infty$ as \ntoo. Hence, similarly to \eqref{kee},
\begin{align*}
  n\P\bigpar{\xi\ge k(n)}&\sim\frac{n}{n_{k(n)}}>\gO_{k(n)}\to\infty,
\\
  n\P\bigpar{\xi\ge k(n)+1}&\sim\frac{n}{n_{k(n)+1}}<\gO_{k(n)+1}\qw\to0,
\end{align*}
and the result follows by \eqref{cdx2}.
\end{proof}

Roughly speaking, the values of $n$ such that $\ya$ takes two values with
significant probabilities thus form intervals around each $n_k$, of the same
length on a logarithmic scale; between these intervals, $\ya$ is
concentrated at one value. 

\begin{example}\label{EPomax2}
  Consider again uniform random labelled unordered rooted trees, as in
  \refE{EPomax}. We have $n_k=\floor{k!/e}$. In this case, it is simpler to
  redefine $n_k\=k!$; \refT{TDB}(ii) is unaffected but (i) is modified to
\begin{align}
  &\P\bigpar{\ya=k(n)-1}\to e^{-x/e}, \\
  &\P\bigpar{\ya=k(n)}\to 1-e^{-x/e}.
\end{align}
Cf.\ \citet{CarrGS}.
\end{example}

\begin{remark}
  We have for simplicity considered only the maximum value $\ya$ in
  \refT{TDB}. It is easily seen, by minor modifications in the proof, that
for any fixed $j$,  
in (ii) also $\yj=k(n)$ \whp, while
in (i) $\yj\in\set{k(n)-1,k(n)}$ \whp, but the two probabilities have limits
depending on $j$; in fact, the number of $j$ such that $\yj=k(n)$ converges
in distribution to $\Po(x)$. We omit the details.
\end{remark}

To make the statement about ``most'' $n$  precise, 
recall that the \emph{upper} and \emph{lower densities} of a set
$A\subseteq\bbN$ are defined as $\limsup_\ntoo a(n)/n$ and $\liminf_\ntoo
a(n)/n$, where $a(n)\=|\set{i\le n:i\in A}|$;
if they coincide, \ie, if the limit $\lim_\ntoo a(n)/n$ exists, it is called
the \emph{density}.
Similarly, the   
the \emph{logarithmic density} of $A$ is
$\lim_\ntoo\frac1{\log n} \sum_{i\le n,\,i\in A}\frac1i$, when this limit
exists, with
\emph{upper} and \emph{lower logarithmic densities} defined using $\limsup$
and $\liminf$. 
It is easily seen that if a set has a density, then it also has a
logarithmic density, and the two densities coincide. (The converse does not
hold.)
Furthermore, 
define
\begin{equation*}
  \pnx\=\max_k\P(\ya=k).
\end{equation*}
It follows from \refT{TDa} that the second
largest probability $\P(\ya=k)$ is $1-\pnx+o(1)$.
Thus, for $n$ in a subsequence, $\ya$ is asymptotically concentrated at one
value if and only if $\pnx\to1$; 
if $\pnx$ stays away from 1,
$\ya$ takes two values with large probabilities.

\begin{theorem}
  \label{TDC}
Suppose that $w_0>0$ and that $w_{k+1}/w_k\to 0$ as \ktoo.
Suppose further that $n\to\infty$ and $m=m(n)$ with 
$m=\gl n+o(\sqrt n)$  
where $0<\gl<\infty$.
\begin{romenumerate}
\item 
If\/ $\frac12<a<1$, then the set \set{n:\pnx<a} has upper density \\
$\log\frac{a}{1-a}/\log\frac{1}{1-a}>0$ and lower density $0$.
\item 
There exists a subsequence of $n$ with upper density $1$ and logarithmic
density $1$ such that $\pnx\to1$.
\end{romenumerate}
\end{theorem}

Note that the upper density in (i) can be made arbitrarily close to 1 by
taking $a$ close to 1.
This was observed by
\citet{CarrGS} for the case in \refE{EPomax2}.
(However, they failed to
remark that the lower density nevertheless is 0.)

\begin{proof}
  \pfitem{i}
Let $b_1\=-\log a$ and $b_2\=-\log(1-a)$; thus $0<b_1<b_2<\infty$. Then
$\max(e^{-x},1-e^{-x})<a\iff x\in(b_1,b_2)$, and it follows from \refT{TDB}
(and a uniformity in $x$ implicit in the proof) that for any $\eps>0$,
if 
$n\in\bigcup_k[(b_1+\eps)n_k,(b_2-\eps)n_k]$, then  $\pnx<a$ for large $n$,
while if
$n\notin\bigcup_k[(b_1-\eps)n_k,(b_2+\eps)n_k]$, then  $\pnx>a$ for large
$n$.
Since $n_{k+1}/n_k\to0$ as \ktoo, it is easily seen that for any
$b'_1,b'_2$ with $0<b'_1<b'_2<\infty$,
$\bigcup_k[b_1'n_k,b_2'n_k]$ has upper density $(b'_2-b'_1)/b'_2$ and lower
density 0; it follows by taking $b_j'\=b_j\pm\eps$ and letting $\eps\to0$
that the set \set{n:\pnx<a} has upper density $(b_1-b_2)/b_2$ and lower
density 0.

\pfitem{ii}
Let $\gO_k$ be an increasing sequence with $\gO_k\upto\infty$ so slowly that
$\log\gO_k=o(\log(n_{k}/n_{k-1}))$. 
Let $A\=\bigcup_k[\gO_k\qw n_k,\gO_kn_k]$. By \refT{TDB}(ii), $\pnx\to1$ as
\ntoo{} with $n\notin A$, so it suffices to prove that $A$ has lower density
0 and logarithmic density 0.

It is easily seen that for the upper logarithmic density of $A$, 
it suffices to consider $n\in\set{\floor{\gO_kn_k}}$, which gives
\begin{equation*}
  \limsup_{\ktoo}\frac{\sum_{j=1}^k \sum_{i=\gO_j\qw n_j} ^{\gO_j n_j} 1/i}
{\log(\gO_k n_k)}
\le
  \limsup_{\ktoo}\frac{\sum_{j=1}^k \bigpar{2\log\gO_j+O(1)}}
{\sum_{j=1}^k\log (n_j/n_{j-1})}
\to0.
\end{equation*}
Hence the logarithmic density exists and is 0.

The lower density is at most, considering the subsequence $\floor{\gO_k\qw
  n_k}$,
\begin{equation*}
  \liminf_{\ktoo} \frac{a(\gO_k\qw n_k)}{\gO_k\qw n_k}
\le  \liminf_{\ktoo} \frac{\gO_{k-1} n_{k-1}}{\gO_k\qw n_k}
=  \lim_{\ktoo} \frac{\gO_{k-1}\gO_k}{ n_k/n_{k-1}}=0,
\end{equation*}
since $\gO_{k-1}\le\gO_k<(n_k/n_{k-1})\qqq$ for large $k$.
(Alternatively, it is a general fact that the lower density is at most the
(lower) logarithmic density, for any set $A\subseteq\bbN$.)
\end{proof}

\subsection{The subcase $\gl=\nu$ and $\gss=\infty$}

We give two examples of the case $\gl=\nu$ and $\gss=\infty$. (In both
examples, we may assume that $\nu=1$ and $m=n-1$, so the examples apply to
simply generated random trees.)
The first example shows that \refT{TDx} does not always hold if
$\gss=\infty$; the second shows that it sometimes does.

\begin{example}
  \label{EL8} \CCreset 
Let $1<\ga<2$ and let $\ww$ be a \pws{} with $w_0>0$ and
$w_k\sim ck^{-\ga-1}$ as
\ktoo, for some $c>0$.
(This is as in \refE{EBpower} with $\gb=\ga+1\in(2,3)$.
If $\ww$ is not a \pws, we may replace $c$ by $c'\=c/\Phi(1)$.)
We have $\rho=1$, and thus $\nu=\Psi(1)=\sum kw_k<\infty$.
(We may obtain any desired $\nu>0$, for example $\nu=1$, by adjusting the
first few $w_k$.)

We consider the case $m=\nu n+O(1)$; thus $m/n\to\gl=\nu$.
(This includes the tree case $m=n-1$ in the case $\nu=1$. Actually, it
suffices to assume $m=\nu n+o(\nga)$.)
Then $\tau=1=\rho$, and $\pi_k=w_k$.

The random variable $\xi$ thus satisfies $\E\xi=\gl=\nu$. Note that
$\gss\=\Var\xi=\infty$. 
(This is the main reason for taking $1<\ga<2$; if we take $\ga>2$, then
$\gss<\infty$ and \refT{TDx} applies.)
Furthermore,
\begin{equation}\label{tailel8}
\P(\xi\ge k)=\sum_{l=k}^\infty w_l\sim c\ga\qw k^{-\ga}.   
\end{equation}
As in the proof of \refT{THstab}, 
there exists by  \cite[Section XVII.5]{FellerII}
a stable random variable $X_\ga$
(satisfying \eqref{stabchf} and  \eqref{tdzetastabL})
such that
\begin{equation}
  \frac{S_n-n\nu}{n^{1/\ga}}\dto X_\ga;
\end{equation}
moreover, by  \cite[\S\ 50]{GneKol},
the local limit law \eqref{l9} holds 
uniformly for all integers $\ell$.
Note that the density function 
$g$ is bounded and uniformly continuous on $\bbR$,
and that $g(0)>0$ by \eqref{g0}. (In fact, $g(x)>0$ for all $x$. 
See also \cite[Section XVII.6]{FellerII} for an explicit 
formula for $g$ as a power series;
$X_\ga$ is, after rescaling, 
the extreme case $\gam=2-\ga$, in the notation there.)

By \eqref{l41a} and \eqref{l9},
\begin{equation}
  \label{l9a}
  \begin{split}
\P(Y_1=k)
&=\frac{w_k\P(S_{n-1}=m-k)}{\P(S_n=m)}
=w_k\frac{g(-k/\nga)+o(1)}{g(0)+o(1)}
\\&
=w_k\frac{g(-k/\nga)+o(1)}{g(0)},	
  \end{split}
\end{equation}
uniformly in $k\ge0$.

For a non-negative function $f$ on $\ooo$, define
\begin{equation}
  \xfn\=\sumin f(Y_i/\nga).
\end{equation}
In particular, if $f$ is the indicator $\ett{a\le x\le b}$ of an interval
$[a,b]$, we write $\xabn$ and have in the notation of \eqref{yna}
\begin{equation}
\xabn\=|\set{i\le n:a\nga\le Y_i\le b\nga}|
= \ny{[a\nga,b\nga]}.
\end{equation}
Suppose that $f$ is either the indicator of a compact interval
$[a,b]\subset\oooy$, 
or a continuous function with compact support in $\oooy$
(or, more generally, any Riemann integrable function with support in a
compact interval in $\oooy$). Then, using \eqref{l9a} and dominated convergence,
\begin{equation}\label{exfn}
  \begin{split}
\E \xfn 
&= n\sumk f(k/\nga)\P(Y_1=k)	
= n\sumk f(k/\nga) w_k\frac{g(-k/\nga)+o(1)}{g(0)}
\\&
= n^{1+1/\ga}\intoo f(\floor{x\nga}/\nga) w_{\floor{x\nga}}
 \frac{g(-\floor{x\nga}/\nga)+o(1)}{g(0)}\dd x
\\&
\to \intoo f(x)c x^{-\ga-1} \frac{g(-x)}{g(0)}\dd x.
  \end{split}
\raisetag\baselineskip
\end{equation}
In the special case when $f(x)=\ett{a\le x\le b}$ with $0<a<b<\infty$, we
further 
similarly obtain,
\begin{equation*}
  \begin{split}
\E \xabn(\xabn-1) 
&= n(n-1)\sum_{k,j\ge0} f(k/\nga)f(j/\nga)\P(Y_1=k,\; Y_2=j)	
\\&
=n(n-1)\sum_{k,j\ge0} f(k/\nga)f(j/\nga)w_kw_l\frac{\P(S_{n-2}=m-k-j)}{\P(S_n=m)}	
\\&
\to c^2\intoo\intoo f(x)f(y) x^{-\ga-1} y^{-\ga-1} 
 \frac{g(-x-y)}{g(0)}\dd x\dd y
\\&=
 c^2\int_a^b\int_a^b  x^{-\ga-1} y^{-\ga-1} 
 \frac{g(-x-y)}{g(0)}\dd x\dd y
  \end{split}
\end{equation*}
and, more generally, for any $\ell\ge1$,
\begin{equation}\label{exabnl}
  \begin{split}
\E (\xabn)_\ell
\to c^\ell\int_a^b\dotsi\int_a^b
\prod_{i=1}^\ell x_i^{-\ga-1}
 \frac{g(-x_1-\dots-x_\ell)}{g(0)}\dd x_1\dotsm\dd x_\ell.
  \end{split}
\end{equation}
For each such interval $[a,b]$, this integral is bounded by $CR^\ell$ for
all $\ell\ge1$, 
for some $C$ and $R$ (depending on $a$ and $b$), and it follows by the method of
moments that $\xabn\dto\xaboo$, where $\xaboo$ is determined by its factorial
moments 
\begin{equation}\label{exabl}
  \begin{split}
\E (\xaboo)_\ell
= 
c^\ell\int_a^b\dotsi\int_a^b
\prod_{i=1}^\ell x_i^{-\ga-1}
 \frac{g(-x_1-\dots-x_\ell)}{g(0)}\dd x_1\dotsm\dd x_\ell.
  \end{split}
\end{equation}
(It follows that $\xaboo$ has a finite \mgf, so the method of moment applies.)
Furthermore, joint convergence for several intervals holds by the same
argument. It follows also (by some modifications or by approximation
with step functions; we omit the details) 
that $\xfn\dto\xfoo$ for every continuous $f\ge0$ with compact support and
some $\xfoo$.

Let $\Xi_n$ be the multiset $\set{Y_i/\nga:Y_i>0}$, regarded as a point
process on $\oooy$. (I.e., formally we let $\Xi_n$ be the discrete measure
$\sum_{i:Y_i>0}\gd_{Y_i/\nga}$. See \eg{} 
\citet{Kallenberg:RM} or \cite{Kallenberg} for details on point processes, or
\citetq{\S\ 4}{SJ136} for a brief summary.)
The convergence $\xfn\dto\xfoo$ for every continuous $f\ge0$
with compact support in $\oooy$ implies, see \cite[Lemma 5.1]{Kallenberg:RM}
or \cite[Lemma 16.15 and Theorem 16.16]{Kallenberg}, that 
$\Xi_n$ converges in distribution, as a point process on $\oooy$, to some point
process $\Xi$ on $\oooy$. The distribution of $\Xi$ is determined by
\eqref{exabl}, where $\xaboo$ is the number of points of $\Xi$ in $[a,b]$.
By \eqref{exabl} or \eqref{exfn}, the intensity measure is given by
\begin{equation}\label{intense}
\E\Xi=cg(0)\qw x^{-\ga-1}g(-x) \dd x.  
\end{equation}

We can also consider infinite intervals. Let $a>0$. Then, 
using again \eqref{l41a} and noting that
$\sum_{k=-\infty}^\infty\P(S_{n-1}=m-k)=1$, 
\begin{equation}\label{exaoon}
  \begin{split}
\E\xaoon&
=n\sum_{k\ge a\nga}\P(Y_1=a)	
=n\sum_{k\ge a\nga}w_k\frac{\P(S_{n-1}=m-k)}{\P(S_{n}=m)}
\\&
\le n\CC (a\nga)^{-\ga-1} \frac{\sum_{k\ge a\nga} \P(S_{n-1}=m-k)}{\P(S_{n}=m)}
\\&
\le \CCx a^{-\ga-1}\ngaw \frac{1}{\ngaw (g(0)+o(1))}
\\&
\le \CC a^{-\ga-1}.
  \end{split}
\end{equation}
By Fatou's lemma, \eqref{exaoon} implies 
$\E\xaoooo \le\CCx a^{-\ga-1}<\infty$. Hence, $\xaoooo<\infty$ \as{} for
every $a>0$, and we
may order the points in $\Xi$ in decreasing order as 
\begin{equation}\label{etaj}
  \Xi=\set{\etaj}_{j=1}^J
\qquad \text{with} \quad 
\etax1\ge\etax2\ge\dots.
\end{equation}
(Here $J=X^{0,\infty}_\infty\le\infty$ is the random number of points in
$\Xi$. We shall see that $J=\infty$ a.s.)

The bound \eqref{exaoon} is uniform in $n$, and tends to 0 as $a\to\infty$.
It follows, see \cite[Lemma 4.1]{SJ136}, 
that if we regard $\Xi_n$ and $\Xi$
as point processes on $\ooox$, the convergence $\Xi_n\dto \Xi$ on $\oooy$ 
implies the stronger result
\begin{equation}
  \label{limxi}
\Xi_n\dto \Xi \quad\text{on }\ooox.
\end{equation}
The points in $\Xi_n$, ordered in decreasing order, are
$\ya/\nga\ge\yb/\nga\ge\dots$. 
If we extend \eqref{etaj} by defining $\etaj\=0$ when $j>J$,
the convergence \eqref{limxi} of point processes on $\ooox$ is 
by \cite[Lemma 4.4]{SJ136}
equivalent to joint convergence
of the ranked points, \ie{}
\begin{equation}\label{limetaj}
  \yj/\nga\dto\etaj, \qquad j\ge1
\text{ (jointly)}.
\end{equation}

We claim that each $\etaj>0$ a.s., and thus 
$J=X^{0,\infty}_\infty=\infty$ a.s.
Suppose the opposite: $\P(\etaj=0)=\gd>0$ for some $j$.
Then, for every $\eps>0$, $\liminf\P(\yj/\nga<\eps)\ge\P(\etaj<\eps)\ge\gd$,
and it follows that there exists a sequence $\eps_n\to0$ such that
$\P(\yj/\nga<\eps_n)\ge\gd/2$ for all $n$.
We may assume that $\eps_n\nga\to\infty$.
Let $A>0$ and take (for large $n$)
$a_n\=\eps_n$ and $b_n\=\bigpar{\eps_n^{-\ga}-\ga c\qw A}^{-1/\ga}$.
Then $a_n,b_n\to0$.
For $k\le b_n \nga=o(\nga)$, \eqref{l9} implies
$\P(S_{n-1}=m-k)/\P(S_n=m)\to1$,  
and the argument in \eqref{exfn}--\eqref{exabnl} yields,
for each $\ell\ge1$,
\begin{equation}
  \begin{split}
\E \bigpar{X^{a_n,b_n}_n}_\ell	
&\sim \lrpar{n\sum_{k=a_n\nga}^{b_n\nga}w_k}^\ell
\sim \lrpar{c\int_{a_n}^{b_n} x^{-\ga-1}\dd x}^\ell
\\&
=\lrpar{c\ga\qw\bigpar{a_n^{-\ga}-b_n^{-\ga}}}^\ell
=A^\ell.
  \end{split}
\end{equation}
Hence, $X^{a_n,b_n}_n\dto\Po(A)$; in particular, 
\begin{equation*}
\gd/2\le  \P(\yj/\nga<\eps_n)\le \P(X^{a_n,b_n}_n<j) \to \P(\Po(A)<j).
\end{equation*}
Taking $A$ large enough, we can make $\P(\Po(A)<j)<\gd/2$, a contradiction
which proves our claim.

We have shown that \eqref{limetaj} holds with $\etaj>0$. 
Furthermore, since the intensity \eqref{intense} is absolutely continuous,
each $\etaj$ has an absolutely continuous distribution.
Hence $\ya$, and
every $\yj$, is of the order $\ngaw$, with a continuous
limit distribution $\etaj$ (and thus no strict concentration at some
constant times $\ngaw$).

Note that if we consider \iid{} variables $\xin$, then
$\set{\xi_i/\nga:\xi_i>0}$ converges (as is easily verified) to a Poisson
process on $\ooox$ with intensity $cx^{-\ga-1}\dd x$.
This intensity differs from the intensity of $\Xi$ in \eqref{intense}, and,
since $g(-x)\to0$ as $x\to\infty$, it is easy to see that $\xia/\nga$ and
$\ya/\nga$ have different limit distributions. Thus, \refT{TDx} does not
hold in this case. (However, $\ya$ and $\xia$ are of the same order $\ngaw$.)
Note also that,
as an easy consequence of \eqref{exabl}, 
the limiting point process $\Xi$ in this example
is \emph{not} a Poisson process.
\end{example}

\begin{remark}\label{Reta} 
  The distribution of the limiting point process $\Xi$ in \refE{EL8} 
is in principle determined by \eqref{exabl} and its extension to joint
convergence for several $X_\infty^{a_i,b_i}$. This  can be made more
explicit as follows. (See \citet{LP} for similar calculations.)

It follows from \citetq{Section XVII.5}{FellerII},
see \eg{} \cite{SJN12} for detailed calculations, 
that $X_\ga$ has the \chf{}
\begin{equation}
  \label{stabchf}
\gf(t)
=\exp\bigpar{c\,\Gamma(-\ga)(-\ii t)^\ga},
\qquad t\in\bbR.
\end{equation}
(Note that $\Gamma(-\ga)>0$ and $\Re (-\ii t)^\ga<0$ for $t\neq0$ since
$1<\ga<2$.) 
The inversion formula gives
\begin{equation}
  g(x)=\frac1{2\pi}\intoooo e^{-\ii xt}\gf(t)\dd t
=\frac1{2\pi}\intoooo e^{-\ii xt+c\Gamma(-\ga)(-\ii t)^\ga}\dd t,
\end{equation}
and \eqref{exabl} yields
\begin{equation}
  \begin{split}
\E (\xaboo)_\ell
&= \frac1{2\pi g(0)}
c^\ell\int_a^b\dotsi\int_a^b
\prod_{j=1}^\ell x_j^{-\ga-1}\intoooo \prod_{j=1}^\ell e^{\ii x_j t}
\gf(t)\dd t\dd x_1\dotsm\dd x_\ell
\hskip-6em
\\&
= \frac1{2\pi g(0)}\intoooo
\Bigpar{ c\int_a^b  x^{-\ga-1}  e^{\ii t x} \dd x }^\ell
\gf(t)\dd t.
  \end{split}
\end{equation}
In particular, $\E (\xaboo)_\ell= O(C^\ell)$ for some $C<\infty$
(with $C$ depending on $a$ but not on $b$).
Hence, $\xaboo$ has \pgf, convergent for all complex $z$,
\begin{equation}\label{sww}
  \begin{split}
\E z^{\xaboo}
&=\E\suml\binom{\xaboo}{\ell}(z-1)^\ell
\\&
=\suml	
\frac{(z-1)^\ell}{\ell!}\cdot \frac1{2\pi g(0)}
\intoooo \Bigpar{ c\int_a^b  x^{-\ga-1}  e^{\ii t x} \dd x }^\ell
\gf(t)\dd t
\\&
= \frac1{2\pi g(0)} \intoooo 
\exp\Bigpar{(z-1)c  \int_a^b  x^{-\ga-1}  e^{\ii t x} \dd x }
\gf(t)\dd t.
  \end{split}
\end{equation}
We can here let $b\to\infty$, so \eqref{sww} holds for $b=\infty$ too. In
particular, taking $z=0$, we obtain, using \eqref{limetaj}, 
the limit distribution of $\ya/n\xga$ as
\begin{equation}\label{eta1}
  \begin{split}
&\P(\eta_1\le x)
=\P(X^{x,\infty}_\infty=0)
\\&
= \frac1{2\pi g(0)} \intoooo 
\exp\Bigpar{-c  \int_a^\infty  x^{-\ga-1}  e^{\ii t x} \dd x }
\gf(t)\dd t
\\&
= \frac1{2\pi g(0)} \intoooo 
\exp\Bigpar{-c  \int_a^\infty  x^{-\ga-1}  e^{\ii t x} \dd x 
+c\,\Gamma(-\ga)(-\ii t)^\ga}
\dd t
\\&
= \frac1{2\pi g(0)} \intoooo 
\exp\Bigpar{c \Bigpar{ 
\int_0^a  x^{-\ga-1}\bigpar{ e^{\ii t x}-1-\ii t x} \dd x 
-\frac{a^{-\ga}}{\ga}-\ii t\frac{a^{1-\ga}}{\ga-1}}}
\dd t,	
  \end{split}
\end{equation}
where the last equality  holds because
\begin{equation}
  \Gamma(-\ga)u^\ga
=\intoo x^{-\ga-1}\bigpar{e^{-ux}-1+ux}\dd u
\end{equation}
when $\Re u\ge0$ and $1<\ga<2$.

Furthermore, by extending \eqref{exabl} to joint factorial moments for
several (disjoint) intervals, it follows similarly, 
for step functions $f$, 
that the \rv{} $\xfoo=\sumji f(\eta_j)$ satisfies
\begin{equation}\label{ultima}
  \begin{split}
\E e^{\xfoo}
&
= \frac1{2\pi g(0)} \intoooo 
\exp\Bigpar{c\int_0^\infty\bigpar{e^{f(x)}-1} x^{-\ga-1} e^{\ii t x} \dd x}
\gf(t)\dd t
\\&
= \frac1{2\pi g(0)} \intoooo 
\exp\Bigpar{c  \int_0^\infty  \bigpar{e^{f(x)}-1}x^{-\ga-1}  e^{\ii t x} \dd x 
+c\,\Gamma(-\ga)(-\ii t)^\ga}
\dd t
\\&
= \frac1{2\pi g(0)} \intoooo 
\exp\Bigpar{c \Bigpar{ 
\intoo  x^{-\ga-1}\bigpar{ e^{f(x)+\ii t x}-1-\ii t x} \dd x }}
\dd t.
  \end{split}
\end{equation}
By taking limits, \eqref{ultima} extends to, \eg, any bounded measurable $f$
with compact support in $(0,\infty]$.
Since $\E e^{s\xfoo}=\E e^{X_\infty^{sf}}$ for $s\in\bbR$, this formula
determines (in principle) the distribution of each $\xfoo$ and thus
of $\Xi$.
\end{remark}

\begin{example}
  \label{EL9}
Let $\ww$ be as in \refE{EL8} but with $\ga=2$, \ie, 
$w_k\sim ck^{-3}$ as \ktoo, for some $c>0$.
(\refE{EBpower} with $\gb=3$.)
We still have \eqref{tailel8}; further,
$\rho=1$, and thus $\nu=\Psi(1)=\sum kw_k<\infty$.
(We may again obtain any desired $\nu>0$, for example $\nu=1$, by adjusting the
first few $w_k$.)

As in \refE{EL8}. we consider the case $m=\nu n+O(1)$, including the tree
case $m=n-1$ 
when $\nu=1$. Thus, again, $m/n\to\gl=\nu$,
$\tau=1=\rho$, $\pi_k=w_k$, and
the random variable $\xi$ satisfies $\E\xi=\gl=\nu$, while
$\gss\=\Var\xi=\infty$. 

As in the proof of \refT{THstab}, 
we have the central limit theorem \eqref{gg}, and 
the local limit law \eqref{l92} holds 
uniformly for all integers $\ell$.

Choose $\bn\=n\qq\log\log n=o(\sqrt{n\log n})$.
Then, by \eqref{l92}, 
\begin{equation}
Z(m,n)=\P(S_n=m)=\frac{g(0)+o(1)}{\sqrt{n\log n}}
\end{equation}
and, uniformly for all $k\le\bn$,
\begin{equation}
Z(m-k,n-1)=\P(S_{n-1}=m-k)=\frac{g(0)+o(1)}{\sqrt{n\log n}}.
\end{equation}
Hence, by \eqref{l41a}, \eqref{1722} holds.
Furthermore, \eqref{l92} yields also,
since $g(0)=\max_{x\in\bbR} g(x)$,
\begin{equation}
Z(m-k,n-1)=\P(S_{n-1}=m-k)\le\frac{g(0)+o(1)}{\sqrt{n\log n}},
\end{equation}
uniformly for all $k\ge0$; hence \eqref{l41a} implies that
\eqref{yabound}--\eqref{y1bound} hold.

For our $\bn$ we have by \eqref{tailel8}
\begin{equation}
\P(\xi\ge\bn)=O(\bn\qww)=o(n\qw),
\end{equation}
so \eqref{xibn} holds, and thus \eqref{ybn} holds.

The proof of \refT{TDx} now holds without further modifications; hence
the conclusions of \refT{TDx} holds for this example, although $\gss=\infty$.
\end{example}

Note that in \refE{EL8}, although the asymptotic distributions of $\ya$ and
$\xia$ are different, they are still of the same order of magnitude.
We do not know whether this is true in general. 
This question can be formulated more
precisely as follows.
\begin{problem}
  In the case $\gl=\nu$, do \refT{TDx}\ref{tdx0}--\ref{tdxoo} hold also when
  $\gss=\infty$? 
\end{problem}

\subsection{The case $\gl>\nu$}\label{SSlarge+}

We turn to the case $\gl>\nu$.  
Then, as briefly discussed in \refS{SBB},  the asymptotic formula for the
numbers $N_k$ in \refT{TBmain} 
accounts only for 
$\sumk k\pi_k n=\mu n=\nu n$ balls, so there are
$m-\nu n\approx(\gl-\nu)n$ balls missing.
A more careful treatment of the limits show that
the explanation is that \refT{TBmain} really 
implies that the ``small'' boxes 
(\ie, those with rather few balls)
have a total
of about $\sumk k\pi_k n=\mu n=\nu n$ balls, while the remaining $\approx
(\gl-\nu)n$ balls are in a few ``large'' boxes.
One way to express this precisely is 
the following simple result.

\begin{lemma}\label{LQ}
  Let $\wwx=(w_k)_{k\ge0}$ be a weight sequence with $w_0>0$ and 
$\go=\infty$.
Suppose that $n\to\infty$ and $m=m(n)$ with $m/n\to\gl$ where
$\nu<\gl<\infty$.
\begin{romenumerate}[-12pt]
\item 
For any sequence $K_n\to\infty$,
\begin{equation*}
\sum_{k\le K_n} kN_k\ge \nu n+\op(n)
\qquad\text{and}\qquad
\sum_{k>K_n} kN_k\le (\gl-\nu)n+\op(n).
\end{equation*}
\item 
There exists a sequence $\gO_n\to\infty$ such that for any sequence
$K_n\to\infty$ with $K_n\le\gO_n$ we have
\begin{equation*}
\sum_{k\le K_n} kN_k= \nu n+\op(n)
\qquad\text{and}\qquad
\sum_{k>K_n} kN_k=(\gl-\nu)n+\op(n).
\end{equation*}
\end{romenumerate}
\end{lemma}

\begin{proof}
The two statements in each part are equivalent, since 
\begin{equation}
  \sumk kN_k=m=\gl n+o(n).
\end{equation}
\pfitem{i}
For every fixed $\ell$, 
\refT{TBmain} implies
\begin{equation}\label{lq}
  \frac1n\sum_{k\le\ell} kN_k \pto\sum_{k\le\ell} k\pi_k.
\end{equation}
Let $\eps>0$. Since $\sumk k\pi_k=\nu<\infty$, there exists $\ell$ such that
$\sum_{k\le\ell} k\pi_k>\nu-\eps$, and \eqref{lq} implies that \whp{}
\begin{equation*}
  \frac1n\sum_{k\le\ell} kN_k >\nu-\eps.
\end{equation*}
Since $\eps$ is arbitrary, this implies
$\sum_{k\le K_n} kN_k\ge \nu n+\op(n)$.

\pfitem{ii}
For each fixed $\ell$, 
$\sum_{k\le\ell} k\pi_k<\sumk k\pi_k=\nu$, and thus
\eqref{lq} implies
$\P(\sum_{k\le\ell} kN_k >\nu n)\to0$.
Hence, there exists an increasing sequence of integers $n_\ell$ such that if
$n\ge n_\ell$, then $\P\bigpar{\sum_{k\le\ell}kN_k>\nu n}<1/\ell$.
Now define $\gO_n=\ell$ for $n_\ell\le n<n_{\ell+1}$.
Then $\sum_{k\le \gO_n} kN_k \le\nu n$ \whp, which together with (i) yields
(ii). 
\end{proof}

Consider the ``large'' boxes. One obvious possibility 
is that there is a
single ``giant'' box with $\approx(\gl-\nu)n$ balls; 
more formally,
$(\gl-\nu)n+\op(n)$ balls (a ``monopoly'').
Applying \refL{LQ}(i) with $K_n=o(n)$, we see that for every $\eps>0$, \whp{}
there are then
less than $\eps n$ balls in all other boxes with more than $K_n$
balls each; thus, either $\yb\le K_n$ or $\yb<\eps n$. Consequently, this
case is defined by
\begin{align}
  \ya&=(\gl-\nu)n+\op(n), \label{ya}
\\
  \yb&=\op(n). \label{yb}
\end{align}
Equivalently, $\ya/n\pto\gl-\nu$ and $\yb/n\pto0$.
This thus describes condensation of the missing balls to a single box.

We will see in \refT{TDzeta} that, indeed, this is the case for the
important example of weights with a power-law.
Another, more extreme  example is \refE{Ek!},
$w_k=k!$, where $\nu=0$, see \refE{EQ0}.

However, if $\ww$ is very irregular, \eqref{ya}--\eqref{yb} do not always hold.
Examples \refand{Eoligo}{Eoligo0}
give examples where, at least for a subsequence, either $\yb/n\pto a>0$, so
there are at least two giant boxes with order $n$ balls each (an ``oligopoly''),
or $\ya/n\pto0$, so there is no giant box with order $n$ balls, and the missing
$(\gl-\nu)n$ balls are distributed over a large number (necessarily
$\to\infty$ as \ntoo) of boxes, each with a large but $o(n)$ number of balls.

\begin{example}\label{EQ}
  We consider \refE{EBpower}; $w_k\sim ck^{-\gb}$ as \ktoo. 
If $\gb\le2$, then $\nu=\infty$, see \eqref{beta2}, and thus $\gl<\nu$ and
Theorems \refand{TD1}{TDx} apply.
We are interested in the case $\gl>\nu$, so we assume $\gb>2$.
In this case, 
\citet{sdf} showed (for the case of random trees)
that when $\gl>\nu$ we have the simple
situation with condensation to a single giant box.
We state this in the next theorem,
which also includes further, more precise, results.
(Note that the case $\gl<\nu$ is 
covered by Theorems \refand{TD1}{TDx}, with $\ya$ of order $\log n$;
the case 
$\gl=\nu$ is studied in Examples \refand{EL8}{EL9} 
for $2<\gb\le3$, and is
covered by \refT{TDx} when $\gb>3$; in both cases $\ya$ is of order
$n^{-1/(\gb-1)}=o(n)$.)
\end{example}


\begin{theorem}
  \label{TDzeta} 
Suppose that 
$w_k\sim ck^{-\gb}$ as \ktoo{} for some $c>0$ and
$\gb>2$. 
Then $\nu<\infty$.
Suppose further $m/n\to\gl>\nu$.
Let $\ga\=\gb-1>1$ and $c'\=c/\Phi(1)$.
\begin{romenumerate}[-5pt]
\item 
The random allocation $\bmn=\YYn$ has largest components
\begin{align}
  \ya&=(\gl-\nu)n+\op(n),\label{tdzeta1}
\\
  \yb&=\op(n).\label{tdzeta2}
\end{align}
\item 
The partition function is asymptotically given by
\begin{equation}\label{tdzeta}
  Z(m,n)\sim c(\gl-\nu)^{-\gb}\Phi(1)^{n-1}n^{1-\gb}.
\end{equation}
\item 
Furthermore,
  \begin{equation}\label{tdzeta=} 
\bigpar{\ya,\yb,\dots,Y\win n}
\dapprox	
\Bigpar{m-\sum_{i=1}^{n-1}\xi_i,\qxi\win1,\dots,\qxi\win{n-1}},
  \end{equation}
where $\qxi\win1,\dots,\qxi\win{n-1}$ are the $n-1$ \iid{}
random variables $\xi_1,\dots,\xi_{n-1}$, with distribution $\ppi$,
ordered in decreasing order. 

\item  
$\ya=m-\nu n+\Op(n\xga)$ and 
  \begin{equation}\label{tdzetastab} 
n\xgaw\bigpar{m-\nu n-\ya}\dto X_\ga,	
  \end{equation}
where $X_\ga$ is an $\ga$-stable random variable with Laplace transform
\begin{equation}\label{tdzetastabL} 
  \E e^{-tX_\ga} = \exp\bigpar{c'\Gamma(-\ga)t^\ga},
\qquad \Re t\ge0.
\end{equation}

\item 
$\yb=\Op(n\xga)$ and
  \begin{equation}\label{tdzetaFre} 
n\xgaw\yb\dto W,	
  \end{equation}
where $W$ has the Fr\'echet distribution 
\begin{equation}
  \P(W\le x)=\exp\Bigpar{-\frac{c'}{\ga}x^{-\ga}},
\qquad x\ge0.
\end{equation}

\item 
More generally, for each $j\ge2$,
$\yj=\Op(n\xga)$ and
  \begin{equation}
n\xgaw\yj\dto W_j,	
  \end{equation}
where $W_j$ has the density function 
\begin{equation}\label{wjdense}
c'x^{-\ga-1}
\frac{\bigpar{c'\ga\qw x^{-\ga}}^{j-2}}{(j-2)!}
\exp\bigpar{-{c'}{\ga\qw}x^{-\ga}},
\qquad x\ge0,
\end{equation}
and $c'\ga\qw W_j^{-\ga}\sim\Gamma(j-1,1)$.
\end{romenumerate}
\end{theorem}

Note that $\pi_k=w_k/\Phi(1)$ and that $\Gamma(-\ga)>0$ in \eqref{tdzetastabL}.

Part
(iii) shows that $\yb,\dots,\yn$ asymptotically are as order statistics of $n-1$
\iid{}  random variables $\xi_i$; 
thus the giant box absorbs the dependency between the variables
  $Y_1,\dots,Y_n$ introduced by the conditioning in \eqref{ebprob2}.

\begin{remark}\label{RDzeta}
\citet{sdf} considered only trees, and thus $m=n-1$ and $\gl=1$,
and then showed the tree versions of (i) and (ii).
(They further showed \refT{Tmain} when $w_k\sim ck^{-\gb}$.)
In the tree
case (i) says that the random tree $\ctn$ has \whp{} a node of largest
degree $(1-\nu)n+o(n)$, while all other nodes have degrees $o(n)$; further,
by \refT{TZ}, (ii) 
becomes 
\begin{equation}\label{rdzeta}
  Z_n\sim c(1-\nu)^{-\gb}\Phi(1)^{n-1}n^{-\gb}
\sim (1-\nu)^{-\gb}\Phi(1)^{n-1}w_n.
\end{equation}
\end{remark}

\begin{proof}[Proof of \refT{TDzeta}]\CCreset
We may assume that  $w_0>0$ by the argument in \refR{Rmin}.
Furthermore, using \eqref{lb1z} for (ii),
by dividing $w_k$ (and $c$) by $\Phi(1)$, we may assume that $\ww$ is a
\pws, and thus $\Phi(1)=1$. 
For $\gl>\nu$ we have $\tau=\rho=1$, and thus then $\pi_k=w_k$.

\pfitem i
$\Phi(t)$ has radius of convergence $\rho=1$, and since $\gb>2$,
  $\Phi(1)=\sum_k w_k<\infty$ and $\nu=\Phi'(1)/\Phi(1)<\infty$.

Consider as in \refE{EBprob} \iid{} random variables $\xin$ with
distribution $\ppi=\ww$ and mean $\mu=\nu$.

Fix a small $\eps>0$. We assume that $\eps<\gl-\nu$.

By the law of large numbers, $S_{n-1}/n\pto\mu=\nu$. We may thus find a
sequence $\gd_n\to0$ such that $|S_{n-1}-n\nu|\le n\gd_n$ \whp.

Since $m/n-\nu-\gd_n\to\gl-\nu>\eps$, we have $m-\nu n-\gd_n n>\eps n$ for
large $n$; we consider only such $n$.

We separate the event ${S_n=m}$ into four disjoint cases (subevents):
\begin{itemize}
\item [$\cE_1$]: Exactly one $\xi_i>\eps n$, and that $\xi_i$ satisfies
  $|\xi_i-(m-\nu n)|\le\gd_nn$.
\item [$\cE_2$]: Exactly one $\xi_i>\eps n$, and that $\xi_i$ satisfies
  $|\xi_i-(m-\nu n)|>\gd_nn$.
\item [$\cE_3$]:  $\xi_i>\eps n$ for at least two $i\in\setn$.
\item [$\cE_4$]: All $\xi_i\le\eps n$.
\end{itemize}
We shall show that $\cE_1$ is the dominating event.
We define also the events
\begin{itemize}
\item [$\eei$]: $S_n=m$, $|\xi_i-(m-\nu n)|\le\gd_nn$ and 
 $\xi_j\le\eps n$ for $j\neq i$.
\item [$\eexi$]: $S_n=m$, $|\xi_i-(m-\nu n)|\le\gd_nn$.
\item [$\eeexi$]: $S_n=m$, $|\xi_i-(m-\nu n)|>\gd_nn$, $\xi_i>\eps n$.
\item [$\dij$]: $S_n=m$, $\xi_i>\eps n$, $\xi_j>\eps n$.
\end{itemize}
Then $\cE_1$ is the disjoint union $\bigcup_{i=1}^n\eei$, so by symmetry
\begin{equation}
  \label{q2}
\P(\cE_1)=n\P(\eeq1).
\end{equation}
Furthermore, for any $i$,
\begin{equation*}
  \eei\subseteq\eexi\subseteq\eei\cup\bigcup_{j\neq i}\dij
\end{equation*}
and thus, again using symmetry,
\begin{equation}\label{q3}
\P(\eexq1)\ge\P(\eeq1)\ge\P(\eexq1)-n\P(\dab).
\end{equation}
Using 
the fact that 
$|k-(m-\nu n)|\le\gd_nn$ implies
$w_k\sim ck^{-\gb}\sim c(\gl n-\nu n)^{-\gb}$, together with
$|S_{n-1}-n\nu|\le\gd_nn$ \whp{}, we obtain
\begin{equation}
  \label{q4}
  \begin{split}
	\P(\eexq1)
&=\sum_{|k-(m-\nu n)|\le\gd_nn} \P(\xi_1=k,\,S_n=m)\\
&=\sum_{|k-(m-\nu n)|\le\gd_nn} \P(\xi_1=k)\P(S_{n-1}=m-k)\\
&=\sum_{|k-(m-\nu n)|\le\gd_nn}
	c(\gl-\nu)^{-\gb}n^{-\gb}\etto\P(S_{n-1}=m-k)\\
&= c(\gl-\nu)^{-\gb}n^{-\gb}\P\bigpar{|S_{n-1}-n\nu|\le\gd_nn}\etto\\
&= c(\gl-\nu)^{-\gb}n^{-\gb}\etto.
  \end{split}
\raisetag\baselineskip
\end{equation}
Similarly, allowing the constants $C_i$ here and below to depend on $\eps$,
\begin{equation}
  \label{q4b}
  \begin{split}
	\P(\eeexq i)
&=\sum_{|k-(m-\nu n)|>\gd_nn,\,k>\eps n} \P(\xi_i=k,\,S_n=m)\\
&\le \CC(\eps n)^{-\gb}\sum_{|k-(m-\nu n)|>\gd_nn,\,k>\eps n}
  \P(S_{n-1}=m-k)\\
&\le \CC n^{-\gb}\P\bigpar{|S_{n-1}-\nu n|>\gd_nn}
=o\bigpar{n^{-\gb}}.
  \end{split}
\end{equation}
For any $i$ and $j$, by symmetry,
\begin{equation}
  \label{q5}
  \begin{split}
\P(\dij)
&=	\P(\xi_n>\eps n,\,\xi_{n-1}>\eps n,\, S_{n}=m)\\
&=\sum_{k>\eps n}\P(\xi_n=k)\P(S_{n-1}=m-k,\,\xi_{n-1}>\eps n)\\
&\le \CC (\eps n)^{-\gb}\sum_{k>\eps n}\P(S_{n-1}=m-k,\,\xi_{n-1}>\eps n)\\
&\le \CCx (\eps n)^{-\gb}\P(\xi_{n-1}>\eps n)
\le \CC (\eps n)^{1-2\gb}.
  \end{split}
\end{equation}
Hence, \eqref{q3} and \eqref{q4} yield
\begin{equation*}
  \P(\eeq1)
=c(\gl-\nu)^{-\gb}n^{-\gb}+o\bigpar{n^{-\gb}}+O\bigpar{n^{2-2\gb}}
=c(\gl-\nu)^{-\gb}n^{-\gb}+o\bigpar{n^{-\gb}}
\end{equation*}
and hence, by \eqref{q2},
\begin{equation}\label{q51}
  \P(\cE_1)=c(\gl-\nu)^{-\gb}n^{1-\gb}+o\bigpar{n^{1-\gb}}.
\end{equation}
Furthermore, \eqref{q4b} yields 
\begin{equation}\label{q52}
  \P(\cE_2)
\le \sumin\P(\eeexq i)
= n\P(\eeexq1)=o\bigpar{n^{1-\gb}},
\end{equation}
and \eqref{q5} also yields
\begin{equation}\label{q53}
  \P(\cE_3)
\le\sum_{i<j}\P(\dij)\le n^2\P(\dab)=O\bigpar{n^{3-2\gb}}
=o\bigpar{n^{1-\gb}}.
\end{equation}
It remains to estimate $\P(\cE_4)$. We define the truncated variables
$\bxi_i\=\xi_i\ett{\xi_i\le\eps n}$ and $\bS_n\=\sumin\bxi_i$.
Thus
$\cE_4\subseteq\set{\bS_n=m}$ and hence, for every real $s$,
\begin{equation}
  \label{q6}
\P(\cE_4)\le e^{-sm}\E e^{s\bS}
= e^{-sm}\lrpar{\E e^{s\bxi_1}}^n.
\end{equation}
Let $s\=a\log n/n$, for a constant $a>0$ chosen later.
Then,
\begin{equation}
  \label{q6b}
  \begin{split}
\E e^{s\bxi_1}
&=1+s\E\bxi_1+\sum_{k=1}^{\eps n}\pi_k\bigpar{e^{sk}-1-sk}	
\\
&\le1+s\nu+\CC\sum_{k=1}^{2\gb/s}k^{-\gb}s^2k^2
+\CCx\sum_{k=2\gb/s}^{\eps n}k^{-\gb}e^{sk}.
  \end{split}
\end{equation}
We have, treating the cases $2<\gb<3$, $\gb=3$ and $\gb>3$ separately,
using $s\to0$,
\begin{equation*}
\sum_{k=1}^{2\gb/s}s^2k^{2-\gb}
\le \CC s^2 \max\lrpar{1,(2\gb/s)^{3-\gb},\log(2\gb/s)}=o(s).
\end{equation*}
Furthermore, for $k>2\gb/s$,
\begin{equation*}
  \frac{k^{-\gb}e^{sk}}{(k+1)^{-\gb}e^{s(k+1)}}
=\Bigpar{1+\frac1k}^\gb e^{-s}
\le e^{\gb/k-s}
\le e^{s/2-s}
= e^{-s/2}.
\end{equation*}
Hence, the final sum in \eqref{q6b} is dominated by a geometric series
\begin{equation*}
\sum_{k\le \fen}(\fen)^{-\gb}e^{s\fen}e^{-s(\fen-k)/2}
\le \CC s\qw n^{-\gb}e^{s\eps n}
=\CCx s\qw n^{-\gb}e^{a\eps \log n}.
\end{equation*}
If we assume $a\eps\le\gb-2$, the sum is thus $\le \CC n^{1-\gb+a\eps} \le
\CCx n\qw=o(s)$.
Consequently, \eqref{q6b} yields 
\begin{equation*}
  \E e^{s\bxi_1}\le 1+s\nu+o(s)\le\exp\bigpar{s\nu+o(s)}
\end{equation*}
and thus \eqref{q6} yields
\begin{equation}
  \label{q8}
\P(\cE_4)\le\exp\bigpar{-sm+ns\nu+o(ns)}
=\exp\bigpar{-ns(\gl-\nu+o(1))}
=n^{-a(\gl-\nu)+o(1)}.
\end{equation}
We choose first $a\=\gb/(\gl-\nu)$ and then $\eps<(\gb-2)/a$, and see by
\eqref{q8} that then
$\P(\cE_4)=n^{-\gb+o(1)}=o\bigpar{n^{1-\gb}}$.
Combining \eqref{q51}, \eqref{q52}, \eqref{q53} and \eqref{q8}, we find
\begin{equation}\label{111111}
\P(S_n=m)= \P(\cE_1)+o\bigpar{n^{1-\gb}}
=c(\gl-\nu)^{-\gb}n^{1-\gb}+o\bigpar{n^{1-\gb}},
\end{equation}
and, in particular, $\P(\cE_1\mid S_n=m)\to1$.
Consequently, by conditioning on $S_n=m$ we see that 
\whp{} $|\ya-(m-\nu n)|\le\gd_nn$ and $\yb\le\eps n$.
Since $\eps$ can be chosen arbitrarily small, this completes the proof of
\eqref{tdzeta1}--\eqref{tdzeta2}. 

\pfitem{ii}
$Z(m,n)=\P(S_n=m)$,
so \eqref{tdzeta} follows from \eqref{111111}, since we assume $\Phi(1)=1$.

\pfitem{iii}
Since $\cE_1\subseteq\set{ S_n=m}$
and $\P(\cE_1\mid S_n=m)\to1$, 
\begin{equation}
  \label{wm1}
\YYn
\eqd\bigpar{(\xin)\mid S_n=m}
\dapprox\bigpar{(\xin)\mid \cE_1}.
\end{equation}
When we consider the ordered variables $\ya,\dots,Y\win n$, we may by
symmetry condition on $\eeq n$ instead of $\cE_1$.
Note that $\eeq n$ is the event $(\xin)\in A$, where $A$ is the set
\begin{equation*}
  \Bigset{(x_1,\dots,x_n):x_j\le\eps n \text{ for $j\le n-1$},
\, x_n=m-\sum_{i=1}^{n-1}x_i,\, \Bigabs{\sum_{i=1}^{n-1}x_i-\nu n}\le\gd_nn}.
\end{equation*}
Since $(x_1,\dots,x_n)\in A$ implies $|x_n-(m-\nu n)|\le\gd_n n$, we then
have, similarly to \eqref{q4},
\begin{equation*}
  \P(\xi_n=x_n)\sim c x_n^{-\gb}\sim c(m-\nu n)^{-\gb}
\sim c(\gl-\nu)^{-\gb}n^{-\gb}.
\end{equation*}
Furthermore, $x_1,\dots,x_{n-1}$ determine $x_n$ by $\sum_1^nx_i=m$. It
follows that, uniformly for all $(x_1,\dots,x_n)\in A$,
\begin{equation*}
  \begin{split}
&\P\bigpar{(\xin)=(x_1,\dots,x_n)}
\\
&\qquad=\etto c(\gl-\nu)^{-\gb}n^{-\gb}
\P\bigpar{(\xi_1,\dots,\xi_{n-1})=(x_1,\dots,x_{n-1})}
\\
&\qquad=\etto c(\gl-\nu)^{-\gb}n^{-\gb}
\P\bigpar{(\xi_1,\dots,\xi_{n-1},m-S_{n-1})=(x_1,\dots,x_n)}.
  \end{split}
\end{equation*}
Hence, since the factor $c(\gl-\nu)^{-\gb}n^{-\gb}$ is a constant for
each $n$,
\begin{equation}\label{wm2}
\bigpar{(\xin)\mid\eeq n}
\dapprox
\bigpar{(\xi_1,\dots,\xi_{n-1},m-S_{n-1})\mid \eqq},
\end{equation}
where $\eqq$ is the event
\begin{equation}\label{wm3}
\bigset{(\xi_1,\dots,\xi_{n-1},m-S_{n-1})\in A}
=
  \bigset{\xi_j\le\eps n \text{ for $j\le n-1$},
\,  \bigabs{S_{n-1}-\nu n}\le\gd_nn}.
\end{equation}
If $\eqq$ holds, then $m-S_{n-1}\ge m-\nu n-\gd_nn>\eps n$ (for large $n$),
so the largest variable among
$\xi_1,\dots,\xi_{n-1},m-S_{n-1}$ is $m-S_{n-1}$. Hence, ordering the
variables, we obtain using \eqref{wm1}--\eqref{wm2}
\begin{equation}\label{wm33}
\bigpar{\ya,\dots,Y\win n}
\dapprox
\bigpar{(m-S_{n-1},\qxi\win 1,\dots,\qxi\win{n-1})\mid \eqq}.
\end{equation}

Finally, observe that $|S_{n-1}-\nu n|\le\gd_nn$ \whp{} and
\begin{equation*}
  \P\bigpar{\xi_j>\eps n \text{ for some } j\le n-1}
\le n\P(\xi_1>\eps n) = O\bigpar{n^{2-\gb}}\to0.
\end{equation*}
Hence, $\P(\eqq)\to1$, and thus
\begin{equation}\label{wm4}
\bigpar{(m-S_{n-1},\qxi\win 1,\dots,\qxi\win{n-1})\mid \eqq}
\dapprox
\bigpar{m-S_{n-1},\qxi\win 1,\dots,\qxi\win{n-1}}.
\end{equation}
The result \eqref{tdzeta=} follows from \eqref{wm33} and \eqref{wm4}.

\pfitem{iv}
By (iii), $m-n\nu-\ya\dapprox \sum_{i=1}^{n-1}\xi_i-n\nu$, and
\eqref{tdzetastab} follows by standard results on domains of attraction for
stable distributions, see \eg{} \citetq{Section XVII.5}{FellerII}.

\pfitem{v}
By (iii), $\yb\dapprox \qxi\win1$, and
\eqref{tdzetaFre} follows by standard results on the maximum of \iid{}
random variables, as in \eg{} \citet{LLR}: using
$\P(\xi> x)\sim c\ga\qw x^{-\ga}$ as \xtoo, we have
\begin{equation*}
  \begin{split}
  \P(\ya\le xn\xga)
&=
  \P(\qxi\win1\le xn\xga)+o(1)
=
  \P(\xi\le xn\xga)^{n-1}+o(1)
\\&
=
\bigpar{1-(c\ga\qw+o(1))( xn\xga)^{-\ga}}^{n-1}+o(1)	
\\&
\to
\exp\bigpar{-c\ga\qw x^{-\ga}}.
  \end{split}
\end{equation*}

\pfitem{vi}
Similar, \cf{}  \citet[Section 2.2]{LLR}.
\end{proof}

\begin{example}
  \label{EQ0} 
If we take $w_k=k!$, then $\nu=\rho=0$. Consider the tree case $m=n-1$.
By \refE{Ek!}, translating to balls-in-boxes,
\whp{}
there are $N_1$ boxes with 1 ball each and a single box with the remaining
$n-1-N_1$ balls, while all other boxes are empty; furthermore, 
$N_1\dto\Po(1)$ so
$N_1=\Op(1)$.
Hence,   $\ya=n-\Op(1)$ and $\yb\le1$ \whp.

If we take $w_k=k!^\ga$ with $0<\ga<1$, and still $m=n-1$, then by
\refE{Ek!a} and \cite{SJ259}, $\ya=n-\Op(n^{1-\ga})=n-\op(n)$ 
and $\yb\le\floor{1/\ga}$ \whp.

If we take  $w_k=k!^\ga$ with $\ga>1$, and still $m=n-1$, then by
\refE{Ek!a}, \whp{} there is a single box containing all $n-1$ balls; thus
$\ya=n-1$ and $\yb=0$ \whp.

In particular,  \eqref{ya}--\eqref{yb} hold, with $\gl=1$ and $\nu=0$, for
all three cases.
We guess that the same is true for any $\gl<\infty$, but we have not checked
the details.
\end{example}

\begin{example}
  \label{Eoligo} 
We consider the tree case $m=n-1$.
Let $\Sigma\=\set{k_0,k_1,\dots}$ be an infinite set with $k_0=0$, $k_1=1$,
$k_2=2$, and 
$k_j$ for $j\ge3$ chosen recursively as specified below.
Let $w_k=(k+1)^{-4}$ for $k\in\Sigma$, and $w_k=0$ otherwise; thus,
$\suppw=\Sigma$.
($\Sigma=\No$ gives \refE{Ezeta} with $\gb=4$.)
Then $\rho=1$ and
\begin{equation}
\label{eoo}
\nu=\Psi(1)=\frac{\sumk kw_k}{\sumk w_k}
\le\frac{\sumk k(k+1)^{-4}}{w_0}
=\zeta(3)-\zeta(4)<0.2<1;
\end{equation}
thus $\tau=\rho=1$.

To begin with, we require that $k_j\ge jk_{j-1}$ for $j\ge3$.
Take $n=k_j$. A good allocation of $n-1$ balls in $n$ boxes has at most
$k_{j-1}$ balls in any box, since $n-1<k_j$, so
\begin{equation}
  \ya\le k_{j-1}\le k_j/j=n/j.
\end{equation}
Hence, for $n$ in the subsequence \set{k_j}, the random allocation $\bnn$
has $\ya=o(n)$.

Next, suppose that $k_0,\dots,k_{j-1}$ are given, and let $\wwkjx$ be $\wwx$
truncated at $k_{j-1}$ as in \eqref{trunc}; for ease of notation we denote
the corresponding generating function by 
$\Phi_j(t)\=\sum_{i=0}^{j-1} w_{k_i}t^{k_i}$
and write $\Psi_j(t)\=t\Phi_j'(t)/\Phi_j(t)$
and $Z_j(m,n)\=Z(m,n;\wwkjx)$.
Note that \eqref{eoo} applies to each $\Psi_j$ too, and thus
\begin{equation}
\label{eoox}
\Psi_j(1)<0.2.
\end{equation}

Take $n=3k_j$ (where $k_j$ is not yet determined). 
A good allocation with $n-1$ balls has at most 2 boxes with $k_j$ balls, and
for the remaining boxes the weights $\wwx$ and $\wwkjx$ coincide. We thus
obtain
\begin{multline}
  \label{zoo}
Z(3k_j-1,3k_j)
=
Z_j(3k_j-1,3k_j)
+3k_jw_{k_j} Z_j(2k_j-1,3k_j-1)
\\
+\binom{3k_j}2w_{k_j}^2 Z_j(k_j-1,3k_j-2).  
\end{multline}
Let the three terms on the \rhs{} be $A_0,A_1,A_2$, where $A_i$ corresponds
to the case when $i$ boxes have $k_j$ balls. 
The generating function $\Phi_j$ is a polynomial, with radius of convergence
$\rho_j=\infty$ and, by \refL{LPsi},
$\nu_j\=\Psi_j(\infty)=\go(\wwkjx)=k_{j-1}\ge2$. Define $\tau,\tau'$ and
$\tau''$  by 
$\Psi_j(\tau_j)=1$,
$\Psi_j(\tau_j')=2/3$,
$\Psi_j(\tau_j'')=1/3$.
Since $\Psi_j(1)<1/3$ by \eqref{eoox},
we have $1<\tau_j''<\tau_j'<\tau_j<\infty$.

\refT{TBZ} applies to each term $A_i$ in \eqref{zoo}, with
$\gl=1,\frac23,\frac13$, respectively; hence, as $k_j\to\infty$,
\begin{align}
  \log A_0 &= 3k_j \log\frac{\Phi_j(\tau_j)}{\tau_j}+o(k_j),
\label{aa0}\\
  \log A_1 & = 3k_j \log\frac{\Phi_j(\tau_j')}{(\tau_j')\qqqb}+o(k_j),
\label{aa1}\\
  \log A_2 &= 3k_j \log\frac{\Phi_j(\tau_j'')}{(\tau''_j)\qqq}+o(k_j).
\label{aa2}
\end{align}
By \eqref{annagl} and $\tau_j''>1$,
\begin{equation*}
  \frac{\Phi_j(\tau_j)}{\tau_j} \le   \frac{\Phi_j(\tau_j'')}{\tau_j''}
<   \frac{\Phi_j(\tau_j'')}{(\tau_j'')\qqq}
\end{equation*}
and
\begin{equation*}
  \frac{\Phi_j(\tau_j')}{(\tau_j')\qqqb} 
\le   \frac{\Phi_j(\tau_j'')}{(\tau_j'')\qqqb}
<   \frac{\Phi_j(\tau_j'')}{(\tau_j'')\qqq}.
\end{equation*}
Hence, the constant multiplying $k_j$ is larger in \eqref{aa2} than in
\eqref{aa0} and \eqref{aa1}, so by choosing $k_j$ large enough, we obtain
$A_2>jA_1$ and $A_2>jA_0$, and thus
\begin{equation}
  \label{aaa}
\P(B_{3k_j-1,3k_j} \text{ has 2 boxes with $k_j$ balls})
=\frac{A_2}{A_0+A_1+A_2}>1-\frac2j.
\end{equation}

This constructs recursively the sequence $(k_j)$ and thus $\Sigma$ and
$\wwx$, and \eqref{aaa} shows that for $n$ in the subsequence $(3k_j)_j$,
$\bnn$ \whp{} has 2 boxes with $n/3$ balls each.

By \refL{Lx}, it follows that, for this subsequence, $\ctn$ \whp{} has 2
nodes with outdegrees $n/3$.

To summarise, we have found a \ws{} with $0<\nu<1$ such that, with $m=n-1$,
for one subsequence 
\begin{equation}
  \ya/n\to0
\end{equation}
and for another subsequence \whp{}
\begin{equation}
  \ya=\yb=n/3.
\end{equation}
Hence, neither \eqref{ya} nor \eqref{yb} holds.
(It is easy to modify the construction such that for every $\ell\ge1$, there
is a subsequence with $\ya=\dots=Y\win\ell=n/(\ell+1)$.)
\end{example}

\begin{example}
  \label{Eoligo0} 
Let $\gS\=\set0\cup\set{2^i:i\ge0}$. We will construct a \ws{} $\wwx$
recursively with
support $\supp(\wwx)=\gS$ and $\rho=0$.
Let $w_0=1$.

Let $i\ge0$. If $w_0,\dots,w_{2^{i-1}}$ are fixed and we let
$w_{2^i}\to\infty$, then for every $m$ with $2^i\le m<2^{i+1}$ and every
$n$,
\begin{equation}
\P(\bmn \text{ contains a box with $2^i$ balls}) \to 1.
\end{equation}
Hence, we can recursively choose $w_{2^i}$ so large that, for every $i\ge0$,
if $2^i\le m<2^{i+1}$ and $2^i\le n\le 2^{2i}$, then, by \eqref{pbmn},
\begin{equation}\label{hjalmar}
\P(\bmn \text{ contains a box with $2^i$ balls}) >1-i\qw.
\end{equation}
We further take $w_{2^i}\ge (2^i)!$; thus $\rho=0$ and $\nu=0$.

Consider the tree case, $m=n-1$. Thus $\gl=1$.
If $2^i<n\le 2^{i+1}$, then \eqref{hjalmar} applies and shows that $\bnn$
\whp{} contains a box with $2^i$ balls, so \whp{}
\begin{equation}
\ya=2^{\floor{\log_2(n-1)}}  =2^{\ceil{\log_2n}-1}  .
\end{equation}
Hence, $\ya/n$ \whp{} is a (non-random) value that oscillates between
$\frac12$ and 1, depending on the fractional part 
$\frax{\log_2n}$ of $\log_2n$.
Consequently, \eqref{ya} holds for subsequences such that 
$0\neq\frax{\log_2n}\to0$, but not in general.

Moreover, conditioned on the existence of a box with $2^i$ balls, the
remainder of the allocation is a random allocation $B_{m-2^i,n-1}$ of the
remaining $m-2^i$ balls in $n-1$ boxes.
For example, if $n=2^{i+1}$, so $m=2^{i+1}-1$, we have $m-2^i=2^i-1$, and we
can apply \eqref{hjalmar} again (with $i-1$) to see that \whp{}
$\yb=2^{i-1}=n/4$. Continuing in the same way we see that for $n$ in the
subsequence $(2^i)$, we have, for each fixed $j$, \whp{}
\begin{equation}  \label{oligo0} 
  \yj=2^{-j} n.
\end{equation}
Hence neither \eqref{ya} nor \eqref{yb} holds in this case.

Similar results follow easily for other subsequences. For example, for $n$
in the subsequence $(\floor{r 2^i})_{i\ge1}$, where $\frac12<r<1$ and $r$
has the infinite binary expansion $r=2^{-\ell_1}+2^{-\ell_2}+\dots$, with
$1=\ell_1<\ell_2<\dots$, we have \whp{} $\yj=2^{-\ell_i}\ceil{n/r}$ for each
fixed $j$.
\end{example}

\begin{example}\label{E00}
Let again $m=n-1$, so $\gl=1$.
Taking $w_k=k!$ for $k\in\supp(\wwx)=\set0\cup\set{i!:i\ge0}$,
we obtain an example with $\rho=0$ and thus $\nu=0$ such that $\ya/n\to0$
for some subsequences, for example for $n=i!$ (since then $\ya\le(i-1)!$). 
\end{example}

\begin{problem}
  Is $\ya/n\pto0$ possible when $0<\nu<\gl$? 
\refE{Eoligo} shows that this is possible for a subsequence, but we
  conjecture that it is not possible for the full sequence,
and, a little stronger, that there always is some $\eps>0$ and some
subsequence along which $\ya\ge\eps n$ \whp.
\end{problem}

\begin{problem}
    Is $\ya/n\pto0$ possible when $\gl>\nu=0$? 
(\refE{E00} shows that this is possible for a subsequence.)
\end{problem}

We expect that bad behaviour as in the examples above only can occur for
quite irregular \ws{s}, but we have no general result beyond \refT{TDzeta}.
We formulate two natural problems.

\begin{problem}
  Suppose that $w_k\ge w_{k+1}$ for all (large) $k$.
Does this imply that \eqref{ya}--\eqref{yb} hold when $\gl>\nu$?
\end{problem}

\begin{problem}
  Suppose that $w_{k+1}/ w_{k}\to\infty$ as \ktoo.
(Hence, $\rho=0$ and $\nu=0$.)
Does this imply that \eqref{ya}--\eqref{yb} hold when $\gl>\nu$?
\end{problem}

\subsection{Applications to random forests}\label{SSmaxforests}
We give some applications of the results above to the size of the largest
tree(s) in different types of random forests witn $n$ trees and $m\ge n$ nodes.
We consider only the case $m/n\to\gl$ with $1<\gl<\infty$; for simplicity we
further assume that $m=\gl n+O(1)$, although this can be relaxed and, moreover,
the general case $m/n\to\gl$ can be handled by using $\gl_n\=m/n$ and the
corresponding $\tau_n\=\tau(\gl_n)$ as in \refT{TBmain2}; for details and 
for results in the cases $m=n+o(n)$ and $m/n\to\infty$, see 
\citet{Pavlov:max,Pavlov:limit,Pavlov,Pavlov:unlabelled}, 
\citet{Kolchin},
\citet{LP},
\citet{KazPavlov} and \citet{Pavlov:unun}.

The random forests considered here are described by balls-in-boxes with
\ws{s} with $w_0=0$ and $w_1>0$, see \refS{Sex+}.
As usual, we use (without further comments) 
the argument in \refR{Rmin} to extend 
theorems above to the case $w_0=0$. (See \refR{RTa}.)

We first consider random rooted forests as in \refE{Eforest}.
We have
\begin{equation}\label{wkb}
 w_k=\frac{k^{k-1}}{k!}
\sim 
\frac1{\sqrt{2\pi}} k^{-3/2}e^k,
\qquad\text{as \ktoo}, 
\end{equation}
and thus $w_{k+1}/w_k\to e$ as \ktoo.
(Alternatively, we may use $\tw_k\=e^{-k}w_k\sim(2\pi)\qqw k^{-3/2}$, 
see \refE{EBpower}.)
Since $\nu=\infty$, see Examples \ref{Eforest} and \ref{EBpower}, $\gl<\nu$ and
\refT{Ta} applies for any $\gl\in(1,\infty)$.

We have $a=e$ and thus, by \eqref{foresttau},
\begin{equation}\label{aq}
  q\=\tau e=\frac{\gl-1}{\gl}e^{1/\gl}\in(0,1)
\end{equation}
and, consequently,
\begin{equation}\label{alq}
  \log(1/q)=-\log q=-\log\Bigpar{1-\frac1\gl}-\frac1\gl >0.
\end{equation}
As \ktoo, by \eqref{treefn}, \eqref{forestTtau} and \eqref{wkb},
\begin{equation}\label{adv}
  \pi_k=\frac{w_k\tau^k}{\Phi(\tau)}
=\frac{\gl}{\gl-1}w_k\tau^k
\sim
(2\pi)\qqw\frac{\gl}{\gl-1}k^{-3/2}q^k.
\end{equation}

It follows that $\pi_{k(n)}=\Theta(1/n)$ for 
\begin{equation}\label{fork}
  k(n)=\frac{\log n-\frac32\log\log n}{\log(1/q)}+O(1),
\end{equation}
and then \eqref{nq} yields
\begin{equation}
  N\sim n \frac{\gl}{\sqrt{2\pi}(\gl-1)(1-q)}k(n)^{-3/2}
\sim \frac{\gl\log^{3/2}(1/q)}{\sqrt{2\pi}(\gl-1)(1-q)} n\log^{-3/2}n.
\end{equation}
Consequently, \refT{Ta}(ii) yields the following theorem
for the maximal tree size $\ya$; this is due to
\citet{Pavlov:max,Pavlov} (in a slightly different formulation),
who also gives further results.
We further use
Theorem \ref{TDa}(i) to give a simple estimate for the size $\yj$ of the
$j$:th largest tree.
(More precise limit results for $\yj$ are also easily obtained from
\eqref{cdx1}.)

\begin{theorem}\label{Tmax1}
For a random rooted forest,
  with $m=\gl n+O(1)$ where $1<\gl<\infty$, 
\begin{equation}\label{rwws}
  \ya\dapprox \lrfloor{\frac{\log n -\frac32\log\log n+\log b +W}{\log(1/q)}},
\end{equation}
where $W$ has the Gumbel distribution \eqref{gumbel} and
\begin{equation}
  b\=\frac{\gl\log^{3/2}(1/q)}{\sqrt{2\pi}(\gl-1)(1-q)} 
\end{equation}
with $q$ given by \eqref{aq}--\eqref{alq}.

Furthermore, $\yj=\ya+\Op(1)$ for each fixed $j$.
\nopf
\end{theorem}

Next, let us, more generally, consider a random \sgf{} as in \refE{EGWforest},
defined by a \ws{} $\wwx$. Then the tree sizes in the random forest are
distributed as balls-in-boxes with the \ws{} 
$(Z_k)_{k=0}^\infty$, where $Z_k$ is the partition function \eqref{zn} for
\sgt{s} with \ws{} $\wwx$ (and $Z_0=0$).

We assume that $\nu(\wwx)\ge1$; thus there exists 
$\tau_1>0$ such that $\Psi(\tau_1)=1$, and then
$\wwx'\=(\tau_1^kw_k/\Phi(\tau_1))_k$ is an equivalent \pws{} 
with
expectation 1, see \refL{LEPsi}. 
($\tau_1$ is the same as $\tau$ in \refT{Tmain}, but here we need to
consider several different $\tau$'s so we modify the notation.)
This \pws{} $\wwx'$ defines the same random forest,
which thus can be realized as a conditioned critical \GWf.
Recall from \eqref{lep4} and \refT{Tmain} that the probability distribution
$\wwx'$ has
variance $\gss=\tau_1\Psi'(\tau_1)$; we assume that $\gss$ is finite, which
always 
holds if $\nu(\wwx)>1$ and thus $\tau_1<\rho(\wwx)$.
We further assume, for simplicity, that $\wwx$ has span 1.
We then have the following generalization of \refT{Tmax1},
see \citet{Pavlov:limit,Pavlov}, where also further results are given.

\begin{theorem}\label{Tmax2}
Consider a \sgrf{} defined by a \ws{} $\wwx$, and assume that 
$m=\gl n+O(1)$ where $1<\gl<\infty$.
Suppose that $\nu(\wwx)\ge1$
and $\spann(\wwx)=1$. 
Define $\tau_1>0$ by $\Psi(\tau_1)=1$, and assume that
$\gss\=\tau_1\Psi'(\tau_1)<\infty$ (this is automatic if $\nu(\wwx)>1$).
Define further $\tau_2>0$ by 
\begin{equation}\label{vspsi}
  \Psi(\tau_2)=1-1/\gl
\end{equation}
and let
\begin{equation}\label{vsq}
  q\=\frac{\tau_2}{\Phi(\tau_2)}\cdot\frac{\Phi(\tau_1)}{\tau_1}.
\end{equation}
Then $0<q<1$ and 
\begin{equation}\label{wws}
  \ya\dapprox \lrfloor{\frac{\log n -\frac32\log\log n+\log b +W}{\log(1/q)}},
\end{equation}
where $W$ has the Gumbel distribution \eqref{gumbel} and
\begin{equation}\label{vsb}
  b\=\frac{\tau_1\log^{3/2}(1/q)}{\tau_2\sqrt{2\pi\gss}(1-q)}. 
\end{equation}

Furthermore, $\yj=\ya+\Op(1)$ for each fixed $j$.
\end{theorem}

\begin{proof}
Replace $\wwx$ by the equivalent \pws{} $\twwx=(\tw_k)$ with
$\tw_k\=\tau_2^kw_k/\Phi(\tau_2)$.
This \pws{} has expectation $\Psi(\tau_2)<1$ by 
\eqref{lep3}, and using it we realize the random forest as a conditioned
subcritical \GWf. 
The partition function $\tZ_k$ for $\twwx$ is by 
\eqref{tz} and \refT{TH1},
\begin{equation}\label{vstz}
\tZ_k
=\frac{\tau_2^{k-1}}{\Phi(\tau_2)^k}Z_k
\sim 
\frac{1}{\sqrt{2\pi\gss}}
\frac{\tau_2^{k-1}}{\Phi(\tau_2)^k}\cdot
\frac{\Phi(\tau_1)^k}{\tau_1^{k-1}}
{k\qqcw}.
\end{equation}
Moreover, by \eqref{znp}, $(\tZ_k)$ is the distribution of the size of a
\GWp{} with offspring distribution $\twwx$. Since this offspring
distribution is subcritical with expectation $\Psi(\tau_2)<1$, the size
distribution $(\tZ_k)$ has finite mean
\begin{equation}\label{etz}
  \sumk k\tZ_k=\frac1{1-\Psi(\tau_2)}=\gl,
\end{equation}
by our choice of $\tau_2$.

The sizes of the trees in the random forest are distributed as
balls-in-boxes with the \ws{} $(\tZ_k)$, see \refE{EGWforest}.
We apply \refT{Ta}, translating $w_k$ to $\tZ_k$.
By \eqref{vstz}, 
\begin{equation}
\tZ_{k+1}/\tZ_k\to  
a\=\frac{\tau_2}{\Phi(\tau_2)}\cdot\frac{\Phi(\tau_1)}{\tau_1},
\qquad \text{as }\ktoo.
\end{equation}
Note further that (with this \ws{} $(\tZ_k)$)
$\tau$ in
\refT{Ta} is chosen such that the equivalent \pws{} 
$\bigpar{\tau^k \tZ_k/\widetilde\cZ(\tau)}$
has expectation $\gl$. We have already constructed $(\tZ_k)$ such that 
it is a \pws{} with this expectation, see \eqref{etz}; hence we have
$\tau=1$
and $q=a$, which yields \eqref{vsq}.

As in \eqref{fork},
$\pi_{k(n)}=\tZ_{k(n)}=\Theta(1/n)$ for
\begin{equation}
  k(n)=\frac{\log n-\frac32\log\log n}{\log(1/q)}+O(1),
\end{equation}
and then \eqref{nq} yields, by \eqref{vstz},
\begin{equation}
  N\sim n \frac{\tau_1}{\sqrt{2\pi\gss}\tau_2(1-q)}k(n)^{-3/2}
\sim \frac{\tau_1\log^{3/2}(1/q)}{\tau_2\sqrt{2\pi\gss} (1-q)} n\log^{-3/2}n.
\end{equation}
The result \eqref{wws} now follows from \refT{Ta}(ii).
Finally, again, 
Theorem \ref{TDa}(i) gives  the estimate for $\yj$.
\end{proof}

\begin{example}
  Consider a random ordered rooted forest.
This is obtained by the \ws{} $w_k=1$, see \refE{EGWforest}, and we 
have by \eqref{euphi}--\eqref{eupsi}
$\Phi(t)=1/(1-t)$ and $\Psi(t)=t/(1-t)$.
Hence, $\tau_1=1/2$ and $\gss=2$ 
(see  \refE{Euniform}); furthermore, \eqref{vspsi} is
$\tau_2/(1-\tau_2)=1-1/\gl$, which has the solution
\begin{equation}
  \tau_2=\frac{\gl-1}{2\gl-1}.
\end{equation}
Consequently, \refT{Tmax2} says that \eqref{wws} holds, 
with the parameters $q$ and $b$ given by, see \eqref{vsq} and \eqref{vsb},
\begin{equation}
q
=\frac{\tau_2(1-\tau_2)}{\tau_1(1-\tau_1)}
=4\tau_2(1-\tau_2)
=\frac{4\gl(\gl-1)}{(2\gl-1)^2} 
=1-\frac{1}{(2\gl-1)^2} 
\end{equation}
and
\begin{equation}
  b=\frac{(2\gl-1)^3} {4\sqrt\pi(\gl-1)}\log\qqc(1/q).
\end{equation}
\end{example}

\begin{example}
The random rooted unlabelled forest in \refE{Eunlabelledforests}
is described by a \ws{} that also satisfies $w_k\sim c_1 k\qqcw\rho^{-k}$ 
as \ktoo,
and we thus again obtain \eqref{wws}, although the parameters $q$ and $b$
now are implicitly defined using the generating function of the number of
unlabelled rooted trees, see \citet{Pavlov:unlabelled}.  
\end{example}

\begin{example}
  For the random recursive forest in \refE{Eperm}, we have 
  \begin{equation}\label{lucia}
w_k=k^{-1}.	
  \end{equation}
Thus \refT{Ta} applies with $a=1$ and $q=\tau\in(0,1)$ given by
\begin{equation}\label{jull}
  \frac{q}{(1-q)|\log(1-q)|}=\gl,
\end{equation}
see \eqref{psiperm}.
(Recall that $\nu=\infty$, so we can take any $\gl>1$ here.)
In this case, see \eqref{phiperm}, 
$\pi_{k(n)}=k(n)\qw q^{k(n)}/|\log(1-q)|=\Theta(1/n)$ for
\begin{equation}
  k(n)=\frac{\log n-\log\log n}{\log(1/q)}+O(1),
\end{equation}
\cf{} \eqref{fork}, and then \eqref{nq} yields
\begin{equation}
  N\sim  \frac{\log(1/q)}{(1-q)|\log(1-q)|}n\log\qw n.
\end{equation}
Consequently, \refT{Ta}(ii) yields
\begin{equation}\label{jul}
  \ya\dapprox \lrfloor{\frac{\log n -\log\log n+\log b +W}{\log(1/q)}},
\end{equation}
where $W$ has the Gumbel distribution \eqref{gumbel} and, using \eqref{jull},
\begin{equation}
  b\=
\frac{\log(1/q)}{(1-q)|\log(1-q)|}
=\frac{\gl\log(1/q)}{q}.
\end{equation}
We thus obtain a result similar to the cases above, but with a different
coefficient for $\log\log n$ in \eqref{jul}.
See \citet{Pavlov:recursive} for further results.
\end{example}

If we consider the random unrooted forest in \refE{Euforest}, we
find different results.
In this case, the tree sizes are described by balls-in-boxes with the \ws{}
$w_k=k^{k-2}/k!$, $k\ge1$ (and $w_0=0$).
Alternatively, we can use the \pws{s} in \eqref{pkk}, in particular the
\pws{}, recalling $\Phi(e\qw)=1/2$ from \eqref{uforestphi},
\begin{equation}\label{utww}
  \tw_k\=\frac{w_ke^{-k}}{\Phi(e\qw)}
=2w_ke^{-k}
=\frac{2k^{k-2}e^{-k}}{k!},
\end{equation}
which by Stirling's formula satisfies
\begin{equation}\label{tull}
  \tw_k\sim 
\frac2{\sqrt{2\pi}}
k^{-5/2},
\qquad \text{as }\ktoo.
\end{equation}

Since we now  have $\nu=2<\infty$, see Examples \refand{Euforest}{EBpower},
there is a phase transition at $\gl=2$. 
We show in the theorem below that
for $\gl<2$ we have a result similar
to Theorems \refand{Tmax1}{Tmax2} with maximal tree size $\ya=\Op(\log n)$, but 
for $\gl>2$ there is a unique giant tree with size of order $n$.
At the phase transition, with $m/n\to2$, the result depends on the rate of
convergence of $m/n$; if, for example, $m=2n$ exactly,  the maximal size is
of order $n\qqqb$;
see further \citet{LP}, where precise results for general $m=m(n)$ are given.
(By the proof below, (iii) in the following theorem holds as soon as
$m/n\to\gl>2$, but (i) and (ii) are more sensitive.)

\begin{theorem}\label{Tmaxu}
Consider a random unrooted forest, and assume that 
$m=\gl n+O(1)$ where $1<\gl<\infty$.
\begin{romenumerate}
\item 
If $1<\gl<2$, let
\begin{equation}\label{uvsq}
  q\=2\frac{\gl-1}{\gl}e^{2/\gl-1}.
\end{equation}
Then $0<q<1$ and 
\begin{equation}\label{uwws}
  \ya\dapprox \lrfloor{\frac{\log n -\frac52\log\log n+\log b +W}{\log(1/q)}},
\end{equation}
where $W$ has the Gumbel distribution \eqref{gumbel} and
\begin{equation}
  b\=\frac{\gl^2\log^{5/2}(1/q)}{2\sqrt{2\pi}(\gl-1)(1-q)} 
\end{equation}

Furthermore, $\yj=\ya+\Op(1)$ for each fixed $j$.

\item 
If $\gl=2$, then 
\begin{equation}\label{uuu}
  \yj/n^{2/3}\dto \eta_j
\end{equation}
for each $j$, where $\eta_j>0$ are some random variables.
The distribution of $\eta_1$ is given by \eqref{eta1} with $\ga=3/2$ and
$c=(2/\pi)\qq$. 
\item 
If $2<\gl<\infty$, then 
$\ya=(\gl-2)n+\Op(n^{2/3})$. More precisely,
\begin{equation}
  n^{-2/3}\bigpar{m-2n-\ya}\dto X,
\end{equation}
where $X$ is a $\frac32$-stable 
random variable with Laplace transform
\begin{equation}
  \E e^{-tX} = \exp\Bigpar{\frac{2^{5/2}}3t\qqc},
\qquad \Re t\ge0.
\end{equation}

For $j\ge2$, $\yj=\Op(n\qqqb)$, and $n\qqqbw \yj\dto W_j$ where $W_2$ has
the Fr\'echet distribution 
\begin{equation}
  \P(W_2\le x)=\exp\Bigpar{-\frac{2\qqc}{3\sqrt\pi}x^{-3/2}},
\qquad x\ge0.
\end{equation}
and, more generally,  $W_j$ has the density function \eqref{wjdense}
with $c'=(2/\pi)\qq$ and $\ga=3/2$.
\end{romenumerate}
\end{theorem}

Note that the exponents $\frac32$, $1$ and $\frac52$ in
\eqref{wkb}, \eqref{lucia} and 
\eqref{tull} appear as coefficients of $\log\log n$ in \eqref{rwws},
\eqref{jul} and \eqref{uwws}, respectively.

\begin{proof}
  \pfitem{i}
This is very similar to the proofs of Theorems \refand{Tmax1}{Tmax2}.
We use $w_k=k^{k-2}/k!$. Then, as for rooted forests and \eqref{wkb} above,
$w_{k+1}/w_k\to e$ as \ktoo.
Further, $\tau$ is given by \eqref{uforesttau}, and thus $q\=\tau e$ is
given by \eqref{uvsq}.
It follows, \cf{} \eqref{fork} and \eqref{tull}, that
$\pi_{k(n)}=\Theta(1/n)$ for 
\begin{equation}\label{ufork}
  k(n)=\frac{\log n-\frac52\log\log n}{\log(1/q)}+O(1),
\end{equation}
and then \eqref{nq} yields
\begin{equation}
  \begin{split}
  N
&\sim 
\frac{n w_{k(n)}e^{-k(n)}}{\Phi(\tau)(1-q)}
\sim n \frac{\gl^2}{2\sqrt{2\pi}(\gl-1)(1-q)}k(n)^{-5/2}
\\&
\sim \frac{\gl^2\log^{5/2}(1/q)}{2\sqrt{2\pi}(\gl-1)(1-q)} n\log^{-5/2}n.	
  \end{split}
\end{equation}
Hence \refT{Ta}(ii) yields \eqref{uwws}.

\pfitem{ii}
We use the equivalent \pws{} $\tww$ given by \eqref{utww}.
By \eqref{tull}, it satisfies the assumptions in \refE{EL8} with $\ga=3/2$
and $c=(2/\pi)\qq$; thus \eqref{uuu} follows from \eqref{limetaj}, and
\eqref{eta1} in \refR{Reta} applies.

\pfitem{iii}
We use again the \pws{} $\tww$ and apply \refT{TDzeta}.
We have $c'=c=(2/\pi)\qq$ by \eqref{utww}, and thus 
$c'\Gamma(-3/2)=c'\frac43\Gamma(1/2)=2^{5/2}/3$ and 
$c'/\ga=2^{3/2}/(3\sqrt\pi)$.
\end{proof}

\begin{example}
The random unrooted unlabelled forest (with labelled trees)
in \refE{Eunlabelledforests}
is described by another \ws{} that satisfies $w_k\sim c k^{-5/2}\rho^{-k}$ 
as \ktoo,
and we thus obtain a result similar to \refT{Tmaxu}, although the parameters
differ (they can be obtained from the generating function of the number of
unlabelled trees); in particular, the phase transition appears when $\gl$ is 
$\nu\approx 2.0513$, 
see \citet{Pavlov:unun} for details.
\end{example}

We do not know any corresponding results for
random completely unlabelled forests ($n$ unlabelled trees consisting of $m$
unlabelled nodes); as said in \refE{Eunlabelledforests}, they
cannot be described by \bib.

\section{Large nodes in simply generated trees with  $\nu<1$}\label{SlargeT}

In the tree case with $\nu<1$, the results in \refSS{SSlarge+} 
show condensation in the
form of one or, sometimes, several nodes with very large degree, together
making up the 
``missing  mass'' of about $(1-\nu)n$. On the other hand, \refT{Tmain} shows
concentration in a somewhat different form, with a limit tree $\hcT$ having
exactly one node of infinite degree. 
This node corresponds to a node with very large degree in $\ctn$ for $n$
large but finite. How large is the degree? Why do we only see one node with
very large degree in \refT{Tmain}, but sometimes several nodes with large
degrees above (Examples \ref{Eoligo} and \ref{Eoligo0})?

The latter question is easily answered: recall that the convergence in
\refT{Tmain} means convergence of the truncated trees (``left balls'')
$T_n\xxmm$, see \refL{LC}; thus we only see a small part of the tree close
to the root, and the two pictures above are reconciled:
if $m$ is large but fixed, then in the set $V(T)\cap V\xxmm$ of nodes, there
is with probability close to 1 exactly one node with very large degree.
(There may be several nodes with very large degree in the tree, but for any
fixed $m$, \whp{} at most one of them is in $V\xxmm$.)
Of course, to make this precise, we would have to define ``very large'', for
example as below using a sequence $\gO_n$ growing slowly to $\infty$ as in
\refL{LQ}, 
but we are at the moment satisfied with an intuitive description.

To see how large the ``very large'' degree is, let us first look at the root.
\refL{Lroot} says that the distribution of the root degree is the
size-biased distribution of $Y_1$. We can write \eqref{lroot} as
\begin{equation}  \label{tom}
  \P(\dx_{\ctn}(o)=d)
=
\frac{d}{n-1} \sumin \P(Y_i=d)
=
\frac{d}{n-1} \sumjn \P(\yj=d);
\end{equation}
hence the distribution of the root degree can be described by:
sample $\YYn$ and then take $\ya$ with probability $\ya/(n-1)$,
$\yb$ with probability $\yb/(n-1)$, \dots.

In particular, if $\ya=(1-\nu)n+\op(n)$, then \eqref{tom}
implies
\begin{equation} 
  \P(\dx_{\ctn}(o)=\ya)
=1-\nu+o(1),
\end{equation}
and comparing with \refT{Troot} we see that \whp{} either the root degree is
small (more precisely, $\Op(1)$), or it is the maximum outdegree $\ya$.
However, we also see that if 
$\ya$ is not $(1-\nu)n+\op(n)$, then this conclusion does not hold; 
for example, in \refE{Eoligo0} for $n$ in the subsequence $(2^i)$ where
\eqref{oligo0} holds for each fixed $j$,
\begin{equation} 
  \P(\dx_{\ctn}(\rot)=2^{-j}n)\to 2^{-j}.
\end{equation}

In the case $\nu=0$, we only have to consider the root, since
the node with infinite degree in $\hcT$ always is the root,
but for $0<\nu<1$, 
the node with infinite degree in $\hcT$ may be somewhere else.
We shall see that it corresponds to a node in $\ctn$ with a large degree having 
(asymptotically) the same distribution as the root degree just considered,
conditioned to be ``large''.

To make this precise,
let $\gO_n\to\infty$ be a fixed sequence which increases so slowly that
\refL{LQ}(ii) holds. We say that an outdegree $\dx(v)$ is \emph{large} if it is
greater than $\gO_n$; we then also say that the node $v$ is large. 
(Note that by \refL{LQ}(ii), \whp{} at least one large node exists.) 
For each $n$, let $\yxx$ by a \rv{} whose distribution is the size-biased
distribution of a large outdegree, \ie{} of $(Y_1\mid Y_1>\gO_n)$:
\begin{equation}\label{lollo}
  \P(\yxx=k)
=\frac{k\P(Y_1=k)}{\sum_{l>\gO_n}l\P(Y_1=l)}
=\frac{k\E N_k}{\sum_{l>\gO_n}l\E N_l}
=\frac{k\E N_k}{(1-\nu+o(1))n},
\end{equation}
for $k>\gO_n$ and $\P(\yxx=k)=0$ otherwise.
Equivalently, in view of \refL{Lroot}, $\yxx$ has the distribution of the
root degree $\dx_{\ctn}(\rot)$ conditioned to be greater than $\gO_n$. 
See also
\eqref{tom}, and note that if $\ya=(1-\nu)n+\op(n)$, then $\yxx\dapprox\ya$,
\ie, we may take $\yxx=\ya$ \whp; in this case (but not otherwise) we thus
have $\yxx=(1-\nu)n+\op(n)$.

Note that if $\gO_n'$ is another such sequence, similarly defining a random
variable $\yxx'$, then 
$\sum_{l>\gO_n}l\P(Y_1=l)\sim(1-\nu)n\sim \sum_{l>\gO'_n}l\P(Y_1=l)$, and
  it follows that $\yxx\dapprox \yxx'$; hence the choice of $\gO_n$ will not
  matter below.

We claim that, \whp, 
the infinite outdegree in $\hct$ corresponds to an outdegree $\yxx$ in $\ctn$.
To formalise this,
recall from \refS{SUlam} that we may consider our trees as subtrees of the
infinite tree $\too$ with node set $\voo$, and that the convergence of trees
defined there
means convergence of each $\dx(v)$, see \eqref{dconv}.
Let $\hcT$ be the random infinite tree defined in \refS{ShGW}; we are in
case \tii, and thus $\hcT$ has a single node $v$ with outdegree
$\dx_{\hct}(v)=\infty$. We assume that $\yxx$ and $\hcT$ are independent,
and define the modified degree sequence
\begin{equation}\label{dxxhct}
  \dxxhct(v)\=
  \begin{cases}
	\dx_\hct(v),& 	\dx_{\hct}(v)<\infty,\\
	\yxx,& 	\dx_{\hct}(v)=\infty.
  \end{cases}
\end{equation}
We thus change the single infinite value to the finite $\yxx$, leaving all
other values unchanged.
(Note that $\dxxhct(v)$ may depend on $n$, since $\yxx$ does.)
We then have the following theorem.

\begin{theorem}\label{TT}
For any finite set of nodes $v_1,\dots,v_\ell\in\voo$,
  \begin{equation}\label{tt}
\bigpar{\dx_{\ctn}(v_1),\dots,\dx_\ctN(v_\ell)}	
\dapprox
\bigpar{\dxx_{\hct}(v_1),\dots,\dxx_\hct(v_\ell)}.
  \end{equation}
\end{theorem}
\begin{proof}
Let $\eps>0$, and let $\vx$ denote the unique node in $\hct$ with
$\dx_\hct(\vx)=\infty$. 
By increasing the set \set{v_1,\dots,v_\ell}, we may assume that it equals
$V\xxmm$  (see \refS{SUlam}) for some $m$, and that $m$ is so large that
$\P(\vx\in V\xxmm)>1-\eps$.
We may then find $K<\infty$ such that 
\begin{equation*}
  \P\bigpar{\dx_\hct(v)\in(K,\infty) \text{ for some } v\in V\xxmm}<\eps.
\end{equation*}
Since $\ctn\dto\hct$ by \refT{Tmain}, we may by
the Skorohod coupling theorem \cite[Theorem 4.30]{Kallenberg} assume that
the random trees are coupled such that $\ctn\to\hct$ a.s., and thus 
$\dx_\ctN(v)\to\dx_\hct(v)$ \as{} for every $v$.
Then, for large $n$, with probability $>1-3\eps$, $\vx\in V\xxmm$, 
$\dx_\ctN(v)=\dx_\hct(v)=\dxxhct(v)\le K$ for all $v\in
V\xxmm\setminus\set{\vx}$, and
$\dx_\ctN(\vx)\to\dx_\hct(\vx)=\infty$.
We may assume that $\gO_n\to\infty$ so slowly that furthermore
$\P(\dx_\ctN(\vx)\le\gO_n)\le\eps$. (Recall that we may change $\gO_n$
without affecting the result \eqref{tt}.)

Let $n$ be so large that also $\gO_n>m$  and $\gO_n>K$.
It follows from \refL{Ldeg} that for each choice of $v'\in V\xxmm$ and
numbers $d(v)$ for $v\in V\xxmm\setminus{v'}$, 
and $k>\gO_n$,
\begin{multline*}
\P\bigpar{\dx_\ctN(v)=d(v) \text{ for }v\in V\xxmm\setminus\set{v'} 
\text{ and }\dx_\ctN(v')=k}
\\
= 
  \bigpar{k+O(1)}C(\set{d(v)},v',n)\P(Y_1=k)
\end{multline*}
for some constant $C(\set{d(v)},v',n)\ge0$ not depending on $k$;
hence, by \eqref{lollo},
\begin{multline*}
\P\bigpar{\dx_\ctN(v')=k \mid
\dx_\ctN(v)=d(v) \text{ for }v\in V\xxmm\setminus\set{v'} 
\text{ and } \dx_\ctN(v')>\gO_n}
\\
= 
  \etto \frac{ k\P(Y_1=k)}{\sum_{k>\gO_n}k\P(Y_1=k)}
=\etto\P(\yxx=k).
\end{multline*}
There is only a finite number of choices of $v'$ and 
$(d(v))_{v\in  V\xxmm\setminus\set{v'}}$,  and it follows
that we may choose the coupling of $\ctn$ and $\hct$ above such that also
$\dx_\ctN(\vx)=\yxx$ \whp; thus, with probability $>1-4\eps-o(1)$, 
$\dx_\ctN(v)=\dxxhct(v)$ for all $v\in V\xxmm$.

The result follows since $\eps>0$ is arbitrary.
\end{proof}

We give some variations of this result, where we replace
$\dxxhct(v)$ by the degree sequences of some random trees obtained by
modifying $\hct$.
(Note that $\dxxhct(v)$ is not the degree sequence of a tree.)

First, 
let $\hcta$ be the random tree obtained by 
pruning the tree $\hct$ at the node $\vx$ with infinite outdegree,
keeping only the first $\yxx$ children of $\vx$. 
Then $\hcta$ is a locally finite tree, and, in fact, it is \as{} finite.
The random tree $\hcta$ can be constructed as $\hct$ in \refS{ShGW},
starting with a spine, and then adding independent \GWt{s} to it, but 
now the number of children of a node in the spine is
given by a finite \rv{} $\xxxi$ with the distribution 
\begin{equation}\label{sp+}
  \P(\xxxi=k)=\P(\hxi=k)+\P(\hxi=\infty)\P(\yxx=k)
=k\pi_k+(1-\nu)\P(\yxx=k).
\end{equation}
The nodes not in the spine (the normal nodes) have offspring distribution
$\ppi$ as before. (This holds also for the following modifications.)

The spine in $\hcta$ stops when we obtain $\hxi=\infty$, but we may also
define another random tree $\hctb$  by continuing the spine
to infinity; this defines a random
infinite but locally finite tree having an infinite spine; each node in the
spine has a number of children with the distribution in \eqref{sp+}, and
the spine continues with a uniformly randomly chosen child. Equivalently,
$\hctb$ can be defined by a \GWp{} with normal and special nodes as in
\refS{ShGW}, but with the offspring distribution for special nodes changed
from \eqref{hxi} to \eqref{sp+}.

Finally, let $\hyy$ by a random variable with the size-biased distribution of
$Y_1$:
\begin{equation}\label{hyy}
  \P(\hyy=k)=\frac{k\P(Y_1=k)}{(n-1)/n}
=\frac{k\E N_k}{n-1},
\end{equation}
recalling that $\sum_k kN_k=n-1$;
\cf{} \eqref{tom} and \eqref{lollo}. (Thus $\hyy\eqd\dx_\ctN(\rot)$ by
\refL{Lroot} and  \eqref{tom}.) 
Define the infinite, locally finite random tree $\hctc$ by the same \GWp{}
again, but now with offspring distribution $\hyy$ for special nodes.
(This does not involve $\yxx$ or $\gO_n$.)
Thus $\hctc$ also has an infinite spine.

We then have  the following version of \refT{TT},
where we also use the metric $\gdd$ on $\stl$ defined by 
\begin{equation}
  \gdd(T_1,T_2)\=
1/\sup\bigset{m\ge1:\dx_{T_1}(v)=\dx_{T_2}(v) \text{ for }v\in V\xxmm}.
\end{equation}

\begin{theorem}
  For $j=1,2,3$, and any 
 finite set of nodes $v_1,\dots,v_\ell\in\voo$,
  \begin{equation}
\bigpar{\dx_{\ctn}(v_1),\dots,\dx_\ctN(v_\ell)}	
\dapprox
\bigpar{\dxx_{\hctj}(v_1),\dots,\dxx_\hctj(v_\ell)}.
  \end{equation}
Equivalently, there is a coupling of $\ctn$ and $\hctj$ such that
$\gdd(\ctn,\hctj)\to0$ as \ntoo.
\end{theorem}

\begin{proof}
  If $\gO_n>m$ we have $\yxx>m$
and then the branches of $\hct$ pruned to make $\hcta$
are all outside $V\xxmm$, and thus 
$\dx_\hcta=\dxxhct$ defined in \eqref{dxxhct} for all $v\in V\xxmm$.
Thus the result for $\hcta$ follows from \refT{TT}.

Next,  for any given $m$, 
and for any endpoint $x$ of the spine of $\hcta$,
 the probability that the  continuation in $\hctb$ of the spine 
contains some node in $V\xxmm$ is less than $m/\gO_n=o(1)$; thus, \whp{}
$\hcta$ and $\hctb$ are equal on 
any $V\xxmm$.

Finally,  \refL{Lhyy} below implies that we can couple $\hctb$ and $\hctc$
such that they \whp{} agree on each $V\xmm$; then  $\hcta$ and $\hctc$
are \whp{} equal on each $V\xxmm$.
\end{proof}

\begin{lemma}\label{Lhyy}
  $\xxxi\dapprox\hyy$.
\end{lemma}

\begin{proof}
  For each fixed $k$, $\P(\xxxi=k)=k\pi_k$ as soon as $\gO_n>k$, and 
$\P(\hyy=k)\to k\pi_k$ by \eqref{hyy} and \refT{TB3}.
Hence, 
\begin{equation}\label{xxxihyy}
|\P(\xxxi=k)-\P(\hyy=k)|\to0.  
\end{equation}

By \eqref{sp+}, \eqref{lollo} and \eqref{hyy}, uniformly for $k>\gO_n$,
\begin{equation*}
  \P(\xxxi=k)=k\pi_k+(1-\nu)\frac{k\P(Y_1=k)}{1-\nu+o(1)}
=k\pi_k+\etto\P(\hyy=k);
\end{equation*}
hence
\begin{equation}\label{cec1}
  \sum_{k>\gO_n} |\P(\xxxi=k)-\P(\hyy=k)|\le
  \sum_{k>\gO_n} \bigpar{k\pi_k+o(1)\P(\hyy=k)}
=  \sum_{k>\gO_n}k\pi_k+o(1).
\end{equation}
Further, for any fixed $K$,
\begin{equation}\label{cec2}
    \sum_{k=K+1}^{\gO_n} \bigpar{\P(\xxxi=k)-\P(\hyy=k)}_+\le
    \sum_{k=K+1}^{\gO_n} \P(\xxxi=k)
=    \sum_{k=K+1}^{\gO_n} k\pi_k.
\end{equation}
Using 
\refL{LM}\ref{lmf} together with
\eqref{xxxihyy} for $k\le K$, \eqref{cec1} and \eqref{cec2} 
we obtain 
\begin{equation}
\dtv(\xxxi,\hyy)=
    \sum_{k=1}^{\infty} \bigpar{\P(\xxxi=k)-\P(\hyy=k)}_+\le
  \sum_{k=K+1}^{\infty} k\pi_k+o(1).
\end{equation}
Since $K$ is arbitrary and
$ \sum_{1}^{\infty} k\pi_k<\infty$,
it follows that $\dtv(\xxxi,\hyy)\to0$.
\end{proof}

\section{Further results and problems}\label{Sfurther}

\subsection{Level widths}

Let, as in \refR{Relz}, 
$\elz_k(T)$ denote the number of nodes with distance $k$ to the root in a
rooted tree $T$.

If $\nu\ge1$, then $\hcT$ is a locally finite tree so all level widths
$\elz_k(\hcT)$ are finite.
It follows easily from the characterisation of convergence in \refL{LClf}
that, in this case, the functional $\elz_k$ is continuous at $\hcT$, and
thus \refT{Tmain} implies 
 (see \citetq{Corollary 1, p.~31}{Billingsley})
\begin{equation}\label{llim}
\elz_k(\ctn)\dto \elz_k(\hcT) <\infty  
\end{equation}
for each $k\ge0$.

On the other hand, if $\nu<1$, 
then $\hcT$ has a node with infinite outdegree; this 
node has a random distance $L-1$ to the root, where $L$ as in \refS{ShGW} is
the length of the spine, and thus $\elz_L(\hcT)=\infty$.

In the case $0<\nu<1$, we have $\pi_0<1$ and $\P(\xi\ge1)=1-\pi_0>0$, so for
any $j$, there is a positive probability that the \GWt{} $\cT$ has height at
least $j$, and it follows that of the infinitely many copies of $\cT$ that
start in generation $L$, \as{} infinitely many will survive at least until
generation $L+j$. Consequently, \as,
$\elz_k(\hcT)=\infty$ for all $k\ge L$,
while $\elz_k(\hcT)<\infty$ for $k<L$.
It follows easily from \refL{LC}, that in this case too, 
for each $k\ge0$,
the mapping  $\elz_k:\st\to\bNo$ is continuous at $\hcT$.
Consequently,
\begin{equation}\label{eri}
\elz_k(\ctn)\dto \elz_k(\hcT) \le\infty,
\qquad k=0,1,\dots,  
\end{equation}
with $\P(\elz_k(\hcT)<\infty)=\P(L>k)=\nu^{k}$.
(Recall that $\mu=\nu$ in this case by \eqref{mainmu}.)

When $\nu=0$, however, \eqref{eri} does not always hold.
By \refE{Emu=0}, $\hcT$ is an infinite star, with 
$\elz_1(\hcT)=\infty$ and 
$\elz_k(\hcT)=0$ for all
$k\ge2$. 
By \refT{Troot},  $\elz_1(\ctn)=\dx_{\ctn}(o)\dto\hxi\eqd\elz_1(\hcT)$,
so \eqref{eri} holds for $k=1$ (and trivially for $k=0$)
in the case $\nu=0$ too (with $\elz_1(\hcT)=\infty$).
However, by \refE{Ek!}, if $w_k=k!$, then
$\elz_2(\ctn)\dto\Po(1)$, so $\elz_2(\ctn)$ does not converge to
$\elz_2(\hcT)=0$. 
Similarly, by \refE{Ek!a},
if $j\ge2$ and
$w_k=k!^\ga$ with $0<\ga<1/(j-1)$, then 
the number of paths of length $j$ attached to the root in $\ctn$ tends to
$\infty$ (in probability), so 
$\elz_j(\ctn)\pto\infty$, while $\elz_j(\hcT)=0$.

Turning to moments, we have for the expectation, by \eqref{elk},
$\E\elz_k(\hcT)=\infty$ if $0<\nu<1$ or $\gss=\infty$; in this case
\eqref{llim}--\eqref{eri} and Fatou's lemma yield
$\E\elz_k(\ctn)\to\E\elz_k(\hcT)=\infty$.

If $\nu\ge1$ and $\gss<\infty$, then \eqref{elk0} yields
$\E\elz_k(\hcT)=1+k\gss<\infty$.
In this case, for each fixed $k$, the random variables $\elz_k(\ctn)$,
$n\ge1$, are uniformly integrable, and thus \eqref{llim} implies
$\E\elz_k(\ctn)\to\E\elz_k(\hcT)$, see \citetq{Section 10}{SJ167}.
(In the case $\nu>1$, this was shown already by \citet{MM78}.)
Consequently, for any $\wwx$ with $\rho>0$ and any fixed $k$, 
\begin{equation}\label{jef}
\E\elz_k(\ctn)\to\E\elz_k(\hcT)\le\infty.
\end{equation}
(When $\rho=0$, this is not always true, by the examples above.)

For higher moments, there remains a small gap. 
Let $r\ge1$. 
When $0<\nu<1$, \eqref{jef} trivially implies
$\E\elz_k(\ctn)^r\to\E\elz_k(\hcT)^r=\infty$,
so suppose $\nu\ge1$.
Then, by \eqref{hxi},
$\E\hxi^r=\E\xi^{r+1}$, so if $\E\xi^{r+1}=\infty$, then
$\E\elz_1(\hcT)^r=\infty$; moreover, each $\elz_k(\hcT)$, $k\ge1$,
stochastically dominates $\hxi$ (consider the offspring of the $k$:th node
on the spine), and thus
$\E\elz_k(\hcT)^r=\infty$ for every $k\ge1$.
Consequently, again immediately by Fatou's lemma and \eqref{eri}, 
$\E\elz_k(\ctn)^r\to\E\elz_k(\hcT)^r=\infty$.
The only interesting case is thus when $\E\xi^{r+1}<\infty$.
If $r\ge1$ is an integer, it was shown in \cite[Theorem 1.13]{SJ167} that
$\E\xi^{r+1}<\infty$ implies that
$\E\elz_k(\ctn)^r$, $n\ge1$, are uniformly bounded for each $k\ge1$.
We conjecture that, moreover, 
$\elz_k(\ctn)^r$, $n\ge1$, are uniformly integrable, which by \eqref{llim}
would yield the following:

\begin{conjecture}
  For every integer $r\ge1$ and every $k\ge1$, if $\nu>0$, then
  \begin{equation}
	\label{jullan}
\E\elz_k(\ctn)^r\to\E\elz_k(\hcT)^r\le\infty.
  \end{equation}
We further conjecture that this holds also for non-integer $r>0$.
\end{conjecture}

One thus has to consider the case $\E\xi^{r+1}<\infty$ only, and the result
from \cite{SJ167} implies that \eqref{jullan} holds if $\E\xi^{r+2}<\infty$,
since then $\E\elz_k(\ctn)^{\floor r+1}$ are uniformly bounded.

\subsection{Asymptotic normality}\label{SSnormal}
In \refT{Tdegree}, we proved that $N_d$, the number of nodes of outdegree
$d$ in the random tree $\ctn$, satisfies $N_d/n\pto\pi_d$.

In our case \Iga{} ($\nu>1$ or $\nu=1$ and $\gss<\infty$), 
Kolchin \cite[Theorem 2.3.1]{Kolchin} 
gives the much
stronger result that the random variable $N_d$ is 
asymptotically normal, for every $d\ge0$:
\begin{equation}
  \label{kol}
\frac{N_d-n\pi_d}{\sqrt n}\dto N(0,\gss_d),
\end{equation}
with
\begin{equation}
  \gss_d\=\pi_d\Bigpar{1-\pi_d-\frac{(d-1)^2\pi_d}{\gss}}.
\end{equation}
(In fact, \citet{Kolchin} gives a local limit theorem which is a stronger
version of \eqref{kol}.)

Under the assumption $\xi^3<\infty$, \citetq{Example 3.4}{SJ132}
gave another proof of \eqref{kol}, and showed further joint convergence for
different $d$, with asymptotic covariances, using $I_k\=\ett{\xi=k}$,
\begin{equation}
  \gss_{kl}
=\Cov(I_k,I_l)-\frac{\Cov(I_k,\xi)\Cov(I_l,\xi)}{\Var \xi}
=
\pi_k\gd_{kl}-\pi_k\pi_l-\frac{(k-1)(l-1)\pi_k\pi_l}{\gss}.
\end{equation}

Moreover, \citet{SJ132} showed that if $\E|\xi|^r<\infty$ for every $r$
(which in particular holds when $\nu>1$ since then $\tau<\rho$ and $\xi$ has
some exponential moment),
then convergence of all moments and joint moments holds in \eqref{kol}; in particular
\begin{equation}
  \E N_k=n\pi_k+o(n)
\quad\text{and}\quad
  \Cov( N_k,N_l)=n\gss_{kl}+o(n).
\end{equation}

In the case $\nu>1$, 
\citet{Minami} and
Drmota \cite[Section 3.2.1]{Drmota} have given other proofs 
of the (joint) asymptotic normality using the
saddle point  method;
\citet{Drmota} shows further the
stronger moment estimates
\begin{equation}
  \E N_d=n\pi_d+O(1)
\quad\text{and}\quad
  \Var N_d=n\gss_d+O(1).
\end{equation}

\begin{problem}
  Do these results hold in the case $\nu=1$, $\gss<\infty$ without extra
  moment conditions?
  Do they extend to the case $\nu=1$, $\gss=\infty$?
What happens when $0\le\nu<1$?
\end{problem}

\begin{problem}
  Extend this to the more general case of \bib{} as in \refT{TBmain}.
(We guess that the case $0<\gl<\nu$ is easy by the methods in 
the references above, in particular \cite{SJ132} and \cite[Section
	3.2.1]{Drmota}, but we have not checked the details.)
\end{problem}

\begin{problem}
  Extend this to the subtree counts in \refT{Tsubtree}.
\end{problem}

\subsection{Height and width}

We have studied the random trees $\ctn$ without any scaling. Since our mode
of convergence really means that we consider only a finite number of
generations at a time, we are really looking at the base of the tree, with
the first generations. The results in this paper thus do not say anything
about, for example, the height and width of $\ctn$.
(Recall that if $T$ is a rooted tree, 
then the  height $H(T)\=\max\set{k:\elz_k(T)>0}$, the maximum distance from
the root, 
and the width  $W(T)\=\max_k\set{\elz_k(T)}$, the largest size of a generation.)
However, there are  other known results. 

In the case $\nu\ge1$, $\gss<\infty$ (the case \Iga{} in \refS{S3}), 
it is well-known that both the height $H(\ctn)$ and the width $W(\ctn)$ of
$\ctn$ typically are of order $\sqrt n$; more precisely,
\begin{align}
H(\ctn)/\sqrt n&\dto 2\gs\qw X,\label{hlim}
\\  
W(\ctn)/\sqrt n&\dto \gs X, \label{wlim}
\end{align}
where $X$ is some strictly positive random
variable (in fact, $X$ equals the maximum of a standard Brownian
excursion and has what is known as a theta distribution), 
see \eg{} 
\citet{Kolchin},
\citet{AldousII}, \citet{ChMY}, \citet{SJ167} and
\citet{Drmota}.
There are also results for a single level giving an asymptotic distribution
for $\elz_{k(n)}(T_n)/\sqrt n$ when the level $k(n)\sim a\sqrt n$ for some
$a>0$,  
see \citetq{Theorem 2.4.5}{Kolchin}.

Since the variance $\gss$ appears as a
parameter in these results, we
cannot expect any simple extensions to the case $\gss=\infty$, and even less
to the 
case $0\le \nu<1$. 
Nevertheless, 
we  conjecture that \eqref{hlim} and \eqref{wlim} extend formally
at least
to the case $\nu=1$ and $\gss=\infty$:

\begin{conjecture}
  If $\nu=1$ and $\gss=\infty$, then $H(\ctn)/\sqrt n\pto 0$.
\end{conjecture}

\begin{conjecture}
  If $\nu=1$ and $\gss=\infty$, then $W(\ctn)/\sqrt n\pto \infty$.
\end{conjecture}

\begin{problem}
    Does $\nu<1$ imply that $H(\ctn)/\sqrt n\pto 0$?
\end{problem}

\begin{problem}
    Does $\nu<1$ imply that $W(\ctn)/\sqrt n\pto \infty$?
\end{problem}

Furthermore, still in the case
$\nu\ge1$, $\gss<\infty$,
\citet{SJ250} have shown sub-Gaussian tail estimates for the height and
width
\begin{align}
  \P(H(\ctn)\ge x\sqrt n)& \le C e^{-c x^2}, \label{hsup}\\
  \P(W(\ctn)\ge x\sqrt n)& \le C e^{-c x^2}, \label{wsup}
\end{align}
uniformly in all $x\ge0$ and $n\ge1$ (with some positive constants $C$ and $c$
depending on $\ppix$ and thus on $\wwx$). In view of \eqref{wlim}, we cannot
expect \eqref{wsup} to hold when $\gss=\infty$ (or when $\nu<1$), but we see
no reason why \eqref{hsup} cannot hold; \eqref{hlim} suggests that $H(\ctn)$
typically is smaller when $\gss=\infty$. 

\begin{problem}
  Does \eqref{hsup} hold for any \ws{} $\wwx$ (with $C$ and $c$ depending on
  $\wwx$, but not on $x$ or $n$)?
\end{problem}

It follows from \eqref{hlim}--\eqref{wlim} and \eqref{hsup}--\eqref{wsup}
that $\E H(\ctn)/\sqrt n$ and $\E W(\ctn)/\sqrt n$ converge to positive
numbers. (In fact, the limits are $\sqrt{2\pi}/\gs$ and $\sqrt{\pi/2}\,\gs$,
see \eg{} \citet{SJ174}, where also joint moments are computed.)

\begin{problem}
What are the growth rates of
$\E H(\ctn)$ and $\E W(\ctn)$ when $\gss=\infty$ or $\nu<1$?
\end{problem}

\subsection{Scaled trees}

The results \eqref{hlim}--\eqref{wlim}, as well as many other 
results on various asymptotics of $\ctn$ in the case $\nu\ge1$, $\gss<\infty$, 
can be seen as consequences of the convergence of the tree 
$\ctn$, after rescaling in a suitable sense in both height and width 
by $\sqrt n$, 
to the \emph{continuum random tree} defined by
\citet{AldousI,AldousII,AldousIII}, see also \citet{LeGall}.
(The continuum random tree is not an  ordinary tree; it is a compact metric
space.) 
This has been extended to the case $\gss=\infty$ when $\ppix$ is in the
domain of attraction of a stable distribution, see
\eg{} \citet{Duquesne} and \citet{LeGall,LeGallreal}; 
the limit is now a different random
metric 
space called a
\emph{stable tree}.

\begin{problem}
  Is there some kind of similar limiting object in the case $\nu<1$
(after suitable scaling)?
\end{problem}

\subsection{Random walks}

Simple random walk on the infinite random tree $\hcT$ has been studied by
many authors in the critical case $\nu\ge1$, in particular when
$\gss<\infty$,
see \eg{}
\citet{Kesten,BarlowKumagai,DurhuusJW,FujiiKumagai},
but also when $\gss=\infty$,
see \citet{CroydonKumagai}
(assuming attraction to a stable law).

A different approach 
is to study simple random walk on $\ctn$ 
and 
study asymptotics os \ntoo.
For example, by rescaling the tree one can obtain convergence to  a process
on the continuum 
random tree (when $\gss<\infty$)
or stable tree (assuming attraction to a stable law), see 
\citet{Croydon,Croydon2010}.

For $\nu<1$, the simple random walk on $\hcT$ does not make sense, since
the tree has 
a node with infinite degree.
Nevertheless, it might be interesting to study simple random walk on $\ctn$
and find asymptotics of interesting quantities as \ntoo.

\subsection{Multi-type \cGWt{s}}
It seems likely that there are results similar to the ones in \refS{Smain}
for multi-type \GWt{s}
conditioned on the total size, or perhaps on the number of nodes of each
type, and for corresponding generalizations of  simply generated random
trees. 
We are not aware of any such results, however,
and leave this as an open problem. (See \citet{KLPP} for related results that
presumably are useful.)

\section{Different conditionings for \GWt{s}}\label{Scondx}
One of the principal objects studied in this paper is the \cGWt{}
$(\cT\mid|\cT|=n)$,  
\ie{} a \GWt{} $\cT$ conditioned on its total size being $n$; we then let
\ntoo. This is one way to consider very large \GWt{s}, but there are also
other similar conditionings.
For comparison, we briefly consider two possibilities;
see further \citet{Kennedy} and \citet{AldousPitman}.
We denote the offspring distribution by $\xi$ and its \pgf{} by $\Phi(t)$.

\subsection{Conditioning on $|\cT|\ge n$.}

If $\E\xi\le1$, \ie, in the subcritical and critical cases,
$\vct<\infty$ \as{} and thus $\cT$ conditioned on $\vct\ge n$ is a mixture
of $(\cT\mid\vct=N)=\cT_N$ for $N\ge n$. It follows immediately from
\refT{Tmain} that
$(\cT\mid\vct=N)\dto\hcT$ as \ntoo.

If $\E\xi>1$,  \ie, in the supercritical case, on the other hand, 
the event $\vct=\infty$ has positive probability, and the events $\vct\ge n$
decrease to $\vct=\infty$. Consequently,
\begin{equation}\label{oct}
  (\cT\mid\vct\ge n)
\dto
  (\cT\mid\vct=\infty),
\end{equation}
a supercritical \GWt{} conditioned on non-extinction.

\begin{remark}\label{Roct}
When $\cT$ is supercritical,
the conditioned \GWt{} $(\cT\mid|\cT|=\infty)$ in \eqref{oct}
can be constructed by  a 2-type \GWp, somewhat similar to the construction
of $\hcT$ in \refS{ShGW}:
Let $q\=\P(|\cT|<\infty)<1$ be the extinction probability, which is given by
$\Phi(q)=q$.
Consider a \GWp{} $\oct$
with individuals of two types, \emph{mortal} and
\emph{immortal}, where a mortal gets only mortal children while an immortal
may get both mortal and immortal children. The numbers $\xi'$ of mortal
and $\xi''$ of immortal children are described by the \pgf{s}
\begin{equation}\label{phim}
\E x^{\xi'}y^{\xi''} =\Phim(x)\=\Phi_q(x)=\Phi(qx)/q
\end{equation}
 for a mortal  
and
\begin{equation}\label{phiim}
\E x^{\xi'}y^{\xi''} =\Phiim(x)\=\frac{\Phi(qx+(1-q)y)-\Phi(qx)}{1-q}
\end{equation}
for an immortal (with the children coming in random order).
Note that the subtree started  by a mortal is subcritical
(since $\Phim'(1)=\Phi'(q)<1$, \cf{} \eqref{lep3}), and thus \as{}
finite, while every immortal has at least one immortal child 
(since $\Phiim(x,0)=0$)
and thus the
subtree started by an immortal is infinite.
It is easily verified that $\cT$ conditioned on non-extinction equals this
random tree $\oct$ started with an immortal, while 
$\cT$ conditioned on extinction equals $\overline\cT$ started with a mortal.
(See \citetq{Section I.12}{AN}, where this is stated in a somewhat
different form.)

One important difference from $\hcT$ is that $\oct$ does not have a single
spine; started with an immortal it has \as{} an uncountable number of
infinite paths from the root.

Note that $\hcT$ in the critical case
can be seen as a limit case of this construction. If we let
$q\upto1$, which requires that we really consider a sequence of different
distributions with generating functions $\Phi\nn(t)\to\Phi(t)$, then taking the
limits in \eqref{phim}--\eqref{phiim} gives for the limiting critical
distribution the offspring generating functions $\Phim(x)=\Phi(x)$ and
$\Phiim(x,y)=y\Phi'(x)$, which indeed are the generating functions for the
offspring distributions in \refS{ShGW} in the critical case (with
mortal = normal and immortal = special), since $\E
x^{\hxi-1}y=y\Phi'(x)=\Phiim(x,y)$ by \eqref{pgfhxi}.
\end{remark}

\subsection{Conditioning on $H(\cT)\ge n$.}\label{Scondh}

To condition on the height $H(\cT)$ being at least $n$ is the same as
conditioning on $\elz_n(\cT)>0$, 
\ie, that the \GWp{} survives for at least
$n$ generations.

If $\E\xi>1$,  \ie, in the supercritical case, 
the events $\elz_n(\cT)>0$
decrease to $\vct=\infty$. Consequently,
\begin{equation}
(\cT\mid H(\cT)\ge n)
=  (\cT\mid\elz_n(\cT)>0)
\dto
  (\cT\mid\vct=\infty),
\end{equation}
exactly as when conditioning on $\vct\ge n$ in \eqref{oct}.
By \refR{Roct}, the limit equals $\oct$, started with an immortal.

In the subcritical and critical cases, the following result, proved by
\citet{Kesten} (at least for $\E\xi=1$, see also \citet{AldousPitman}),
shows convergence to 
the size-biased \GWt{} $\ctx$ in \refR{Rsizebias}.

\begin{theorem}\label{Tsb}
  Suppose that $\mu\=\E\xi\le1$. Then, as \ntoo, 
\begin{equation}
(\cT\mid H(\cT)\ge n)=
  (\cT\mid\elz_n(\cT)>0)
\dto
  \ctx.
\end{equation}
\end{theorem}

\begin{proof}
  Let $r_n\=\P(\elz_n(\cT)>0)$, the probability of survival for at least $n$
  generations. Then $r_n\to0$ as \ntoo.
Fix $\ell>0$ and a tree $T$ with height $\ell$. Conditioned on $\cT\nnx\ell=
T$, the remainder of the tree consists of $\elz_\ell(T)$ independent branches,
each distributed as $\cT$, and thus, for $n>\ell$,
\begin{equation}\label{rr1}
  \begin{split}
	\P(\cT\nnx\ell=T\mid H(\cT)\ge n)
&=
\frac{\P(\cT\nnx\ell=T \text{ and } H(\cT)\ge n)}{\P(H(\cT)\ge n)}
\\&
=
\frac{\P(\cT\nnx\ell=T)\bigpar{1-(1-r_{n-\ell})^{\elz_\ell(T)}}}
 {\P(H(\cT)\ge n)}.
  \end{split}
\end{equation}
Let $\stf\nnx\ell$ be the set of finite trees of height $\ell$.
Summing \eqref{rr1} over $T\in\stf\nnx\ell$ yields 1, and thus
\begin{equation}
\P(H(\cT)\ge n)=
\sum_{T\in\stf\nnx\ell}
{\P(\cT\nnx\ell=T)\bigpar{1-(1-r_{n-\ell})^{\elz_\ell(T)}}}
\end{equation}
Dividing by $r_{n-\ell}$, and noting that for any $N\ge1$, 
$\bigpar{1-(1-r)^N}/r\upto N$ as $r\downto0$, we find by monotone convergence
\begin{equation}
  \begin{split}
\frac{\P(H(\cT)\ge n)}{r_{n-\ell}}
&=
\sum_{T\in\stf\nnx\ell}
\P(\cT\nnx\ell=T)\frac{{1-(1-r_{n-\ell})^{\elz_\ell(T)}}}{r_{n-\ell}}
\\&
\to
\sum_{T\in\stf\nnx\ell}
\P(\cT\nnx\ell=T)\elz_\ell(T)
=\E\elz_\ell(\cT)=\mu^\ell.
  \end{split}
\end{equation}
Hence, by \eqref{rr1} and \eqref{c3},
\begin{equation}
  \begin{split}
	\P(\cT\nnx\ell=T\mid H(\cT)\ge n)
&\sim
\frac{\P(\cT\nnx\ell=T){\elz_\ell(T)}r_{n-\ell}}{\P(H(\cT)\ge n)}
\\&
\to
\frac{\P(\cT\nnx\ell=T){\elz_\ell(T)}}{\mu^\ell}
=\P(\ctxx\ell=T).
  \end{split}
\end{equation}
Thus, $(\cT\mid H(\cT)\ge n)\nnx\ell\dto\ctxx\ell$, and
the result follows by \eqref{jobl}.
\end{proof}

Note that if $\E\xi=1$, then $\ctx=\hcT$, see \refR{Rsizebias},
so the limits in Theorems
\refand{Tmain}{Tsb} of $\cT$ conditioned on $\vct=n$ and $H(\cT)\ge n$ have
the same limit. 
However, in the subcritical case $\E\xi<1$, $\ctx\neq\hcT$; moreover, $\ctx$
differs also from the limit in \refT{Tmain}, which is $\hcT$ for a
conjugated distribution, and the same is true in the supercritical case.
Hence, as remarked by \citet{Kennedy}, conditioning on 
$\vct=n$ and $H(\cT)\ge n$ give similar results (in the sense that the
limits as \ntoo{} are the same) in the critical case,
but quite different results in the subcritical and supercritical cases.
Similarly, conditioning on 
$\vct\ge n$ and $H(\cT)\ge n$ give quite different results in the
subcritical case.  \citet{AldousPitman} remarks that the two different
limits as \ntoo{} both can be intuitively interpreted as ``$\cT$ conditioned
on being infinite'', which shows that one has to be careful with such
interpretations.

\begin{ack}
This research was started during a visit to NORDITA, Stockholm, 
during the  program \emph{Random Geometry and Applications}, 2010.
I thank the participants, in particular 
Zdzis{\l}aw  Burda,
Bergfinnur Durhuus, 
Thordur Jonsson
and  Sigur\dh ur Stef\'ansson, for stimulating discussions. 
\end{ack}

\newcommand\AAP{\emph{Adv. Appl. Probab.} } 
\newcommand\JAP{\emph{J. Appl. Probab.} }
\newcommand\JAMS{\emph{J. \AMS} }
\newcommand\MAMS{\emph{Memoirs \AMS} }
\newcommand\PAMS{\emph{Proc. \AMS} }
\newcommand\TAMS{\emph{Trans. \AMS} }
\newcommand\AnnMS{\emph{Ann. Math. Statist.} }
\newcommand\AnnPr{\emph{Ann. Probab.} }
\newcommand\CPC{\emph{Combin. Probab. Comput.} }
\newcommand\JMAA{\emph{J. Math. Anal. Appl.} }
\newcommand\RSA{\emph{Random Struct. Alg.} }
\newcommand\ZW{\emph{Z. Wahrsch. Verw. Gebiete} }
\newcommand\DMTCS{\jour{Discr. Math. Theor. Comput. Sci.} }

\newcommand\AMS{Amer. Math. Soc.}
\newcommand\Springer{Springer-Verlag}
\newcommand\Wiley{Wiley}

\newcommand\vol{\textbf}
\newcommand\jour{\emph}
\newcommand\book{\emph}
\newcommand\inbook{\emph}
\def\no#1#2,{\unskip#2, no. #1,} 
\newcommand\toappear{\unskip, to appear}

\newcommand\urlsvante{\url{http://www.math.uu.se/~svante/papers/}}
\newcommand\arxiv[1]{\url{arXiv:#1.}}
\newcommand\arXiv{\arxiv}


\begin{thebibliography}{99}

\bibitem[Addario-Berry, Devroye and Janson(2010+)]{SJ250}
L. Addario-Berry, L. Devroye \& S. Janson,
Sub-Gaussian tail bounds for the width and height of conditioned
Galton--Watson trees.
Preprint, 2010.
\arXiv{1011.4121}


\bibitem[Aldous(1991)]{AldousFringe} 
D. Aldous,
Asymptotic fringe distributions for general families of random trees.
\emph{Ann. Appl. Probab.} \vol1 (1991), no. 2, 228--266.

\bibitem[Aldous(1991)]{AldousI} 
D. Aldous,
The continuum random tree I.  
\AnnPr  \vol{19} \no1  (1991),  1--28.


\bibitem[Aldous(1991)]{AldousII} 
D. Aldous, 
The continuum random tree II: an overview.
\emph{Stochastic Analysis (Durham, 1990)}, 23--70, 
London Math. Soc. Lecture Note Ser. 167, Cambridge Univ. Press, 
Cambridge, 1991. 

\bibitem[Aldous(1993)]{AldousIII} 
D. Aldous, 
The continuum random tree III.
\AnnPr \vol{21} \no1 (1993), 248--289.

\bibitem[Aldous and Pitman(1998)]{AldousPitman}
D. Aldous \& J. Pitman, 
Tree-valued Markov chains derived from Galton--Watson processes. 
\emph{Ann. Inst. H. Poincar\'e Probab. Statist.} \vol{34} (1998), no. 5,
637--686. 

\bibitem[Arratia, Barbour and Tavar\'e(2003)]{ABT}
R. Arratia, A.~D. Barbour \& S. Tavar\'e, 
\emph{Logarithmic Combinatorial Structures: a Probabilistic Approach}, 
EMS, Z\"urich, 2003.  


\bibitem[Athreya and Ney(1972)]{AN}
K. B. Athreya \& P. E. Ney,
\book{Branching Processes}.
Springer-Verlag, Berlin, 1972.

\bibitem[Barlow and Kumagai(2006)]{BarlowKumagai}
M. T. Barlow \& T. Kumagai, 
Random walk on the incipient infinite cluster on trees. 
\emph{Illinois J. Math.} \vol{50} (2006), no. 1--4, 33--65.

\bibitem{frobenius}
D. Beihoffer, J. Hendry, A. Nijenhuis \& S. Wagon,
Faster algorithms for Frobenius numbers.
\emph{Electron. J. Combin.} \vol{12} (2005), R27.

\bibitem[Bennies and Kersting(2000)]{BenniesK}
J. Bennies \& G. Kersting, 
A random walk approach to Galton--Watson trees. 
\emph{J. Theoret. Probab.} \vol{13} (2000), no. 3, 777--803. 

\bibitem[Bernikovich and Pavlov(2011)]{Pavlov:unun}
E. S. Bernikovich \&  Yu. L. Pavlov,
On the maximum size of a tree in a random unlabelled unrooted forest. 
\emph{Diskret. Mat.} \vol{23} (2011), no. 1, 3--20
(Russian). 
English transl.: 
\emph{Discrete Math. Appl.} \vol{21} (2011), no. 1, 1--21.

\bibitem[Bialas and Burda(1996)]{BialasB}
P. Bialas \& Z. Burda,
Phase transition in fluctuating branched geometry.
\emph{Physics Letters B} \vol{384} (1996), 75--80.

\bibitem[Bialas, Burda and Johnston(1997)]{BialasetalNuPh97}
P. Bialas, Z. Burda \& D. Johnston,
Condensation in the backgammon model.
\emph{Nuclear Physics} \vol{493} (1997), 505--516.

\bibitem[Billingsley(1968)]{Billingsley}
P. Billingsley,
\book{Convergence of Probability Measures}.
Wiley, New York, 1968.

\bibitem[Bingham, Goldie and Teugels(1987)]{BinghamGoldieTeugels}
N. H.\ Bingham, C. M.\ Goldie \& J. L.~Teugels,
\book{Regular Variation}.
Cambridge Univ. Press, Cambridge, 1987.


\bibitem[Borchardt(1860)]{Borchardt}
C. W. Borchardt,
Ueber eine der Interpolation entsprechende Darstellung der
Eliminations-Resultante.
\emph{J. reine und angewandte Mathematik} \vol{57} (1860), 111--121.

\bibitem[Borel(1942)]{Borel}
\'E. Borel, 
Sur l'emploi du th\'eor\`eme de Bernoulli pour faciliter le calcul
d'une infinit\'e de coefficients. 
Application au probl\`eme de l'attente \`a un guichet.  
\emph{C. R. Acad. Sci. Paris} \vol{214} (1942), 452--456.

\bibitem[Boyd(1971)]{Boyd}
A. V. Boyd,
Formal power series and the total progeny in a branching process.
\emph{J. Math. Anal. Appl.} \vol{34} (1971), 565--566. 

\bibitem[Britikov(1988)]{Britikov}
V. E. Britikov, 
Asymptotic number of forests from unrooted trees.  
\emph{Mat. Zametki} \vol{43} (1988), no. 5, 672--684, 703 (Russian). 
English transl.: 
\emph{Math. Notes} \vol{43} (1988), no. 5--6, 387--394. 

\bibitem[Carr, Goh and Schmutz(1994)]{CarrGS}
R. Carr,  W. M. Y. Goh \& E. Schmutz,
The maximum degree in a random tree and related problems.
\RSA \vol5 (1994), no. 1, 13--24.

\bibitem[Cayley(1889)]{Cayley}
A. Cayley,
A theorem on trees.
\emph{Quart. J. Math.} \vol{23} (1889), 376--378.

\bibitem[Chassaing and Durhuus(2006)]{ChDur} 
P. Chassaing \& B. Durhuus,
Local limit of labeled trees and expected volume growth in a random
quadrangulation. 
\emph{Ann. Probab.} \vol{34} (2006), no. 3, 879--917.

\bibitem[Chassaing and Louchard(2002)]{CL} 
P. Chassaing \& G. Louchard,
Phase transition for parking blocks, Brownian excursion and coalescence.
\RSA \vol{21} \no1 (2002),  76--119.

\bibitem[Chassaing, Marckert and Yor(2000)]{ChMY}
P. Chassaing, J.-F. Marckert \& M. Yor,
The height and width of simple trees.  
\inbook{Mathematics and Computer Science (Versailles, 2000)},  17--30, 
Trends Math., Birkh\"auser, Basel, 2000.

\bibitem{Corless}
R. M. Corless, G. H. Gonnet, D. E. Hare, D. J. Jeffrey \& D. E. Knuth,
On the Lambert $W$ function. 
\emph{Adv. Comput. Math.} \vol5 (1996), no. 4, 329--359.

\bibitem[\Cramer(1938)]{Cramer}
H. \Cramer, 
Sur un noveau th\'eor\`eme-limite de la th\'eorie des probabilit\'es.
\emph{Les sommes et les fonctions de variables al\'eatoires},
{Actualit\'es Scientifiques et Industrielles} 736, Hermann, Paris, 1938, 
pp. 5--23.

\bibitem[Croydon(2008)]{Croydon}
D. Croydon, 
Convergence of simple random walks on random discrete trees to Brownian
motion on the continuum random tree. 
\emph{Ann. Inst. H. Poincar\'e Probab. Statist.}
\vol{44} (2008), no. 6, 
987--1019. 

\bibitem[Croydon(2010)]{Croydon2010}
D. Croydon, 
Scaling limits for simple random walks on random ordered graph
trees. 
\AAP \vol{42} (2010), no. 2, 528--558. 

\bibitem[Croydon and Kumagai(2008)]{CroydonKumagai}
D. Croydon \& T. Kumagai,
Random walks on Galton--Watson trees with infinite variance offspring
distribution conditioned to survive.  
\emph{Electron. J. Probab.} \vol{13} (2008), no. 51, 1419--1441. 

\bibitem{DemboZ}
A. Dembo \& O. Zeitouni, 
\book{Large Deviations Techniques and Applications.} 
2nd ed., Springer, New York, 1998.

\bibitem[Devroye(1998)]{Devroye}
L. Devroye,
Branching processes and their applications in the analysis of tree
structures and tree algorithms.
\inbook{Probabilistic Methods for Algorithmic Discrete Mathematics},
eds. M. Habib, C. McDiarmid, J. Ramirez and B. Reed, 
Springer, Berlin, 1998,
pp. 249--314.

\bibitem[Drmota(2009)]{Drmota}
M. Drmota,
\book{Random Trees},
Springer, Vienna, 2009.

\bibitem[Duquesne(2003)]{Duquesne}
T. Duquesne, 
A limit theorem for the contour process of conditioned Galton--Watson trees.
\emph{Ann. Probab.} \vol{31} (2003), no. 2, 996--1027. 

\bibitem[Durhuus, Jonsson and Wheater(2007)]{DurhuusJW}
B. Durhuus, T. Jonsson \& J.~F. Wheater, 
The spectral dimension of generic trees.  
\emph{J.~Stat.~Phys.} \textbf{128} (2007),
1237--1260.

\bibitem[Dwass(1969)]{Dwass}
M. Dwass, 
The total progeny in a branching process and a related random walk.
\emph{J. Appl. Probab.} \textbf{6} (1969), 682--686.

\bibitem[Eggenberger and \Polya(1923)]{EggPol}
F. Eggenberger \& G. \Polya,
\"Uber die Statistik verketteter Vorg\"ange.
\jour{Zeitschrift Angew. Math. Mech.}
\vol3 (1923), 279--289.

\bibitem[Feller(1957)]{FellerI}
W. Feller, 
\emph{An Introduction to Probability Theory and its Applications,
Volume I}, 2nd ed., Wiley, New York, 1957.

\bibitem[Feller(1971)]{FellerII}
W. Feller, 
\emph{An Introduction to Probability Theory and its Applications,
Volume II}, 2nd ed., Wiley, New York, 1971.

\bibitem[Flajolet and Sedgewick(2009)]{Flajolet} 
P.~Flajolet \& R.~Sedgewick, 
\emph{Analytic Combinatorics}.
Cambridge Univ. Press, Cambridge, UK, 2009.

\bibitem[Franz and Ritort(1995)]{BackII}
S. Franz \& F. Ritort, 
Dynamical solution of a model without energy barriers. 
\emph{Europhysics Letters} \vol{31} (1995), 507--512

\bibitem[Franz and Ritort(1996)]{BackIII}
S. Franz \& F. Ritort, 
Glassy mean-field dynamics of the backgammon model.
\emph{J. Stat. Phys.}  
\vol{85} (1996), 131--150. 

\bibitem[Fujii and Kumagai(2008)]{FujiiKumagai}
I. Fujii \& T. Kumagai, 
Heat kernel estimates on the incipient infinite cluster for critical
branching processes.
\emph{Proceedings of German--Japanese Symposium in Kyoto 2006}, 
RIMS K\^oky\^uroku Bessatsu B6 (2008), pp. 8--95.

\bibitem[Geiger(1999)]{Geiger}
J. Geiger, 
Elementary new proofs of classical limit theorems for Galton--Watson
processes. 
\emph{J. Appl. Probab.} \vol{36} (1999), no. 2, 301--309.

\bibitem[Geiger and Kauffmann(2004)]{GeigerK}
J. Geiger \& L. Kauffmann, 
The shape of large Galton--Watson trees with possibly infinite
variance. 
\RSA{} \vol{25} (2004), no. 3, 311--335.

\bibitem[Gnedenko and Kolmogorov(1949)]{GneKol}
B. V. Gnedenko \&  A. N. Kolmogorov,
\emph{Limit Distributions for Sums of Independent Random Variables}.
Gosudarstv. Izdat. Tehn.-Teor. Lit., Moscow--Leningrad, 1949
(Russian). 
English transl.:
Addison-Wesley, Cambridge, Mass., 1954. 

\bibitem[Grimmett(1980)]{Grimmett80}
G. R. Grimmett,
Random labelled trees and their branching networks.
\emph{J. Austral. Math. Soc. Ser. A} \vol{30} (1980/81), no. 2, 229--237.

\bibitem[Grimmett(2006)]{Grimmett:RCM}
G. R. Grimmett,
\book{The Random-Cluster Model}, 
Springer, Berlin, 2006.

\bibitem[Gut(2005)]{Gut}
A. Gut,
\emph{Probability: A Graduate Course}.
Springer, New York, 2005.

\bibitem{HLP}
G. H. Hardy, J. E. Littlewood \& G. P\'olya,
\book{Inequalities}. 
2nd ed., Cambridge, at the University Press, 
1952. 

\bibitem[Harris(1960)]{Harris}
T. E. Harris,
A lower bound for the critical probability in a certain percolation process. 
\emph{Proc. Cambridge Philos. Soc.} \vol{56} (1960), 13--20. 

\bibitem[Holst(1979)]{Holst:conditional}
L. Holst, 
Two conditional limit theorems with applications. 
\emph{Ann. Statist.} {\bf7} (1979), no. 3, 
551--557.  

\bibitem[Holst(1979)]{Holst:urn}
L. Holst, 
A unified approach to limit theorems for urn models.
\emph{J. Appl. Probab.} {\bf16} (1979),
154--162.  


\bibitem[Ibragimov and Linnik(1965)]{IbragimovLinnik}
I. A. Ibragimov \&  Yu. V. Linnik,
\emph{Independent and Stationary Sequences of Random Variables}.
Nauka, Moscow, 1965 (Russian).
English transl.:
Wolters-Noordhoff Publishing, Groningen, 1971. 

\bibitem[Janson(2001)]{SJ132}
S. Janson,
Moment convergence in conditional limit theorems.
\emph{J. Appl. Probab.} \vol{38} (2001), no. 2, 421--437.

\bibitem[Janson(2001)]{SJ133}
S. Janson,
Asymptotic distribution for the cost of linear probing hashing. 
\RSA
\vol{19} (2001), no. 3--4, 438--471.

\bibitem[Janson(2003)]{SJ136}
S. Janson,
Cycles and unicyclic components in random graphs. 
\CPC 
\vol{12} (2003), 27--52. 

\bibitem[Janson(2004)]{SJ154}
S. Janson,
Functional limit theorems for multitype branching processes
and generalized P\'olya urns.
\emph{Stochastic Process. Appl.} \vol{110} (2004), no. 2,  177--245.

\bibitem[Janson(2006)]{SJ167}
S. Janson,
Random cutting and records in deterministic and random trees. 
\RSA \vol{29}  (2006), no. 2, 139--179.

\bibitem[Janson(2006)]{SJ175} 
S. Janson,
Rounding of continuous random variables and oscillatory asymptotics. 
\emph{Ann. Probab.} \vol{34} \no5 (2006),
1807--1826.

\bibitem[Janson(2008)]{SJ174}
S. Janson,
On the asymptotic joint distribution of height and width in random trees,
\emph{Studia Sci. Math. Hungar.}
\vol{45} (2008), no. 4,
451--467. 

\bibitem[Janson(2009)]{SJN6}
S. Janson,
Probability asymptotics: notes on notation.
Institute Mittag-Leffler Report 12, 2009 spring.
\arxiv{1108.3924}

\bibitem[Janson(2011)]{SJN12}
S. Janson,
Stable distributions.
Unpublished notes, 2011.
\arxiv{1112.0220}


\bibitem[Janson, Jonsson and Stef\'ansson(2011)]{SJ259} 
S. Janson, T. Jonsson \&  S.~\"O.~Stef\'ansson, 
Random trees with superexponential branching weights.
\emph{J. Phys. A: Math. Theor.} \vol{44} (2011), 485002.

\bibitem{JLR}
S.~Janson, T.~{\L}uczak \& A.~Ruci{\' n}ski, 
\emph{Random Graphs}.
Wiley, New York, 2000.

\bibitem{JohnsonKotz}
N. L. Johnson \& S. Kotz,
\emph{Urn Models and their Application}.
Wiley, New York, 
1977. 

\bibitem[Jonsson and Stef\'ansson(2011)]{sdf} 
T. Jonsson \&  S.~\"O.~Stef\'ansson, 
Condensation in nongeneric trees.
\emph{J. Stat. Phys.} 
\textbf{142} (2011), no. 2,
277--313.

\bibitem[Kallenberg(1983)]{Kallenberg:RM}
O. Kallenberg,
\emph{Random Measures}.
Akademie-Verlag, Berlin, 1983.

\bibitem[Kallenberg(2002)]{Kallenberg}
O. Kallenberg,
\book{Foundations of Modern Probability.}
2nd ed., Springer, New York, 2002. 

\bibitem[Kazimirov(2002)]{Kazimirov}
N. I. Kazimirov,
On some conditions for absence of a giant component in
the generalized allocation scheme.
\emph{Diskret. Mat.} \vol{14} (2002), no. 2, 107--118
(Russian). 
English transl.: 
\emph{Discrete Math. Appl.} \vol{12} (2002), no. 3, 291--302.   

\bibitem[Kazimirov(2003)]{Kazimirov:permutation}
N. I. Kazimirov,
Emergence of a giant component in a random
permutation with a given number of cycles.
\emph{Diskret. Mat.} \vol{15} (2003), no. 3, 145--159
(Russian). 
English transl.: 
\emph{Discrete Math. Appl.} \vol{13} (2003), no. 5, 523--535. 

\bibitem[Kazimirov and Pavlov(2000)]{KazPavlov}
N. I. Kazimirov \& Yu. L. Pavlov,
A remark on the Galton--Watson forests. 
\emph{Diskret. Mat.} \vol{12} (2000), no. 1, 47--59
(Russian). 
English transl.: 
\emph{Discrete Math. Appl.} \vol{10} (2000), no. 1, 49--62.

\bibitem[Kennedy(1975)]{Kennedy}
D. P. Kennedy,
The Galton--Watson process conditioned on the total progeny.
\JAP \vol{12}  (1975),  800--806. 

\bibitem[Kesten(1986)]{Kesten}
H. Kesten,  
Subdiffusive behavior of random walk on a random cluster. 
\emph{Ann. Inst. H. Poincar\'e Probab. Statist.} \vol{22} (1986), no. 4,
425--487. 

\bibitem[Knuth(1998)]{KnuthIII} 
D. E. Knuth, 
\emph{The Art of Computer Programming. Vol. 3: Sorting
 and Searching}. 
2nd ed., Addison-Wesley,
Reading, Mass., 1998.

\bibitem[Kolchin(1984)]{Kolchin} 
V. F. Kolchin,  
\emph{Random Mappings}.
Nauka, Moscow, 1984 (Russian).
English transl.:
Optimization Software, New York, 1986.

\bibitem[Kolchin, {Sevast\cprime yanov} and Chistyakov(1976)]{KolchinSCh}
V. F. Kolchin, B. A. Sevast\cprime yanov \& V. P. Chistyakov, 
\emph{Random Allocations}.
Nauka, Moscow, 1976 (Russian).
English transl.:
Winston, Washington, D.C., 1978.

\bibitem[Kurtz, Lyons, Pemantle and Peres(1997)]{KLPP}
T. Kurtz, R. Lyons, R. Pemantle \& Y. Peres, 
A conceptual proof of the Kesten--Stigum Theorem for multi-type branching
processes. 
\inbook{Classical and Modern Branching Processes (Minneapolis, MN, 1994)}, 
IMA Vol. Math. Appl., 84, 
Springer, New York, 1997,
pp. 181--185.

\bibitem{Lagrange}
J.-L.  Lagrange, 
Nouvelle m{\'e}thode pour r{\'e}soudre les {\'e}quations litt{\'e}rales par
le moyen des s{\'e}ries.
\emph{M{\'e}moires de l'Acad{\'e}mie royale des Sciences et Belles-Lettres
de Berlin}, 
\vol{XXIV} (1770), 5--73.

\bibitem[Le Gall(2005)]{LeGall}
J.-F. Le Gall, 
Random trees and applications. 
\jour{Probab. Surveys} \vol{2} (2005), 245--311.

\bibitem[Le Gall(2006)]{LeGallreal}
J.-F. Le Gall,
Random real trees.
\emph{Ann. Fac. Sci. Toulouse Math.} (6) \vol{15} (2006), no. 1, 35--62. 

\bibitem[Leadbetter,  Lindgren and  Rootz{\'e}n(1983)]{LLR}
M. R. Leadbetter, G. Lindgren \& H. Rootz{\'e}n,
\emph{Extremes and Related Properties of Random Sequences and Processes}.
Springer-Verlag, New York, 1983.

\bibitem[\L uczak and Pittel(1992)]{LP}
T. \L uczak \& B. Pittel,
Components of random forests. 
\emph{Combin. Probab. Comput.} \vol1 (1992), no. 1, 35--52.

\bibitem[Lyons, Pemantle and Peres(1995)]{LPP}
R. Lyons, R. Pemantle \& Y. Peres,
Conceptual proofs of $L\log L$ criteria for mean behavior of branching
processes.
\emph{Ann. Probab.} \textbf{23}  (1995), no.\ 3, 1125--1138. 


\bibitem[Meir and Moon(1978)]{MM78}
A.~Meir \& J. W.~Moon, 
On the altitude of nodes in random trees.  
\emph{Canad.\ J.\ Math.}, \textbf{30} (1978), 997--1015.

\bibitem[Meir and Moon(1990)]{MM90}
A.~Meir \& J. W.~Moon, 
On the maximum out-degree in random trees. 
\emph{Australas. J. Combin.} \vol2 (1990), 147--156. 

\bibitem[Meir and Moon(1991)]{MM91}
A.~Meir \& J. W.~Moon, 
On nodes of large out-degree in random trees. 
\emph{Congr. Numer.} \vol{82} (1991), 3--13. 

\bibitem[Meir and Moon(1992)]{MM92}
A.~Meir \& J. W.~Moon, 
A note on trees with concentrated maximum degrees. 
\emph{Utilitas Math.} \textbf{42} (1992), 61--64.
Coorigendum:
\emph{Utilitas Math.} \vol{43} (1993), 253.

\bibitem[Minami(2005)]{Minami}
N. Minami,
On the number of vertices with a given degree in a Galton--Watson tree.
\AAP \vol{37} (2005), no. 1, 229--264. 

\bibitem[Moon(1968)]{Moon68}
J. W. Moon,
On the maximum degree in a random tree. 
\emph{Michigan Math. J.} \vol{15} (1968), 429--432. 

\bibitem[Neveu(1986)]{Neveu}
J. Neveu,
Arbres et processus de Galton--Watson. 
\emph{Ann. Inst. H. Poincar\'e Probab. Statist.} \vol{22} \no2 (1986),  
199--207.

\bibitem[Otter(1948)]{Otter:trees}
R. Otter, 
The number of trees.
\emph{Ann. of Math. (2)} \vol{49} (1948), 583--599.

\bibitem[Otter(1949)]{Otter}
R. Otter, 
The multiplicative process. 
\emph{Ann. Math. Statistics} \vol{20} (1949), 206--224.

\bibitem[Pavlov(1977)]{Pavlov:max}  
Yu. L. Pavlov,
The asymptotic distribution of maximum tree size in a
random forest.
\emph{Teor. Verojatnost. i Primenen.} 
{\bf22} (1977), no. 3, 523--533 (Russian). 
English transl.:
\emph{Th.  Probab. Appl.} {\bf22} (1977), no. 3, 509--520.

\bibitem[Pavlov(1995)]{Pavlov:limit}  
Yu. L. Pavlov,
The limit distributions of the maximum size of a tree in a random forest.
\emph{Diskret. Mat.} \vol7 (1995), no. 3, 19--32
 (Russian).
English transl.:
\emph{Discrete Math. Appl.} \vol5 (1995), no. 4, 301--315.

\bibitem[Pavlov(1996)]{Pavlov}  
Yu. L. Pavlov,
\emph{Random Forests}.
Karelian Centre Russian Acad. Sci., Petrozavodsk, 1996
(Russian).
English transl.:
VSP, Zeist, The Netherlands, 2000.

\bibitem[Pavlov(2005)]{Pavlov:unlabelled}  
Yu. L. Pavlov,
Limit theorems on sizes of trees in a random unlabelled forest.
\emph{Diskret. Mat.} \vol{17} (2005), no. 2, 70--86
(Russian). 
English transl.:  
\emph{Discrete Math. Appl.} \vol{15} (2005), no. 2, 153--170.

\bibitem[Pavlov and Loseva(2002)]{Pavlov:recursive}  
Yu. L. Pavlov \& E. A. Loseva,
Limit distributions of the maximum size of a tree in a random recursive
forest. 
\emph{Diskret. Mat.} \vol{14} (2002), no. 1, 60--74
(Russian). 
English transl.:  
\emph{Discrete Math. Appl.} \vol{12} (2002), no. 1, 45--59.

\bibitem[Pitman(1998)]{Pitman:enum}  
J. Pitman, 
Enumerations of trees and forests related to branching processes and
random walks. 
\emph{Microsurveys in Discrete Probability (Princeton, NJ, 1997)},
DIMACS Series in Discrete Mathematics and Theoretical Computer Science, 41, 
\AMS, Providence, RI, 1998, pp. 163--180.

\bibitem[Ritort(1995)]{BackI}
F. Ritort, 
Glassiness in a model without energy barriers. 
\emph{Physical Review Letters} \vol{75} (1995), 1190--1193.

\bibitem[Rudin(1970)]{Rudin}
W. Rudin,
\emph{Real and Complex Analysis}.
McGraw-Hill, London, 1970

\bibitem[Sagitov and Serra(2009)]{SagitovS}
S. Sagitov \& M. C.  Serra,
Multitype Bienaym\'e--Galton--Watson processes escaping extinction.
\AAP \vol{41} (2009), no. 1, 225--246. 

\bibitem{Stanley2}
R. P. Stanley,
\emph{Enumerative Combinatorics,
Volume 2}.
Cambridge Univ. Press, Cambridge, 1999.

\bibitem[Sylvester(1857)]{Sylvester}
J. J. Sylvester,
On the change of systems of independent variables,
\emph{Quart J. Math.} \vol1 (1857), 42--56.

\bibitem[\Takacs(1962)]{Takacs}
L. \Takacs,
A generalization of the ballot problem and its application in the theory of
queues. 
\emph{J. Amer. Statist. Assoc.} \vol{57} (1962), 327--337. 

\bibitem[\Takacs(1989)]{Takacs:ballots}
L. \Takacs,
Ballots, queues and random graphs. 
\JAP \vol{26} (1989), no. 1, 103--112. 

\bibitem[Tanner(1961)]{Tanner}
J.C. Tanner, 
A derivation of the Borel distribution. 
\emph{Biometrika} \vol{48} (1961), 222--224.

\bibitem[Wendel(1975)]{Wendel} 
J. G. Wendel,
Left-continuous random walk and the Lagrange expansion. 
\emph{Amer. Math. Monthly} {\bf 82} (1975), 494--499.

\bibitem{Wilf}
Herbert S. Wilf,
\emph{generatingfunctionology}.
2nd ed., Academic Press, 1994.
\end{thebibliography}
\end{document}